\documentclass[11pt]{amsart}
\usepackage[leqno]{amsmath}
\usepackage{amssymb}
\usepackage{hyperref}
\usepackage{enumerate}
\usepackage[normalem]{ulem}
\usepackage{color}
\usepackage{subfigure}
\usepackage{graphicx}
\usepackage{tabularx}
\usepackage[table]{xcolor}
\usepackage[protrusion=true,expansion=true]{microtype}
\usepackage[all]{xy}
\usepackage{bm}

\numberwithin{equation}{section}
\numberwithin{figure}{section}

\newtheorem{theorem}{Theorem}[section]
\newtheorem{remark}[theorem]{Remark}
\newtheorem{lemma}[theorem]{Lemma}
\newtheorem{proposition}[theorem]{Proposition}
\newtheorem{corollary}[theorem]{Corollary}

\newtheorem{definition}[theorem]{Definition}

\newtheorem{prop}[theorem]{Proposition}
\usepackage{comment}

\newcommand{\C}{\mathbb{C}}
\newcommand{\D}{\mathbb{D}}

\newcommand{\h}{\mathbb{H}}
\renewcommand{\H}{\h}
\newcommand{\N}{\mathbb{N}}

\newcommand{\R}{\mathbb{R}}

\newcommand{\HH}{\mathbb{H}}
\renewcommand{\P}{\mathbb{P}}

\newcommand{\ccwBCLE}{\BCLE^{\boldsymbol {\circlearrowleft}}}
\newcommand{\cwBCLE}{\BCLE^{\boldsymbol {\circlearrowright}}}

\newcommand{\Fh}{\mathfrak {h}}

\newcommand{\CF}{\mathcal {F}}

\newcommand{\CL}{\mathcal {L}}

\newcommand{\CS}{\mathcal {S}}

\newcommand{\strip} {\CS}
\newcommand{\eps}{\epsilon}

\newcommand{\bSLE}{{\rm bSLE}}
\newcommand{\SLE}{{\rm SLE}}
\newcommand{\CLE}{{\rm CLE}}
\newcommand{\BCLE}{{\mathrm{BCLE}}}

\newcommand{\wt}{\widetilde}

\newcommand{\ol}{\overline}
\newcommand{\ul}{\underline}

\begin{document}

\title{CLE percolations}

\author{Jason Miller}
\address{Statslab, Center of Mathematical Sciences, University of Cambridge, Wilberforce Road, Cambridge CB3 0WB, UK}
\email{jpmiller@statslab.cam.ac.uk}

\author{Scott Sheffield}

\address{
Department of Mathematics, MIT,
77 Massachusetts Avenue,
Cambridge, MA 02139, USA}
\email{sheffield@math.mit.edu}

\author{Wendelin Werner}

\address{Department of Mathematics, ETH Z\"urich, R\"amistr. 101, 8092 Z\"urich, Switzerland} 
\email{wendelin.werner@math.ethz.ch}


\begin{abstract}
Conformal loop ensembles are random collections of loops in a simply
connected domain, whose laws are characterized by a natural
conformal invariance property. The set of points not surrounded by any
loop is a canonical random connected fractal set --- a random and conformally
invariant analog of the Sierpinski carpet or gasket.

In the present paper, we derive a direct relationship between the conformal loop ensembles with simple loops ($\CLE_\kappa$ for $\kappa \in (8/3, 4)$, whose loops are Schramm's $\SLE_\kappa$-type curves) and the corresponding conformal loop ensembles with non-simple loops ($\CLE_{\kappa'}$ with $\kappa' := 16/\kappa \in (4, 6)$, whose loops are $\SLE_{\kappa'}$-type curves).  This correspondence is the continuum analog of the {\em Edwards-Sokal} coupling between the $q$-state Potts model and the associated FK random cluster model, and its generalization to non-integer~$q$.

Like its discrete analog, our continuum correspondence has two directions. First, we show that for each $\kappa \in (8/3,4)$, one can construct a variant of $\CLE_{\kappa}$ as follows: start with an instance of $\CLE_{\kappa'}$, then use a biased coin to independently color each $\CLE_{\kappa'}$ loop in one of two colors, and then consider the outer boundaries of the clusters of loops of a given color.  Second, we show how to interpret $\CLE_{\kappa'}$ loops as interfaces of a continuum analog of critical Bernoulli percolation within $\CLE_\kappa$ carpets --- this is the first construction of continuum percolation on a {\em fractal} planar domain. It extends and generalizes the continuum percolation on open domains defined by $\SLE_6$ and $\CLE_6$.
 
These constructions allow us to prove several conjectures made by the second author in \cite{she2009cle} and provide new and perhaps surprising interpretations of the relationship between conformal loop ensembles and the Gaussian free field. Along the way, we obtain new results about generalized $\SLE_\kappa (\rho)$ curves for $\rho < -2$, such as their decomposition into collections of $\SLE_\kappa$-type ``loops'' hanging off of $\SLE_{\kappa'}$-type ``trunks'', and vice-versa (exchanging $\kappa$ and $\kappa'$).  We also define a continuous family of natural $\CLE$ variants called boundary conformal loop ensembles (BCLEs) that share some (but not all) of the conformal symmetries that characterize $\CLE$s, and that should be scaling limits of critical models with special boundary conditions. We extend the $\CLE_\kappa$/$\CLE_{\kappa'}$ correspondence to a $\BCLE_\kappa$/$\BCLE_{\kappa'}$ correspondence that makes sense for the wider range $\kappa \in (2,4]$ and $\kappa' \in [4,8)$.

\end{abstract}

\maketitle

\setcounter{tocdepth}{1}

\vspace{-.4in}

\tableofcontents

\parindent 0 pt
\setlength{\parskip}{0.20cm plus1mm minus1mm}

\section{Introduction}
\label{sec:intro}

\subsection {CLE background}
Before describing the results and content of the present paper, let us first recall some facts about the conformal loop ensembles (CLEs), which will be 
our central object of study. CLEs are natural random collections of planar loops that possess conformal invariance 
properties, and that have been defined and studied in \cite{she2009cle,sw2012cle}. The set of points {\em not} surrounded by a $\CLE$ loop is a random closed connected fractal subset of the plane known as either a {\em $\CLE$ gasket} or {\em $\CLE$ carpet}, depending on whether the loops intersect each other. These sets are the 
natural random generalizations of deterministic self-similar fractals, with simple built-in conformal symmetries.
As we will see, these symmetries make it possible to perform analysis that is currently out of reach for most deterministic fractals. 
CLEs arise in a number of settings (for instance as conjectured or proven scaling limits of a number of lattice models, or as level-lines of random surfaces) and they are very closely and directly related to Schramm's SLE processes \cite{schramm2000sle} and to the Gaussian free field (GFF).

One can view a $\CLE$ as a random collection of loops defined in the closed unit disk (here, we consider loops modulo time-reparameterization; note in particular that a loop is not oriented). The law of this collection is invariant under any M\"obius transformation of the unit disk, which implies that $\CLE$ is well-defined in any simply connected domain $D \not= \C$ as the image of the $\CLE$ in the disk under any given conformal transformation from the unit disk into $D$. 
Each $\CLE$ comes in two closely-related versions: nested and non-nested. A non-nested $\CLE$ is a random collection of loops in the (closed) unit disk with the property that no loop surrounds another loop; once one has defined non-nested CLE, one may construct an instance of nested $\CLE$ using an iteration procedure.

\begin{figure}[ht!]
\begin{center}
\includegraphics[width=3.1in]{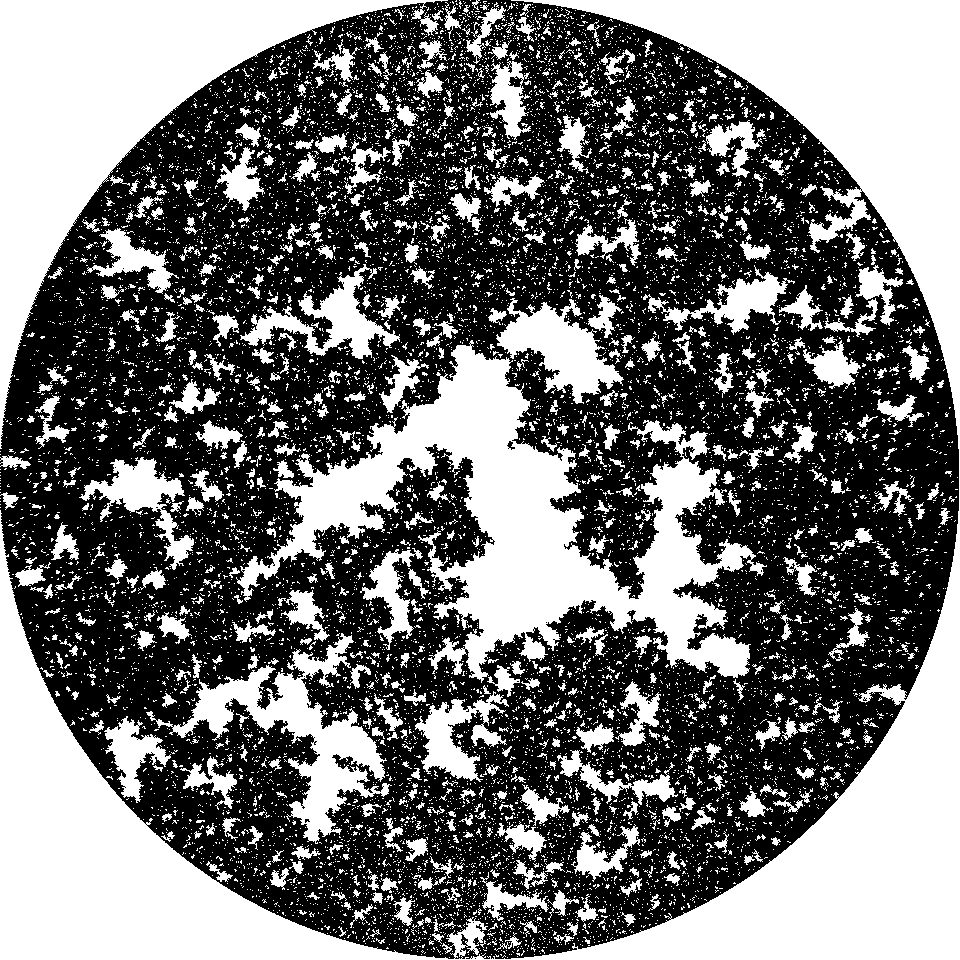}\hspace{0.03\textwidth}\includegraphics[width=3.1in]{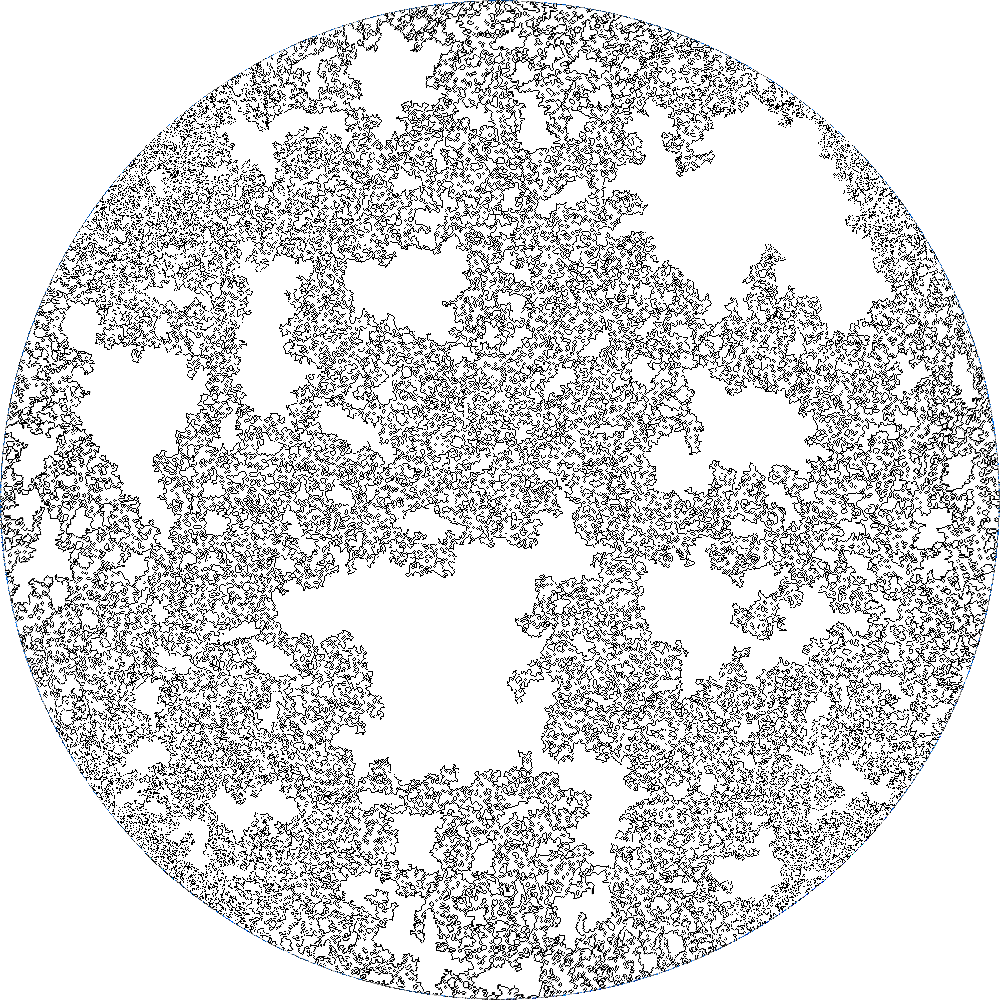}
\caption{\label{fig:simulations}{\bf Left:} Simulation of a $\CLE_4$ carpet (due to David B.\ Wilson); the carpet, in black, is a fractal set with zero Lebesgue measure. 
{\bf Right:} Simulation of a non-nested $\CLE_6$, the closure of the union of loops is the $\CLE_6$ gasket.}
\end{center}
\end{figure}

The family of CLEs is indexed by a parameter $\kappa \in (8/3, 8)$; this notation comes from the fact that the loops in a $\CLE_\kappa$ are loop variants of $\SLE_\kappa$. Recall that the Hausdorff dimension of an $\SLE_\kappa$ curve is almost surely $1+ (\kappa/8)$ \cite{rs2005basic,BEF_DIM}. The Hausdorff dimension of the $\CLE_\kappa$ carpets/gaskets is almost surely equal to $1 + (2 / \kappa) + (3 \kappa /32)$, as shown in \cite{ssw2009radii, nw2011carpets, msw2012cle_gasket}.  Like $\SLE_\kappa$ itself, $\CLE_\kappa$ has different properties depending on whether or not $\kappa \leq 4$ (see Figure~\ref{fig:simulations} for some simulations\footnote{Some of the simulations in this article are based on discrete models which are only conjecturally related to $\CLE$.}): 

\begin{itemize}
\item
When $\kappa \in (8/3, 4]$, each loop in the $\CLE_\kappa$ is a simple loop, and it does not intersect either the boundary of the disk or any other loop in the CLE. As mentioned above, if one removes the interiors of all the $\CLE$ loops, one is left with a random closed and connected fractal carpet reminiscent of the Sierpinski carpet. Note also that the knowledge of all of the loops of a (non-nested) simple $\CLE$ encapsulates exactly the same information as the knowledge of the carpet. These simple CLEs can be characterized by their conformal restriction axioms \cite{sw2012cle}, which explains why they arise in many different settings.

The present paper will not really directly deal with any discrete models, but it is nevertheless useful to recall that simple CLEs are the conjectural scaling limits of certain lattice models. The nested versions of the $\CLE_\kappa$ conjecturally correspond to the scaling limits of the so-called critical $O(n)$ models, which are natural models for random collections of loops within discrete lattices. 
The three $\CLE_\kappa$'s for $\kappa = 3$, $10/3$ and $4$ have been respectively conjectured to correspond to the scaling limits of critical Potts models for $q=2$, $3$ and $4$. The Potts model is a natural measure on colorings of the lattice using $q$ colors; to obtain loops, one considers the Potts models with monochromatic boundary conditions and then looks at the loops that form the inner boundaries of the outermost monochromatic cluster. It is worth emphasizing that except in the special case $q=2$ (which is the Ising model and corresponds also to the $O(n)$ model for $n=1$) for which the scaling limit of single interfaces is now well understood thanks the 
discrete analyticity features of the model (see \cite {cdhks2014ising} and \cite {smirnov2007ising,cs2012universality,ks2012loewner} or \cite {smirnov2010icm} for a survey), the Potts-CLE correspondence only describes {\em non-nested} $\CLE_\kappa$.  When $q = 3$, for example, the critical Potts model does {\em not} describe a collection of loops that conjecturally correspond to {\em nested} $\CLE_{\kappa}$ (since in this case there are three colors of clusters, and the law of the set of {\em all} cluster boundaries is expected to be more complicated). The present paper will however provide insight into this.

In the critical case $\kappa=4$, $\CLE_4$ is still a carpet but (roughly speaking) the probability that two big holes get very close does not decay in a power-law fashion anymore but much more slowly, which yields for instance the rather frequent presence of ``narrow bottlenecks'' between big holes.  This explains why, for the properties that we study in the present paper, $\CLE_4$ behaves somewhat differently from the $\CLE_\kappa$ with $\kappa<4$.

 \begin{figure}[ht!]
\begin{center}
\includegraphics [width=3in]{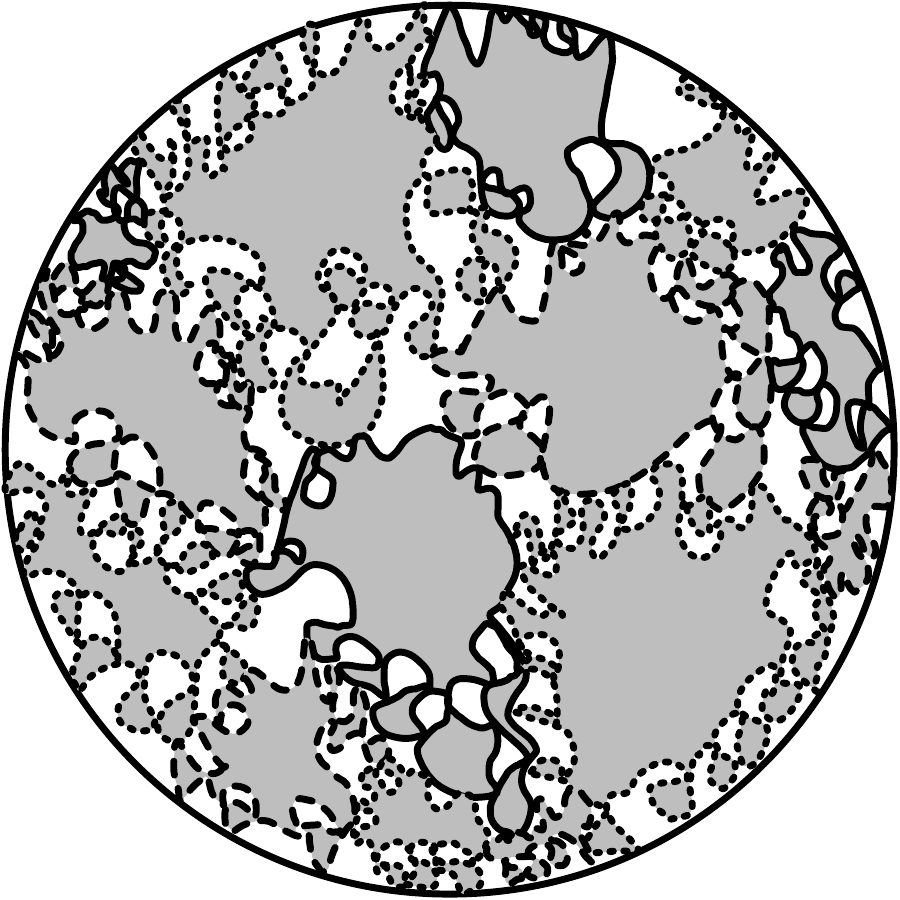}
\caption{Sketch of the non-nested larger loops in a $\CLE_{\kappa'}$ for $\kappa' \in (4,8)$.  Different loops have boundaries which are dashed/dotted/plain to differentiate them from each other in the illustration.}
\label {fig12}
\end{center}
\end{figure} 

\item
When $\kappa' \in (4, 8)$, the loops are not simple anymore (see Figure~\ref{fig12}), but because each loop is not ``self-crossing'' one can still define its interior and its exterior.  Also a loop can intersect the boundary of the disk and it can intersect other loops (but they cannot cross). The complement of the union of the interiors of all the loops is now a random gasket.  It turns out that in this case the information provided by the collection of outermost loops is richer than that provided by the gasket (viewed simply as a closed set).  This will follow from the results of \cite{msw2016randomness}. However, by a slight abuse of  terminology, we will often use ``CLE gasket'' also for the collection of outermost loops.

 Each non-simple conformal loop-ensemble is conjectured to be the scaling limit of a critical dependent percolation model called the critical FK$_q$-random cluster model for $q = q (\kappa') \in (0,4)$. 
 In this setting, the loops of the nested $\CLE_{\kappa'}$ will alternatively correspond to inner or outer interfaces of FK$_q$-clusters (that are respectively the outer and inner interfaces of the dual clusters for the dual FK$_q$-model).  This has been proved mathematically
 for $\kappa'=6$ (this corresponds to $q=1$ which is ordinary critical percolation, see \cite {smirnov2001percolation,cn2006percolation,smirnov2010icm}) and
 for the boundary touching loops of $\CLE_{16/3}$ that corresponds to the special case $q=2$ \cite{ks2015fkcle}. 
 It has also been proved for FK$_q$ models on random planar maps for all values of $q \in (0,4)$ in the so-called {\em peanosphere topology} \cite{she2011qginv,dms2014mating} for infinite volume surfaces.  See also \cite{gms2015cone_times,gs2015finite_volume,gs2015finite,quantum_spheres} for the corresponding results in the finite volume setting and \cite{gm2016topology} for a strengthening of this topology.
   
 It is natural in view of this FK framework to separate the loops of a nested $\CLE_{\kappa'}$ into the even and odd ones, depending on the parity of their nesting (all outermost loops would be odd ones, the next level ones would be even, and so on). If we work with ``free boundary conditions'', the $\CLE_{\kappa'}$ clusters will correspond to the gasket that is squeezed inside an odd loop and outside of all the even loops that it surrounds, and the ``wired boundary conditions'' clusters would be the complementary ones (the $\CLE_{\kappa'}$ clusters correspond to the  gasket that is squeezed inside an even loop and outside of all the even loops that it surrounds; the boundary of the domain would also count as an even loop here). 
\end{itemize}
 
The limiting cases $\kappa=8/3$ and $\kappa =8$ respectively correspond to the empty collection of loops and to the space-filling $\SLE_8$ loop (which has been shown to be the scaling limit of the Peano curve associated with the uniform spanning tree \cite{lsw2004lerw_ust}, which can be viewed as the critical  FK$_q$-model for $q=0$). 

Let us mention that while CLEs do provide a description of the conjectural full scaling limit of the discrete models that we have mentioned, which provides one motivation to study them, 
other possible descriptions of these scaling limits than via these interfaces do exist, via scaling limits of correlation functions or other observables and conformal fields 
(see for instance in the case of the Ising model \cite {cardy2010conformal,hs2013isingenergy,chi2015isingcorrelation,dtw2016ising, cgn2015ising} and the references therein).

\subsection{Overview of our CLE duality results}

The  main purpose of the present paper is to derive a direct
coupling between $\CLE_\kappa$ where $\kappa \in (8/3, 4)$ and $\CLE_{\kappa'}$ where $\kappa' = 16/\kappa \in (4,6)$. We will 
  also sometimes refer to this coupling as the $\CLE_\kappa$/CLE$_{\kappa'}$ duality. 
As we shall explain in the next section, the existence of this coupling
is not so surprising in view of the aforementioned facts that for three values of $\kappa$, the $\CLE_\kappa$ and $\CLE_{\kappa'}$ are (conjectured) scaling limits of critical discrete Potts 
models and FK$_q$-models  respectively (for $q=2$, $3$ and $4$), and of the existence of  a coupling (sometimes referred to as the Edwards-Sokal coupling) between the $q$-Potts models and the FK$_q$ models in the discrete setting
 (see \cite{grimmett2006cluster} for more background on the random cluster model and this coupling -- see \cite {fk1972fkmodel,es1988fk} for the original papers on the coupling). In some sense, we will derive a continuous analog and generalizations of this coupling that works for the continuum of values of $\kappa$. Several of our results had been conjectured in the $\CLE$ paper \cite {she2009cle}.
 
\subsubsection {From $\CLE_{\kappa'}$ to $\CLE_{\kappa}$}

Let us first describe how, starting from a $\CLE_{\kappa'}$, one can construct a $\CLE_\kappa$ or variants of $\CLE_\kappa$. 
  
The setup in this subsection is going to be the following. 
Fix any $\beta \in [-1,1]$, and consider a nested $\CLE_{\kappa'}$ for $\kappa' \in (4,8)$. 
For each $\CLE_{\kappa'}$ cluster, one tosses an independent coin to declare it to be colored in black or in white (with respective probability $(1-\beta)/2$ and $(1+ \beta)/2$).  Then, we will be interested in 
the clusters of clusters that one obtains by agglomerating all $\CLE_{\kappa'}$ clusters of the same color that touch each other. Consider for instance all clusters of black $\CLE_{\kappa'}$ clusters.  See Figure~\ref{fig::bcle_nested} for a sketch in the unit disk.

Let us describe here three type of results among those that we shall derive in the present paper.  All of them say in some ways
that the outer boundaries of such clusters of clusters form a variant of $\CLE_\kappa$.

\begin {itemize}
 \item 
If we start with the exact setup described above, we can consider the ``free boundary'' definition of clusters and focus on those clusters of black clusters that do touch the boundary. Then, the outer boundary of these clusters will be formed of unions of interfaces with clusters of white clusters that also touch the boundary. We will see that for any choice of $\beta$, this collection of boundary-touching interfaces will form a variant of $\CLE_\kappa$ that we define and call a boundary conformal loop ensemble $\BCLE_\kappa$ (see Section~\ref{subsubsec::BCLE} below and Figure~\ref{fig:perc_sim} -- for each $\kappa$, there will be a one-parameter family of such $\BCLE_\kappa$ corresponding to the choice of $\beta$). Note that this $\BCLE_\kappa$ cannot be exactly a $\CLE_\kappa$ because the clusters of clusters do touch the unit circle, while $\CLE_\kappa$ loops do not (and when $\kappa'\in [6,8)$ for which the previous 
statement is valid, $\CLE_{16/\kappa'}$ does not even exist). Note also that for this result, only the outermost $\CLE_{\kappa'}$ loops matter and that one could have started with a non-nested $\CLE_{\kappa'}$. 
The precise statements are described in Section~\ref{S7}, after all the definitions of these $\CLE$ variants have been properly laid out. 

 \begin{figure}[ht!]
\begin{center}
\includegraphics [width=3in]{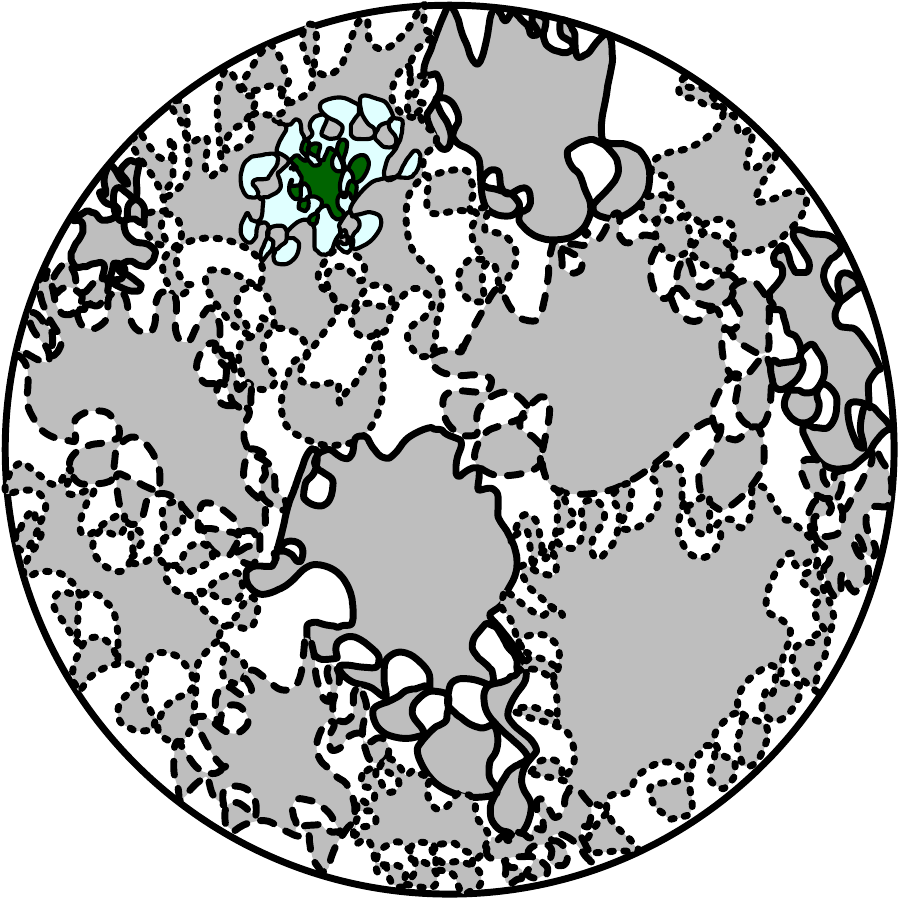}
\includegraphics [width=3in]{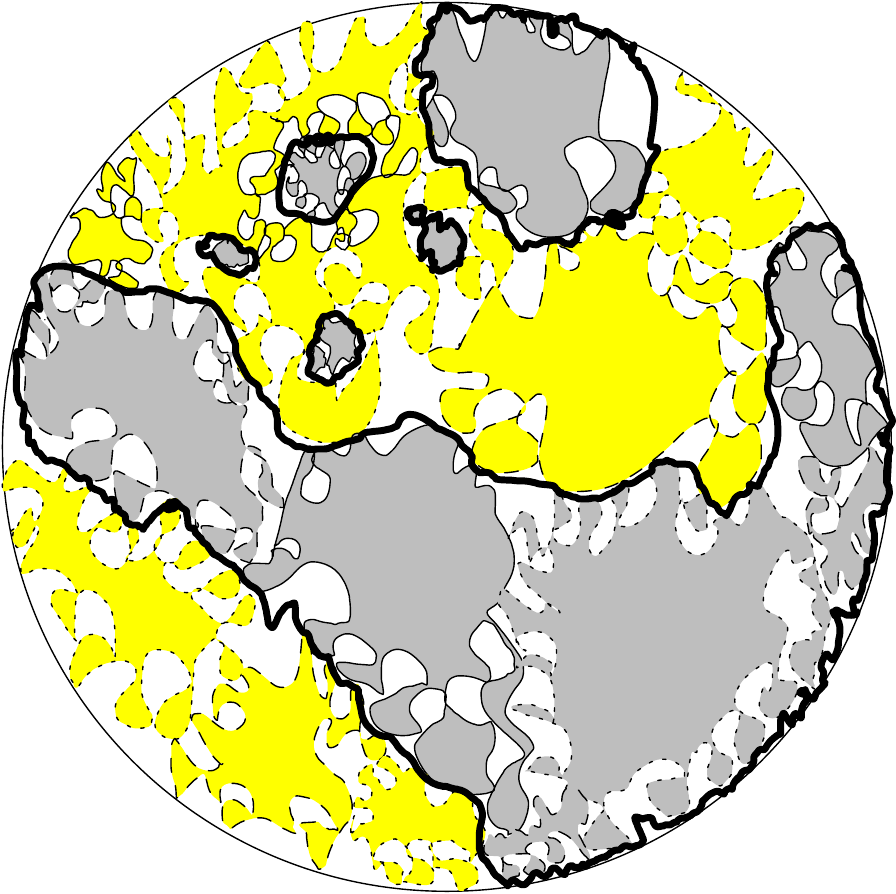}
\caption{\label{fig::bcle_nested}From nested $\CLE_{\kappa'}$ (on the left -- only one second-layer loop and one third-layer loop are drawn on top of the sketch of Figure~\ref{fig12})  to a variant of $\CLE_\kappa$  (on the right) by considering percolation of loops (sketch).}
\end{center}
\end{figure} 
 
 One very closely related result goes as follows.  
Consider a simple $\CLE_{\kappa'}$ in the upper half-plane, and toss an independent (not necessarily fair) coin for each of the $\CLE_{\kappa'}$ clusters just as before (and here again, it is actually in fact enough to 
use a non-nested $\CLE_{\kappa'}$). We then consider the union of all clusters of black $\CLE_{\kappa'}$ clusters that touch the positive half-line, and the union of all clusters of white $\CLE_{\kappa'}$ clusters that touch the negative half-line. It then turns out that these two sets have a common boundary, which is a simple curve from $0$ to infinity in the closed half-plane. 
We will describe precisely the law of this interface as an $\SLE_{\kappa} ( \rho ; \kappa - 6 - \rho)$ curve for some value of $\rho = \rho (\kappa', \beta)$,  which is a rather simple variant of the chordal $\SLE_\kappa$ curve.

\begin{figure}[ht!]
\begin{center}
\includegraphics[width=3.2in]{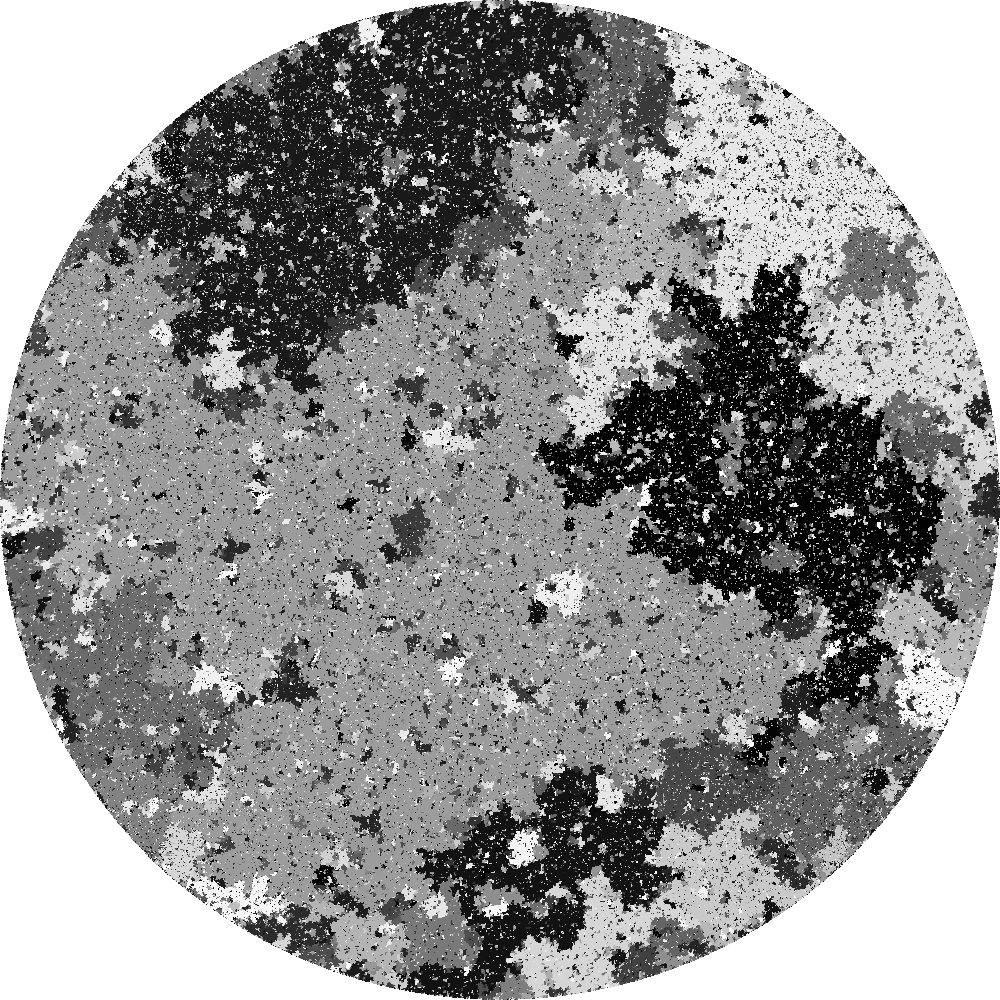}\hspace{0.01\textwidth}\includegraphics[width=3.2in]{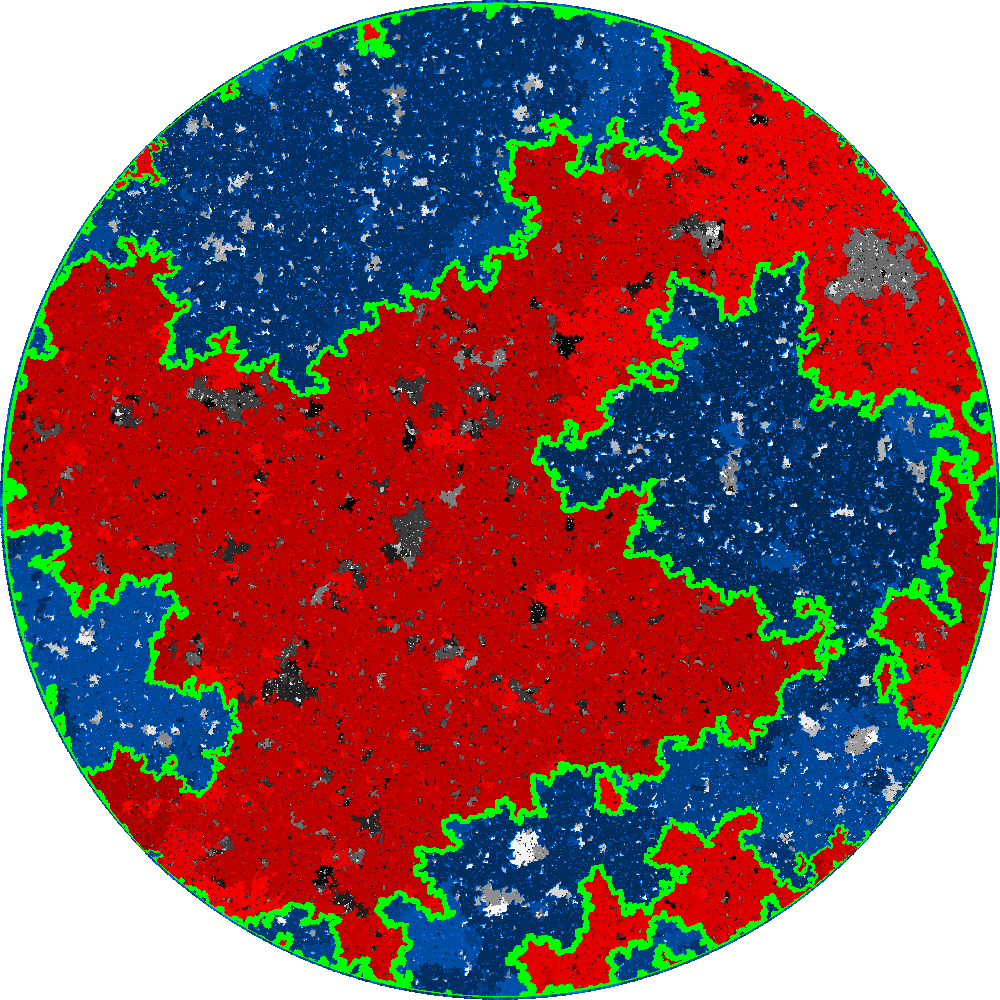}
\end{center}
\caption{\label{fig:perc_sim} Percolation of $\CLE_6$ loops: {\bf Left:} Critical bond percolation on a $1000 \times 1000$ square grid, conformally mapped to $\D$.  Each cluster is colored according to an i.i.d.\ label chosen uniformly in $[0,1]$.  {\bf Right:} Cluster of clusters with label at most (resp.\ at least) $1/2$ which is connected to the boundary is colored in red (resp.\ blue).  The interface between the red and blue clusters is indicated in green and forms a $\BCLE$.}
\label {figcle6}
\end{figure}

\item 
Let us now describe another result in the same direction, that can be viewed as the exact scaling limit of the Edwards-Sokal correspondence. Here we start with $\CLE_{\kappa'}$ for $\kappa' \in (4,6)$ (we exclude this time the values $\kappa' \in [6,8)$) and it will be important to consider a nested version, and this time to look at the wired way to define its clusters via the parity of the $\CLE_{\kappa'}$ loops. With this definition, the outermost cluster is a cluster that has the boundary of the domain as its outside boundary and the first-level $\CLE_{\kappa'}$ as its inside boundary.    

We then color all these nested clusters in white or black with probability $p$ or $1-p$, except that we force the outside cluster to be white. We then look at the law of the collection of all outermost black cluster boundaries.  
Theorem~\ref{thm:cle_k_from_kp} will state that for each $\kappa' \in (4,6)$, there exists a value $p (\kappa') \in (0,1)$ such that this collection is exactly a $\CLE_\kappa$ for $\kappa \in (8/3, 4)$. This is the exact continuous analog of the construction of the $q$-state Potts model out of the corresponding FK$_q$ model, where $p$ plays here the role of $1/q$. 
\end {itemize}

Let us now list some consequences of these couplings:
A first observation is that it provides a new derivation of some basic properties of $\CLE_{\kappa}$ carpets (such as its M\"obius invariance and the fact that the branching tree construction of \cite{she2009cle} does not depend of the chosen root of the tree) that does not rely on the loop-soup construction of these $\CLE_\kappa$ of \cite {sw2012cle} 
(up to now, this was the only existing approach to these results). 

A second feature is related to the scaling limit of discrete models. 
Given an instance of the $q$-state Potts model, one can construct a two-coloring by simply dividing the $q$ states into two parts and assigning each part one of the two colors. The corresponding two-color model belongs to a more general family of models (known as {\em fuzzy} Potts models in the literature, see for instance \cite {mv1995fuzzypotts,hagg1999fuzzypotts}) that can be constructed by starting with an instance of FK$_q$ percolation model and then using i.i.d.\ biased coin tosses to assign a color to each cluster. 
Clearly, it follows from our description of clusters of $\CLE_{\kappa'}$ clusters that (modulo some discrete considerations), 
if one has proved that the scaling limit of a discrete FK$_q$ model is a $\CLE_{\kappa'}$, then one will be able to deduce that the scaling limits of clusters of randomly colored FK$_q$ clusters (the fuzzy Potts models) are given by the aforementioned variants of $\CLE_{\kappa}$ for $\kappa = 16/ \kappa'$ that we will describe.  
This is for instance useful in the case of the Ising model, as it is essentially known that the critical FK$_{q=2}$ model (see \cite {ks2015fkcle} for the boundary-touching loops)  converges to $\CLE_{16/3}$ and the randomly colored FK$_{q=2}$ model for $\beta=0$ is exactly the Ising model. Hence, this will make it possible to deduce that the scaling limit of critical Ising model interfaces is $\CLE_3$ from the fact that the scaling limit of the corresponding FK$_{q=2}$ is $\CLE_{16/3}$ -- see the upcoming paper \cite {dtw2016ising}).

An interesting feature of our coupling is that the $\CLE_\kappa$ variant that one constructs by
coloring $\CLE_{\kappa'}$ clusters exhibits a certain special symmetry for $\beta =0$ (i.e.\ when one colors the clusters using a fair coin) that corresponds to the fact that $p(\kappa') = 1/2$ 
only when $\kappa' = 16/3$. 
On the other hand, if the scaling limit of the critical FK$_{q=2}$ model is conformally invariant, it should satisfy this special symmetry.
Hence, one can conclude, based on our results only and without reference to discrete observables, that the values $\kappa' =16/3$ and $\kappa=3$ are the only possible 
candidates for a conformally invariant scaling limit for the critical FK$_{q=2}$ and Ising models. 
 
\subsubsection {From $\CLE_{\kappa}$ to $\CLE_{\kappa'}$}
 
We now describe the converse coupling: 
Let us first consider a simple $\CLE_\kappa$ carpet for $\kappa \in (8/3, 4)$, and let us try to list properties that a  ``continuous conformally critical percolation interface within the $\CLE_\kappa$ carpet'' would have to satisfy. Intuitively, one can imagine that it traces the outer boundary of what critical percolation clusters within the sparse $\CLE_\kappa$ 
 carpet would be (see Section~\ref{subsec:cpi} for more details, Figure~\ref{fig:percolation_sketch} for an illustration and Figure~\ref{figcleising} for a simulation). This would  be a random curve within this carpet, that would never be allowed to self-cross; it would bounce off 
 its past trajectory (just like non-simple SLEs do) as well as each of the holes of the $\CLE_\kappa$ carpet (leaving always the holes on the left-hand side of the percolation curve, when it is exploring the boundary of a cluster clockwise). Furthermore, it should be local in the sense that it should feel the holes of the $\CLE$ only when it hits them. Finally, we expect the joint law of the $\CLE_\kappa$
 and the percolation interface to be conformally invariant. 
 
In general fractal deterministic carpets (e.g., the Sierpinski carpet), such continuous percolation interfaces should in principle exist,
but showing this seems out of reach of the current mathematical knowledge,  as does describing or even guessing the conjectural features (such as their dimension) of these interfaces.    
 However, the conformal invariance properties of $\CLE$ carpets make things possible. We show that for each $\kappa \in (8/3, 4)$, a random process with the aforementioned properties does indeed exist
 in $\CLE_\kappa$ and that its distribution is unique. We describe it precisely in terms of $\SLE_{\kappa'}$-type processes, and more generally, we describe the corresponding collection of ``all continuous percolation clusters'' within a $\CLE_\kappa$ carpet in terms of a $\CLE_{\kappa'}$. In this way, we indeed interpret $\CLE_{\kappa'}$'s in terms of ``percolation clusters'' within the $\CLE_\kappa$ as is suggested by the corresponding relation between FK$_q$-models and the Potts models.

 \begin{figure}[ht!]
\begin{center}
\includegraphics [width=3in]{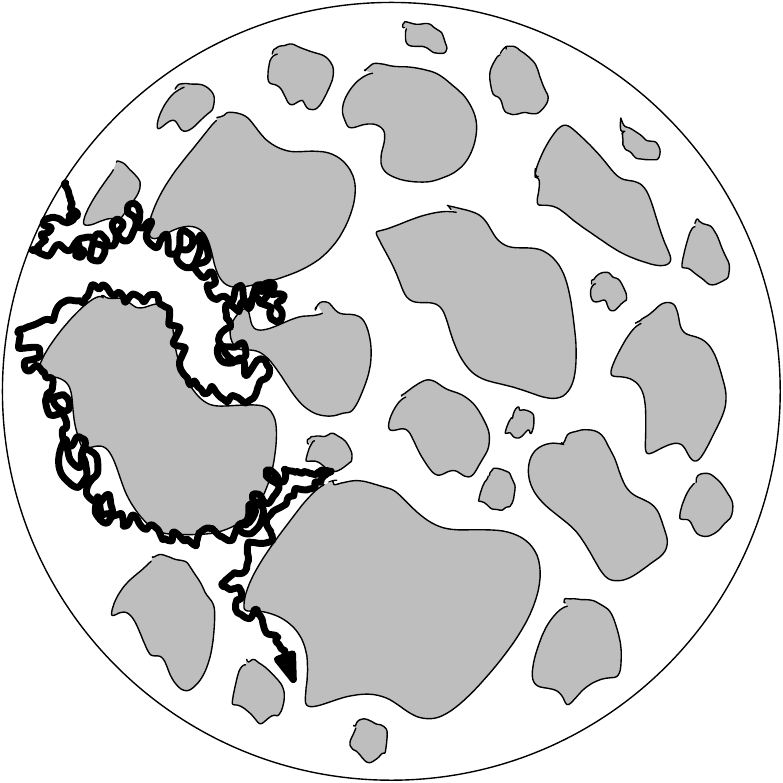}
\includegraphics [width=3in]{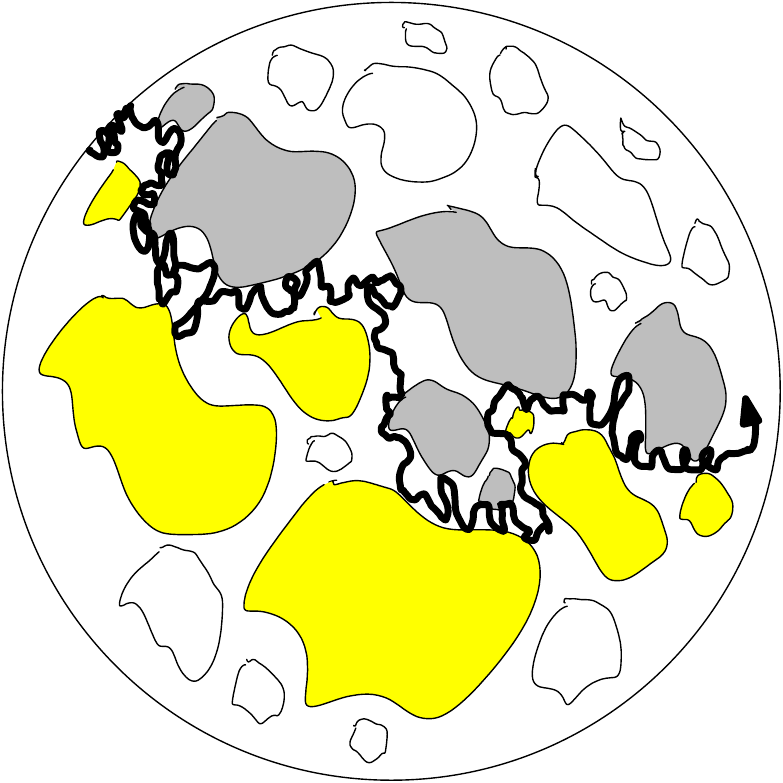}
\caption{\label{fig:percolation_sketch} A sketch of $\CLE_\kappa$ ``percolation interfaces'' for $\beta=-1$ and $\beta=0$}
\end{center}
\end{figure}

\begin{figure}[ht!]
\begin{center}
	\includegraphics[width=3.2in]{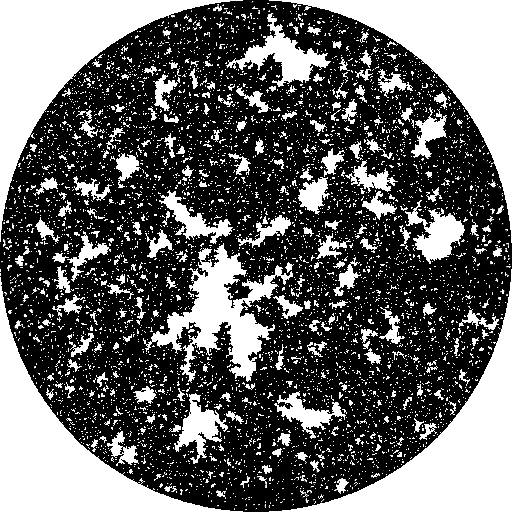}\hspace{0.01\textwidth}\includegraphics[width=3.2in]{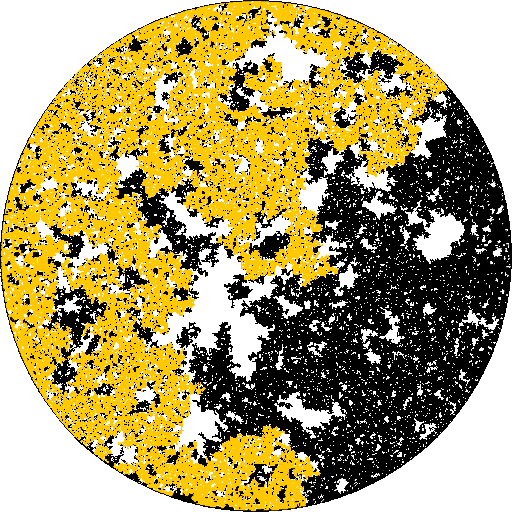}
\end{center}
\caption{Percolation within $\CLE_3$: {\bf left:} Boundary touching cluster of sites with $+$ spin in an Ising model instance (black) --- a.k.a.\ Ising carpet --- on a $512 \times 512$ square grid with $+$ boundary conditions together with Ising carpet holes (white) conformally mapped to $\D$.  {\bf Right:} Clusters of percolation performed on Ising carpet incident to an arc of $\partial \D$ in yellow.}
\label {figcleising}
\end{figure}

This construction can be generalized as follows. One first tosses an independent coin for each $\CLE_\kappa$ hole to decide whether it is closed or open (with probability $(1+\beta)/2$ and $(1-\beta)/2$). This then changes of course the connectivity rules within the $\CLE_\kappa$ but it is then still possible to describe, construct and characterize the corresponding ``continuous percolation interface'' (this process would now leave the closed holes to its left and the open holes on its right, when it is exploring the boundary of a cluster clockwise) in terms of a variant of $\SLE_\kappa$ that will depend on the parameter $\beta$. The case $\beta=-1$ is then exactly the previously described case, where all $\CLE_\kappa$ holes are closed, i.e.\ obstacles for the percolation within the CLE. Note that for general $\beta$, one can choose the status of a $\CLE_\kappa$ hole only once the percolation interfaces does indeed hit it, which explains the relation to the so-called side-swapping $\SLE_\kappa(\kappa-6)$ 
processes, that will be at the heart of our analysis. 

In the subsequent paper \cite{msw2016randomness}, we shall prove that (when $\kappa \in (8/3, 4)$), these percolation interfaces are indeed still random (when one conditions on the $\CLE_\kappa$) for all choices of $\beta \in [-1,1]$. This means that this percolation process captures additional randomness that is located ``inside'' the $\CLE_\kappa$ carpet.

An important remark is that, just as in the discrete Edwards-Sokal correspondence, if one constructs a $\CLE_\kappa$ from a $\CLE_{\kappa'}$ by looking at clusters of $\CLE_{\kappa'}$-clusters, and then looks at the conditional law of the 
CLE$_{\kappa'}$ clusters given the obtained $\CLE_{\kappa}$ carpet that one did construct, then one can interpret the former as the critical percolation clusters in the latter. So both types of couplings are in fact the same.

\subsection {Other results and comments}

\subsubsection {$\CLE_4$ percolation}

In the special boundary case $\kappa =4$, it turns out  not to be  possible to define such  continuous percolation interfaces within the $\CLE_4$ carpet, unless 
one uses the previous procedure with $\beta =0$ (in other words, each hole will be closed or open with probability $1/2$). But when $\beta =0$, there  exists in fact a one-parameter family of 
such continuous percolation interfaces.  Another important difference is that these continuous percolation interfaces in the $\CLE_4$ turn out to be deterministic functions
of the $\CLE_4$ carpet and of the status of all its holes. One can interpret this as follows: When one colors the $\CLE_4$ loops at random in black and white using fair coins, then there is a one-dimensional family of rules to define clusters of black loops.  The situation 
has therefore similarities with the $\kappa' > 4$ case, where one agglomerates deterministically $\CLE_{\kappa'}$ loops of the same color that touch each other, but recall that the $\CLE_4$ loops are all disjoint so that the very existence of such a deterministic gluing mechanism is quite surprising and intricate.

\subsubsection {Boundary conformal loop ensembles} \label{subsubsec::BCLE}

We define and study a natural family of $\CLE$ variants that we call boundary conformal loop ensembles (BCLEs) that share some (but not all) of the conformal symmetries that characterize CLEs and that loosely speaking should describe the scaling limit of critical models for a  large class (actually a continuum) of boundary conditions (there will be a one parameter family of loop ensemble models called $\BCLE_\kappa (\rho)$). These arise naturally when we study the CLE-variants that one needs to properly describe the relation between variants of $\CLE_\kappa$ and $\CLE_{\kappa'}$. They are defined in a natural and elementary way via target-independent variants of the $\SLE_\kappa$ processes (that we will denote by $\bSLE_\kappa (\rho)$), and they heuristically correspond to the scaling limit of discrete models for a certain continuous family of ``constant'' boundary values.  It is worthwhile noticing that while the definition of $\CLE_\kappa$ is restricted to $\kappa \in (8/3, 8]$, these
boundary conformal loop ensembles make sense also (for some values of $\rho$) in the range $\kappa \in (2, 8/3]$.

\subsubsection{$\SLE_\kappa(\rho)$ processes with $\rho < -2$}
The $\SLE_\kappa(\rho)$ processes are an important variant of $\SLE$ in which one keeps track of just one extra marked point.  These processes were first introduced in \cite{lsw2003restriction} and, just like $\SLE$ itself, they appear naturally in many different contexts.  Roughly, $\SLE_\kappa(\rho)$ is defined similarly to ordinary $\SLE$ except the driving process is described by a multiple of a Bessel process (the unique continuous one-dimensional Markov processes with the same scaling property as Brownian motion). These processes can be defined in a natural and 
simple way for all values of $\rho > -2$ and the continuity of the trace and couplings with the GFF in this case have been analyzed in \cite{ms2012ig1}. 
However, it turns out that there are natural ways to generalize these $\SLE_\kappa (\rho)$ processes to some values of $\rho \le -2$, and that these generalized processes do play an important role 
in many instances (such as in the present paper about CLE).  
In the present article, we will establish the continuity of the trace and couplings with the GFF of certain generalized $\SLE_\kappa(\rho)$, and we will shed some light on the structure of these processes.  When $\kappa \in (0,4]$ there are in fact two phases of process behavior depending on whether $\rho \in (-2-(\kappa/2), (\kappa/2)-4)$ or $\rho \in ( (\kappa/2)-4,-2)$.  The present work will focus on the former range, while the latter range will be the focus of \cite{ms2016lightcone}.  The two regimes differ in that in the former the process is ``loop-forming'' (and therefore 
related to loop ensembles) and the latter is where the process can be constructed as a ``light cone'' of angle-varying GFF flow lines.

The critical value $\rho=(\kappa/2)-4$ is quite interesting because, as we will show in this paper, the law of its range is exactly equal to the law of the range of an $\SLE_{\kappa'}(\rho')$ process with $\rho'=(\kappa'/2)-4$ but it visits the points of its range in a different order.

The minimal value $\rho = -2-(\kappa/2)$ is also interesting because, as we will show later, these processes have exactly the same law as an $\SLE_{\kappa'}$ process. More precisely, it will follow from our results that as $\eps$ decreases to $0$, the generalized processes $\SLE_{\kappa} ( -2 - (\kappa/2) + \eps )$ do converge in a rather strong sense to an $\SLE_{\kappa'}$. For instance, $\SLE_{8/3}(-10/3)$ can be interpreted as an $\SLE_6$.  (We remark that the interpretation of $\SLE_{8/3}(-10/3)$ as $\SLE_6$ was pointed out in \cite[Section~4.5]{she2009cle} by considering a limit of $\SLE_\kappa(\kappa-6)$ as $\kappa \downarrow 8/3$.)

Thus, for a fixed $\kappa$, as $\rho$ increases from $-2-(\kappa/2)$ to $(\kappa/2)-4$, we have a family of random curves whose laws interpolate between that of an $\SLE_{\kappa'}$ curve and the law of a random curve which has the same range as an $\SLE_{\kappa'}$-type curve but visits the points of its range in a different order.  The intermediate curves look like $\SLE_{\kappa'}$-type curves decorated with extra $\SLE_\kappa$-type loops.  More details on the definition of these processes will be given later in the paper.  The generalized $\SLE_{\kappa'}(\rho')$ processes with $\rho' < -2$ have a similarly interesting structure.

 \subsubsection {Further comments}
 
A key role in the present paper will be played by the coupling of some SLE processes with the GFF. Recall that as pointed out in \cite{ms2012ig1,ms2013ig4}, it is natural to couple some $\SLE_\kappa$ processes with $\SLE_{\kappa'}$ processes by coupling each with the same instance of the GFF.  This is also what we shall be doing in the present work in the $\CLE$ context. The very rough idea of the derivation of our
$\CLE_\kappa$/$\CLE_{\kappa'}$ duality results will go as follows: It is possible to couple an $\SLE_\kappa$ exploration tree started from a given point $x \in \partial D$ with a GFF in $D$ with fairly simple boundary conditions but that depend on $x$. In this coupling the branches of the $\SLE_\kappa$ exploration tree then in turn define the $\CLE_\kappa$ collection of loops. 
(Mind that the dependence on the $\CLE$ with respect to the choice of boundary point is quite delicate, except when $\kappa =4$.) 
 One can then define, associated with the same choice of boundary point, the same GFF and the same boundary data, an $\SLE_{\kappa'}$ exploration tree that in turn will define a $\CLE_{\kappa'}$, and a careful analysis of the interaction and commutation relations between the branches of these two trees will enable us to conclude. Here, as in the series of papers on imaginary geometry \cite{ms2012ig1,ms2012ig2,ms2012ig3,ms2013ig4}, the concept of local sets of the GFF as introduced in \cite{ss2010continuumcontour}
 (see also \cite {rozanov1982markov} for an earlier version of such ideas).
 Also, just as in these imaginary geometry papers, while some of the commutation relations that we will use are very close to some cases established by Dub\'edat \cite {dub2007commutation} for SLE curves as long as 
 they do not intersect, the GFF flow line framework will be crucial to control the joint law and interaction between the coupled SLE curves also when they touch each other. 
 
  In a sequel paper \cite{msw2016fan}, we will explain how to interpret the SLE fan defined in \cite{ms2012ig1} as the collection of all $\SLE_\kappa (\rho)$ type curves from a point to another coupled with a given GFF, in terms of this $\CLE_{\kappa'}$ decomposition. This will give another example (but still building on the present work) of how the FK/Potts coupling is also intrinsically embedded in the GFF and directly related to imaginary geometry concepts.

Let us stress that the duality type results between variants of $\CLE_\kappa$ and $\CLE_{\kappa'}$ that we derive here are of a rather different type from the duality results between variants of $\SLE_{\kappa}$ and 
variants of $\SLE_{\kappa'}$ derived by \cite{zhan2008duality,dubedat2009duality,zhan2010duality2,ms2012ig1,ms2013ig4}, who were viewing single $\SLE_\kappa$-type curves as outer boundaries of single $\SLE_{\kappa'}$-type curves (and did not mention any percolation of $\CLE_{\kappa}$ or 
$\CLE_{\kappa'}$ loops).

\subsection {Outline of the paper}
We conclude this introduction with a general description of the structure of the paper: We start by recalling some background and motivation, which leads to the definition of what we call ``continuous percolation interfaces'' (CPI) in labeled $\CLE$ carpets (this corresponds to the percolation in $\CLE_\kappa$ for $\kappa \in (8/3, 4]$ question) in Section~\ref{Sec2}. We then spend some time in Section~\ref{Sec3} discussing the various definitions and generalizations of $\SLE_\kappa (\rho)$ processes to some values of $\rho \le -2$. Then, mostly building on ideas related to conformal restriction and $\CLE$ characterizations (in the spirit of \cite{sw2012cle}), we derive a description and characterization of these CPI processes in Sections~\ref{Sec4} and~\ref{Sec5}, but {\em conditionally} on what looks like a rather technical but tricky assumption (roughly speaking, the existence and continuity of the trace of certain $\SLE_\kappa (\rho)$ processes) that will be proved later in the paper.
In the case $\kappa =4$ which is the focus of Section~\ref{Sec5}, one can prove everything though, building on existing results relating the $\CLE_4$ to the GFF. We then turn our attention to the case of $\CLE_{\kappa'}$ percolation: After describing various aspects of the construction of $\CLE_{\kappa'}$, we again derive a conditional result about $\CLE_{\kappa'}$ percolation interfaces in Section~\ref{Sec6}, this time conditionally on the existence and continuity of the trace of $\SLE_{\kappa'}(\rho)$-type processes (that will be proved later in the paper).

In Section~\ref{Sec7},  we  
define the boundary conformal loop ensembles (BCLEs) and derive some of their basic properties. We are then ready to state precisely the general duality results (Theorems~\ref{thm:duality1} and~\ref{thm:duality2}) between $\BCLE_\kappa$ and $\BCLE_{\kappa'}$. These results prove all the assumptions that the conditional results of Section~\ref{Sec4}, Section~\ref{Sec5} 
and Section~\ref{Sec6} were based on, turning them into unconditional statements. They also imply the duality results that we have described in the introduction and that are formally stated in Section~\ref{Sec7}.

The proofs of these two key theorems, based on the GFF imaginary geometry couplings of various SLE curves, is then the goal of Sections~\ref{Sec8} and~\ref{Sec9}: We first recall some imaginary geometry background, and some of the results from \cite{ms2012ig1,ms2012ig2,ms2012ig3,ms2013ig4} that are instrumental in the proof. 
In particular, we recall the interaction rules between various flow and counterflow lines coupled with the GFF (mostly derived in \cite{ms2012ig1}), and the definition and basic features of the ``space-filling SLE'' that has been derived in \cite{ms2013ig4}.  We then use those to derive the theorems in Section~\ref{sec:proofs}.

Finally, we conclude in Section~\ref{sec:conclusion} with some comments and outlook.

\section{Heuristics and motivation from discrete models}
\label {Sec2}

In the following, we will first provide some motivation, background, and a possible natural interpretation for our subsequent study and results. We will then define axiomatically what we call ``continuous percolation'' within $\CLE$ carpets.

\subsection{Percolation in fractal carpets}

It is possible (see \cite{kumagai1997percolation,hw2008uniqueness}) to study models that
can be viewed as percolation in deterministic self-similar fractals such as the Sierpinski carpet, and to show the existence of a non-trivial phase transition. Intuitively speaking, even though the fractal set has zero Lebesgue measure (and a fractal dimension that is strictly smaller than $2$), taking a percolation parameter that on the discrete level is sufficiently close to one makes it possible to create macroscopic connections within this sparse fractal domain.

One way to proceed is for instance to look at the infinite discrete connected subgraph of the square lattice, where one has removed the larger and larger squares at each scale (more precisely, consider the graph $\N \times \N$, and remove first all interiors of squares of the type $S_{N,j,k}= 3^N ([3j + 1 ,  3j + 2] \times [3k+1 , 3k+2])$ for non-negative integers $j, k, N$, and to see that with probability one, percolation with parameter $p$ on this graph has an infinite open cluster provided $p$ is chosen to be large enough. 
Furthermore, in this particular case, at the critical value, one can show that there are macroscopic connections at each scale, but no infinite cluster. This is like what classical Russo-Seymour-Welsh  type results yield 
for ordinary critical percolation in the plane: At whatever scale one looks at the critical percolation picture, the existence of connections remains very much random.

This suggests that there might exist a continuous object that could be interpreted as the scaling limit of these critical percolation interfaces in the carpet. It would be the analog of  $\SLE_6$ paths (the scaling limits of percolation interfaces in the plane \cite{smirnov2001percolation,cn2006percolation}), but within this fractal set. It is also natural to conjecture that it should satisfy a number of remarkable properties (target-invariance, local growth, some notion of scale-invariance, or even conformal invariance).

In the case of $\SLE_6$, conformal invariance was the key-property that enabled Oded Schramm \cite{schramm2000sle} to construct a candidate for this scaling limit directly in the continuum via Loewner's equation. In the present case, this feature seems delicate to handle. Even though this exploration path would not ``feel the holes in the fractal set before hitting them,'' one has to keep track of their location in the remaining-to-be-explored domain in order to describe its future evolution, and this (conformally speaking) means to keep track of infinitely many parameters. So, there does not seem to be a simple Loewner-growth way to describe it.

However, in the case where the fractal carpet is itself a random $\CLE_\kappa$ for $\kappa \in (8/3, 4]$,  the conditional distribution of the holes remaining to be discovered by this exploration is easy to describe thanks to the $\CLE$ restriction property and the locality property of the percolation path. As we shall see in the present paper, this enables us to give a full understanding of the joint distribution of the $\CLE_\kappa$ with what can be interpreted as the percolation path in the $\CLE_\kappa$ carpet. 

Let us mention that these heuristic arguments are still valid for the following variant of the model. Choose first an additional parameter 
$p_0 \in [0,1]$, and decide independently for each hole in the fractal carpet whether it is ``completely open'' (with probability $1-p_0$)  or ``completely closed'' (with probability $p_0$). In the discrete graph approach that we briefly described above, this would correspond to open/close all the edges (or sites) that are located inside the squares $S_{N,j,k}$. The previously described case where all the holes are closed corresponds to $p_0=1$, whereas in some loose sense (this is best seen rigorously on some triangular site-percolation analog) the case where all the holes are open, corresponds to a dual-version of the same problem. The value $p_0=1/2$ is quite natural too in this setting, as it induces a natural duality within the carpet between open and closed. Hence, for each value of $\kappa \in (8/3, 4]$ and $p_0 \in [0,1]$, one could  hope to obtain a continuous critical percolation interface path in the 
$\CLE_\kappa$ carpet, that leaves the open holes it encounters to one of its sides (say to its left) and the closed ones to its other side (i.e.\ its right).

\subsection{FK clusters as critical percolation in Potts clusters}

The standard coupling between FK-clusters for parameters $(p,q)$ where $q$ is a positive integer and the corresponding $q$-state Potts model with parameter $\beta = \beta (p,q)$ is usually described as follows (and this is why Fortuin and 
Kasteleyn introduced the FK percolation model in the first place): First sample the FK-percolation model on the edges of a given finite graph, then color independently each cluster with one of the $q$ possible colors, and note that the obtained coloring of the sites of the graph follows exactly the law of the Potts model. In other words, the Potts clusters are obtained by randomly coloring clusters of the associated FK model. 
   
This coupling has also a simple and classical reverse version: Sample first a $q$-state Potts model coloring of the sites of a finite graph with parameter $\beta$. Then, for each edge of the graph, if the two ends have different colors, declare the edge closed, otherwise, toss a coin with probability $p=p(\beta, q)$ to declare it open (or closed). Then, the obtained configuration on edges follows exactly the $(p,q)$-FK percolation distribution. 
In other words, one can view the FK-percolation clusters as having been obtained by performing independent Bernoulli percolation in the $q$-state Potts clusters (see, e.g.,
\cite{grimmett2006cluster} for background on FK-percolation and the Potts model).

In the scaling limit,  when $q \le 4$ and for a critical value $p_c = p_c (q)$ (that depends on the considered regular graph), the FK-percolation model should behave randomly at every scale, and therefore the corresponding Potts model should as well (when $q=2,3,4$), and the FK clusters are strict subsets of the Potts clusters. 
Furthermore, it is believed (and proved for many graphs in the case $q=2$) that the obtained pictures are conformally invariant in law. 
On top of this, it is easy to note that these critical Potts-model-clusters should be random carpets that possess the properties that enable us to characterize axiomatically CLEs for $\kappa \in (8/3,4]$. 
Loosely speaking, conditionally on its outside boundary, the law of the inside holes of (the scaling limit) of a cluster of $1$'s should be distributed like a $\CLE_\kappa$ in the domain delimited by this outer boundary. 

Hence, this suggests that in the case where one considers a $\CLE_\kappa$ for $\kappa \in (8/3, 4]$ and $p_0=0$, then the ``continuous critical percolation interface'' in this random carpet could describe the boundary of the scaling limit of a coupled FK-percolation cluster, and therefore be related to the $\CLE_{\kappa'}$ loop-ensembles.

However, some care has to be given to the boundary conditions of the FK-models and the Potts models in order to make such statements more precise. For instance, in the $q$-state Potts model, one would wish to take a Potts model with mono-chromatic boundary conditions (say all boundary points have color $1$) and look at the carpet created by all $1$-clusters that touch the boundary (the law of the coloring inside holes would therefore be Potts models conditioned to have ``only non-$1$'' colors on the boundary etc.). This should correspond to a $\CLE_\kappa$. The monochromatic boundary conditions for the FK model corresponds to a wired boundary condition for the corresponding FK model, which does not exactly correspond to a $\CLE_{\kappa'}$ (the $\CLE_{\kappa'}$ should rather correspond to the contours of the dual FK clusters).  

Figure~\ref{FKexploration} provides an illustration of the FK correspondence in the case of the Ising model (i.e., $q=2$) and how it is then possible to explore successively portions of the FK pictures and of the associated Ising pictures. These should give rise to variants of $\CLE_3$ and $\CLE_{16/3}$ in the scaling limit, with other types of ``boundary conditions'' -- these will be the some of the ``boundary CLEs'' that we will construct in the present paper. 
\begin{figure}[ht!]
\begin{center}
\includegraphics [width=5in, angle=180]{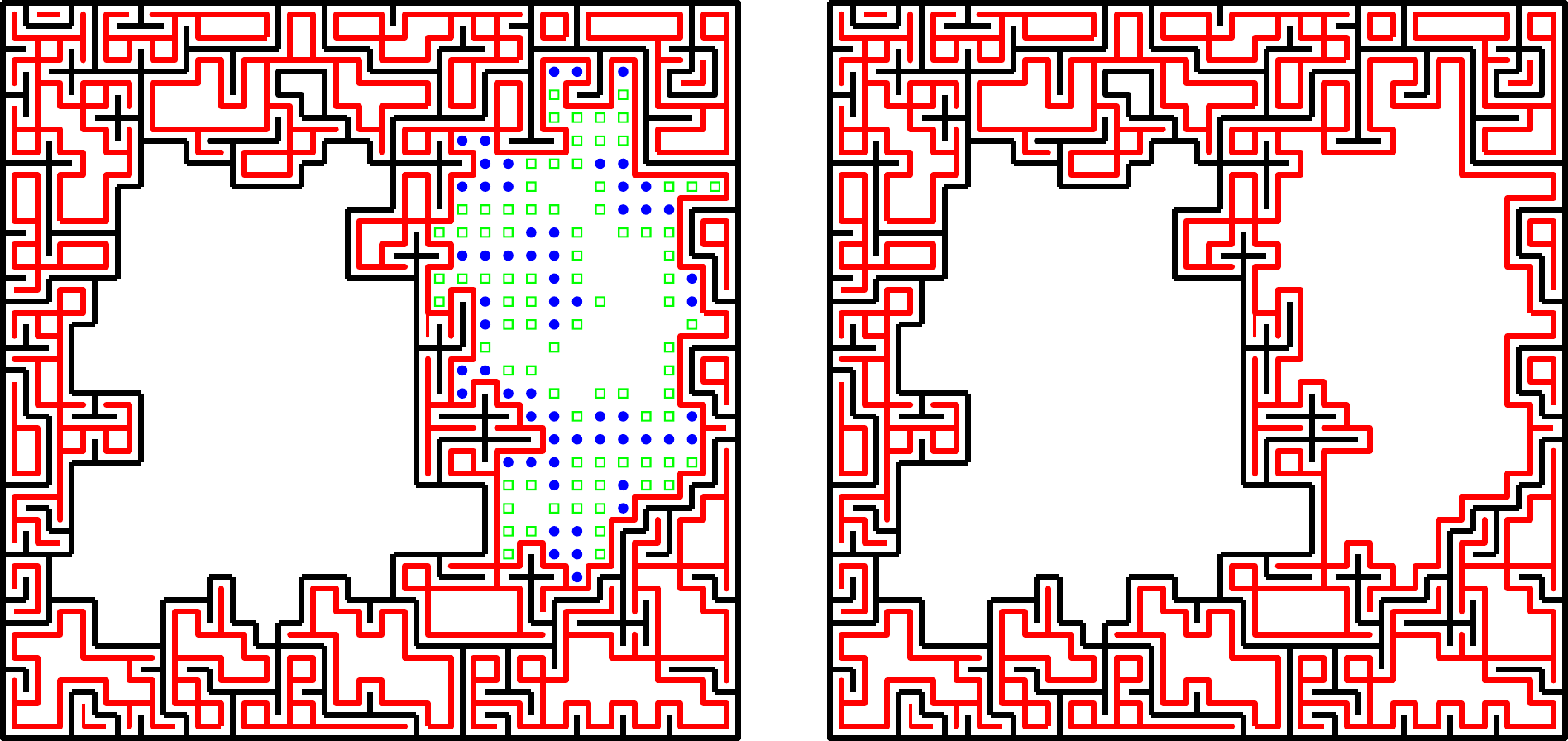}
\caption{\label{FKexploration} Illustration of iterative discovery from the boundary of a critical FK$_{q=2}$ model with wired boundary conditions. (i) First exploring the outer boundaries of the dual FK clusters that touch the boundary (or equivalently, all interfaces that touch the boundary of the domain). (ii) Conditionally on this, the boundary conditions in the remaining holes are wired and free depending on the holes.  In a hole with free boundary conditions, we discover the Ising clusters of $+$ (circles) that touch the boundary. In the remaining part inside this loop, the boundary conditions are now given by $-$ (squares), and correspond to wired boundary conditions for FK. }
\end{center}
\end{figure} 

Hence, this FK-Potts analogy heuristic suggests that a candidate for the critical percolation paths within a $\CLE_\kappa$ carpet for $\kappa \in (8/3, 4]$ could be expressed in terms of a $\CLE_{\kappa'}$ coupled to it.

\subsection{Continuous percolation interfaces within $\CLE$ carpets}
\label{subsec:cpi}

The ideas laid out in the previous subsections lead us to try to define directly in the continuum and in abstract terms what a continuous conformally invariant ``critical percolation interface''
within the $\CLE_\kappa$ carpets for $\kappa \in (8/3, 4]$ should satisfy.

We first choose a parameter $\beta \in [-1, 1]$ which we will use as follows:  When $\beta=-1$ all the holes are closed, when $\beta=1$, all the holes are open, and in general, we choose randomly and independently each hole to be open or closed with respective probabilities $p_0:=(1+\beta)/2$ and $1-p_0=(1-\beta)/2$. This gives rise to a ``labeled'' $\CLE_\kappa$ that we call a $\CLE_\kappa^\beta$. One way to represent the open or closed status of a hole is to decide to respectively orient its boundary counterclockwise or clockwise. A $\CLE_\kappa^\beta$ is then a random collection of oriented, disjoint, and non-nested loops. 

Suppose that $\Gamma$ is such a $\CLE_\kappa^\beta$ in a simply connected planar domain $D$ (with $D \not= \C$) and that it is coupled with a random continuous curve $\gamma$ in $\overline D$ (in the sense that $\overline D$ is the union of $D$ with its prime ends) from one boundary point (or prime end) $x$ of $D$ to another one $y$. We assume that almost surely: 
\begin{enumerate}[(i)]
 \item\label{it:non-self-crossing} The curve $\gamma$ is non self-crossing (double points are nevertheless allowed). 
 \item\label{it:does-not-hit-interior} The curve $\gamma$ does not intersect the interior of any of the loops of $\Gamma$ (but it can touch the loops). 
 \item\label{it:bouncing-orientation} When $\gamma$ intersects a counterclockwise (open) loop of $\Gamma$, it leaves it to its right on its way from $x$ to $y$.  If the loop is clockwise (closed), it leaves it to its left.
\item\label{it:conformal-invariance} The joint law $\mu$ of the coupling $(\Gamma, \gamma)$ in $D$ is invariant under any conformal automorphism of $D$ (which makes it possible to define it by conformal invariance in any other simply connected proper subset of the complex plane with two marked prime ends).
\item\label{it:does-not-trace} Almost surely, for all $t >0$, $\gamma (0,t] \cap D \not= \emptyset$ (this loosely speaking means that almost surely, $\gamma$ does not start by sneaking along $\partial D$). 
\end{enumerate}

In the sequel, we will call $\gamma_t^*$ the set consisting of  $\gamma[0,t]$ together with all the loops of $\Gamma$ that it intersected, and let $\CF_t^*$ be the $\sigma$-algebra generated by  $((\gamma_s^*, \gamma_s) :  s \le t)$. We let $D_t^0$ the set obtained by removing from $D$ the set $\gamma_t^*$ and all the interiors of the loops of $\Gamma$ that have been discovered. One connected component of $D_t^0$ that we denote by $D_t$  has $y$ on its boundary, and 
we define $x(t)$ to be the prime end in this domain corresponding to the point where $\gamma$ is currently growing; when at time $t$ one discovers no $\CLE$ loop, then this is just $\gamma (t)$, and when at time $t$ one 
discovers a $\CLE$ loop, then one has to choose $x(t)$ among the two prime ends corresponding to $\gamma_t$, depending on the label of this loop. 
We will also use the conformal map  $\varphi_t$ from this connected component $D_t$ onto $D$ with $x(t)$ mapped to $x$, $y$ onto itself, and $\varphi_t (z) \sim z$ in the neighborhood of $y$.

\begin{definition}
\label{def:cpi}
Suppose that, in addition to~\eqref{it:non-self-crossing}-\eqref{it:does-not-trace}, we have for any $(\CF_t^*)$ stopping time $\tau$ that:
\begin{enumerate}[(i)]
 \item The conditional law of $\varphi_\tau (\gamma,\Gamma)$ restricted to $D_\tau$ given $\CF_\tau^*$ is equal to the joint law $\mu$ of the coupling $(\Gamma, \gamma)$.
 \item The conditional law of $\Gamma$ in the other connected components of $D_\tau^0$ is independently that of a labeled $\CLE_{\kappa}^\beta$.
\end{enumerate}
Then, we say that $\gamma$ is a continuous percolation interface (CPI) in the labeled $\CLE_\kappa^\beta$ $\Gamma$.
\end{definition}

As explained in the previous paragraphs, if a continuous path that could be interpreted as the continuous analog of a critical percolation interface in a $\CLE_\kappa^\beta$  exists at all, then one would expect it to be such a \hyperref[def:cpi]{CPI}. 
In the present paper, we shall provide a characterization and description of \hyperref[def:cpi]{CPIs} in labeled CLEs.  The aforementioned relations between Potts and FK-clusters (in particular that FK clusters can be viewed as percolation clusters within Potts clusters) suggests that when $\beta = -1$ or $1$, the paths $\gamma$ will turn out to be $\SLE_{\kappa'}$-type curves  that are coupled to  the $\CLE_\kappa$.

\section{Classical and generalized $\SLE_{\kappa}(\rho)$ processes}
\label{subsubsec:classical_sle_kr}
\label {Sec3}

Let us now very briefly review some basic facts concerning the $\SLE_\kappa(\rho)$ processes that will be relevant to the present paper. These ideas have been described in several previous papers and we will therefore just give a brief overview.  We shall assume that the reader is familiar with Bessel processes (the standard reference on this subject is \cite[Chapter~XI]{ry2004martingale}).

\subsection{Definition and characterization up to the swallowing time}
The $\SLE_\kappa(\rho)$ processes were introduced in \cite{lsw2003restriction} as the natural generalization of $\SLE$, where one keeps track of an additional marked force point on the boundary of the considered domain, which for simplicity we take to be $\h$.  Suppose  first that $O_0 < W_0$ on $\R$, and that $(K_t, t < \tau)$ is a Loewner chain in the upper half-plane, started at $W_0$ and stopped at its (possibly infinite) first swallowing time $\tau$ of $O_0$ (i.e., $\tau$ is the first time at which $O_0 \in K_\tau$). We use here the parameterization of Loewner chains by half-plane capacity (as is customary for chordal SLE) and denote the driving function of the Loewner chain by $W_t$. With the usual notation, we let $(g_t, t < \tau)$ be the associated family of conformal maps and define $O_t = g_t(O_0)$. 

For any time $t < \tau$, we denote by $F_t$ the conformal map from $\H \setminus K_t$ normalized so that $F_t (O_0) = 0$, $F_t ( \infty)=\infty$, and where the tip of the Loewner chain is mapped onto $1$. In other words, $F_t (z) = (g_t (z) - O_t) / (W_t - O_t)$. 

We say that the random Loewner chain satisfies the conformal Markov property with the extra marked point $O_0$ if for any stopping time $\sigma < \tau$, the conditional distribution of the process $(F_{\sigma} (K_{\sigma + t}), t < \tau - \sigma)$ given its realization up to time $\sigma$ is always the same up to time-reparameterization. 
It implies immediately that in fact, the process $(Z_t:= W_t - O_t, t < \tau)$ is a multiple of a Bessel process of some dimension $\delta$ up to its (possibly infinite if $\delta \ge 2$) first hitting time $\tau$ of $0$. As we also know that necessarily, for all $t < \tau$,  $\partial_t O_t = - 2/Z_t$, and therefore that 
\begin{equation}
\label{eqn:sle_w_def}
W_t = Z_t - 2 \int_0^t  \frac{1}{Z_s} ds.
\end{equation}
 This driving function fully characterizes these Loewner chains up to the first swallowing time $\tau$ of the marked point. When $Z$ is $\sqrt{\kappa}$ times a Bessel process of dimension $\delta$, then one says that the chain is an $\SLE_\kappa (\rho)$ where 
 \begin{equation}
 \label{eqn:sle_kp_bes_dim}
 \delta= \delta (\kappa, \rho):= 1+ \frac{2 (\rho +2 )}{\kappa}
\end{equation}
up to this stopping time $\tau$. $\SLE_\kappa (\rho)$ processes are therefore the only random Loewner chain satisfying the conformal Markov property with one extra marked point, at least up to the first swallowing time of this point.

When $O_0 > W_0$, then exactly the same analysis defines the $\SLE_\kappa (\rho)$ processes with {\em marked point on the right}. In this case, $Z$ is simply $-\sqrt{\kappa}$ times a Bessel process of dimension $\delta$.

\subsection{Classical $\SLE_\kappa (\rho)$ for $\rho > -2$}
The next question is whether and how one can define these Loewner chains after this swallowing time, i.e., how one can update the position of the marked point in an intrinsic way after the time $\tau$. An alternative and almost equivalent question is whether one can extend the definition and characterization of $\SLE_\kappa (\rho)$ when $O_0 = W_0$, and if so, how. 
This turns out to be almost immediate when $\rho > -2$ i.e.\ when $\delta >1$. Indeed, one can choose $(Z_t, t \ge 0)$ to be $\sqrt{\kappa}$ times an instantaneously reflecting Bessel process (recall that these reflected processes exist for all $\delta >0$), and to note that for $\delta > 1$, the integral
$\int_0^t ds / Z_s  $ is finite for all $t$ (recall that this is not true for $\delta \le 1$), which in turn allows us to define the driving function $W$ by $W_t := Z_t - 2 \int_0^t ds/Z_s $ at all positive times as in~\eqref{eqn:sle_w_def}. 

Note that with this definition, the marked point always stays (conformally speaking) to the left of the tip of the Loewner chain (i.e.,  $0 \le W_t- O_t = Z_t$ at all times). It is then in particular possible to define  $\SLE_\kappa (\rho)$ from $0$ to infinity in $\HH$ with marked point at $0^-$, and this process is scale-invariant. 

 \begin{figure}[ht!]
\begin{center}
\includegraphics[width=5in]{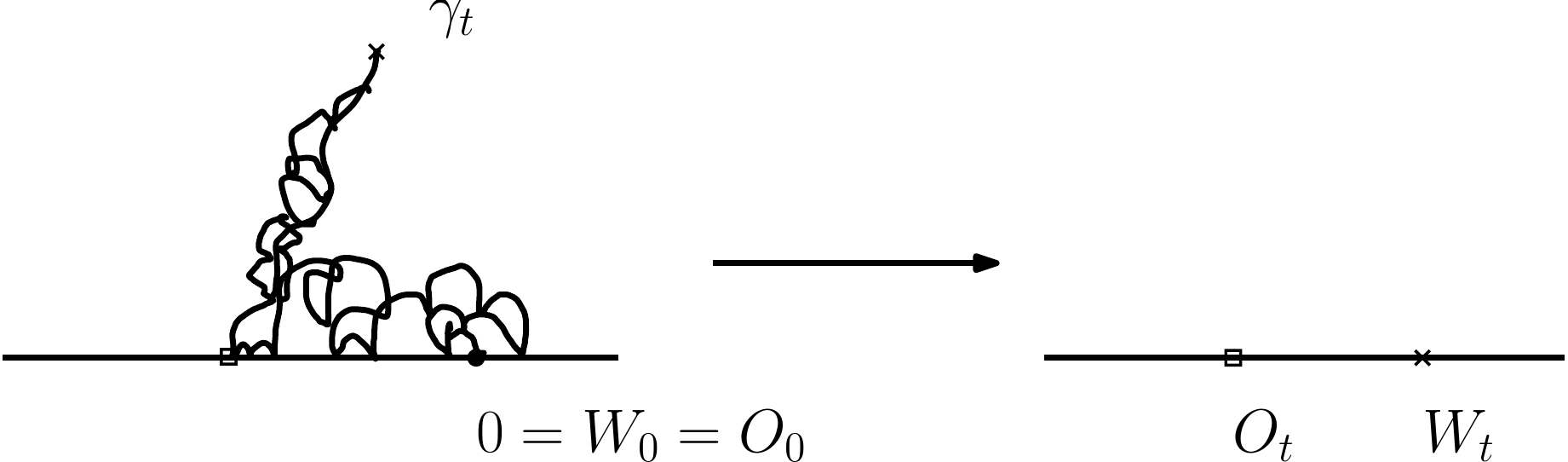}
\caption{\label{fig:slekrhosketch} A sketch of $\SLE_{\kappa'} (\rho)$ for $\kappa' > 4$ and $\rho \in (\kappa'/2-4,\kappa'/2-2)$.}
\end{center}
\end{figure} 

We will refer to such $\SLE_\kappa (\rho)$ processes for $\rho > -2$ (defined for all positive times) as the \emph{standard} or \emph{classical} $\SLE_\kappa (\rho)$ processes, as opposed to the generalized ones that we will define next.

It is also explained in \cite{ms2012ig1} that the marked point then always corresponds to the left-most visited point by $\gamma$ of the half-line to the left of $O_0$.  Also, each excursion of the process $Z$ away from~$0$ corresponds to an excursion of the $\SLE_\kappa(\rho)$ path away from the  half-line to the left of $O_0$. If $\delta \ge 2$, equivalently $\rho \geq \kappa/2-2$, then $Z$ does not hit $0$ and the $\SLE_\kappa (\rho)$ does not hit this half-line.

Similarly, one can define also the classical SLE$_\kappa (\rho)$ from $0$ to $\infty$ in $\HH$ with marked point at $0^+$ or at $x \in (0,\infty)$, which now lies \emph{on the right} of the starting point,
by choosing $Z$ to be $-\sqrt{\kappa}$ times an instantaneously reflecting Bessel process.

The continuity of the trace of ordinary $\SLE$ was proved by Rohde and Schramm in \cite{rs2005basic}.  By a Girsanov theorem \cite{ry2004martingale} argument using absolute continuity, it follows that the $\SLE_\kappa(\rho)$ processes are continuous in the intervals of time in which they are not hitting the boundary.  In particular, they are continuous for all $\rho \geq \kappa/2-2$.  The continuity of all of the classical $\SLE_\kappa(\rho)$ processes was established using the GFF in \cite{ms2012ig1}.

We note that the coupling results of the $\SLE_\kappa(\rho)$ processes with the GFF in \cite{ms2012ig1} only deal with classical $\SLE_\kappa (\rho)$ processes. One of the outcomes of the present paper will precisely be to shed some light on this coupling also in the generalized cases that we will describe next.

Let us summarize some of the special ranges of $\rho$ values for classical $\SLE_\kappa(\rho)$ for $\rho > -2$. Here, $\partial$ will refer to the boundary half-line
of $\partial \h$ between the marked point and infinity that does not contain the origin:  
\begin{itemize}
\item $\rho \in (-2,\kappa/2-4]$: The process fills $\partial$.  This phase is non-empty only for $\kappa > 4$.
\item $\rho \in ( \max (\kappa/2-4, -2),\kappa/2-2)$: The process hits $\partial$, but it bounces off and does not fill $\partial$. This phase is non-empty for every positive value of $\kappa$.
\item $\rho \geq \kappa/2-2$: The process does not hit $\partial$.
\end{itemize}
See Figure~\ref{fig:slekrhosketch} for an illustration of the middle phase in the case $\kappa > 4$.

The processes $\SLE_\kappa (\kappa-6)$ are special because this is the value of $\rho$ for which the law of the process does not depend on its target point (here which point is chosen to correspond to infinity (see, e.g., \cite{dub2007commutation,sw2005coordinate}). We will reexplain this in the generalized setting. 

\subsection{Generalized $\SLE_\kappa (\rho)$ processes}

We now list the ways to generalize these definitions in order to treat also the case where $\rho \le -2$.  As in the case of classical $\SLE_\kappa(\rho)$, the evolution of the chordal Loewner driving function $W$ is described by a pair $(W,O)$.  These processes, however, have a different character than the classical $\SLE_\kappa(\rho)$ processes in that the set $\{t : W_t = O_t\}$ of collision times of $W$ and $O$ does not coincide with the set of times in which the process is swallowing a point on the domain boundary.  The so-called \emph{trunk} of such a process is the set of points visited at a collision time.  We will show in this article that for certain ranges of $\rho$ values, the trunk of such processes can be understood as a continuous curve.  Moreover, the law of such a process can be sampled from by first sampling a continuous $\SLE$-type curve which corresponds to the trunk and then sampling $\SLE$-type loops which hang off the trunk.  See Figure~\ref{fig:trunksketch} for an illustration 
of the general picture when $\kappa \leq 4$.

In the next few paragraphs, we will describe all the natural generalized $\SLE_\kappa (\rho)$ processes (with one marked boundary point) -- which corresponds to the range $-2 \ge \rho > -2 - (\kappa/2)$.  Note that $\rho=-2-(\kappa/2)$ corresponds to $\delta = 0$. 
It is however worth mentioning already that the results in the present paper will provide insight into these processes only for $\rho < (\kappa/2) - 4$ (which means in particular that $\kappa > 2$).
As we shall see, this includes a number of important cases, but there will still be 
a range of values of $(\kappa, \rho)$ that is not treated here. This is the ``light-cone regime'' (where $\max (-2 - (\kappa/2), (\kappa/2 -4)) \le \rho \le -2$ and $\kappa \in (0,4)$) that will be 
studied in \cite {ms2016lightcone}.

\subsubsection{Symmetric side swapping}
\label{subsubsec:sym_side_swapping}

The first possible extension is to allow the process to choose at random on which side the marked point is with respect to the tip of the curve.  In this case, $Z$ takes on both positive and negative values corresponding to when the marked point is to the left or to the right of the tip of the curve.

The following side-swapping construction will work for all $\kappa >0$ and $\rho > -(\kappa/2) - 2$.  One samples a multiple of a Bessel process of dimension $\delta$ as in~\eqref{eqn:sle_kp_bes_dim} and then modifies it to yield $Z$ by tossing an independent fair coin for each excursion of $Z$ away from the origin in order to decide its sign.  Then, one makes sense of the integral of $1/Z$ up to a given time $t$ (for instance, by taking the limit when $\eps \to 0$ of the contribution to this integral of all $Z$-excursions of length at least $\eps$, and one sees that this converges thanks to the cancellation induced by the random signs.  The integral, in particular, converges in the case that $\rho \leq -2$ even though the integral of $1/|Z|$ blows up.  As in the case $\rho > -2$, one defines $W_t:= Z_t - 2 \int_0^t ds / Z_s$.  We refer to the resulting Loewner chain as a symmetric side-swapping $\SLE_\kappa (\rho)$ process, or as the $\SLE_\kappa^{0} (\rho)$ process.  This is a member of the family of side-swapping $\SLE_
\kappa^\beta(\rho)$ processes, $\beta \in [-1,1]$, that we will introduce in full generality just below.

Note that, as explained for instance in \cite{ww2013conformally}, the times at which the process $Z$ hits the origin does not necessarily correspond to times at which the Loewner chain hits the real line (when $\kappa \in (8/3, 4]$, one constructs exactly the $\CLE_\kappa$ loops in this way, and these loops do not touch the real line).  Our construction of $\SLE_\kappa^\beta(\rho)$ given later in this paper will shed further light on this.

\subsubsection{L\'evy compensation}
There is another natural  way to generalize $\SLE_\kappa(\rho)$ processes to values of $\rho$ in 
$(-\kappa/2 -2, -2)$.  Note that in this case, by~\eqref{eqn:sle_kp_bes_dim}, the dimension $\delta$ of the corresponding Bessel process is in $(0,1)$.  As opposed to the  construction explained in Section~\ref{subsubsec:sym_side_swapping}, this  generalization does not work for $\delta=1$, i.e., for $\rho = -2$.  It will also be the case that the marked point always stays on the same side of the tip of the SLE, as in the case of the classical $\SLE_\kappa (\rho)$.  The details of this construction are explained, for example, in \cite{she2009cle, ww2013conformally}.  See also \cite[Chapter~XI]{ry2004martingale} for a discussion of Bessel processes with $\delta \in (0,1)$.

The starting point for the construction is a multiple of a non-negative Bessel process $Z$ of dimension $\delta$ as before, but one now overcomes the difficulty that the integral $\int ds/Z_s$ can blow up  (because $\delta \le 1$) by adding add a ``compensation/renormalization'' which serves to prevent the explosion of this integral.  The process $W_t$ obtained in this way satisfies the SDE $ dO_t:= d (Z_t - W_t) = - 2dt/ Z_t $ at all times where $Z_t \not= 0$.   We will refer to the corresponding process as the \emph{totally asymmetric generalized $\SLE_\kappa^{1} (\rho)$ process}.

Similarly, if we take the symmetric image, i.e. starting with a negative multiple of a non-negative Bessel process $Z$, we will get what we call the {\em totally asymmetric 
generalized $\SLE_\kappa^{-1} (\rho)$ process.} 

Note that when $\rho >-2$, no  compensation is needed in these constructions because in this case the $1/|Z|$ is integrable up to any fixed time.  Consequently, $\SLE_\kappa^{-1} (\rho)$ 
and SLE$_\kappa^1 (\rho)$ for $\rho > -2$ correspond to the classical $\SLE_\kappa (\rho)$ processes, with marked points on the right and on the left respectively.

\begin{figure}[ht!]
\begin{center}
\includegraphics [width=4in]{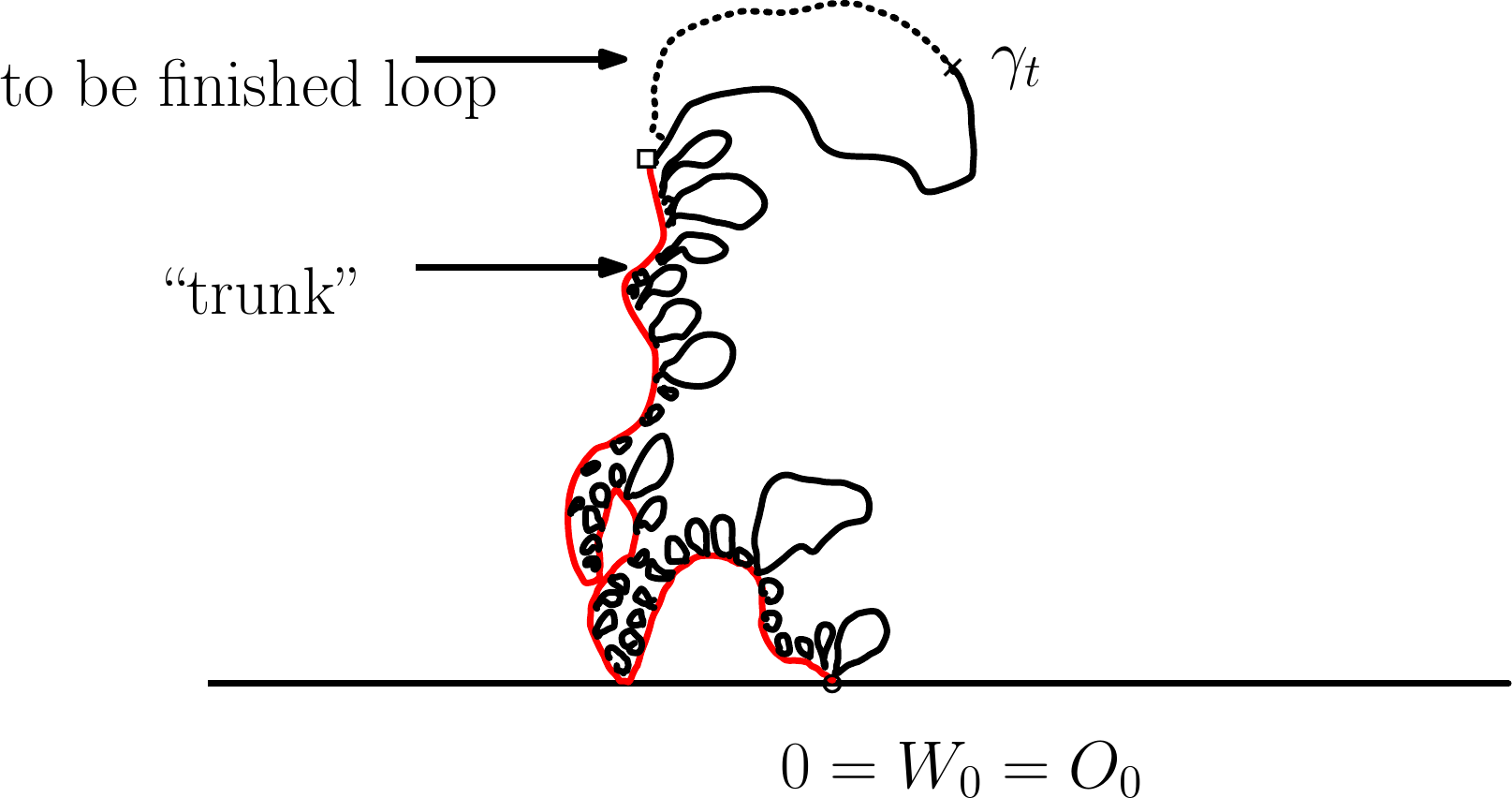}
\caption{\label{fig:trunksketch}A sketch of the conjectural trunk and loops traced by $\SLE_\kappa^1 (\rho)$ (for $\kappa \in (0,4)$ and $\rho < -2$). NB. It will follow from 
our results that the trunk looks itself more like a (non-simple) $\SLE_{\kappa'}$ curve.}
\end{center}
\end{figure}

\subsubsection{Asymmetric side-swapping}
One can also interpolate between the previous two constructions when $\rho \in (-\kappa/2 -2, -2)$, and allow the process to choose at random on which side the marked point is with respect to the tip of the curve, but with a biased coin.  One chooses a parameter $\beta \in [-1,1]$, and first modifies the multiple $Z$ of a Bessel process by tossing an independent coin for each excursion of $Z$ away from $0$ in order to decide its sign.  The sign of a given excursion is positive with probability $(1+\beta)/2$ and negative with probability $(1-\beta)/2$.  Then, one proceeds exactly as before (define $W_t$ using the same formula as above, and use a compensation in the case where $\rho < -2$ to make sense of the integral of $1/Z$). The choice of the sign of the excursion of $Z$ decides  whether the process tries to grow to the right of the marked point or to its left when the tip and the marked point coincides.  This then defines a new Loewner chain: The $\SLE_\kappa^\beta (\rho)$ process. This definition indeed 
interpolates between
the previous symmetric (corresponding to $\beta=0$) and totally asymmetric cases (corresponding to $\beta= \pm 1$).  For details about these side-swapping $\SLE_\kappa^\beta (\rho)$ processes, we again refer to \cite{she2009cle, ww2013conformally}.

\subsubsection{The case $\rho=-2$}
In the case where $\rho=-2$, the symmetric $\SLE_\kappa^0 (\rho)$ definition works, but none of the asymmetric ones does. It is however possible (and this extension is specific to this $\rho=-2$ case) to introduce an asymmetry by introducing an additional drift in the driving process of the symmetric side-swapping process, that is equal to $\mu$ times the local time at the origin of the underlying Bessel process (which is in fact a Brownian motion because $\beta=0$). This gives rise to 
a process denoted by  $\SLE_\kappa^{0, \mu} (-2)$. For details, we refer again to \cite{she2009cle, ww2013conformally}.

\subsubsection{Characterization of $\SLE_\kappa(\rho)$ processes}

The following simple conformal Markov characterization of the $\SLE_\kappa^\beta(\rho)$ processes will be useful later on:  

\begin{lemma}
\label{lem:slekrcharacterization}
Suppose that we have a pair $(W,O)$ of continuous processes with $W_0 = O_0 = 0$ which together form a  Markov process such that the following conditions hold: 
\begin {itemize}
 \item There exists $\rho$ such that during the intervals in which $W - O \not= 0$, the process $|W-O|$ evolves as $\sqrt{\kappa}$ times a Bessel process of dimension $\delta$ as in~\eqref{eqn:sle_kp_bes_dim} and $dO_t = 2 dt/ (O_t - W_t)$.
 \item For each $t >0$, if $\tau$ denotes the first time after $t$ at which $W-O$ hits $0$, then the conditional law of $(W_{\tau+s}-W_\tau, O_{\tau+s} - O_\tau)_{s \ge 0}$ given the information up to the stopping time $\tau$ is equal to the (unconditional) law of $(W, O)$. 
 \item The process $(W,O)$ satisfies Brownian scaling.  
 \end {itemize}
Then, if $\rho \not=-2$, there exists $\beta \in [-1,1]$ such that $(W,O)$ generates an $\SLE_\kappa^\beta(\rho)$ process, and if $\rho = -2$, there exists $\mu \in \R$ such that $(W,O)$ generates an $\SLE_\kappa^{0, \mu} (-2)$.
\end{lemma}
\begin{proof}
When $\rho \not= -2$, the scale invariance of $(W,O)$ implies that $|W-O| / \sqrt {\kappa}$ is a Bessel process of dimension $\delta$ that is instantaneously reflected at $0$. 
The strong Markov property at the stopping times $\tau$ implies readily that the signs of the excursions that $W-O$ makes away from $0$ are i.i.d., so that $ W-O$ is distributed like the $\beta$-side-swapping process $Z$ described 
above.  We then further define from this process $Z=W-O$ the continuous process $\tilde O_t := 2 \int_0^t ds / Z_s$ defined in the previous generalized sense, and we note that the first condition in 
the lemma implies that the continuous process $ O - \tilde O$ is constant during the excursions of $W-O$ away from $0$, and the last two properties imply  readily that $O=\tilde O$ at all times.
The proof for $\rho=-2$ follows the same general lines.
\end{proof}

\subsection{Further discussion}

In all these definitions of $\SLE_\kappa (\rho)$ processes for $\rho > -2 - (\kappa/2)$, it is possible to start with the marked point $O_0$ equal to the origin. Then, the trace of all these $\SLE_\kappa (\rho)$ processes is scale-invariant in distribution which makes it  possible to define them in any other simply connected domain by conformal invariance.  These families of generalized $\SLE_\kappa (\rho)$ processes are important, as they form the only Loewner chains with continuous driving functions that satisfy  scale-invariance and their Markovian property, see \cite{she2009cle, ww2013conformally} and recall Lemma~\ref{lem:slekrcharacterization}. 
This in turn makes it possible by conformal invariance to define the corresponding $\SLE_\kappa (\rho)$ targeting any given fixed point $x$ on the real line instead of infinity (as the image of the previous one by any given M\"obius transformation of the upper half-plane that fixes the origin and maps $\infty$ to $x$), and more generally $\SLE_\kappa (\rho)$ from one boundary point of a simply connected domain to another.  

One consequence of the characterization of the $\SLE_\kappa (\rho)$ as the only processes with the conformal Markov property with one additional marked point is that if one considers an $\SLE_\kappa$ from $0$ to $z \in \R_+ \setminus \{0\}$, and views it ``parameterized'' from infinity, then it is in fact an $\SLE_\kappa (\rho)$ from the origin to infinity, with marked point at $z$ (at least up to the time at which the SLE disconnects $z$ from infinity) for some $\rho$. A simple computation shows that this value is $\kappa - 6$ (see also \cite{dub2007commutation,sw2005coordinate}; this can also be derived using the SLE/GFF coupling and the change of coordinates rule \cite{ms2012ig1}).  This explains 
the following target-invariance property of all the generalized $\SLE_\kappa (\kappa -6)$ processes that will play a very important role in the present paper: 
When  $\kappa > 8/3$ (so that $\kappa - 6 > -2 - (\kappa/2)$), the laws of a given generalized $\SLE_\kappa (\kappa -6)$ targeting two different points
$z$ and $z'$ coincide up to their first disconnection time of $z$ from  $z'$.

In the sequel, for $\kappa \in (8/3, 8)$, we will use the acronym $\bSLE_\kappa^\beta$ and $\bSLE_4^{0,\mu}$  (for ``branchable $\SLE_\kappa$'') for these generalized SLE$_\kappa (\kappa -6)$ 
processes. Again, the nature of these bSLE processes (and the corresponding bSLE tree that we will now discuss) will be quite different depending on whether $\kappa> 4$, $\kappa=4$ and $\kappa <4$  (because  the $\SLE_\kappa$ is then not a simple curve anymore when $\kappa > 4$, and also because $\kappa - 6 < -2$ when $\kappa < 4$).

As explained in \cite{she2009cle}, the target-invariance property enables us to define for each branching tree of these generalized $\bSLE_\kappa$ processes a random collection of loops as follows. 
When $\kappa \not= 4$ and $\beta \in [-1,1]$, choose a boundary point $x$ of $D$ and from $x$ launch a branching tree of $\bSLE_\kappa^\beta$ targeting any point in $D$.
Each point $z$ of $D$ will almost surely be surrounded by a loop corresponding to an excursion of the Bessel process and created by the part of this branching tree that targets $z$. Call $\gamma(z)$ this first loop. We then consider the countably family of loops surrounding say the points with rational coordinates.

It turns out that the law of this family depends only on $\kappa$, but neither the choice of $\beta$ (or of $\mu$ when one considers $\bSLE_4^{0,\mu}$ instead of $\bSLE_\kappa^\beta$) nor  of 
the choice of $x$. This random collection of loops is called a $\CLE_\kappa$ (it is sometimes referred to as the \emph{branching tree definition of $\CLE$}).

To see that the law of the obtained loops does not depend on $x$ is far from trivial. In the case $\kappa' \in (4,8)$, this is stated in \cite[Theorem~5.4]{she2009cle} conditionally on the reversibility and the continuity (i.e., the fact that it is generated by a continuous curve) of $\bSLE_\kappa $, and these two facts have since then been derived in \cite{ms2012ig1,ms2012ig2,ms2012ig3}.  In the case where $\kappa \in (8/3, 4]$, prior to the present article, the only existing proof builds on a different set of tools and techniques (in particular the Brownian loop-soup) for $\kappa \in (8/3, 4]$ in \cite{sw2012cle}.   As a consequence of our analysis here, we will obtain a new proof of this statement by reducing it to the reversibility of $\SLE$ processes.

The fact that the law of the obtained collection of loops does not depend on $\beta$ (or $\mu$) 
 is explained in \cite{ww2013conformally} in the case $\kappa \in (8/3, 4]$, and we will give some details about the case $\kappa>4$ in a few paragraphs.  In fact, when $\kappa \not=4$, a $\CLE_\kappa$ loop that is traced via the $\bSLE_\kappa^\beta$ branching tree is either traced counterclockwise or clockwise 
 (loosely speaking, this  corresponds to 
 whether this loop corresponds to a portion of the exploration where the marked point is on the left or or the right of the tip). For instance, all the loops traced
 by a $\bSLE_\kappa^1$ will be counterclockwise loops and all loops traced by a $\bSLE_\kappa^{-1}$ will be clockwise loops. 
 The same argument shows that in fact, for this construction, the conditional law of the orientations of these loops given the $\CLE_\kappa$ is given by i.i.d.\ $(1+\beta)/2$ versus $(1-\beta)/2$ coin tosses. This $\CLE_\kappa$ with i.i.d.\ orientations will be referred to as a $\CLE_\kappa^\beta$ in this paper.

Another consequence of the characterization of $\SLE_\kappa (\rho)$ processes as the only chains with the conformal Markov property with one additional marked point goes as follows: 
 
Suppose that one considers an $\SLE_\kappa (\rho)$ from one boundary point $a$ to another boundary point $b$ with marked point at $c$ (where $a$, $b$ and $c$ are different boundary points of the simply connected domain $D$ -- we choose them here to be ordered counterclockwise on $\partial D$, so this SLE$_\kappa (\rho)$ has its marked point on the left). Then, up to the time at which the Loewner chain disconnects $b$ from $c$, it satisfies also the conformal Markov property with one additional marked point when one swaps the role of $b$ and $c$ and views it as targeting $c$. It is therefore an $\SLE_\kappa (\wt \rho)$ from $a$ to $c$ with marked point at $b$ (therefore, on the right), and it is rather simple to check that $\wt \rho = \kappa - 6 - \rho$ (see for instance \cite {dub2007commutation,sw2005coordinate}). 
Note that when $\rho =0$, then $\wt \rho =  \kappa -6$ which is no surprise given that in this case, this is the very same question as the previous one. 
When one now introduces a {fourth boundary} point $d$ located in the arc from $b$ to $c$ in $\partial D$ that does not contain $a$, one can also view the $\SLE_\kappa (\rho)$ from $a$ to $b$ with marked point at $c$ from there. One way to describe it is via the $\SLE_\kappa (\rho_1; \rho_2)$ processes with two marked points -- it is an $\SLE_\kappa (\rho; \kappa-6-\rho)$ process from $a$ to $d$ with marked points at $b$ and $c$ (here $\rho$ corresponds to the marked point on the left, which is $c$, and $6 -\kappa - \rho$ corresponds to the marked point 
on the right, which is $b$). Clearly, $d$ plays no role in the law of this chain: The $\SLE_\kappa (\rho; \kappa-6-\rho)$ process is therefore target-invariant.
These processes will be important in our definition of boundary conformal loop ensembles.

\subsection{Further remarks on exploration trees and side-swapping}

Let us now make a few further remarks about the case in which $\rho >-2$ but where 
one is using the $\beta$-side-swapping construction.

In this case, let us recall that the construction of $\SLE_\kappa^\beta (\rho)$ is very direct. One starts with the multiple of a Bessel process with dimension $\delta > 1$ and one chooses at random and independently for each excursion away from the origin, whether it is positive or negative,  using a $(1+\beta)/2$-coin. Then, as the integrals 
$\int_0^t ds / Z_s$ are absolutely convergent, there is no problem to define 
$W_t = Z_t - 2 \int_0^t ds / Z_s$ and the corresponding Loewner chain. 

Suppose now that for each positive $\eps$, we have a way that is measurable with respect to the filtration generated by the Poisson point process of excursions in order to decide whether the excursion of $Z$ is going to be forced to be positive or whether one uses a coin-toss (here we are allowed use information about this excursion in order to decide this), then we can define another driving function $W^\eps$ in the very same way. A typical example is for instance to only toss a coin for the excursions of $Z$ of time-length at least $\eps$ and to decide that all others are positive. We will also describe another useful cut-off procedure based on the diameter of completed $\SLE_\kappa$ loops later on. Then, as $\eps \to 0$, the process $W^\eps$ converges almost surely (provided the scheme is consistent as $\eps \to 0$) and uniformly on any compact time-interval to $W$. 

This type of argument will for instance be useful in the case of $\bSLE_{\kappa'}^\beta$ processes for $\kappa' > 4$. 
It follows in particular readily that for such a process started at the origin, for any given $x < 0 < y$, the probability that the $\eps$-approximation of the $\bSLE_{\kappa'}^\beta$ disconnects $x$ from infinity before disconnecting $y$ from infinity does tend to the probability that the actual $\bSLE_{\kappa'}^\beta$ does disconnect~$x$ from infinity before~$y$.

The $\bSLE_\kappa$ process has a very simple radial version: If we are given a point $z$ in $\HH$, we can first follow a $\bSLE_\kappa$ (from $0$ to $\infty$) up to the first time at which it disconnects $z$ from $\infty$, and at this point, instead of continuing in the connected component of the complement that contains infinity, one continues using a $\bSLE_\kappa$ targeting a boundary point of the connected component that contains $z$ (here one can choose this boundary point a little before the disconnection time, and use the target-invariance), and so on. This is the radial $\bSLE_\kappa$ process from $0$ to $z$. Again, when $\rho > -2$, there is no difficulty in defining the side-swapping    
version of these radial processes (and in fact, also for all generalized ones as well). The observation about cut-offs that we have just made in the previous paragraph can be easily generalized to this radial case. This will play an instrumental role in the proof of Proposition~\ref{prop:unique_interface}.

{\bf Notation warning.}
In the sequel, as in the introduction, we will often have to treat separately the cases $\kappa \le 4$ and $\kappa \ge 4$. When this is the case, we will specify 
 at the beginning of the 
(sub)section the range of values under consideration and we will use the notation $\kappa \le 4$ and $\kappa' = 16/ \kappa \ge 4$.  
However, we will also occasionally (as in this past section for instance) want to make statements that are valid in the entire range $(2,8)$ or in the range $(8/3, 8)$. In this case, we will use 
the symbol $\kappa$ and emphasize that those statements are valid for some values of $\kappa$ greater than $4$.

\section{Conformal percolation in $\CLE_\kappa^\beta$ carpets for $\kappa< 4$}
\label{sec:cpi_carpet}
\label {Sec4}

In the present section and the next one, we will restrict ourselves to the $\CLE_\kappa$ carpets. We will first focus on the case where $\kappa \in (8/3, 4)$, and we will 
study the special case where $\kappa =4 $ in the next section.  Some of the arguments that we will give in the present section will also be valid for $\kappa =4$. 

We will mostly use here the definition of $\CLE$ carpets via the aforementioned exploration tree \cite{she2009cle} (see e.g.\ \cite{ww2013conformally} 
for more aspects of the SLE tree construction). Recall that because $\bSLE_\kappa$ is an $\SLE_\kappa (\rho)$ with $\rho = \kappa -6 $ which is smaller than $-2$, this 
process involves  side-swapping and/or L\'evy compensation. 
It is known that the $\bSLE_\kappa^\beta$ is a random Loewner chain from $0$ to infinity (if one chooses the target point to be $\infty$ in the upper half-plane) that traces simple disjoint loops (that also do not touch the real line -- the 
derivation of the fact that these are proper loops follows from the loop-soup construction) on the way (and these loops correspond to the excursions of the corresponding Bessel process in the branching SLE construction).
So, the intuitive structure is that one has proper loops hanging off a ``trunk'' (this trunk corresponds to the moments where the Bessel process in the construction of the $\bSLE_\kappa^\beta$ is equal to $0$; see Figure~\ref{fig:cpi} for an illustration).  However, at this point 
of the paper, we have not yet shown that this trunk is almost surely a continuous path.  The goal of the present section is to derive the following result on \hyperref[def:cpi]{CPIs} in $\CLE_\kappa^\beta$ as defined in Section~\ref{subsec:cpi}:
\begin{prop}
\label{prop:CME_characterization}
\begin{enumerate}[(i)]
 \item 
There exists at most one \hyperref[def:cpi]{CPI} distribution in $\CLE_\kappa^\beta$ for each given choice of $\beta \in [-1,1]$ and $\kappa \in (8/3, 4)$. The \hyperref[def:cpi]{CPI} path has then necessarily the same law as a $\bSLE_\kappa^\beta$ process viewed only at those times at which it does not trace an $\SLE_\kappa$ loop (and the traced loops are labeled loops of the corresponding labeled $\CLE_\kappa$). 
\item
Conversely, if $\bSLE_\kappa^\beta$ is almost surely a continuous path, then its trunk (which is the subpath corresponding to the times at which it is not tracing a $\CLE_\kappa$ loop) is a \hyperref[def:cpi]{CPI} in the corresponding $\CLE_\kappa^\beta$.
\end{enumerate} 
\end{prop}

 \begin{figure}[ht!]
\begin{center}
\includegraphics [width=3in]{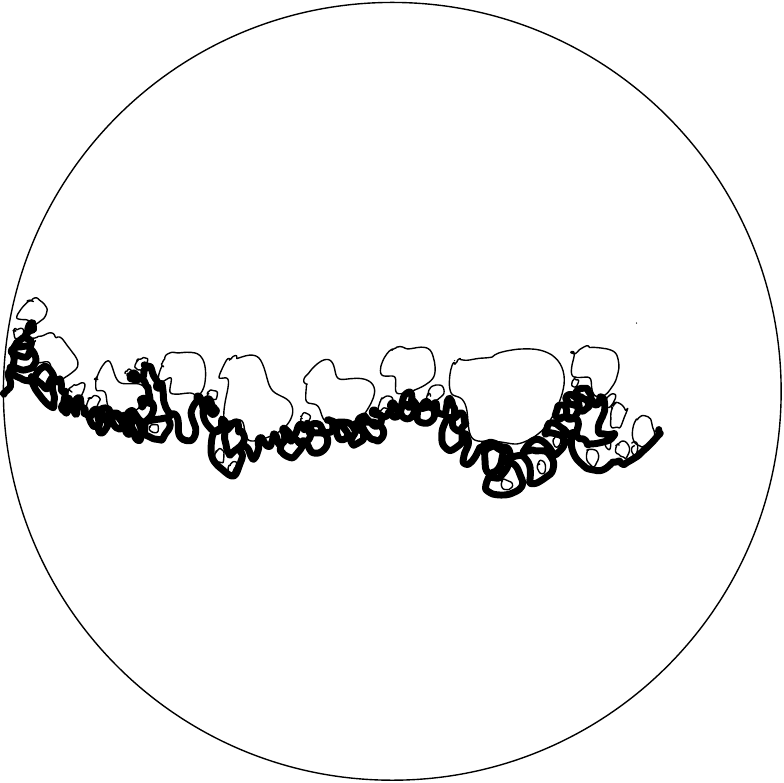}
\includegraphics [width=3in]{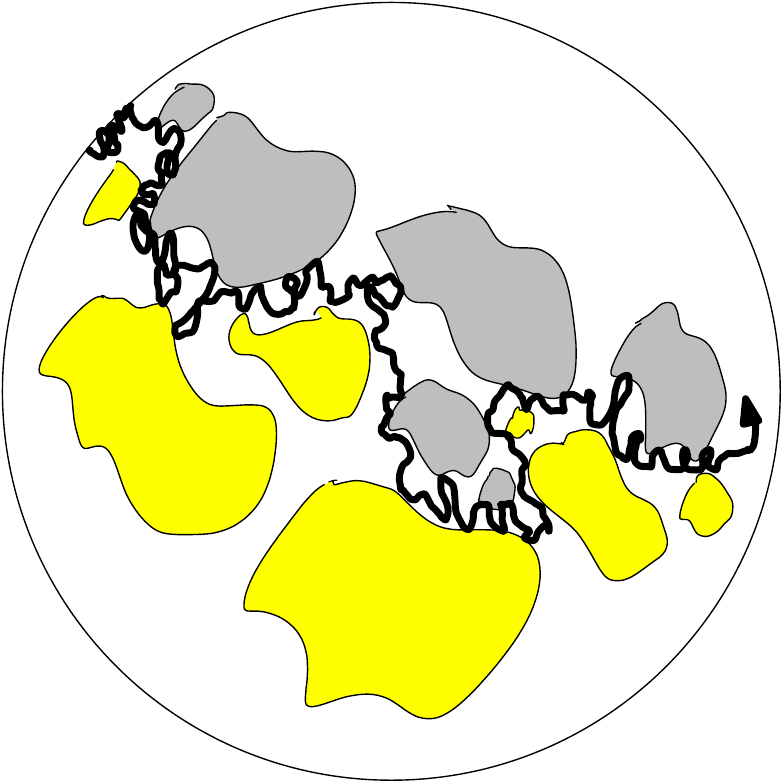}
\caption{\label{fig:cpi} If the \hyperref[def:cpi]{CPI} in $\CLE_\kappa^\beta$ exists then it is a trunk of a $\bSLE_\kappa^\beta$.  {\bf Left:} Sketch for $\beta=-1$.  {\bf Right:} Sketch for $\beta = 0$.}
\end{center}
\end{figure}

Note that this proposition does not yet state the existence of such \hyperref[def:cpi]{CPIs} for $\kappa \in (8/3, 4]$ (we only say ``at most'') because we have not yet shown at this point that the $\bSLE_\kappa^\beta$ traces are almost surely  continuous paths.  We will however prove this in Theorem~\ref{thm:duality2} and 
we will furthermore describe the law of 
the trunk of a $\bSLE_\kappa^\beta$ in terms of $\SLE_{\kappa'}$-type processes. 
Combined with this proposition, this will give a rather detailed description of all of the  \hyperref[def:cpi]{CPIs} in $\CLE_\kappa^\beta$ carpets for $\kappa \in (8/3, 4)$. 
We will in particular also describe the conditional distribution of the $\CLE_{\kappa}^\beta$ given the 
\hyperref[def:cpi]{CPI} (when one first samples the \hyperref[def:cpi]{CPI} as an $\SLE_{\kappa'}$ type process, and then traces the $\CLE_\kappa$ loops that it encountered).  

In the subsequent paper \cite{msw2016randomness}, we shall prove that when $\kappa \in (8/3,4)$, the \hyperref[def:cpi]{CPI} are not deterministically determined from the labeled CLE, meaning that it captures additional randomness that is not present in the labeled CLE. This contrasts with the case $\kappa=4$ that we will discuss in the next section.

In order to prove this proposition, we will use ideas in the spirit of the $\CLE$ properties studied in \cite{sw2012cle}.
Let $\Gamma$  now denote a $\CLE_\kappa^\beta$ in the upper half-plane, and let $\gamma$ be a \hyperref[def:cpi]{CPI} in $\Gamma$ from $0$ to $\infty$ (by conformal invariance, it is enough to consider this case). We use the same notation as in Section~\ref{subsec:cpi}.
We parameterize the continuous curve $\gamma$ in some way (for instance by half-plane capacity, would work here), and we can define the ordered family $(L_{t_i}):= (\varphi_{t_i-}^{-1} ( l_{t_i}))$, where the  discrete set $(t_i)$ denotes the times at which the \hyperref[def:cpi]{CPI} hits a $\CLE$ loop (that we then call $\CL_i$)
for the first time, and $\varphi_t$ is defined as in Section~\ref{subsec:cpi}.
We also denote the orientation (clockwise or counterclockwise) of $\CL_i$ by $s(\CL_i)$ and define also $s(t_i)= s(\CL_i)$.

\begin{lemma}
\label{lem:loop_distribution}
The distribution of the ordered collection $(L_{t_i}, s ({t_i}))$ is equal to that of the ordered family associated with a Poisson point process with intensity given by the $\SLE_\kappa$ bubble measure and an independent sign which is chosen to be counterclockwise with probability $(1+\beta)/2$ and clockwise with probability $(1- \beta)/2$.
\end{lemma}
\begin{proof}
This proof will be based on $\CLE$ exploration arguments developed in \cite{sw2012cle}.
A first observation is that the \hyperref[def:cpi]{CPI} property implies that for all stopping times $\tau$, the conditional distribution of the ordered collection of $(L_{t_i}, s(t_i))$ corresponding to $t_i > \tau$, given all the information provided by what one has discovered until~ $\tau$ is always the same. In other words, this  ordered collection of to be discovered labeled bubbles is 
independent of those discovered so far. This implies that it is distributed like the ordered collection  defined via a Poisson point process. It therefore remains to identify the intensity measure of this Poisson point process.

Recall that the definition of a \hyperref[def:cpi]{CPI} implies that almost surely, for all $\eps >0$, the curve $\gamma$ up to its first hitting time $\tau_\eps$ of the circle of radius $\epsilon$ does intersect the upper half-plane and has a positive half-plane capacity. This makes it possible to adapt the arguments of \cite{sw2012cle} to prove that when $\eps \to 0$, the law of the labeled $\CLE$ loop that surrounds a given point $z$, conditioned on the event that $\gamma [0, \tau_\eps]$ intersects this loop, does converge to the $\SLE_\kappa$ bubble measure restricted (and then renormalized to make a probability measure) to surround $z$, defined in \cite{sw2012cle}, with an independent $p_0$ versus $1-p_0$ labeling. 

The idea is to proceed exactly as in the proof of \cite[Proposition~4.1]{sw2012cle}.
For small $\eps$, we are going to iterate the following experiment: Choose a boundary point in the upper half-plane, launch a \hyperref[def:cpi]{CPI} from that point until it hits the
circle of radius $\eps$ around this starting point (and collect all $\CLE$ loops encountered). Consider the connected component that contains $z$ of the complement of the traced CPI and  the discovered loops, and map back this domain 
conformally onto the upper half-plane leaving $z$ fixed, and look at the image under this map of the collection of $\CLE$ loops that remain to be discovered in this domain.
The \hyperref[def:cpi]{CPI} property ensures that on the event where the CLE loop that surrounds $z$ has not been discovered yet, the conditional law of the collection of loops
obtained in this way is that of a $\CLE$, which makes it possible to iterate the same experiment again. 

As explained in \cite{sw2012cle}, by choosing iteratively the starting points of the explorations in an appropriate way, one can approximate (taking $\eps$ small) the deterministic exploration of all $\CLE$ loops 
discovered along some deterministic slit from $0$ to $z$, and this leads, exactly as in \cite{sw2012cle} to the fact that ``at the iteration step at which one discovers the loop that surrounds $z$'' and in the $\eps \to 0$ limit, it is distributed
as an ``$\SLE_{\kappa}$ bubble measure conditioned to surround $z$'' in the remaining to be discovered domain (and the fact that its label is independently chosen follows from our construction).  

All this argument is non-trivial, but it is a very direct adaptation of the proof of \cite[Proposition~4.1]{sw2012cle} with no other ingredient (one just replaces each iteration step by the discovery of a \hyperref[def:cpi]{CPI} with the loops that it intersects instead of the half-disk and the loops that it intersects)  so that we refer to that proof for details.
\end{proof}

We can now turn to the actual proof of Proposition~\ref{prop:CME_characterization}:

\begin{proof}[Proof of Proposition~\ref{prop:CME_characterization}]
Lemma~\ref{lem:loop_distribution} suggests that a \hyperref[def:cpi]{CPI} will necessarily be distributed like the trunk of a $\bSLE_\kappa^\beta$ from $x$ to $y$ (if this trunk exists and if it is a continuous curve) used to construct the $\CLE_\kappa^\beta$.

We consider the path $\eta$ obtained by gluing to the path $\gamma$ the loops of the $\CLE_\kappa^\beta$ that it discovered and decide to trace each of these loops following their orientation. 
The path $\gamma$ passes  to the left (resp.\ right) of this loop (on its way from $x$ to $y$) if the loop is traced counterclockwise (resp.\ clockwise). This path $\eta$ is then a continuous (because the $\CLE$ loops are locally finite)
non-self-crossing (because of this clockwise versus counterclockwise choice) path from $x$ to $y$, and the fact that $\gamma_t^*$ is increasing implies in fact that $\eta$ (when seen from $y$) can be defined via a Loewner chain with a continuous driving function (the Loewner chain is generated by the non-self crossing $\eta$ and the hull generated by this curve is increasing, so the driving function is just the conformal image of the tip of the curve).
Our goal is to prove that $\eta$ is necessarily a $\bSLE_\kappa^\beta$ process.

First, Lemma~\ref{lem:loop_distribution} implies that at those times when $\eta$ is away from $\gamma$, it evolves exactly like an $\SLE_\kappa$ process targeting the last point of $\gamma$ that it has visited, because of the definition of the $\SLE_\kappa$ bubble measure.  Viewed as a process targeting $\infty$ (just via change of variables), it means that it evolves like an $\SLE_\kappa (\kappa -6)$ when it is away from $\gamma$. Furthermore, when one discovers 
a loop, one tosses a $(1+\beta)/2$ versus $(1-\beta)/2$ coin to decide its orientation (i.e., to choose the sign of the corresponding excursions of the multiple of the Bessel process $(W_t-O_t)$).  

Second, scale-invariance and a zero-one argument (and the Markov property) shows that either the set of capacity-times for $\eta$ during which it traces a $\CLE$ loop is almost surely empty, or it is almost surely dense. It is not difficult to rule out the former case
(i.e.\ to rule out the possibility that  $\gamma$ does not intersect $\Gamma$ at all). For instance, if this would be the case, then by the conformal Markov property, we could first sample the entire path $\gamma$, and then independent CLEs in the connected components of its complement, and the union of the obtained collection of loops that one obtains would be exactly a CLE. If one considers a given point $a$  in the interior of $D$, one can get a contradiction by looking at the law of the conformal radius of the loop that surrounds $a$ in the CLE.

Let us now consider  the process $A_t$ defined to be the total Lebesgue measure of the set of times in $[0,t]$ at which $\eta$ (parameterized by capacity seen from $y$) is on its 
``trunk'' $\gamma$. Because of the conformal Markov property, this is necessarily the inverse of a subordinator, and its jumps are those of a stable process (because of the $\CLE$ property and the scaling property of the bubble measure). But conformal invariance of the whole process implies that the stable subordinator has no drift part (because a non-zero drift would not scale in the same way as the jumps of the subordinator), so that the Lebesgue measure of the set of times at which $\eta$ in on $\gamma$ is almost surely equal to zero.

We can therefore conclude that the path $\eta$ is one of the $\bSLE_\kappa^\beta$ processes because the process $ W_t - O_t$ has to be a sign-swapping Bessel process and the fact that $W$ can have no inverse local time type drift when the Bessel process hits zero.
\end{proof}

Note that this \hyperref[def:cpi]{CPI} process would then also satisfy target-invariance and other properties that we would have also expected from such a continuous percolation interface in CLEs. Indeed, these properties are known to hold for the generalized $\bSLE_{\kappa}^\beta$ processes.

\section{Conformal percolation in the $\CLE_4^0$ carpet}
\label {Sec5}
In the present section, we study \hyperref[def:cpi]{CPIs} in the labeled $\CLE_4$ carpets, and will make use of the relation of $\CLE_4^0$ with the GFF.  This type of GFF-coupling based argument will be instrumental in the remainder of the present paper.   

\subsection{Background and preliminaries on $\SLE_4$, $\CLE_4$, local sets and the GFF}
\label{subsec:cle4_local}

Let us first briefly recall a few features about local sets. This notion, introduced in \cite{ss2010continuumcontour}, was instrumental in \cite{ms2012ig1,ms2012ig2,ms2012ig3,ms2013ig4} and will be important in the present paper as well. When $\Fh_0$ is a harmonic function in a domain $D$, we say that $h$ is a GFF with boundary conditions given by $\Fh_0$ if $h-\Fh_0$ is a (usual) centered GFF with covariance given by the Dirichlet Green's function in $D$.  We use the same normalization as in \cite{ms2012ig1} for the Green's function (another choice would affect some of the constants in what follows). The terminology ``boundary conditions'' comes from the fact that a harmonic function is fully determined by its ``trace on the boundary,'' so that it is in fact enough to specify the latter to define the former. For instance, when~$\Fh_0$ is a harmonic function that extends continuously to the set of the prime ends of $D$, then one can recover~$\Fh_0$ from its boundary values (and we just say that $h$ is a GFF with boundary conditions given by the boundary values of~$\Fh_0$).  

When $U$ is a deterministic open subset of $D$, it is possible to decompose the GFF $h$ into the sum of two independent parts: The projection $h^U$ of $h$ onto the set of generalized functions that are harmonic in $U$ (in other words, this is equal to $h$ in $D \setminus U$ and to the harmonic extension of this generalized function in $U$) and a GFF with zero boundary conditions in $U$. This decomposition can be understood as a generalization of the standard Markov property for Brownian motion or Brownian bridge where the one-dimensional time-set is here replaced by the two-dimensional set $D$. 

Local sets form an important and natural class of random sets in relation to the GFF; they correspond to the random sets for which a ``strong Markov property''  holds: A  random closed set $A$ is said to be local for the GFF $h$ defined on $D$ if there exists a random distribution $h_A$ that has the property that $h_A$ is almost surely continuous and harmonic  in $D \setminus A$, and such that, conditionally on $A$ and $h_A$,  $h - h_A$ is a GFF with zero boundary conditions in $D \setminus A$ (for more information and surveys about this, we refer to \cite{ss2010continuumcontour,ms2012ig1} or \cite{ww2016ln}).   Let us now just briefly review some features of this theory that we shall use here.  We begin with a restatement of part of \cite[Lemma~3.9]{ss2010continuumcontour}, which will be especially important for our later arguments.
\begin{proposition}
\label{prop:local_set_char}
Suppose that $A$ is a random closed subset of $\overline D$ which is coupled with a GFF $h$ on $D$.   If for each given $U \subseteq D$ open, the event  $\{ A \cap U \neq \emptyset \} $ is a measurable function of $h^U$, then $A$ is a local set of the GFF $h$.
\end{proposition}

A useful subclass of local sets are formed by sufficiently small local sets (referred to as  ``thin'' local sets in \cite{asw2016boundedthin,ww2016ln,sep2016thinlocalsets}). Indeed, if we know for instance that for some $d < 2$, the Minkowski dimension of the local set $A$ is almost surely smaller than $d$, then the random distribution $h_A$ is in fact equal to the harmonic function $h_A$ times the indicator function on $D \setminus A$ (in loose words, it carries no mass on $A$).

The local sets that we will consider in the present section are all closely related to the level lines of the GFF.  Recall that $h$ is a distribution and does therefore not take values at points, so 
that it does not have level lines in the literal sense.  But these have been made sense of in \cite{ss2009contours,ss2010continuumcontour} using two different approaches. The first is to define the GFF lines to be the scaling limits of the level lines of the discrete GFF and the second construction is done directly in the continuum and builds on the idea that such a level line should be local for the GFF since changing the field values away from a level line does not affect the level line itself.  The way that this construction proceeds is to first sample a random curve $\eta$ according to some well-chosen law and then construct a distribution on $D$ by sampling a GFF in the complement of $D$ with given boundary conditions and then show that this defines a GFF on all of $D$.  It turns out that the boundary conditions that one should use are $-\lambda$ (resp.\ $\lambda$) on the left (resp.\ right) side of the curve where $\lambda=\pi/2$.  This discrepancy between the field heights between the left and right 
sides of $\eta$ was coined the ``height-gap'' by Schramm and Sheffield in their original article \cite{ss2009contours}.  

More precisely, one can consider an $\SLE_4$ from $-1$ to $1$ in the unit disk $\D$ and define the harmonic function in the complement of the $\SLE_4$ that is equal to $+ \lambda$ in the bottom connected component of the complement of the curve (the one with the lower half of the circle $\partial \D$ on its boundary), and to $- \lambda$ in the top one. Then this is a local set for a GFF $h$ with boundary values equal to $+\lambda$ and $-\lambda$ on the bottom and top half-circles of $\partial \D$. Conversely, and this is not a trivial fact, it turns out that the $\SLE_4$ path can be deterministically recovered from $h$ (and it is therefore sometimes referred to as a zero-level-line of $h$), see \cite{ss2010continuumcontour}.

A variant of the previous result that will be important for our purposes is the following: Choose $\rho \in (-2, 0)$ and define first an $\SLE_4 (\rho ; -2- \rho)$ from $-1$ to $1$ in $\D$. This is a simple continuous curve in the closed unit disk, that touches almost surely both the top and bottom half-circle of $\partial \D$. Then, in each connected component of the complement of the curve (we can either look at the entire curve, or stop it at some stopping time), define the harmonic function with boundary conditions given by: (i) on $\partial \D$, the boundary condition is $0$, (ii) on the ``left-hand'' side of the curve, the boundary condition is $-\lambda + c$, (iii) on the ``right-hand'' side of the curve, the boundary condition is $\lambda + c$, where $c= \lambda ( \rho + 1)  \in (- \lambda, \lambda)$ (see Figure~\ref{sle411} for the case $c=0$).  
\begin{figure}[ht!]
\begin{center}
\includegraphics [width=2.4in]{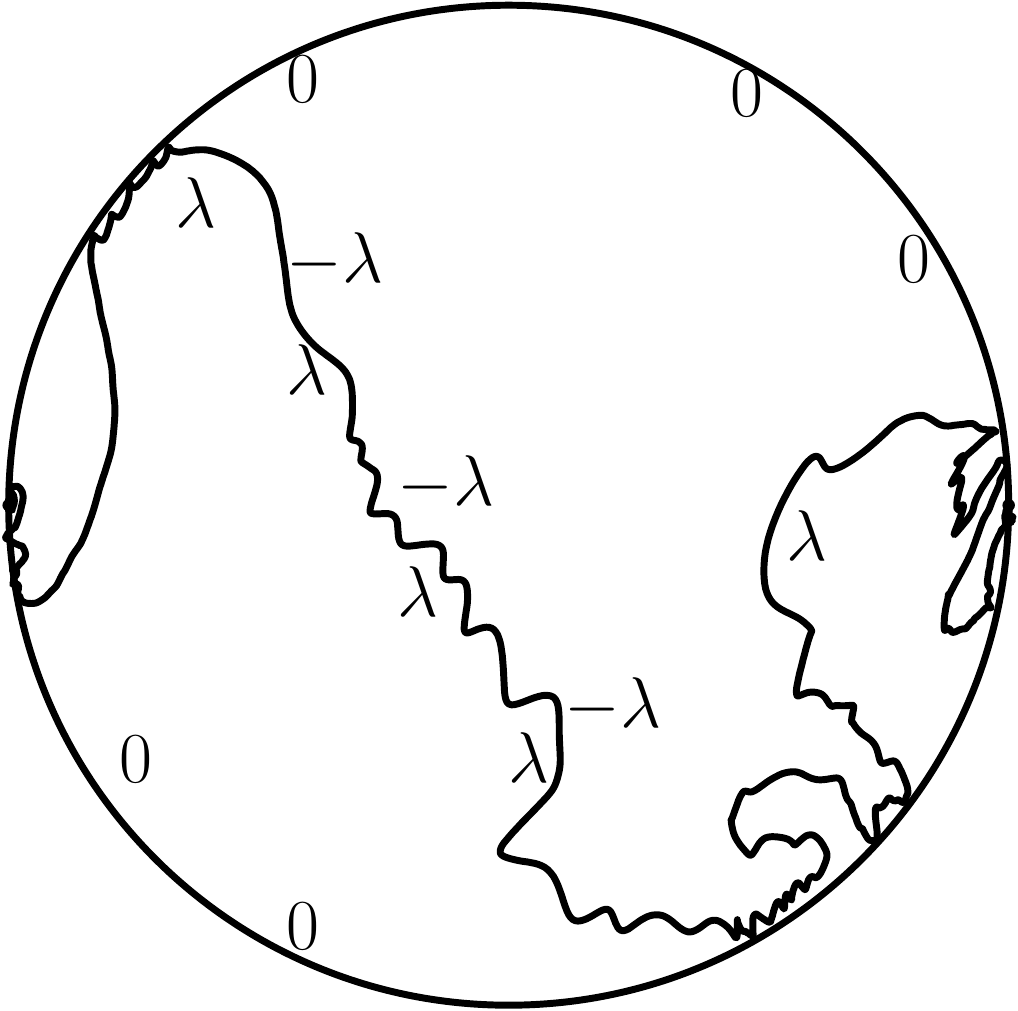}
\caption{\label{sle411}
Sketch of the $0$-level line from $-1$ to $1$ in $\D$ with the corresponding boundary conditions.}
\end{center}
\end{figure} 
Then this is a local set of GFF with zero boundary conditions. Again, it can be proved that this local set can be deterministically recovered from the GFF \cite{ss2010continuumcontour}.

Another variant that we will also use in a few paragraphs goes as follows. We consider the bounded harmonic function $\Fh_0$ in $\h$ with boundary values equal to $0$ on the negative half-line, to $c+\lambda$ on $(0,x)$ for some $x \in (0, \infty]$, and to $c- \lambda$ on $(x, \infty)$, where $c \in (- \lambda, \lambda)$. If one now samples a well chosen $\SLE_{4} (\rho_1; \rho_2)$ process from $0$ to $x$ with marked points at $0^+$ and $\infty$ (so here, the force points lie on different sides of the tip of the curve), then this process will trace a curve from $0$ to $x$, that will intersect the interval $(0,x)$ but not the rest of the real line. Furthermore, if one defines the ``boundary conditions'' on the complement of the curve to be: (i) as $\Fh_0$ on the real line, and (ii) equal to  $0$ and $2 \lambda$ on the left-hand and right-hand side of the curve respectively, one gets a local set in the upped half-plane with boundary conditions given by $\Fh_0$ (see Figure~\ref{sle4rhorho}). 
\begin{figure}[ht!]
\begin{center}
\includegraphics [width=3in]{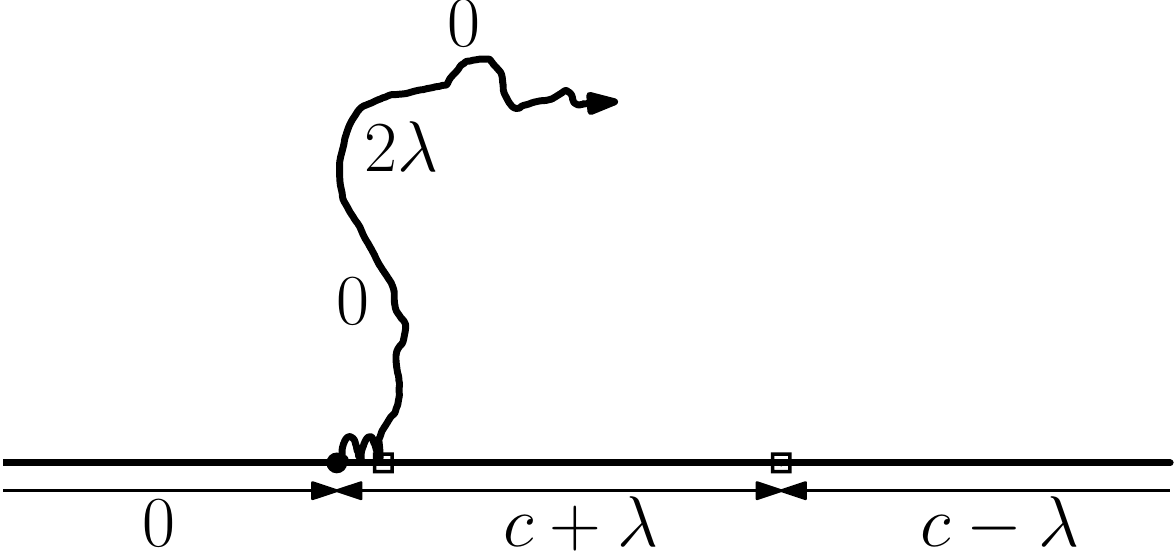}
\includegraphics [width=3in]{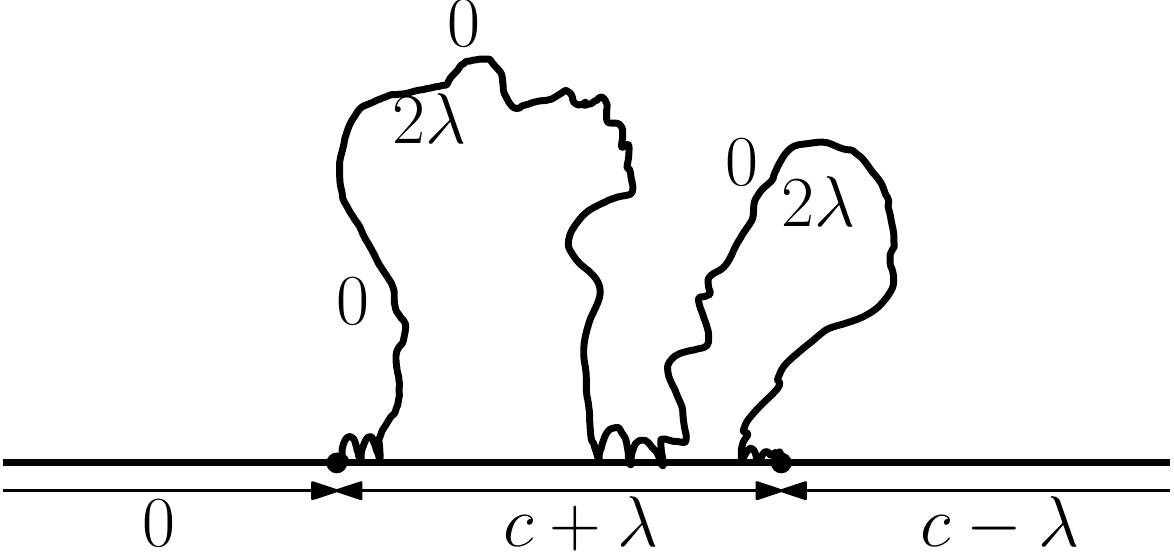}
\caption{A sketch of the $\SLE_4 ( \rho_1, \rho_2)$ coupling. \label{sle4rhorho}}
\end{center}
\end{figure}

Recall that $\CLE_4$ is a random collection of disjoint non-nested simple loops in the unit disk. Each loop is an $\SLE_4$-type loop, and the complement of the inside of all loops in the unit disk is a set with zero Lebesgue measure and Hausdorff dimension that is strictly smaller than $2$ (it is actually $15/8$, see \cite{ssw2009radii,nw2011carpets}) that is  called the $\CLE_4$ carpet. We then toss an i.i.d.\ fair coin inside each loop of the $\CLE_4$ to choose between the constant values $+2 \lambda$ or $-2 \lambda$ inside this loop. In this way, we have a random set $A$ (the CLE carpet) with labels that define a harmonic function $h_A$ in the complement of this carpet. It turns out that this couple $(A, h_A)$ is also a thin local set (to define a GFF in $D$, one just has to sample independent GFFs in the connected components of the complement of the carpet, with boundary conditions $h_A$). Again, the couple $(A, h_A)$ can be deterministically recovered from the GFF it defines. These facts are due 
to Miller and Sheffield \cite{mscle}, see also the self-contained presentation in \cite{asw2016boundedthin}.  

In the present section, we will use features of the following type, that are very closely related to the fact that the $\CLE_4^0$ is a deterministic function of the GFF: 
\begin{lemma}[\cite{asw2016boundedthin}]
\label{cle4charact}
 Suppose that one can construct a local set $A$ with Minkowski dimension almost surely smaller than some $d<2$, and with harmonic function $h_A$ such that $h_A \in  \{ -2 \lambda, 2 \lambda, 0 \}$ almost surely (i.e., the harmonic function in each of the connected components $D_j$ of the complement of $A$ is constant and equal to one of these three values). Then $A$ can be coupled with a $\CLE_4^0$ carpet and a GFF in such a way that both $A$ and the $\CLE_4^0$ are local with respect to $h$. The components $D_j$ with $h_A \in \{ - 2 \lambda, 2 \lambda \}$ are then also connected components of this $\CLE_4^0$ carpet. 
\end{lemma} 
Indeed, one can first complete $A$ by sampling a labeled $\CLE_4$ carpet in all of the components of its complement with $h_A = 0$. In this way, one gets a local set with $h_A \in \{ -2 \lambda, 2 \lambda \}$
and Minkowski dimension strictly smaller than $2$, and it is known (see \cite{asw2016boundedthin}) that the labeled $\CLE_4^0$ is the only possible one. We shall use also a slight variation of this result that we will describe during the proof.

\subsection{$\CLE_4^0$ percolation}

We are now going to derive the following proposition that characterizes and describes all possible CPIs in labeled $\CLE_4$ carpets (see Figure~\ref{fig:cle4_perc}): 

\begin{prop}
\label{prop:CME_characterization2}
\begin{enumerate}[(i)]
 \item 
There are no \hyperref[def:cpi]{CPIs} in $\CLE_4^\beta$ for $\beta \not= 0$. 
\item  
There is exactly a one-parameter family of \hyperref[def:cpi]{CPIs} in $\CLE_4^{0}$ parameterized by $c \in (- \lambda, \lambda)$ or equivalently by $\mu \in \R$. These \hyperref[def:cpi]{CPIs} have the following properties:
\begin{itemize}
\item Each of these \hyperref[def:cpi]{CPIs} is a deterministic function of the labeled $\CLE_4^0$. 
\item In the coupling between the $\CLE_4^0$ and the GFF described above, a \hyperref[def:cpi]{CPI} corresponds to the $c$-level line for some $c \in (- \lambda, \lambda)$. 
\item The $\bSLE_4^{0,\mu}$ Loewner chain is almost surely generated by a continuous path, and its trunk (when one erases the $\CLE_4$ loops it creates when going from its starting point to its target point) is an $\SLE_4 ( \rho; -2- \rho)$ process (these processes will also be called $\bSLE_4 (\rho)$ later in this paper) for some $\rho \in (-2, 0)$.
This trunk is then a \hyperref[def:cpi]{CPI} of the $\CLE_4^0$.
\end{itemize}
\end{enumerate}
\end{prop}

Note that by symmetry, the relation between $c$ and $\mu$ satisfies $\mu (-c) = - \mu (c)$ and $\mu (0) = 0$. The arguments presented in the present paper will not provide a formula for the relation between $\mu$ and $c$ (or equivalently, between $\mu$ and $\rho$) but the explicit formula should follow from our upcoming paper \cite{msw2016betarho}.  We will see a similar feature in our study of $\CLE_\kappa$ percolations for other $\kappa$'s, and will comment further on this after the statements of the main results (Theorems~\ref{thm:duality1} and~\ref{thm:duality2}) at the end of Section~\ref{sec:bcle}.

The first statements of this proposition have some similarities with the previous $\kappa \in (8/3, 4)$ case, but we note already that the description here is complete in that this proposition contains both the existence of the \hyperref[def:cpi]{CPIs} and the description of their distribution. In the course of the proof, we will also describe the conditional distribution of the $\CLE_4^0$ given the \hyperref[def:cpi]{CPI}.

\begin{figure}[ht!]
\begin{center}
\includegraphics [width=3in]{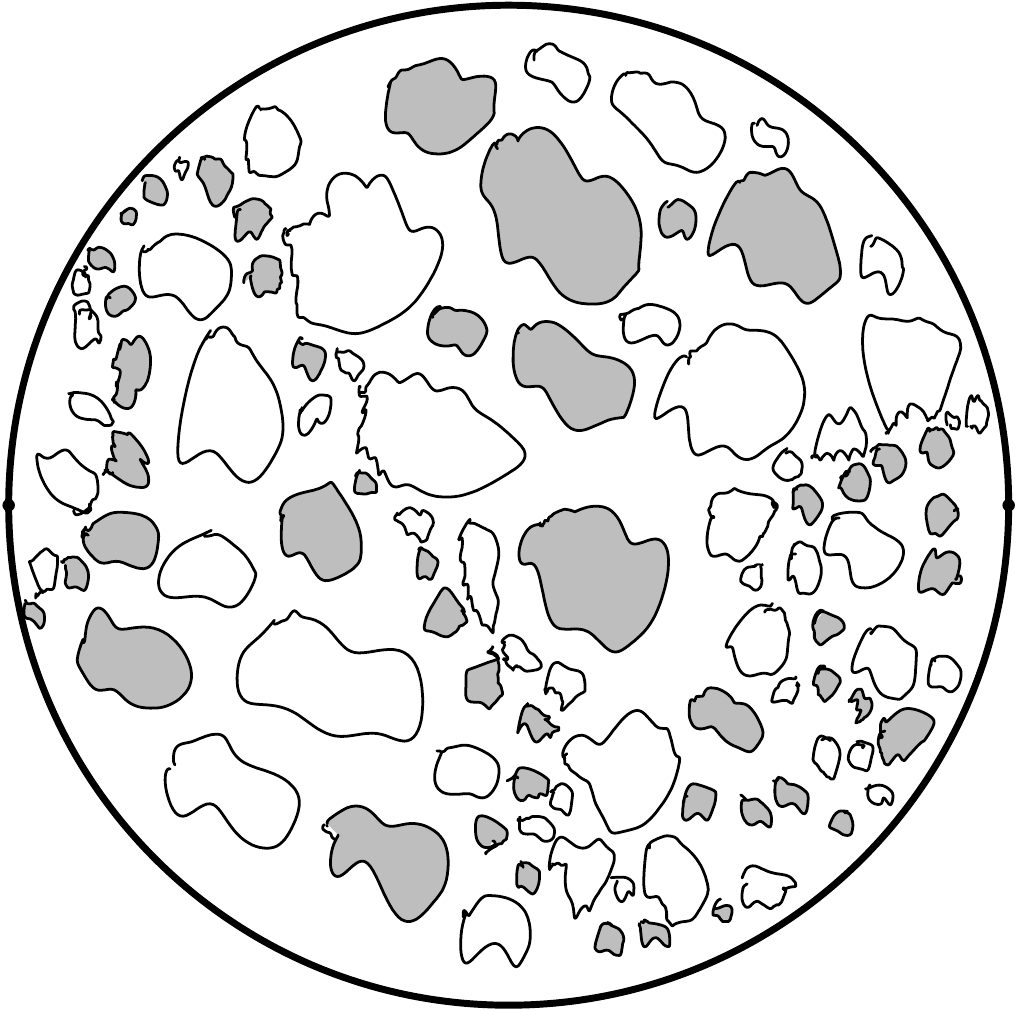}
\includegraphics [width=3in]{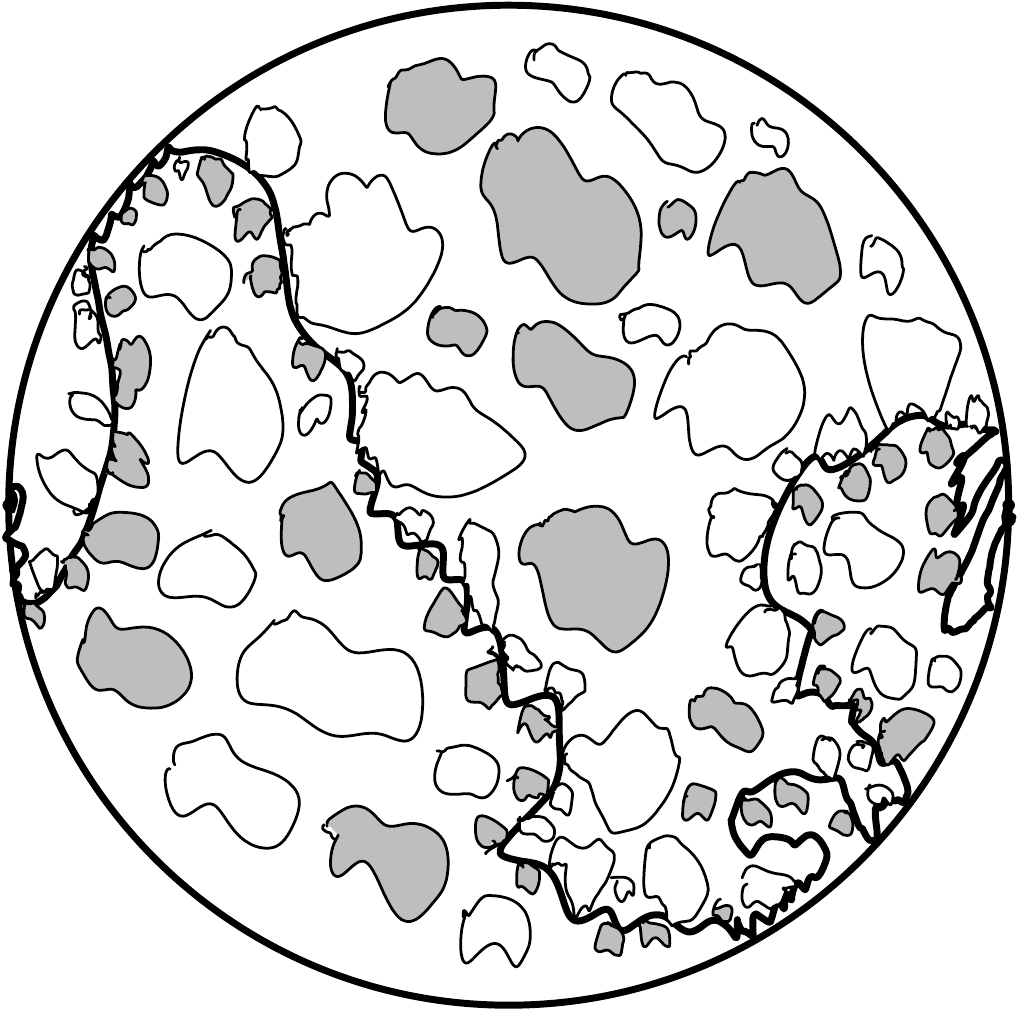}
\caption{\label{fig:cle4_perc} {\bf Left:} A sketch of $\CLE_4$ (shaded loops have boundary values $2 \lambda$ and the others have boundary values $-2 \lambda$ for the GFF). {\bf Right:} A sketch of the joint realization of the $\CLE_4^0$ and the $\SLE_4 (-1; -1)$. Proposition~\ref{prop:CME_characterization2} shows that the latter is a \hyperref[def:cpi]{CPI} in the former, and that the law of the process obtained when tracing the encountered loops along the way of the level line is a $\bSLE_4^{0,0}$.}
\end{center}
\end{figure}

\begin{proof}[Proof of Proposition~\ref{prop:CME_characterization2}]
The arguments of the previous section can be repeated almost word for word, in order to derive the first few statements in the proposition, as well as the fact that for $\CLE_4^0$, there is at most a one-parameter family of \hyperref[def:cpi]{CPIs} in $\CLE_4^0$ and that the process $\gamma_t^*$ must have the same law as a $\bSLE_4^{0,\mu}$ process, viewed at those times at which it does not trace a $\CLE_4$ loop, and that finally, if one knows that $\bSLE_4^{0,\mu}$ is almost surely a continuous path, then its trunk is a \hyperref[def:cpi]{CPI} in a corresponding $\CLE_4^0$.  The only difference is that in the last part of the argument, the subordinator can have a drift part (this is as in \cite{she2009cle,ww2013conformally}).

We will now show that each $c$-level line  of the GFF $h$ from $x$ to $y$ coupled to the $\CLE_4^0$ as described above does indeed define a \hyperref[def:cpi]{CPI} from $x$ to $y$ in that GFF. Combined with the above, it shows that each of these $\CLE_4^0$ level lines is the trunk of the $\bSLE4^{0, \mu}$ for some $\mu$, and this $\bSLE_4^{0,\mu}$ is a continuous curve. 

We proceed as follows (in the remainder of this proof, we shall work with processes in the unit disk~$\D$): For a given $c \in ( - \lambda, \lambda)$, we first consider the height $c$ level-line started from $-1$ and targeting $1$ of a GFF $h$, that we call $\eta_c$ (recall that this is a continuous simple curve in the closed unit disk that touches almost surely both the bottom and the top half-circle, and that it is a deterministic function of $h$). Note that this GFF $h$ also deterministically defines a $\CLE_4^0$.

Sample first the path $\eta_c|_{[0, \tau]}$, where $\tau$ is some stopping time for $\eta_c$. We know the conditional distribution of $h$ given $\eta_c|_{[0,\tau]}$ in the complement of this slit (it is a GFF with boundary conditions~$0$ on the unit circle and $c - \lambda$ and $c+\lambda$ on the two sides of~$\eta_c$). Each connected component of the complement of this slit has three boundary arcs (one of which can be empty): an arc of the unit circle, a portion of $\eta$ seen from below, and a portion of $\eta$ seen from above. Then, we draw the $-\lambda$-level lines (i.e.\ boundary values on the two sides of such lines are~$0$ and~$-2 \lambda$) of this GFF in each of the components that sees a portion of $\eta_c$ from above, as indicated in Figure~\ref{fig:cle4_bcle4_construction}, and we also draw symmetrically the $\lambda$-level lines below the curve. In all connected components except the one that has $\eta_c (\tau)$ on its boundary, one has drawn just one level line.  Again, we know (from the 
level-lines couplings) the form of the conditional distribution of the GFF in the complement of the union of this first level line with these new level line. 

In particular, for those components that have part of $\partial \D$ on their boundary, the boundary conditions are identically $0$. In all other ones, the boundary consists of a piece of the curve $\eta_c$ (seen either from below or above) and a piece of the $\lambda$ (or $-\lambda$) level line, so the corresponding boundary conditions are $\lambda +c$ and $2\lambda$, or $- \lambda + c$ and $-2 \lambda$.

\begin{figure}[ht!]
\begin{center}
\includegraphics [width=3in]{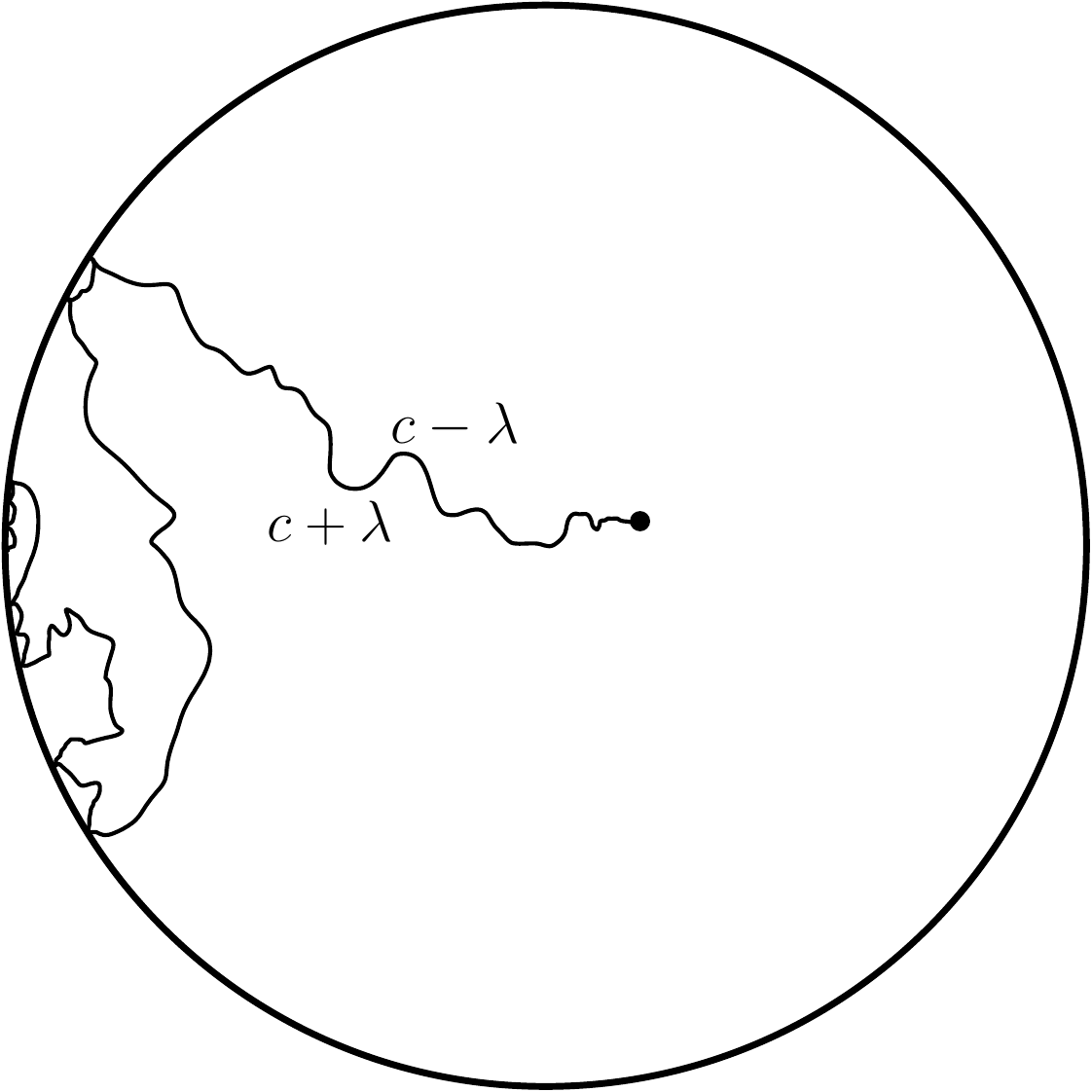}
\includegraphics [width=3in]{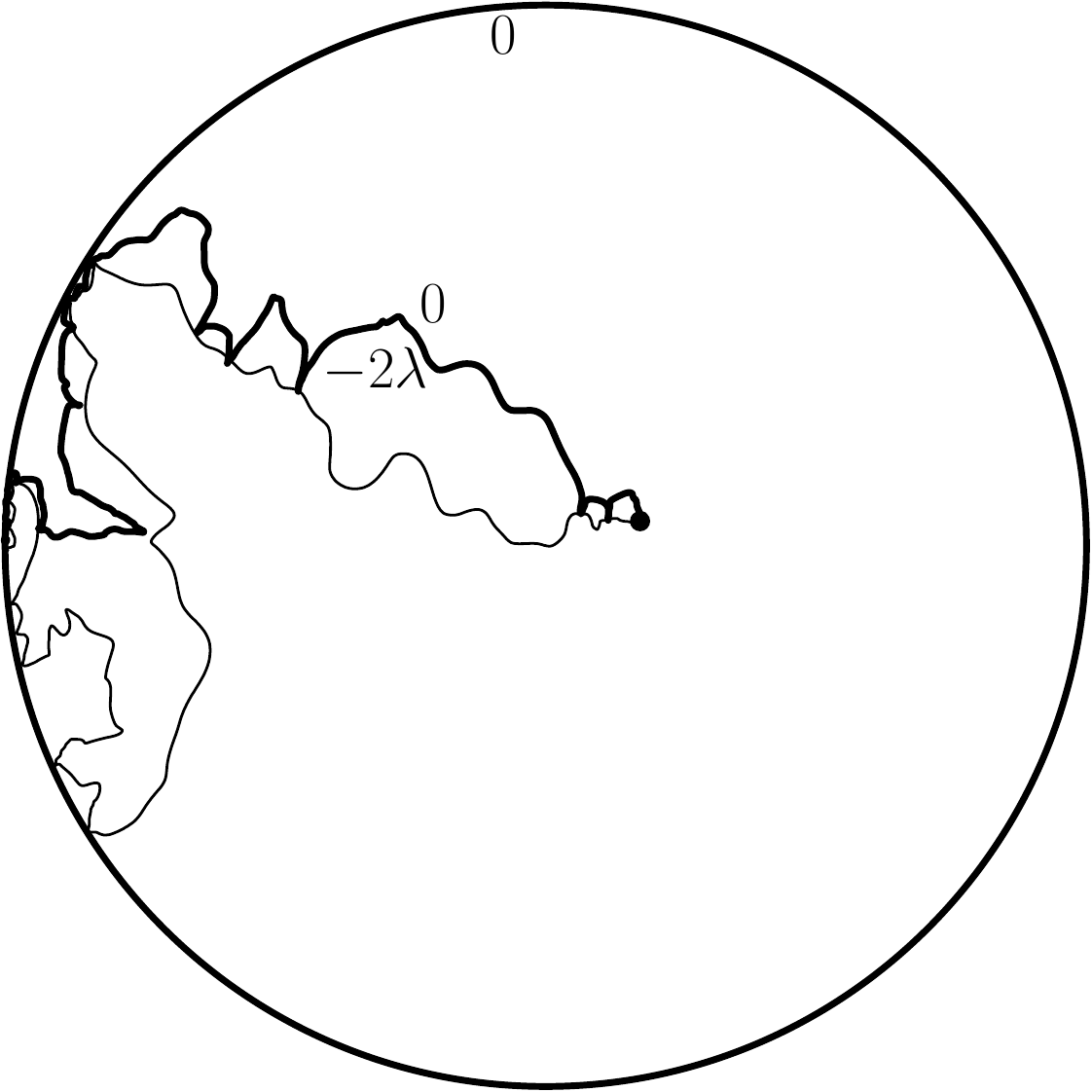}
\caption{\label{fig:cle4_bcle4_construction}The level line, and the first layer}
\end{center}
\end{figure}

\begin{figure}[ht!]
\begin{center}
\includegraphics [width=2in]{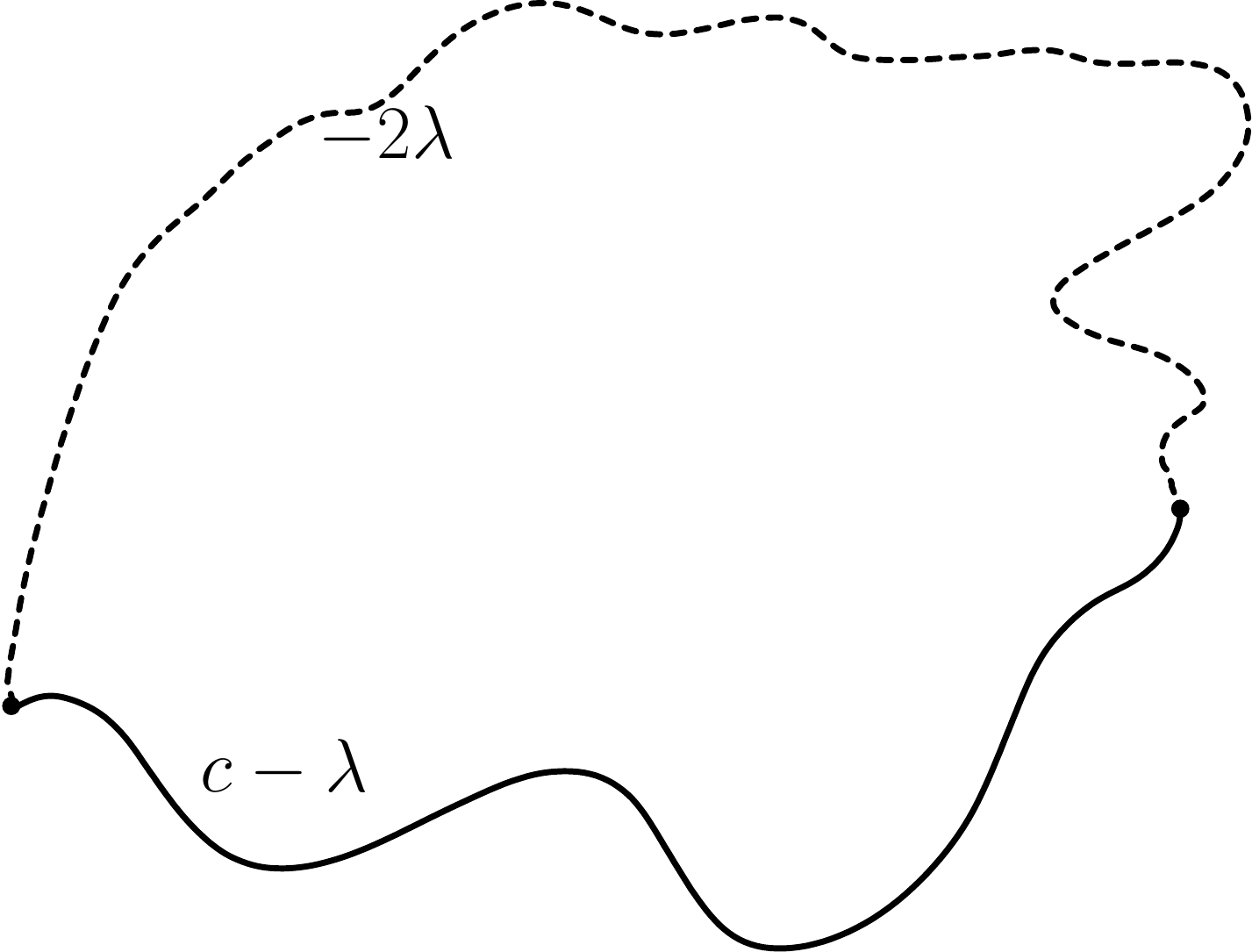}
\includegraphics [width=2in]{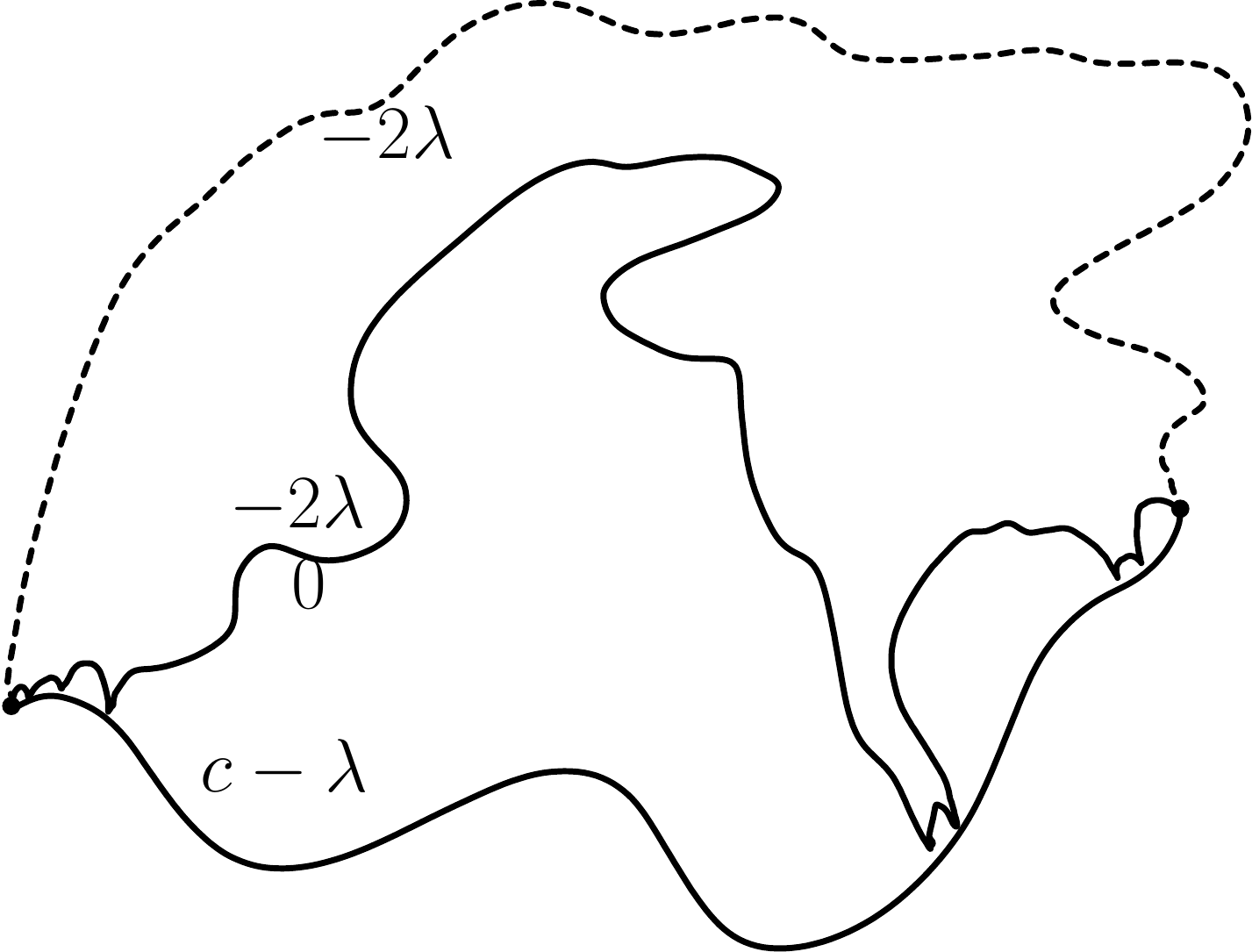}
\includegraphics [width=2in]{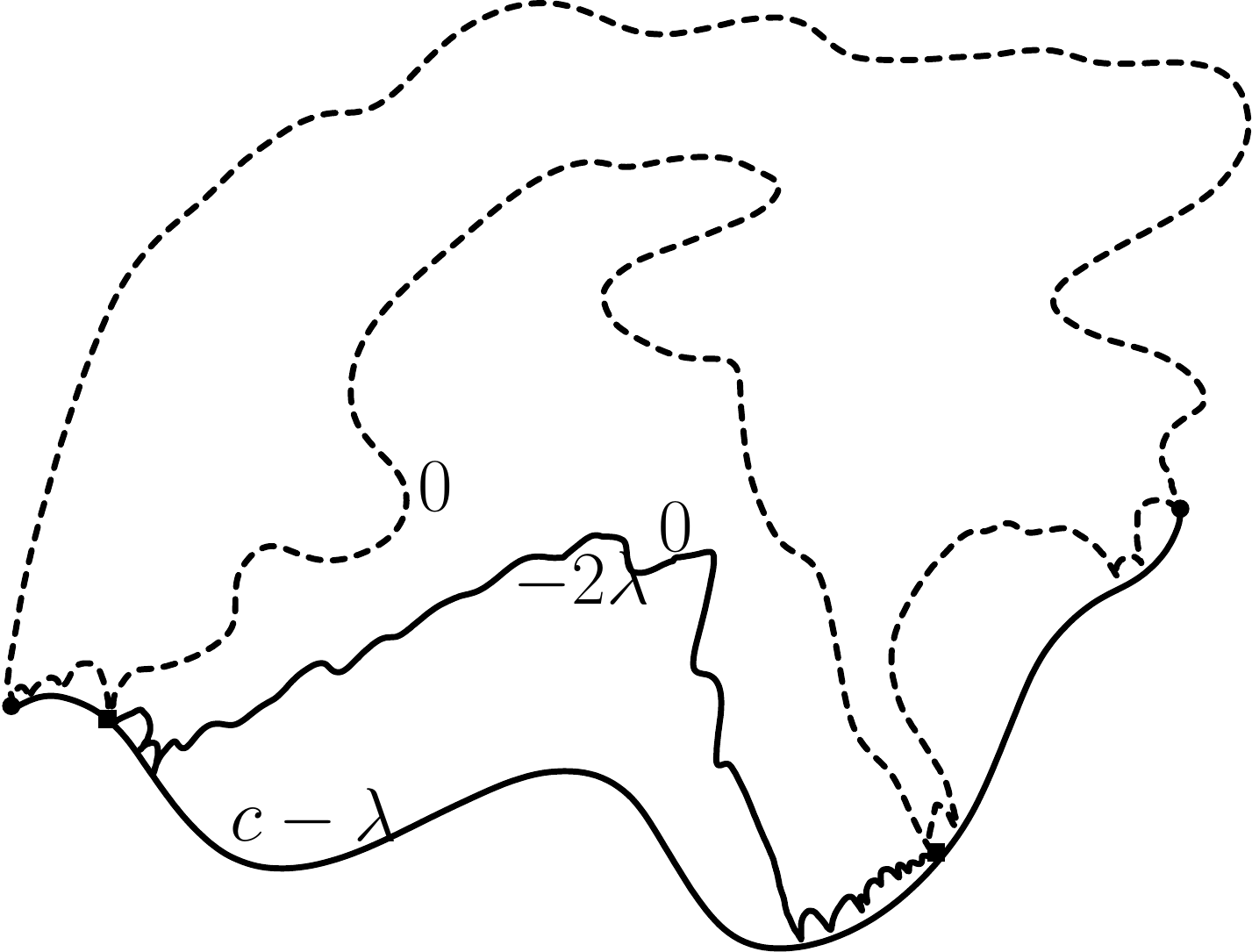}
\caption{\label{fig:cle4_bcle4_construction2}Iterating $-2\lambda / 0$ layers in a pocket of the first layer and discovering $-2 \lambda$ $\CLE_4^0$ loops (such as the dashed loop in the third picture) on the way}
\end{center}
\end{figure}

Iterating the procedure once in each of these connected component as indicated in Figures~\ref{fig:cle4_bcle4_construction} and~\ref{fig:cle4_bcle4_construction2}, one obtains domains with constant $\pm 2 \lambda$ boundary conditions, and domains where the boundary conditions are $0$ and $\lambda +c$ or $0$ and $\lambda -c$. We then iterate the procedure in the latter ones. This defines an increasing sequence of local sets $A_n$. We then define $A$ to be the closure of the union of all $A_n$. Its complement is then just the union of all the interiors loops discovered along the way and the value of the corresponding harmonic function is almost surely equal to $\pm 2 \lambda$ or to $0$ (and it is constant in each connected component of the complement of the local set). 

For each $n$, the Minkowski dimension of $A_n$ is strictly smaller than $2$ (this is easy to check because it is just the union of SLE interfaces), so that (using one of the main propositions in \cite{asw2016boundedthin}) they are almost surely all contained in the same stopped iterated $\CLE_4$ defined there (which is known to have Minkowski dimension smaller than~$2$) and therefore the same holds for $A$, which in turn implies that the  Minkowski dimension of $A$ is smaller than~$2$.  We can then apply Lemma~\ref{cle4charact} to deduce that the loops that have been discovered in this way are all part of the $\CLE_4^0$, and that all other loops in the $\CLE_4^0$ do not touch $\eta_c$. In other words, our procedure to define loops with $2\lambda$ boundary conditions via the iteration of layers does define a subset of the $\CLE_4^0$ loops that are coupled to the GFF (this is the same GFF that is coupled with the level line $\eta_c$).  By construction, all these loops do touch $\eta_c$, the positive ones touch the ``right-hand side'' of $\eta_c$ 
while the negative one are on its ``left-hand side'', and no other loop of the $\CLE_4^0$ touches $\eta_c|_{[0,\tau]}$ (because they are sampled out of a GFF with zero boundary conditions in the remaining connected components, and therefore do not touch $\partial A$). 

Summarizing things, we have constructed a coupling of the curve $\eta_c$, of the $\CLE_4^0$ loops of the GFF $h$ that do touch $\eta_c$ and of the GFF, with the following properties: 
\begin{itemize}
 \item  The $\CLE_4^0$ loops that do touch the curve $\eta_c$ are above or below $\eta_c$ depending on their label. 
 \item  The curve $\eta_c$ does not enter in the inside of any of the loops. 
 \item  One can then define the curve $\eta$ that is obtained by attaching in their order of appearance along $\eta_c$ from $-i$ to $i$ the loops of the labeled $\CLE_4^0$ that do intersect $\eta_c$. The local finiteness of $\CLE_4$ ensures that $\eta$ is indeed a continuous curve.  If one chooses to trace the positive loops counterclockwise and the negative loops clockwise, then the obtained curve $\eta$ is non-self-crossing.
\item Both $\eta$ and the $\CLE_4^0$ are deterministic functions of the GFF, so that the path $\eta$ is also a deterministic function of the GFF. 
\end{itemize}

It is now easy to show that $\eta_c$ is a \hyperref[def:cpi]{CPI} in the $\CLE_4^0$.  If $\tau$ is a stopping time for the filtration generated by the curve $\eta_c$ together with the collection of $\CLE_4^0$ loops in encounters, let us define ${\mathcal F}_\tau^*$ to be the corresponding $\sigma$-algebra. 
The previous description of $\eta_c|_{[0, \tau]}$ and of the loops it encounters shows that the conditional distribution of $h$ in $D_\tau$ (the connected component of the 
remaining to be discovered domain that has $1$ on its boundary) is then a GFF with zero boundary conditions. In particular, as $\eta_c$ is the $c$ level line in $h$, the picture in $D_\tau$ will be that of a $0$-boundary GFF with its $c$-level line, which proves the Markovian part of the definition of a \hyperref[def:cpi]{CPI}. In the other
unexplored connected components  (that are not surrounded by an already discovered $\CLE_4^0$ loop), the conditional distribution of $h$ is that of a GFF with zero boundary conditions, and the restriction of the $\CLE_4^0$ defined by $h$ to this domain is then distributed like a $\CLE_4^0$ in 
this connected component. We therefore conclude that $\eta_c$ is indeed a \hyperref[def:cpi]{CPI} in the $\CLE_4^0$.

To conclude the proof, we need to check that the mapping $c \mapsto \mu$ is a monotone bijection from $(-\lambda, \lambda)$ into $\R$.  We extend this map to be a map from $[-\lambda,\lambda]$ to $\R \cup\{\pm \infty\}$ by declaring that $-\lambda$ (resp.\ $\lambda$) is sent to $-\infty$ (resp.\ $+\infty$).  Since the extended map is injective and sends $-\lambda$ to $-\infty$ and $\lambda$ to $+\infty$, it suffices to show that $c \mapsto \mu$ is continuous.  Suppose that $(c_n)$ is any sequence in $[-\lambda,\lambda]$ which converges to $c \in [-\lambda,\lambda]$.  For each $n$, we let $\mu_n$ be the image of $c_n$ under the map and let $\mu$ be the image of $c$ under the map.  By passing to a subsequence if necessary, we may assume that $\mu_n \to \wt{\mu} \in \R \cup \{ \pm \infty\}$ as $n \to \infty$.  It suffices to show that $\mu = \wt{\mu}$.  It is easy to see from the construction that the law of the driving process associated with $c_n$ converges weakly as $n \to \infty$ to the law of the driving 
process associated with $c$ simply because the hulls of the corresponding processes converge.  Similarly, the law of the process associated with $\mu_n$ converges weakly to the law of the process associated with $\wt{\mu}$ as $n \to \infty$.  Indeed, this can be seen by inspecting the equation satisfied by the driving process.  Therefore $\mu = \wt{\mu}$, as desired.
\end{proof}

\section{$\CLE_{\kappa'}$ percolation}
\label {Sec6}

\subsection{Boundary-touching $\CLE_{\kappa'}$ loops}

Recall that the $\CLE_{\kappa'}$ for $\kappa' \in (4,8)$ correspond to gaskets (as opposed to carpets, as in the case that $\kappa \in (8/3,4]$) because different loops can touch each other and can touch the boundary.  In the case where $\kappa' \in (4,8)$, the existence and first properties of the $\CLE_{\kappa'}$ follow directly from the combination of the results in \cite{she2009cle} and \cite{ms2012ig1,ms2012ig2,ms2012ig3,ms2013ig4}.  In particular, the existence and basic properties of $\CLE_{\kappa'}$ were stated in \cite{she2009cle} conditionally on a continuity and reversibility assumption for $\bSLE_{\kappa'}$ that was then proved in \cite{ms2012ig1,ms2012ig2,ms2012ig3}.  The local finiteness of $\CLE_{\kappa'}$ was proved in \cite{ms2013ig4} by using the relationship between space-filling $\SLE_{\kappa'}$ and $\CLE_{\kappa'}$ (we will come back to this in Section~\ref{sec:ig}). We are now going to describe some consequences of this reversibility, in the spirit of the arguments in \cite{she2009cle}.

Suppose that $\kappa' \in (4, 8)$ and let us now recall from \cite{she2009cle} how to concretely define parts of the $\CLE_{\kappa'}$ using the $\bSLE_{\kappa'}$. Let us first consider a time-indexed Poisson point process of $\SLE_{\kappa'}$ bubbles. The intensity measure of this process is given by the Lebesgue measure on $\R_+$ times the so-called  $\SLE_{\kappa'}$ bubble measure 
(which is the appropriately rescaled limit when $\eps \to 0^+$ of the law of an $\SLE_{\kappa'}$ from $0$ to $\eps$ in the upper half-plane). 
This Poisson point process is therefore a countable random collection of pairs $(u_i, e_{u_i})$ where $e_{u_i}$ is a bubble and $u_i \ge 0$ (and we think then of $e_{u_i}$ as ``appearing'' at time $u_i$). 

The previous bubbles $e_{u_i}$ are oriented clockwise (for the previous definition)
but we can note that the bubble measure is invariant under the operation of taking the symmetry with respect to the imaginary axis of its 
counterclockwise orientation (this follows from the reversibility of SLE$_{\kappa'}$).

For each bubble $e$, we define $x(e)$ (resp.\ $y(e)$) to be the rightmost (resp.\ leftmost) point of the bubble on the real axis and $\varphi_e$ to be the conformal transformation from the unbounded connected component of $\HH \setminus e$ onto $\HH$ with $\varphi_e (z) \sim z$ as $z \to \infty$ and $\varphi_e (x (e))=0$. 
\begin{figure}[ht]
\begin{center}
\includegraphics[width=3in]{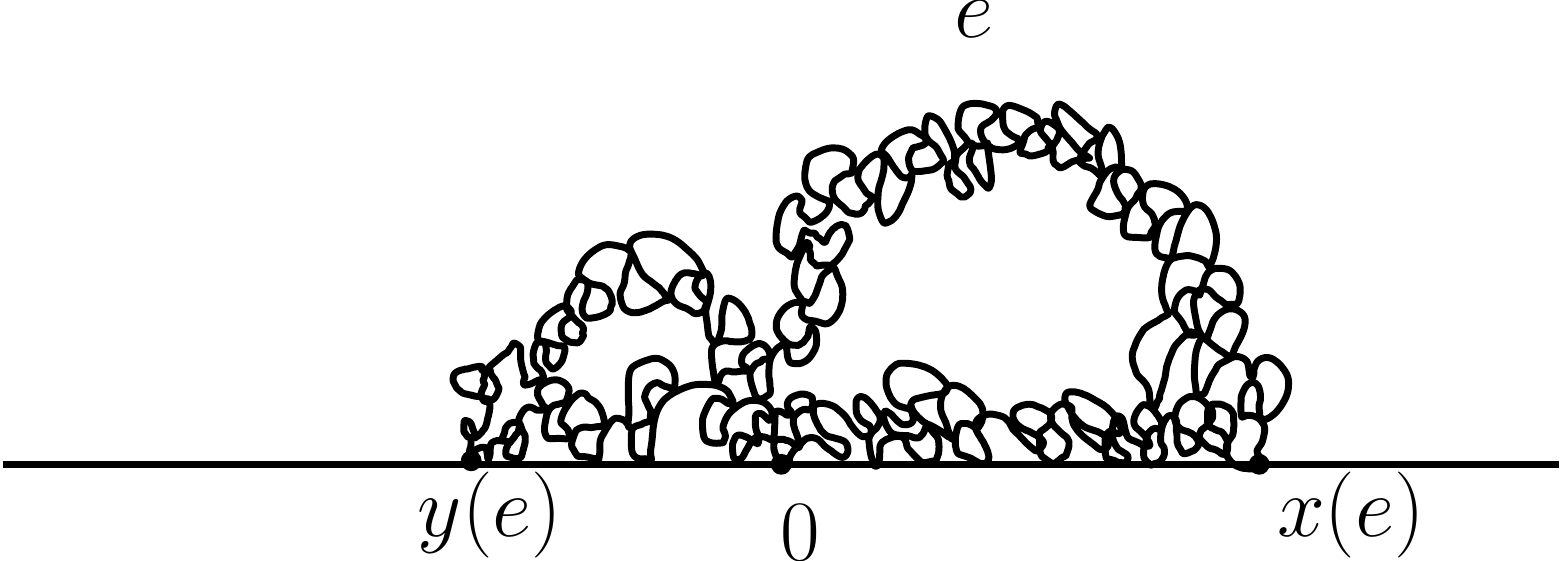}
\end{center}
\caption{\label{fig:sle_bubble} An $\SLE_{\kappa'}$ bubble $e$}
\end{figure}
We also define  $e^+$ to be the clockwise part of the bubble $e$ from $0$ to $x(e)$, and $e^-$ the counterclockwise part of the bubble between $0$ and $y(e)$ (see Figures~\ref{fig:sle_bubble} and~\ref{fig:sle_bubble2}). 
\begin{figure}[ht]
\begin{center}
\includegraphics[width=3in]{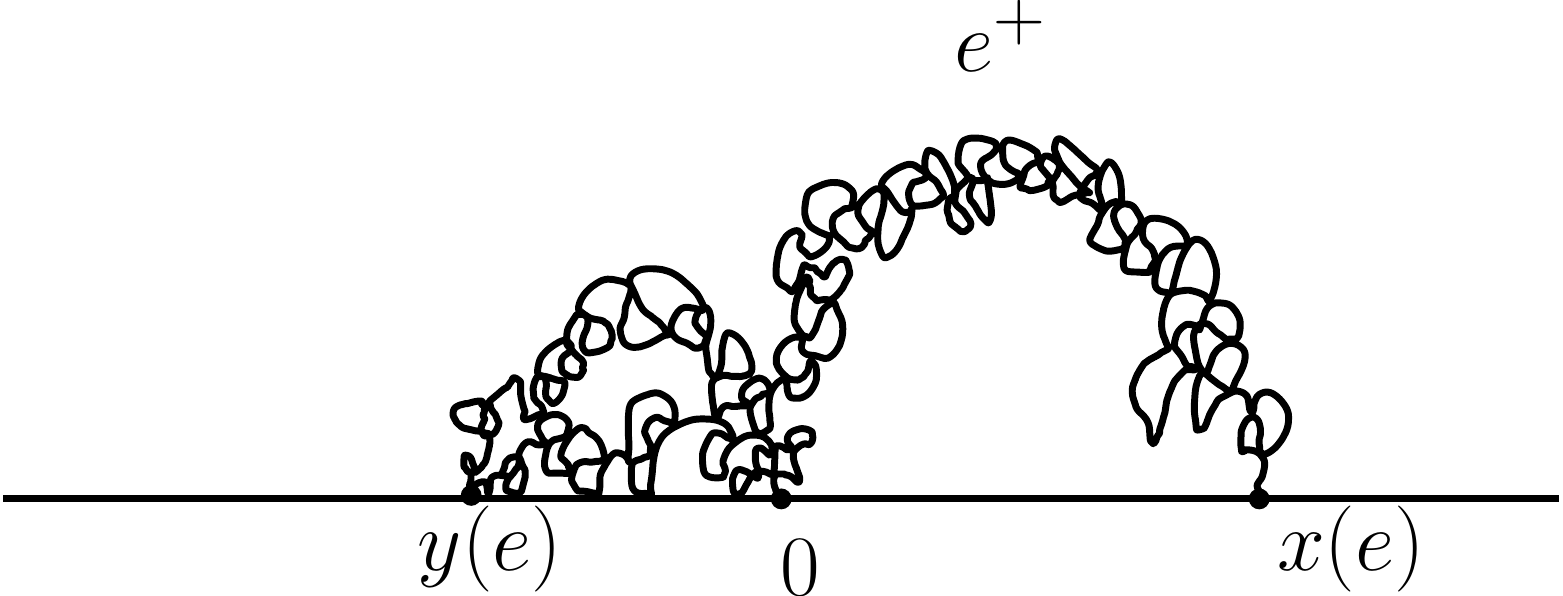}
\quad
\includegraphics[width=3in]{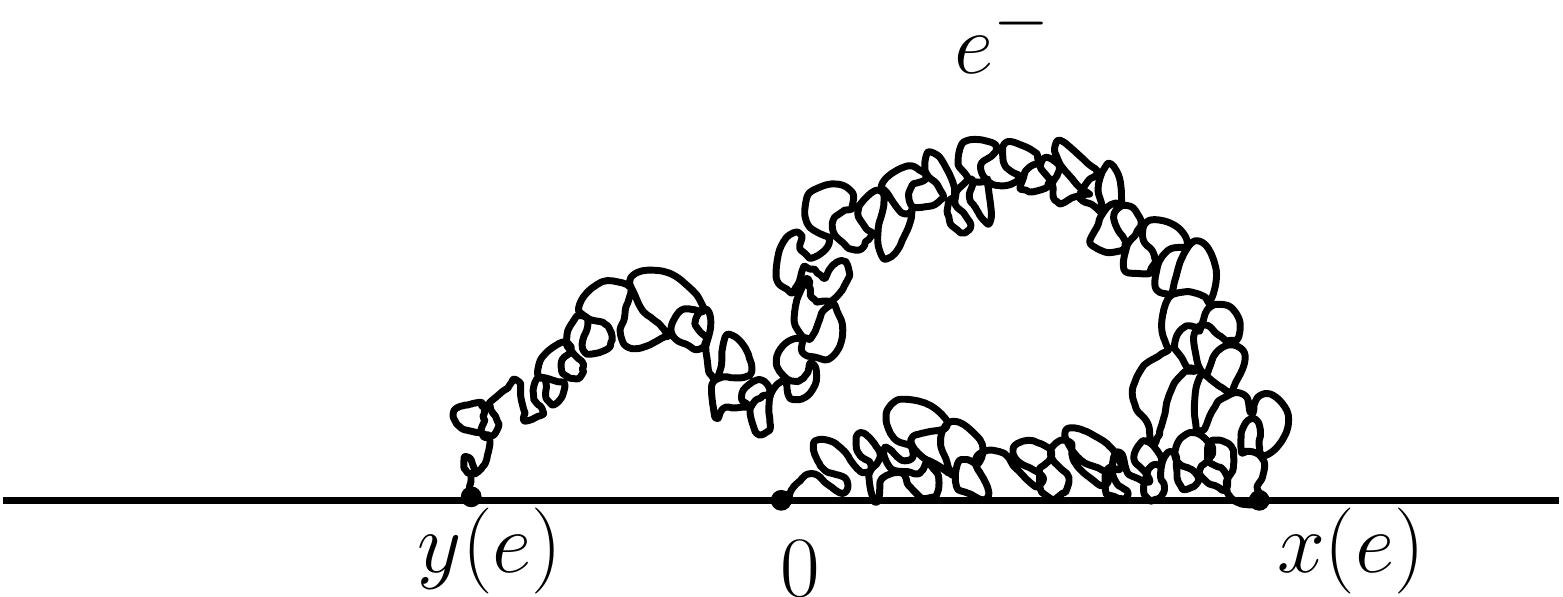}
\end{center}
\caption{\label{fig:sle_bubble2} The corresponding $e^+$ and $e^-$}
\end{figure}

If one iterates the maps $\varphi_e$ in their order of appearance, one obtains a process of conformal maps $(\Phi_u, u \ge 0)$. If one now concatenates the paths $\Phi_{u-} (e_u^+)$ in their order of appearance (see Figure~\ref{fig:sle_bubble3}), one gets exactly the ordinary $\bSLE_{\kappa'}$ (i.e.\ an $\SLE_{\kappa'} (\kappa' -6)$) from the origin to infinity (recall that $\kappa' >4$ so that $\kappa' -6 > -2$, hence there is no issue with the accumulation of small bubbles) -- one can also invoke here the target independence of these paths.  The family of loops $\Phi_{u-} ( e_u)$ form now a part of the $\CLE_{\kappa'}$.  

\begin{figure}[ht]
\begin{center}
\includegraphics[width=3in]{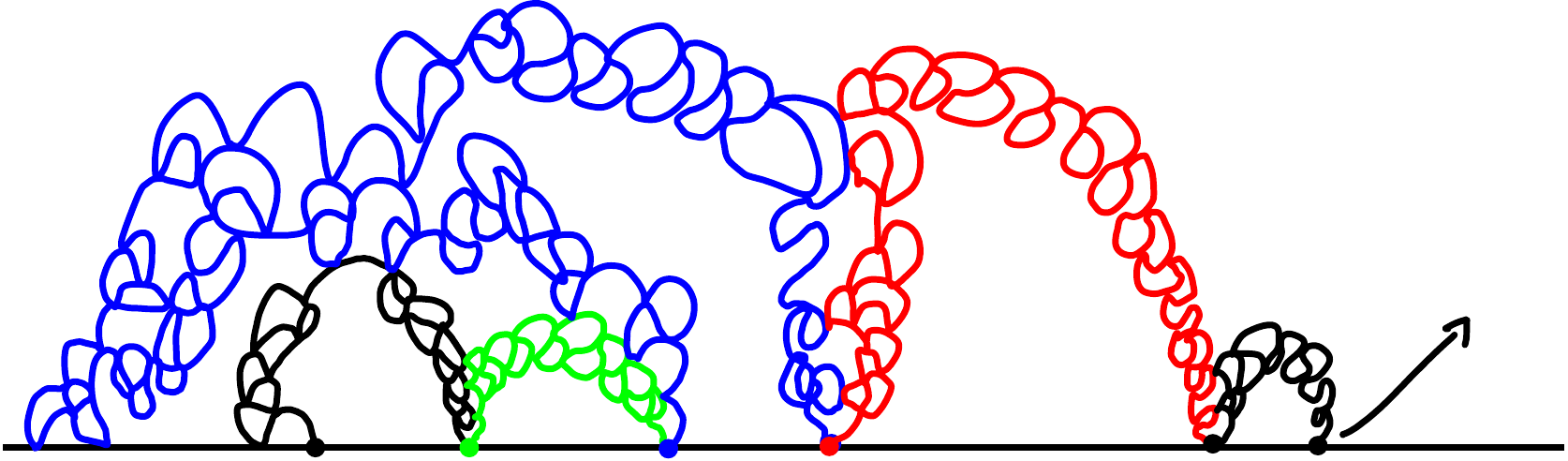}
\includegraphics[width=3in]{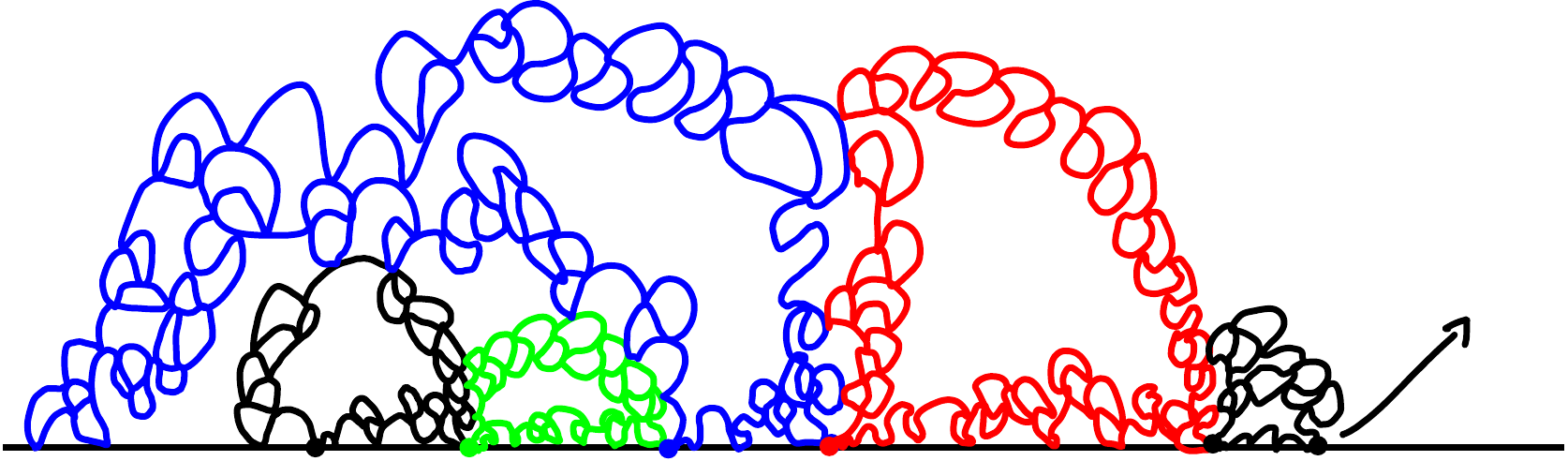}
\end{center}
\caption{\label{fig:sle_bubble3} A $\bSLE_{\kappa'}$ and a corresponding collection of $\Psi_{u-}(e_u)$'s}
\end{figure}

In this way, starting from the Poisson point process of $\SLE_{\kappa'}$ bubbles, one constructs a collection 
of $\CLE_{\kappa'}$ loops that intersect the positive half-line, but one does not construct all the loops of this $\CLE_{\kappa'}$ that do intersect the positive half-line. 
Indeed, there are a number of additional loops that are squeezed ``under'' the ones that one has constructed.

It is however not difficult to construct them as well. Indeed, it suffices to iterate the procedure inside each of the pockets that are located underneath all the loops that one has constructed, and then to iterate the procedure inside the pockets that are underneath all of the newly traced loops. Note (for instance in Figure~\ref{Fpockets}), that in this way, one creates pockets that are squeezed in between two loops that touch the boundary, and that a point in the upper half-plane will be either on a loop, or inside such a clockwise loop, or inside one pocket. We can also note that one can view a pocket as being surrounded counterclockwise by a concatenation of parts of $\CLE_{\kappa'}$ loops i.e.\ of excursions away from the boundary of $\CLE_{\kappa'}$ loops, that are concatenated at boundary points.

\begin{figure}[ht]
\begin{center}
\includegraphics[width=3.8in]{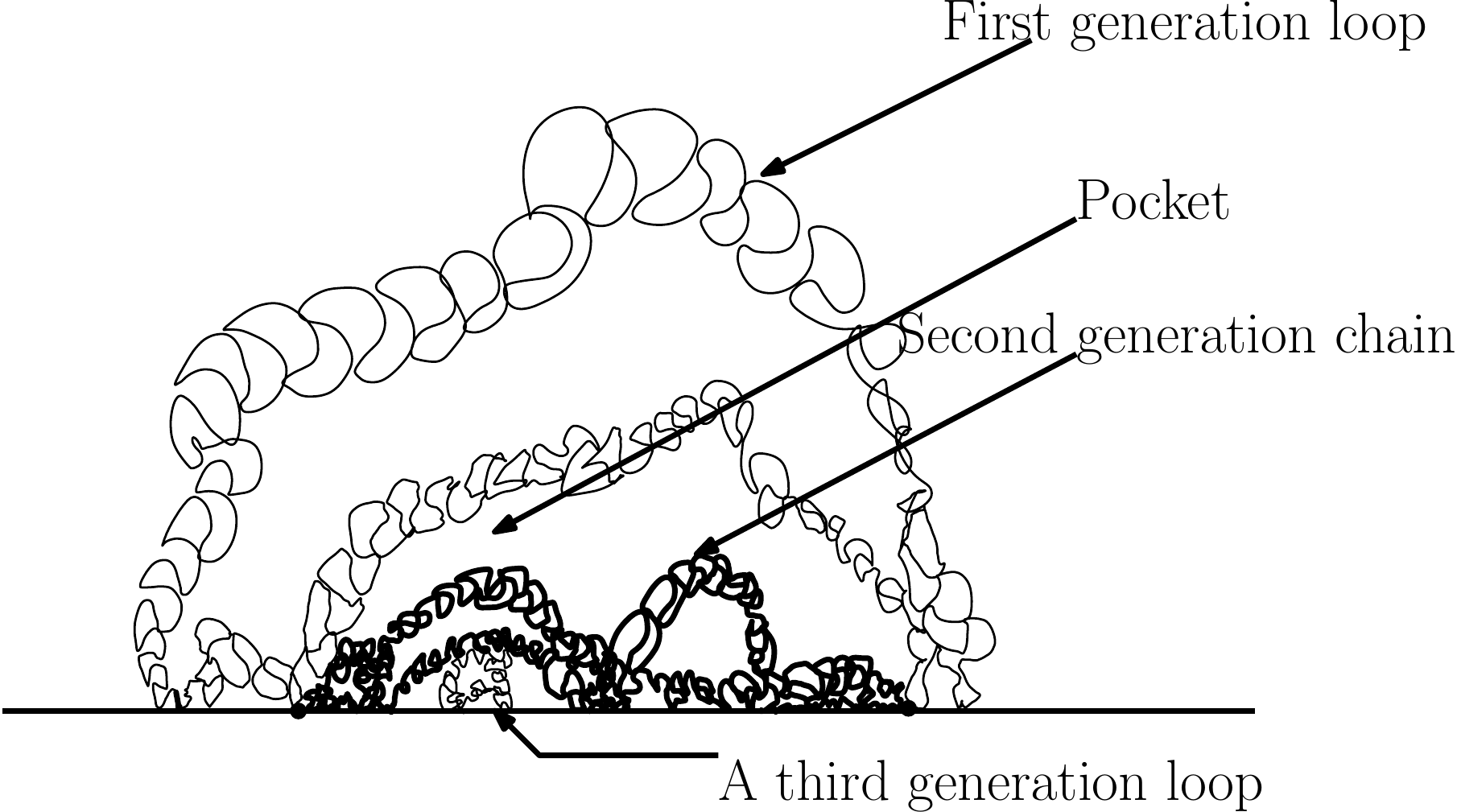}
\end{center}
\caption{Adding the missing loops iteratively}
\label{Fpockets}
\end{figure}

The local finiteness (i.e.\ for all positive $\eps$, there are almost surely finitely many loops of diameter greater than $\eps$ when one looks at the image under a given conformal map from the upper half-plane onto the unit disk)  \cite{ms2013ig4} of the $\CLE_{\kappa'}$ ensures that if one traces (in a clockwise manner) all the $\CLE_{\kappa'}$ loops that touch the positive half-line, in the order in which they ``appear'' alongside the positive half-line, one obtains in fact a continuous path, formed of the concatenation of all these loops (i.e., there exists a parameterization that turns this into a continuous path). One may note that the proof of the local finiteness of $\CLE_{\kappa'}$ in \cite{ms2013ig4} relies on the continuity of the space-filling SLE defined there, which is directly related to this continuous path that we are drawing here, so that in a way, the argument goes rather in the other direction i.e.\ the local finiteness of $\CLE_{\kappa'}$ follows from the continuity of a 
path that is related to the concatenation of the loops. 

 Similarly, one could consider the process started from the boundary point at infinity, that moves on the real line and traces on the way all loops that intersect the real line (not just the positive half-line) in the order given by their left-most intersection points with the real line. By applying a conformal transformation, one can then discover for a given domain $D$ and a given boundary point $x$, the process that traces all loops of a $\CLE_{\kappa'}$ in $D$ that touch $\partial D$, in the counterclockwise order of appearance on $\partial D$ (when one starts from $x$).  In this procedure, after a given time $t$, one has discovered a certain 
collection of whole loops, and one is typically in the process of tracing one.

Finally, in order to define the whole $\CLE_{\kappa'}$ one can to iterate this procedure, and define a second layer of boundary intersecting loops in the domains obtained when removing the interiors of all the loops that intersect the real line and so on.  

It is important at this point to observe that conversely, once one samples the entire $\CLE_{\kappa'}$, then this boundary-intersecting loops tracing procedure (and the decomposition into layers) 
is a deterministic function of the $\CLE_{\kappa'}$. 

In the next paragraphs, we will focus on two variations of this construction. The first one will be to consider side-swapping $\bSLE_{\kappa'}^\beta$ (and its ``full'' version) and to define its trunk. The second one will be to replace the $\bSLE_{\kappa'}$ by a $\bSLE_{\kappa'} (\rho)$ variant, and this will define the boundary conformal loop ensembles.

\subsection{Side-swapping, $\CLE_{\kappa'}^\beta$, $\bSLE_{\kappa'}^\beta$, the full $\bSLE_{\kappa'}^\beta$ and its trunk}
\label{subsec:side-swapping}

At each end-time of a macroscopic excursion $e_u$ in the previous construction we can stop the process, and then, the conditional distribution of the not-yet discovered loops of the $\CLE_{\kappa'}$ is that of a collection of independent $\CLE_{\kappa'}$ in each of the connected components that remain to be discovered. This indicates that it is actually possible to then change the starting point of the ``discovery'' process in the unbounded component at that time. One natural possibility is to start the discovery at $\Phi_{u} (y(e_u))$
instead of $\Phi_u (x (e_u))$. This corresponds in fact to the Loewner chain that one would have obtained when going along the loop $e_u$ in the counterclockwise direction.  This leads to the following construction:
\begin{itemize}
\item Independently, for each given $\beta \in [-1,1]$, toss an independent $p_0 = (1 - \beta)/ 2$ versus $1-p_0$ coin to decide whether one 
defines $\psi_u^\beta$ to be equal to  $\varphi_{e_u}$ or to be equal to  $\varphi_{e_u}$ shifted horizontally so that $\psi_u^\beta (y(e_u)) = 0$. 
\item Then, iterate the conformal maps $\psi_u^\beta$ in their order of arrival, i.e., $\Psi_u^\beta$ is the composition of all $\psi_v^\beta$ for $v <u$ in chronological order. 
\end{itemize}
Then again, the obtained loops $\Psi_{u-} ( e_u)$ will be part of a labeled $\CLE_{\kappa'}$ and the appropriate concatenation of the clockwise/counterclockwise parts (when one does not choose $\varphi_{e_u}$ in the coin-tossing, then one takes the counterclockwise part $e^-$ of $e$ from $0$ to $y(e)$ instead of $e^+$) of the half-loops $\Psi^\beta_{u-} ( e_u^\pm)$ will form a side-swapping $\bSLE_{\kappa'}^\beta$ process. 

One way to make sense of this is to first do the side-swapping (i.e.\ to decide to toss a coin) only for the bubbles
$e_u$ that give rise to $\CLE_{\kappa'}$ loops of diameter greater than $\eps$, when the entire picture is mapped onto the unit disk; the set of such swapping times is then discrete, the procedure therefore also defines loops that are part of a $\CLE_{\kappa'}$ and one can also see that the obtained path is a deterministic function of the labeled $\CLE_{\kappa'}^\beta$. It starts for instance like the non-swapping exploration along the boundary, until the first discovered positive loop with diameter at least $\eps$ and so on. 

One natural way to couple all these cut-offs is to first sample the entire labeled $\CLE_{\kappa'}^\beta$, and then define all these deterministically defined 
$\eps$-side-swapped explorations. As explained at the end of the section on  generalized $\SLE_\kappa (\rho)$ processes, the driving function of the obtained process does indeed converge in distribution to that of $\bSLE_{\kappa'}^\beta$ as $\eps \to 0$, and the collections of traced loops as well (for instance, if one picks $n$ given points and looks at the loops surrounding these points if they exist, that are traced by the $\eps$-approximation, their distribution converges to the corresponding one for the $\bSLE_{\kappa'}^\beta$. In particular, we see that the latter are still 
part of a $\CLE_{\kappa'}^\beta$. 

Another possibility is to first sample the Poisson point process of bubbles, and then do the iteration procedure described above for all $\eps$ and the same Poisson point process of bubbles (Figure~\ref{fig:changed_orientation}). In this case, the driving function will converge almost surely as $\eps \to 0$, but the $\CLE_{\kappa'}$ ensembles that are constructed (both before and after the cut-off) then vary from one $\eps$ to another.

\begin{figure}[ht]
\begin{center}
\includegraphics[width=3in]{figures/bui2.pdf}
\includegraphics[width=3in]{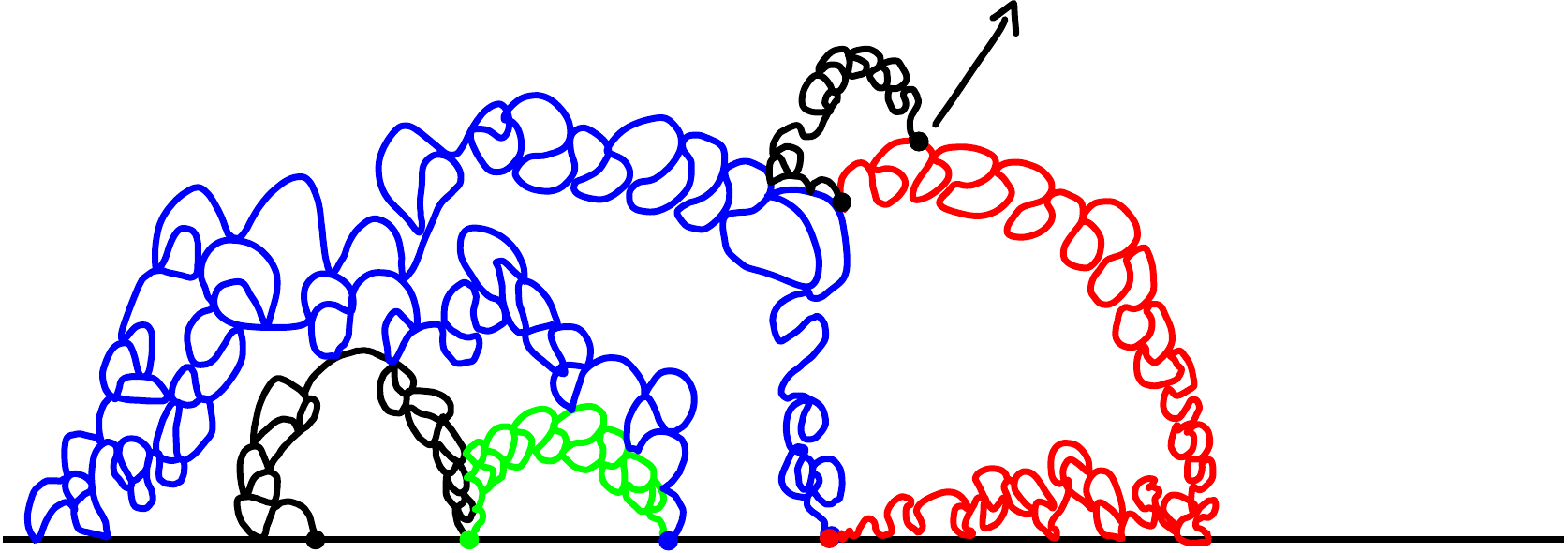}
\end{center}
\caption{\label{fig:changed_orientation} Changing the exploration direction of one loop and side-swapping}
\end{figure}

We now describe what we will call the full  $\bSLE_{\kappa'}^\beta$ process and its trunk (we will first define them both for the version from $0$ to $\infty$ in the upper half-plane). We 
first consider the side swapping $\bSLE_{\kappa'}^\beta$ as defined above when starting from a Poisson point process of labeled bubbles. As we have explained, this process only traces one part of each of the loops it encounters (as it has to branch off to infinity instead of completing them). In order to define the more complete picture, we now complete each of these loops. On the way back to its starting point, each of these loops will bounce off some earlier traced loops of the opposite type, thereby creating a countable family of pockets in between them. These pockets are naturally ordered from $0$ to $\infty$, as indicated in Figure~\ref{fig:pockets} (see also Figure~\ref{fig:cle_6_perc_sim} for a simulation), and each of them contains two marked points that we refer to as their entrance and exit points. 

This procedure therefore defines in the upper half-plane a collection of labeled $\CLE_{\kappa'}$ loops, ordered in their order of appearance along the side-swapping $\bSLE_{\kappa'}^\beta$ from $0$ to $\infty$, and an ordered  collection of simply connected pockets with entrance and exit points. All these collections are invariant in distribution under multiplication by a positive constant, so that we can also define them in other simply connected domains with two marked boundary points (or prime ends). We can therefore iterate the procedure by defining inside each pocket, a second layer of labeled $\CLE_{\kappa'}^\beta$ loops from the entrance point to the exit point of each pocket. We can also then clearly order all the loops in each pocket, and decide the loops in a given pocket gets discovered just after the completion of the first-layer loop that creates that pocket. We then further iterate the procedure. In this way, we define a countable and ordered collection of labeled $\CLE_{\kappa'}$ loops (we 
can note that in the particular case where $\beta = \pm 1$, this corresponds exactly to the discovery of all the loops that touch the real half-line from $0$ to infinity, in their order of appearance along this half-line).  
 
At this point, it is not yet clear whether the concatenation of all of these loops in this order does indeed create a continuous path. However, if it does (and this will be established in Theorem~\ref{thm:duality1}), we call it the {\em full} $\bSLE_{\kappa'}^\beta$ from $0$ to infinity. Then we can erase again all these loops and obtain in this way a continuous path from the origin to infinity, that we call the $\bSLE_{\kappa'}^\beta$ trunk from the origin to infinity.  That is, the $\bSLE_{\kappa'}^\beta$ trunk is the interface between the loops discovered by the full $\bSLE_{\kappa'}^\beta$ with different orientations.  It will follow from our analysis that the trunk is indeed a continuous curve.
However, at this point, even if we do not know whether the full $\bSLE_{\kappa'}^\beta$ from $0$ to $\infty$ is a continuous curve, we know that it is 
a deterministic function of a $\CLE_{\kappa'}^\beta$ that traces some of the oriented loops of the $\CLE_{\kappa'}^\beta$.

\begin{figure}[ht]
\begin{center}
\includegraphics[height=2in]{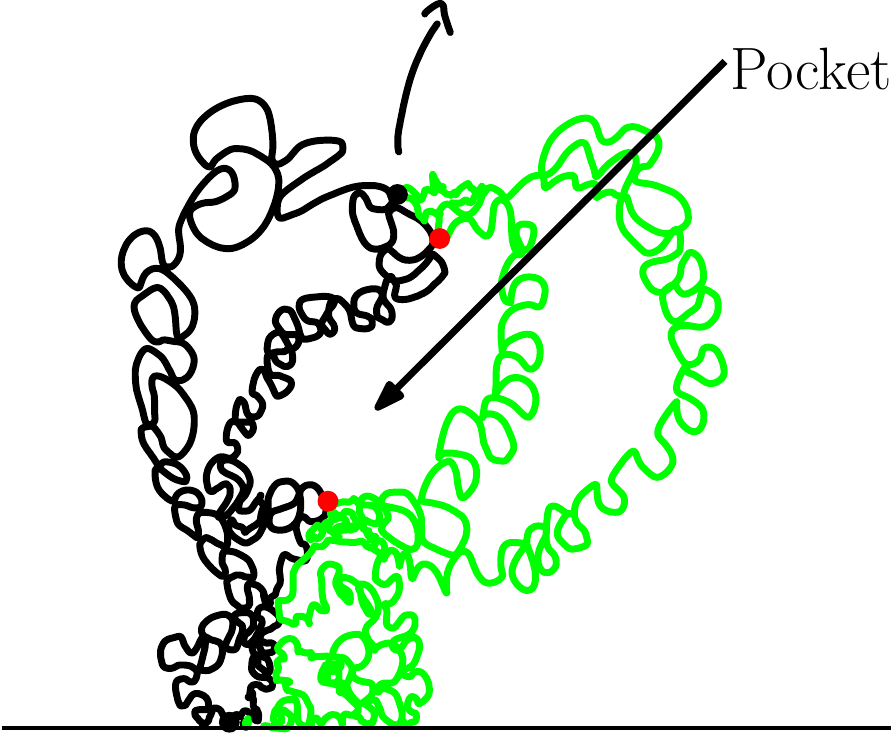} 
\hskip 2mm
\includegraphics[height=2in]{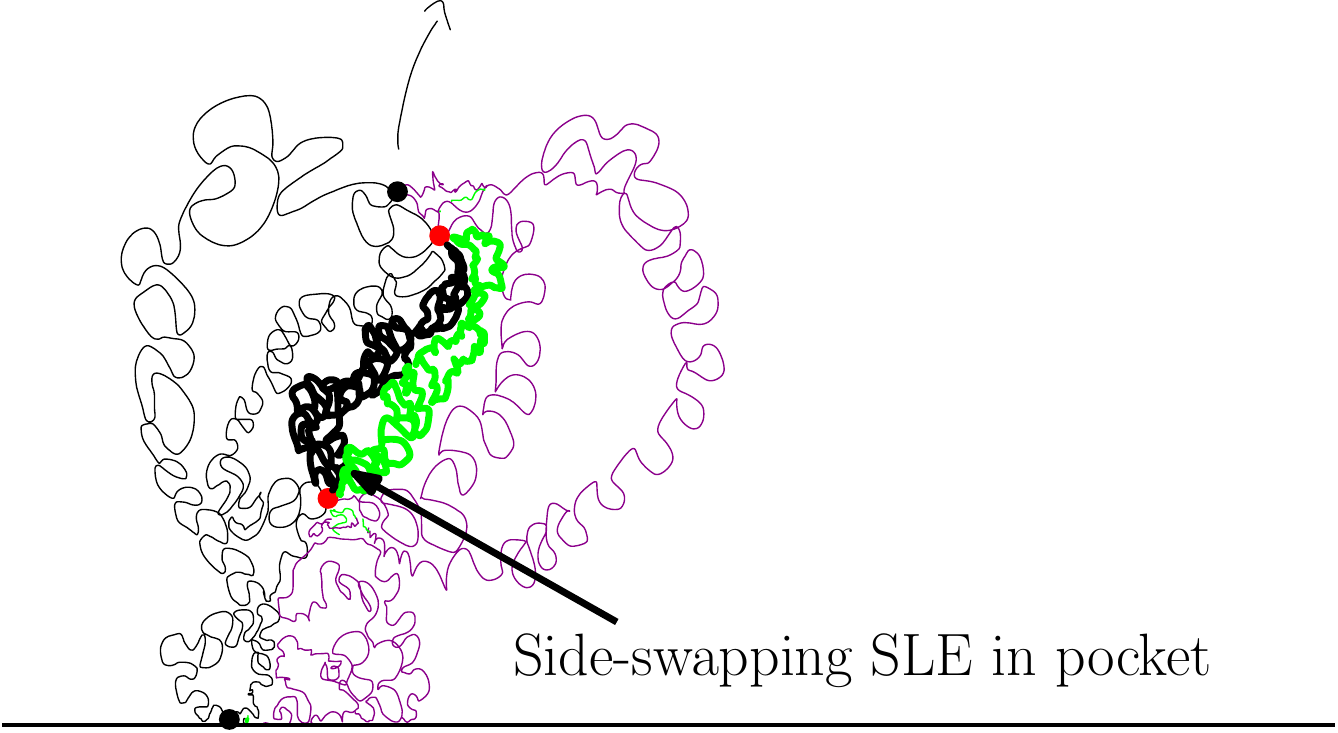}
\end{center}
\caption{\label{fig:pockets}Filling in the pockets in a side-swapping SLE iteratively to construct a full $\bSLE_{\kappa'}^\beta$.  The interface between loops with different orientations is the trunk of the full $\bSLE_{\kappa'}^\beta$.}
\end{figure}
\subsection{Relation between trunk and $\CLE_{\kappa'}^\beta$ percolation}

We now explain how to relate this (conjecturally existing) trunk to percolation of labeled $\CLE_{\kappa'}^\beta$ loops. 

When two $\CLE_{\kappa'}^\beta$ loops have the same sign and touch each other, we say that they belong to the same cluster.  We now consider the union $C^-$ of all negative clusters that touch the negative half-axis (by this we mean that one of the loops in the cluster touches the negative half-axis), and the union $C^+$ of all positive clusters that touch the positive half-axis.  As it turns out, these clusters are closely related to the previously described trunk: 
\begin{proposition}
\label{prop:unique_interface}
If we assume that the full $\bSLE_{\kappa'}^\beta$ process from $0$ to $\infty$ in the upper half-plane is almost surely a continuous curve with zero Lebesgue measure, and that its trunk $\eta$ is a continuous simple curve from $0$ to $\infty$ in $\ol{\h}$, then almost surely, this trunk is equal to the intersection between the boundaries of $C^+$ and $C^-$. 
\end{proposition}
This result will become useful because we shall prove (see Theorem~\ref{thm:duality1}) that the full $\bSLE_{\kappa'}^\beta$ is indeed continuous (and furthermore that its trunk is distributed like an $\SLE_{\kappa} ( \rho; \kappa- 6 - \rho)$ (we will define them in the next section).  This result will also play an important role in \cite{msw2016fan}.

\begin{figure}[ht!]
\begin{center}
	\includegraphics[width=3.2in]{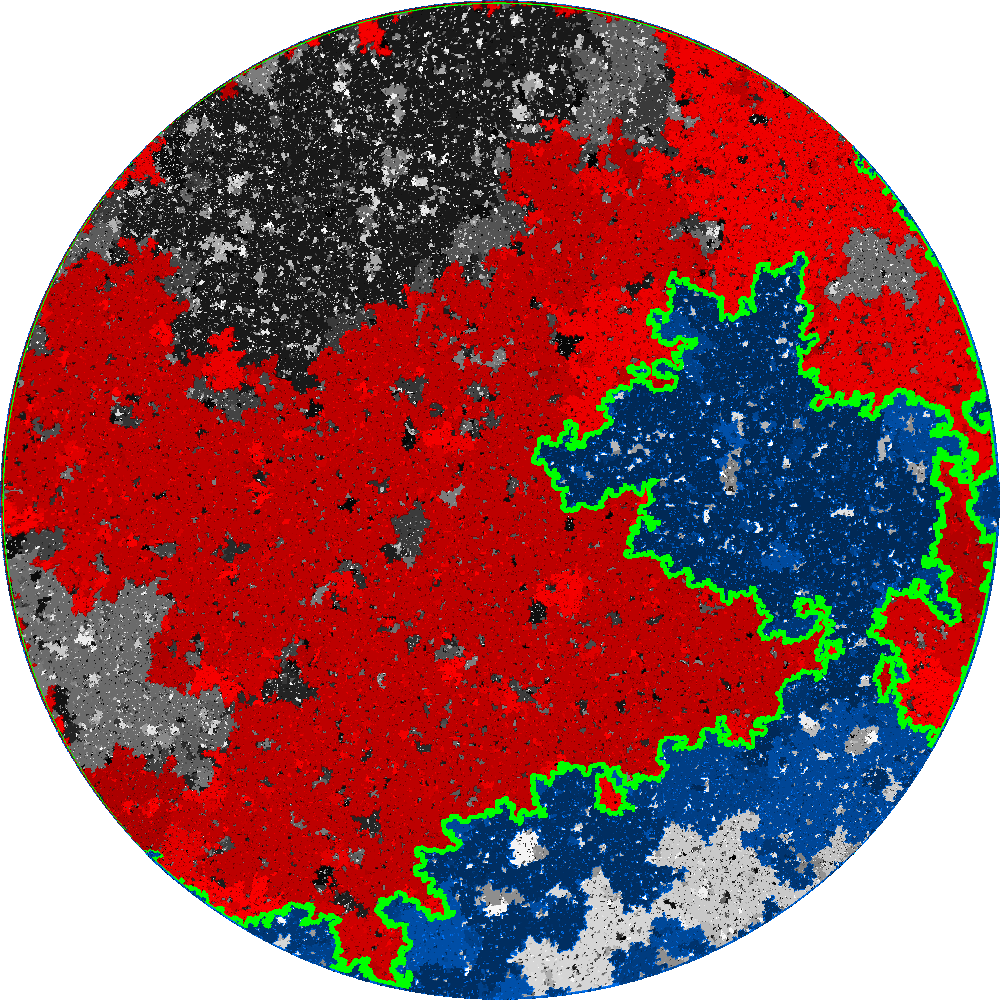}\hspace{0.01\textwidth}\includegraphics[width=3.2in]{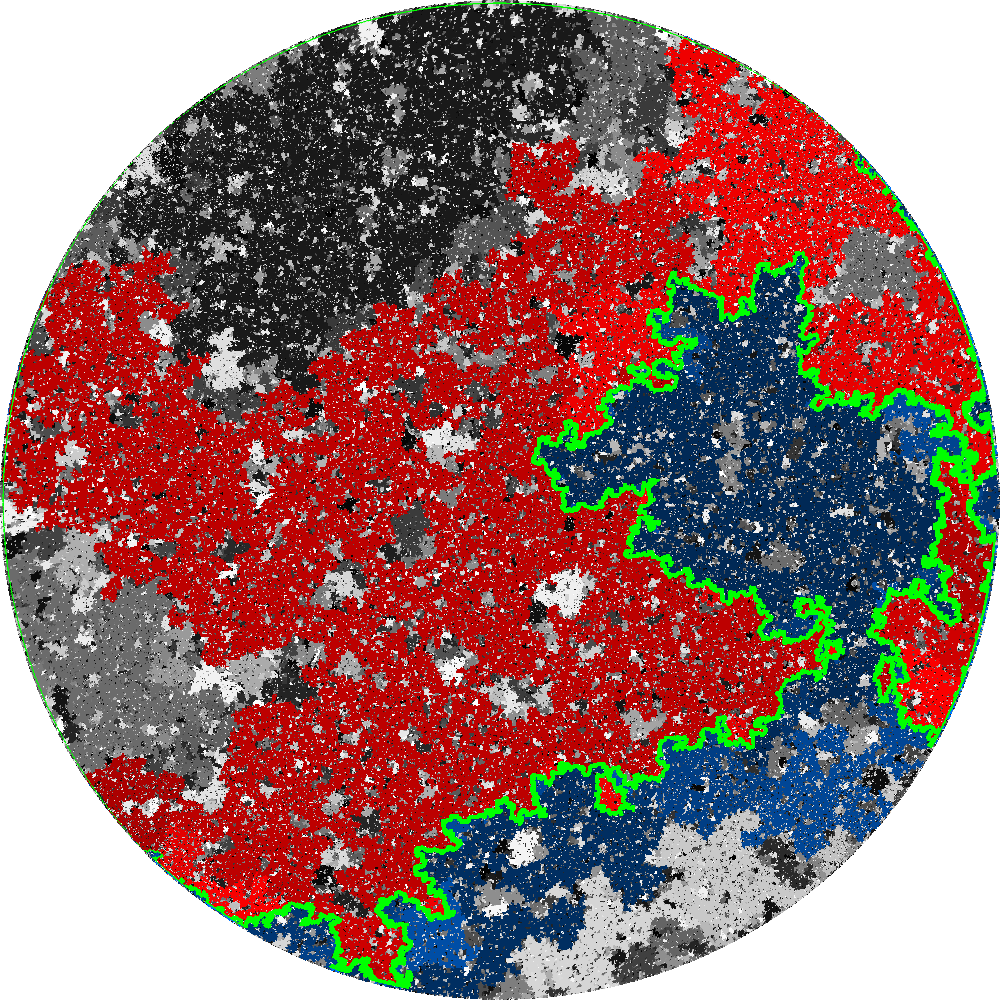}
\end{center}
\vspace{-0.01\textheight}
\caption{\label{fig:cle_6_perc_sim} The $\CLE_6^{\beta}$ interface in the unit disk for $\beta=0$: In the percolation simulation of Figure~\ref{figcle6}, the boundary is divided into two arcs. {\bf Left:} Cluster of percolation clusters with label at at most (resp.\ at least) $1/2$ which 
touches left (resp.\ right) arc is shown in blue.  Their interface is shown in green.  In the scaling limit, the green path should be an $\SLE_{8/3}(-5/3;-5/3)$.  {\bf Right:} Same picture, but only the clusters with label at most (resp.\ at least) $1/2$ which do touch the green interface are colored red (resp.\ blue).  In the scaling limit, the path which follows the red and blue clusters in the order in which they are visited by the green path should be an $\SLE_6^0$.}
\end{figure}

 \begin{proof}[Proof of Proposition~\ref{prop:unique_interface}]
The proof will be based essentially on the two following ingredients: First, using general properties of $\CLE_{\kappa'}$, we will see that $C^+$ and $C^-$ can touch on their boundaries, but cannot overlap i.e.\ that 
$\overline C^+$ and $\overline C^-$ cannot respectively intersect open domains that are respectively to the left-hand or to the right-hand side of $\eta$. Then, using the $\eps$-approximation of the $\bSLE_{\kappa'}^\beta$, we will see that the set of points squeezed between $\overline C^+$ and $\overline C^-$ has zero Lebesgue measure. 

We will work in the upper half-plane. In order to use the notion of local finiteness, we therefore use the $h$-diameter of a set, which is the diameter of its image under a given conformal map from $\HH$ onto the unit disk. Local finiteness of $\CLE_{\kappa'}$ in the $\h$ then means that almost surely, for all $\eps > 0$, only finitely many have an $h$-diameter greater than~$\eps$. 

Let us first make some general observations about the ``graph'' of $\CLE_{\kappa'}$ loops.  We can approximate $C^+$ by the union $C_\eps^+$ of the collection of positive $\CLE_{\kappa'}^\beta$ loops of $h$-diameter greater than~$\eps$, that are connected to the positive half-line by a finite chain of positive $\CLE_{\kappa'}^\beta$ loops of diameter at least~$\eps$.  Clearly $\cup_\eps C_\eps^+ = C^+$.  We call $c_\eps^+$ to be the ``left boundary'' of $\R_+ \cup C_\eps^+$. The local finiteness of $\CLE_{\kappa'}$ ensures that $c_\eps^+$ is a continuous curve from~$0$ to~$\infty$ in~$\overline \HH$ (note that some $\CLE_{\kappa'}^\beta$ loops of both open and closed type will touch both the positive and the negative half-line, which  shows that $c_\eps^+$ will touch both the positive and the negative half-line when~$\eps$ small).  We define also $c_\eps^-$ symmetrically (via chains of negative loops attached to the negative half-line), and the open sets $O_\eps^+$ and $O_\eps^-$ corresponding to the part of 
$\HH$ lying to the right of~$c_\eps^+$ and to the left of~$c_\eps^-$ respectively.  See Figure~\ref{fig:cutoff_clusters} for an illustration. A point in $O^+_\eps$ is then a point that is disconnected from the negative half-line by a chain of positive $\CLE_{\kappa'}^\beta$ loops of $h$-diameter greater than~$\eps$. The maps $\eps \mapsto O_\eps^+$ and $\eps \mapsto O_\eps^-$ are non-increasing, and we define $O^+ := \cup_\eps O_\eps^+$ and $O^- := \cup_\eps O_\eps^-$. The open set $O^+$ is therefore exactly  the collection of points that are disconnected from the negative half-line by a finite chain of positive $\CLE_{\kappa'}^\beta$ loops.  

We can already list a few further  easy consequences of the fact that almost surely, $\CLE_{\kappa'}$ is locally finite, that if a $\CLE_{\kappa'}$ loop touches another loop or the real line, then it does so at infinitely many other points, and that a $\CLE_{\kappa'}$ loop does not have any cut-points (i.e., that its outer boundary is a simple loop). It is indeed then easy to see that $O_\eps^+ \cap O_\eps^- = \emptyset$,  and that a point $z \in O_\eps^+$ can be linked to $\R^+$ by a continuous path that stays a positive distance from $c_\eps^+$, and therefore also of  the closure of $O^-$ (as $O^-$ lies on ``the other side'' of $c_\eps^+$).

\begin{figure}[ht]
\begin{center}
\includegraphics[width=3.1in]{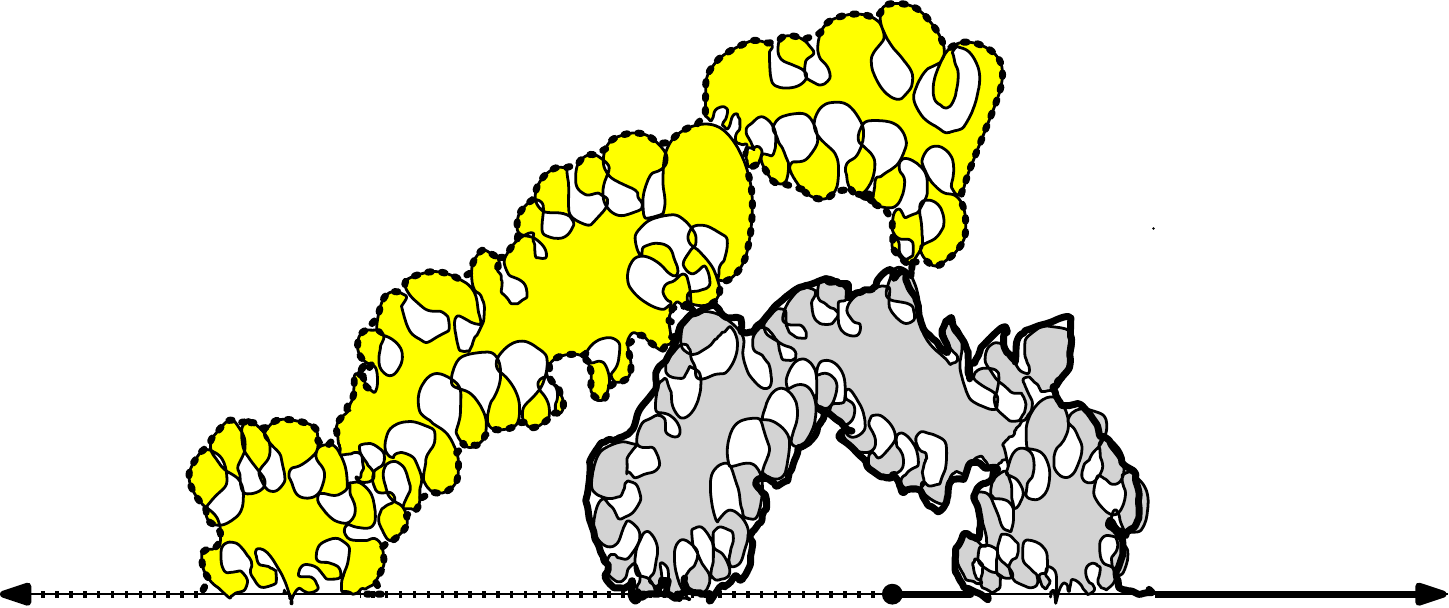} 
\includegraphics[width=3.1in]{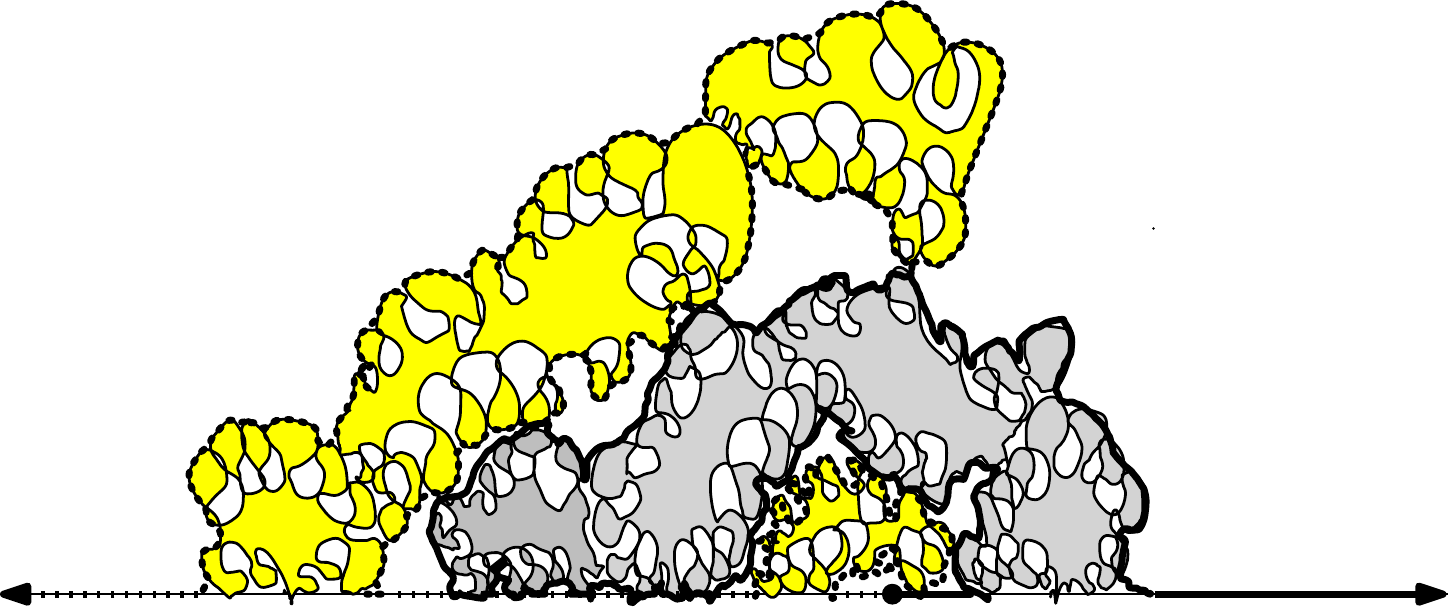} 
\end{center}
\caption{\label{fig:cutoff_clusters} The clusters $C_\eps^+$, $C_\eps^-$ and the paths $c_\eps^+$ and $c_\eps^-$ (in dotted) for two different values of $\eps$}
\end{figure}

We now are going to study the position of the trunk of the full $\bSLE_{\kappa'}^\beta$ trunk from $0$ to infinity with respect to the sets $O^+$ and $O^-$.  Let us call $T^+$ and $T^-$ the two open sets that lie respectively to the right and to the left of the trunk $T$ (recall that we assume that the trunk is a simple curve in $\overline \HH$, but  we expect it to almost surely hit the positive and the negative half-line infinitely many times, so that $O^+$ will actually have infinitely many connected components as soon as $|\beta| \not= 1$).  Let us now argue why $O^+ \subset T^+$ almost surely, using the coupling of all $\eps$-cutoffs of the side-swapping procedure with a single $\CLE_{\kappa'}^\beta$. We note that almost surely, for all $\eps$, the exploration path will almost surely go around $O_\eps^+$ (it may trace the positive loops that are creating the chains of positive loops, but will not do anything else on the right of $c_\eps^+$). It therefore follows readily that in the limiting coupling of 
an 
$\CLE_{\kappa'}^\beta$ with the trunk $T$, $O_\eps^+ \subset T^+$ almost surely.  Hence, we get indeed that $O^+ \subset T^+$, and (symmetrically) $O^- \subset T^-$. It follows that the intersection of the closure of $O^+$ with the closure of $O^-$ is necessarily contained in $T$.  

Next, we are going to argue that for every given $z$, the point $z$ is almost surely not in   $T^+ \setminus O^+$.  As (under the assumptions of the proposition) the trunk has zero Lebesgue measure, it then follows that $O^+ \cup O^-$ has full Lebesgue measure, that $O^+$ is dense in $T^+$, that $O^-$ is dense in $T^-$, which is enough to deduce that $T$ is actually equal to the intersection between the boundary of $O^+$ with the boundary of $O^-$ and concludes the proof of the proposition.

As illustrated in Figures~\ref{statusviaradial} and~\ref{fig:approxclusters}, one  can observe that in order to decide whether $z$ ends up to the right or to the left of the trunk, it suffices to study a radial $\bSLE_{\kappa'}^\beta$ that targets $z$, and that will exactly follow the iterated definitions within the pockets.
\begin{figure}[ht]
\begin{center}
\includegraphics[width=2.4in]{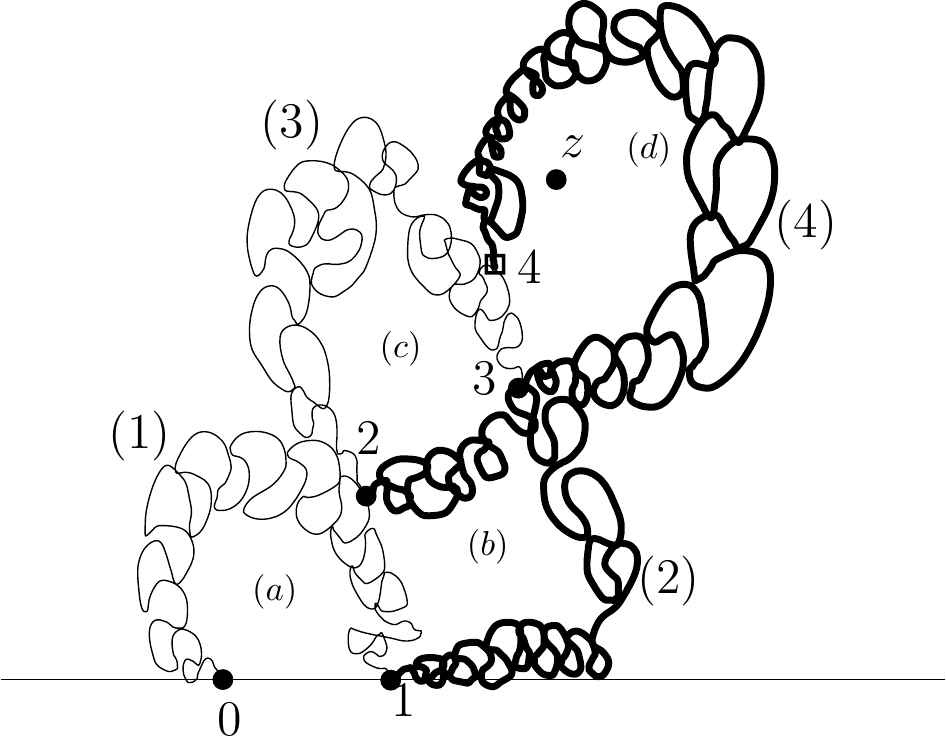} 
\includegraphics[width=2.4in]{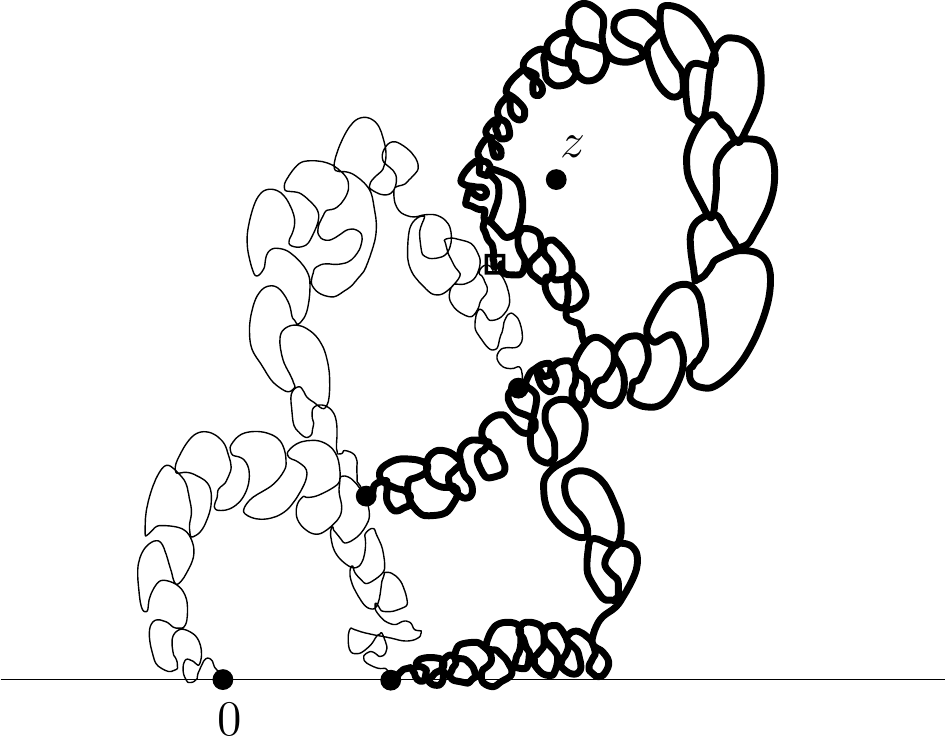} 
\end{center}
\label {fig69}
\caption{\label{statusviaradial}(i) The radial and chordal $\bSLE_{\kappa'}^\beta$ coincide until the point labeled~$4$ is reached. In bold we indicate the parts of the counterclockwise loops (to the right of the trunk). When visiting the points~$1$ to~$4$, it is creating the pockets (a)-(d) (but there are many more pockets, created just after the starting point, and after visiting the points~$1$, $2$ and~$3$). (ii) The radial exploration starts exploring within the pocket (d) containing $z$ and discovers on which side of the trunk $z$ is.}
\end{figure}
As mentioned before, the radial $\bSLE_{\kappa'}^\beta$ can be well approximated via a ``side-swapping'' cut-off, where all bubbles of $h$-diameter smaller than $\eps$ have a negative sign, and only the signs of the finitely many ones with $h$-diameter greater than $\eps$ get tossed at random with the $p_0$-coin. In particular, the probability that $z$ ends up to the right of the trunk of this approximated radial $\bSLE_{\kappa'}^\beta$ converges to $\P( z \in T^+)$ as $\eps$ tends to $0$.

\begin{figure}[ht]
\begin{center}
\includegraphics[width=2.4in]{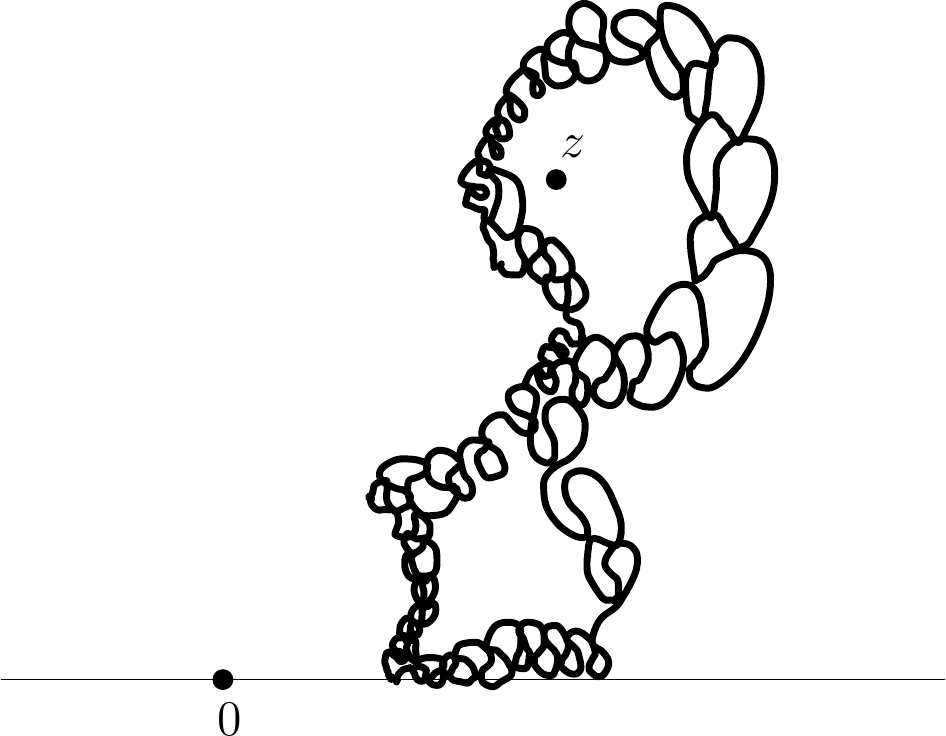} 
\end{center}
\caption{\label{fig:approxclusters} In the $\eps$-approximation of the configuration of Figure~\ref{fig69}, the point $z$ is in $O^+_\eps$}
\end{figure}
But we also observed that for this $\eps$-approximation, if a point $z$ ends up to the right of this (radial) trunk,  it is in the $O^+_\eps$ of the $\CLE_{\kappa'}^\beta$ that can be coupled to it. Hence, letting $\eps \to 0$, we see that indeed 
\[ \P(z \in O^+ ) = \lim_{\eps \to 0^+} \P(z \in O_\eps^+) \ge  \P(z \in T^+).\]
\end{proof}

\section{Boundary conformal loop ensembles and their duality}
\label{sec:bcle}
\label{S7}
\label {Sec7}

\subsection{BCLE definition and first properties}

\subsubsection {Heuristics}
We are now going to define the \emph{boundary conformal loop ensemble} with parameters~$\kappa$ and~$\rho$ that we will denote by $\BCLE_{\kappa} (\rho)$. 
This object will be well-defined for each $ \kappa \in (2,8)$ and each $\rho$ in a certain $\kappa$-dependent range specified below. Let us stress here already that some $\BCLE_{\kappa} (\rho)$ exist for $\kappa \in (2, 8/3]$ even if $\CLE_{\kappa}$ itself does not exist. We will talk about $\BCLE_{\kappa'} (\rho')$ and $\BCLE_{\kappa}(\rho)$ when we will want to differentiate between the cases $\kappa' >4$ and $\kappa \le 4$. 
A BCLE will be a random countable collection of loops defined in the closure of any simply connected domain $D$ with the property that each loop looks locally like an $\SLE_{\kappa}( \rho)$ in the inside of $D$ and intersects the boundary of $D$ on a fractal uncountable set. As it will be a conformally invariant object, it suffices to define it in the upper half-plane or in the unit disk. 

The intuition is that this class of loop ensembles should describe the scaling limits of the family of boundary-intersecting loops for a class of discrete critical models, with the idea that for a given $\kappa$, each value of $\rho$ corresponds to a different boundary condition that for instance can be viewed as a different way to weight the boundary-touching loops compared to the interior loops. Our results will in fact make this intuition more precise.

In Section~\ref{Sec6}, we have explained how to build the collection of $\CLE_{\kappa'}$ loops that touch the boundary of the upper half-plane 
for $\kappa' \in (4,8)$, starting from a Poisson point process of $\SLE_{\kappa'}$ excursions, or equivalently from a $\bSLE_{\kappa'}$ branching tree. The construction of the BCLEs will follow along the same lines, but where one replaces the $\SLE$ bubbles (and henceforth the $\bSLE$) by $\SLE_{\kappa} ( \rho)$ bubbles (and henceforth a $\bSLE_{\kappa}(\rho)$ generalization of $\bSLE_{\kappa}$).
In particular, we already emphasize that  the construction of boundary conformal loop ensembles will involve neither side-swapping nor principal value compensation  arguments,  and will  
be based on $\SLE_\kappa (\rho)$ ideas for $\rho > -2$ solely. 

Recall that the $\SLE_{\kappa}$ bubble measure was the limit as $\eps \to 0^+$ of the appropriately rescaled distribution of an $\SLE_{\kappa}$ from $0$ to $\eps$ in $\HH$. This is an infinite measure 
on loops from $0$ to $0$ in the upper half-plane that possess the following scaling property. If $A$ is a set of loops and $\lambda A$ denotes the set of loops $\gamma$ such that $\gamma / \lambda \in A$, then the mass of $\lambda A$ is equal to $\lambda^{-b}$ times the mass of $A$ for $b = b (\kappa) = 8/\kappa - 1$.
Similarly, the mass of the image of $A$ under a M\"obius transformation $\psi$ of the upper half-plane that keeps the origin fixed is equal to 
$\psi' (0)^{-b}$ times the mass of $A$. 

The fact that $b(\kappa')\in (0,1)$ when $\kappa' \in (4,8)$ can be viewed as the reason that makes it possible  to 
define $\SLE_{\kappa'} (\kappa' -6)$ via a Poisson point process of such bubbles in this way only for this range of~$\kappa'$. This construction (using the conformal invariance property of the bubble measure)
provides a direct way to see that this process is target-invariant. 
The reversibility properties of $\SLE_{\kappa'} (\kappa'-6)$ then allows one to show that the loops defined by this branching $\bSLE_{\kappa'}$ tree does not depend on the choice of its root.

In Section~\ref{subsubsec:classical_sle_kr}, we have recalled the definition of $\SLE_\kappa (\rho)$ from~$0$ to~$\infty$ with marked point at~$0^+$ or at~$0^-$, when $\rho > -2$. By applying a M\"obius transformation $\H \to \H$, this yields the definition of an $\SLE_\kappa (\rho)$ from $x$ to $y$ in $\HH$ with marked point at $x^-$ for any two distinct reals $x$ and $y$.  We can then easily define the $\SLE_\kappa (\rho)$ excursion/bubble measure as the properly renormalized limit when $\eps \to 0^+$ of the law of $\SLE_\kappa (\rho)$ from $0$ to $\eps$ with marked point at $0^-$. Again, this measure satisfies a scaling property (as one can expect from its definition) with some exponent $b$. 
When $b \in (0,1)$ (which imposes the constraint $ -2 < \rho < \kappa-4$), the same procedure as for $\SLE_{\kappa'}(\kappa' -6)$ 
then enables us to start from a Poisson point process of such $\SLE_\kappa (\rho)$ excursions to construct a target-invariant process that we denote by $\bSLE_\kappa (\rho)$. This target-invariant version can be equivalently described as an $\SLE_{\kappa} (\rho ; \kappa -6- \rho)$ process.

\subsubsection{$\BCLE_\kappa(\rho)$ for $\kappa \in (2,4]$}

We now turn to the more formal definition of these boundary conformal loop ensembles. 
Let us first consider the case where $\kappa \in (2,4]$.  When $\kappa$ is in this range, we will define $\BCLE_\kappa(\rho)$ for each $\rho$ such that
\begin{equation}
\label{eqn:bcle_k_range}
	-2 < \rho <  \kappa-4.
\end{equation} 
Note already that $\rho$ is in this admissible interval if and only if $\kappa  - 6 - \rho$ is in this admissible interval; in fact this admissible interval is the set of 
$\rho$'s for which $\rho > -2$ and $\kappa - 6 -\rho > -2$.  Hence SLE$_\kappa (\rho)$ or SLE$_\kappa (\rho ; \kappa - 6- \rho)$ will all be classical processes with no side-swapping or compensation.
 Note also that $\rho$ has to be negative and in particular that the value $\rho = 0$ is not in this admissible range.  

The $\BCLE_\kappa (\rho)$ is then defined to be the set of boundary touching loops that are traced by a branched $\bSLE_\kappa (\rho)$.  Let us consider the case where $D$ is the unit disk (in other simply connected domains, the definition will then be given via conformal invariance). For 
every given boundary point $x$, one can then define a branching tree of $\SLE_\kappa (\rho; \kappa -6- \rho)$
processes starting from $x$ and targeting all other boundary points. 
The trace of this branching tree in the unit disk will consist of the union of a countable family of disjoint ``boundary-to-boundary'' arcs. Each of these arcs 
looks locally like an $\SLE_\kappa$-curve (when it is away from $\partial D$) and  comes with an orientation (the tree being naturally oriented from $x$ towards the other boundary points).
If one considers any fixed point $z \in D$, the boundary of the connected component of the complement of this branching tree that contains $z$ will then consist almost surely of the concatenation of such boundary-to-boundary arcs, that will form either a clockwise loop around $z$ or a counterclockwise loop around $z$. 

The $\BCLE_\kappa (\rho)$ is going to be the knowledge of all these oriented boundary-to-boundary arcs. Clearly, this information can be either encapsulated by the collection  of 
all the clockwise loops that it defines and that we will denote as $\cwBCLE_\kappa (\rho)$, or equivalently by the collection of counterclockwise loops that it defines. We will often refer 
to the clockwise loops as the {\em true  loops} of $\cwBCLE_\kappa (\rho)$ and to the counterclockwise loops as the {\em false loops} of $\cwBCLE_\kappa (\rho)$. 

Conformal invariance and target-invariance show immediately that the law of this $\cwBCLE_\kappa (\rho)$ is invariant under any conformal automorphism of $D$ that keeps $x$ fixed (at this point, we have not yet argued that it does not depend on the choice of the root $x$). 
 Furthermore, the time-reversal symmetry also shows that the $\SLE_\kappa (\rho)$ bubble measure is invariant under the transformation that traces the loop counterclockwise  (instead of clockwise) and then takes its image under the symmetry $x + iy \mapsto -x +iy$. Hence, it follows that if one reverses the orientation of all the  loops of a $\cwBCLE_\kappa (\rho)$ and then takes the image under any given anti-conformal transformation, one gets again 
 a $\cwBCLE_{\kappa} (\rho)$.

Let us now recall (see for instance \cite {sw2005coordinate}) that when one considers in $D$ an $\SLE_\kappa (\rho)$ from $x$ to $y$ with marked point at $y'$, it is also an $\SLE_\kappa (\wt \rho)$ process from $x$ to $y'$ with marked point at $y$ (until the first time it disconnects $y$ from $y'$) with $\wt \rho = \kappa - 6 - \rho$. Note also that any given point $z$ in $D$ is almost surely surrounded 
either by a true loop or by a false loop of a $\cwBCLE_\kappa(\rho)$. The collection of false loops traced by  this $\cwBCLE_\kappa(\rho)$ is 
therefore also exactly created by the image of a $\bSLE_\kappa (\kappa - 6  - \rho)$ branching tree under an anti-conformal automorphism of $D$ that keeps $x$ fixed. 

We can reformulate this by saying that the following two ways to construct a family of counterclockwise loops (using the branching tree rooted at $x$) are identical in law: 
\begin {itemize}
 \item Consider a $\cwBCLE_\kappa (\rho)$ and just reverse the orientation of all its loops. 
 \item Consider a  $\cwBCLE_\kappa ( \kappa -6- \rho)$ and look at the collection of false loops it defines. 
 \end {itemize}
We will refer to the oriented loops obtained in this way as a $\ccwBCLE_\kappa (\rho)$. The true loops of $\ccwBCLE_\kappa (\rho)$ will be this time these counterclockwise loops, and 
the false loops of $\ccwBCLE_\kappa (\rho)$ will be this time the corresponding clockwise loops.

Then, the reversibility property of the $\SLE_\kappa (\rho)$ processes established in \cite{ms2012ig1,ms2012ig2} using the arguments of \cite[Proposition~5.1]{she2009cle}
implies (just as for the case of $\CLE_{\kappa'}$) that the law of $\BCLE_\kappa(\rho)$ does not depend on the choice of the root $x$.  
The law of a  boundary conformal loop ensemble is therefore invariant under the whole group of M\"obius transformations of $D$.

Let us finally note that in the boundary case $\rho=-2$, one can view and define $\cwBCLE_\kappa(-2)$ as consisting of 
just of one single true loop which traces the domain boundary clockwise (and there are no false loops),
whereas the case $\rho=\kappa-4$ corresponds to the case where there is no true loop at all and a single counterclockwise false loop along the domain boundary.  In the sequel, we will use this as a definition of $\cwBCLE_\kappa (-2)$ and $\cwBCLE_{\kappa} (\kappa -4)$. This therefore extends the range of admissible $\rho$-values to 
the closed interval $[\kappa -4,  -2]$.

\subsubsection{$\BCLE_{\kappa'}(\rho')$ for $\kappa' \in (4,8)$}

 We now consider the case of  $\cwBCLE_{\kappa'}(\rho')$  and  $\ccwBCLE_{\kappa'}(\rho')$ for $\kappa' \in (4,8)$. The definition is basically identical, except that the range of admissible $\rho'$ is now 
\begin{equation}
\label{eqn:bcle_kp_range}
	\frac{\kappa'}{2}-4 < \rho' <  \frac{\kappa'}{2}-2.
\end{equation}
 As before, this implies that $\rho' > -2$, so that we will be dealing with classical $\SLE_{\kappa'} ( \rho')$ processes. Note that again, the condition on $\rho'$ remains the same when one changes $\rho'$ into $\kappa' - 6 -\rho'$, but there are some little differences with the previously described case of $\BCLE_\kappa (\rho)$ for $\kappa \in (2,4)$:
 \begin {itemize}
  \item The obtained $\bSLE_{\kappa'} ( \rho')$ branching tree defines a random countable collection of non-simple  boundary-to-boundary arcs. 
  The concatenation of these arcs then defines a random collection $\cwBCLE_{\kappa'}(\rho')$ of (true) clockwise loops and also a collection of false counterclockwise loops. 
 The properties that we have derived in the $\kappa < 4$ case hold as well for these $\BCLE_{\kappa'} (\rho')$ families but the loops are not simple curves anymore, so that they cannot be directly viewed as the outer boundaries of the connected components of the complement of the branching tree.
  \item This time $\rho'=0$ is in the allowed interval. The $\bSLE_{\kappa'} (0)$ is the usual SLE$_{\kappa'} (\kappa' -6)$ process, and the set of true loops of a $\cwBCLE_{\kappa'} (0)$ is then just the set of loops in a $\CLE_{\kappa'}$ that intersect the boundary and that are traced clockwise. 
  \item This condition on $\rho'$ and $\kappa'$ is stronger than $\rho' \in (-2, \kappa'-4)$. This corresponds to the fact that it is necessary to ensure that the $\bSLE_{\kappa'}(\rho')$ does not trace the entire boundary of the domain, so that the boundary branching process really does branch (see \cite{ms2013ig4}) and is reversible (see the non-reversibility problems for $\rho < \kappa'/2-4$ in \cite{ms2012ig3}).  One also requires that $\kappa' < 8$ for these  processes to be reversible (see the non-reversibility problems for $\kappa' > 8$ in \cite{ms2013ig4}).
  \end {itemize}
In the boundary case $\rho'=\kappa'/2-4$, we can define $\cwBCLE_{\kappa'}(\kappa'/2-4)$ as a single clockwise loop which fills the whole domain boundary.
However in this $\kappa' \in (4,8)$ regime, this single loop is not just the boundary itself (as the limiting case $\kappa'  \to 8^-$ shows, where it becomes one single space-filling loop).  We similarly define its reverse-orientation $\ccwBCLE_{\kappa'}(\kappa'/2-4)$ as a single counterclockwise loop which fills the boundary of the domain, and 
$\ccwBCLE_{\kappa'} ( \kappa'/2 -2 )$ and $\cwBCLE_{\kappa'} ( \kappa'/2 -2 )$ as the collection of false loops defined by $\cwBCLE_{\kappa'}(\kappa'/2-4)$ and $\ccwBCLE_{\kappa'}(\kappa'/2-4)$ 
respectively. Since the latter {$\BCLE$s} consist of one boundary-filling true loop, each of these false loops touch the boundary at just one point).

\subsubsection{Basic properties of $\BCLE$}

Let us sum up the basic properties of all these $\BCLE$s in the form of the following proposition.

\begin{proposition}
\label{prop::bcle_prop} Let us consider a $\cwBCLE_{\kappa}(\rho)$ process $\Gamma$ in the simply connected domain $D$ as defined above for all $\kappa \in (2, 8)$ and $\rho$ in the corresponding admissible range. Then: 
\begin{enumerate}[(i)]
 \item The law of $\Gamma$ is invariant under any given conformal automorphism of $D$.
 \item The image of $\Gamma$ under an anti-conformal automorphism of $D$ is a $\ccwBCLE_{\kappa}(\rho)$, a sample of which can be produced by reversing the orientations of the loops of a $\cwBCLE_{\kappa}(\rho)$. 
 \item The collection of false loops traced by $\Gamma$ (i.e., the counterclockwise loops traced by the union of boundary-to-boundary arcs of $\Gamma$-loops around the points that are inside no loop of $\Gamma$) 
 form a $\ccwBCLE_{\kappa} (\kappa- 6 - \rho )$. 
 \end{enumerate}
\end{proposition}

The relation between the clockwise and counterclockwise BCLEs indicates a special symmetry feature of the $\BCLE_\kappa((\kappa-6)/2)$.  
For a simple discrete analog of this duality, one can consider critical Bernoulli percolation on the faces of a hexagonal lattice in a simply connected domain. The collection of outer boundaries of all white clusters that touch the boundary describes a discrete analog of the boundary touching loops of a $\CLE_6$. 
The collection of outer boundaries of all black clusters that touch the boundary has of course the same distribution (as that of the white clusters). The samples of the two discrete loop ensembles differ from each other, but almost surely, each of them is made of the union of all black-white interfaces that touch the boundary.  In this setting, both loop ensembles correspond, in the scaling limit, to the set of boundary-intersecting loops of a $\CLE_{6}$, and this set has the law of $\BCLE_6(0)$.  Similarly, one could for instance guess that 
the set of outer boundaries of Ising clusters that touch the boundary of a domain (with free boundary conditions) should give rise to a discrete approximation of the $\BCLE_3(-3/2)$.  This is indeed the case, as has been shown in rigorous work on the Ising model on a grid \cite{HONGLER_KYTOLA} (see also \cite {iz2015freeboundary}). The $\kappa=4$ case (i.e.\ $\BCLE_4(-1)$) corresponds to the zero level-lines that touch the boundary in a GFF with zero boundary conditions \cite{ss2010continuumcontour} that is also the scaling limit of a white/black coloring of the hexagonal lattice that is symmetric in distribution.

\subsection{BCLE nesting and duality statements}

\subsubsection{Main duality statements}

Let us start with a $\cwBCLE_\kappa (\rho)$ as defined above for some fixed $\kappa \in (2, 4)$. Throughout this section, $\kappa'$ and $\kappa$ will always be related by $\kappa \kappa' = 16$. 
Denote by $\Lambda$ the set of 
its true clockwise loops thus created.  We are now going to construct $\BCLE$s within the true and the false loops of  $\cwBCLE_\kappa (\rho)$:
Inside of each clockwise true (resp.\ counterclockwise false) loop $\CL$ of $\Lambda$, we sample an independent $\ccwBCLE_{\kappa'}(\rho_R')$ (resp.\ $\cwBCLE_{\kappa'} (\rho_L')$) for some admissible $\rho_R'$, $\rho_L'$ (see Figure~\ref{fig:bcle_k_rho} as well as Figure~\ref{fig:perc_sim_later}).  As we will see later, it will be useful that in this procedure, the nested $\BCLE$s have the \emph{opposite} orientation of the true loop or false loop that contains them. The actual choice of the values of $\rho$, $\rho_R'$ and $\rho_L'$ will be important, and we will discuss this later. 

\begin{figure}[ht!]
\begin{center}
\includegraphics [width=3in]{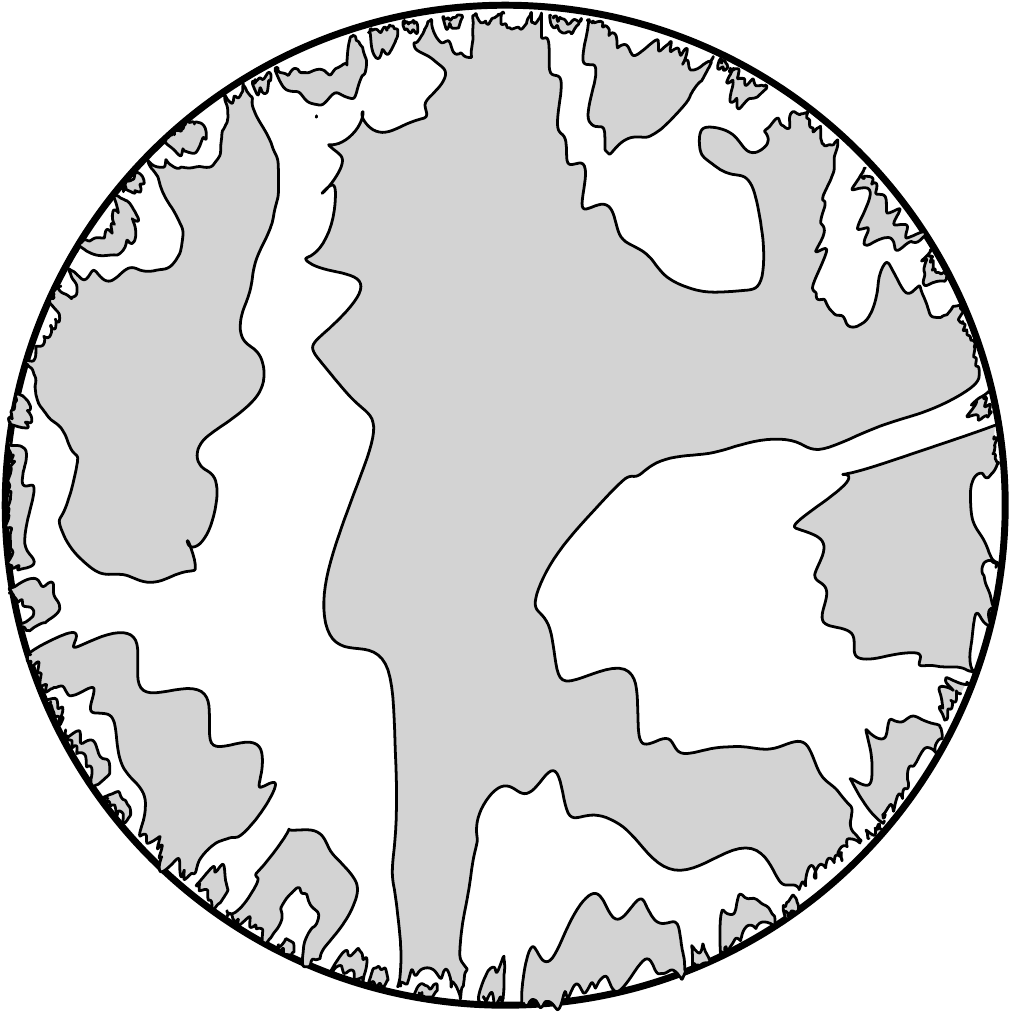}
\includegraphics [width=3in]{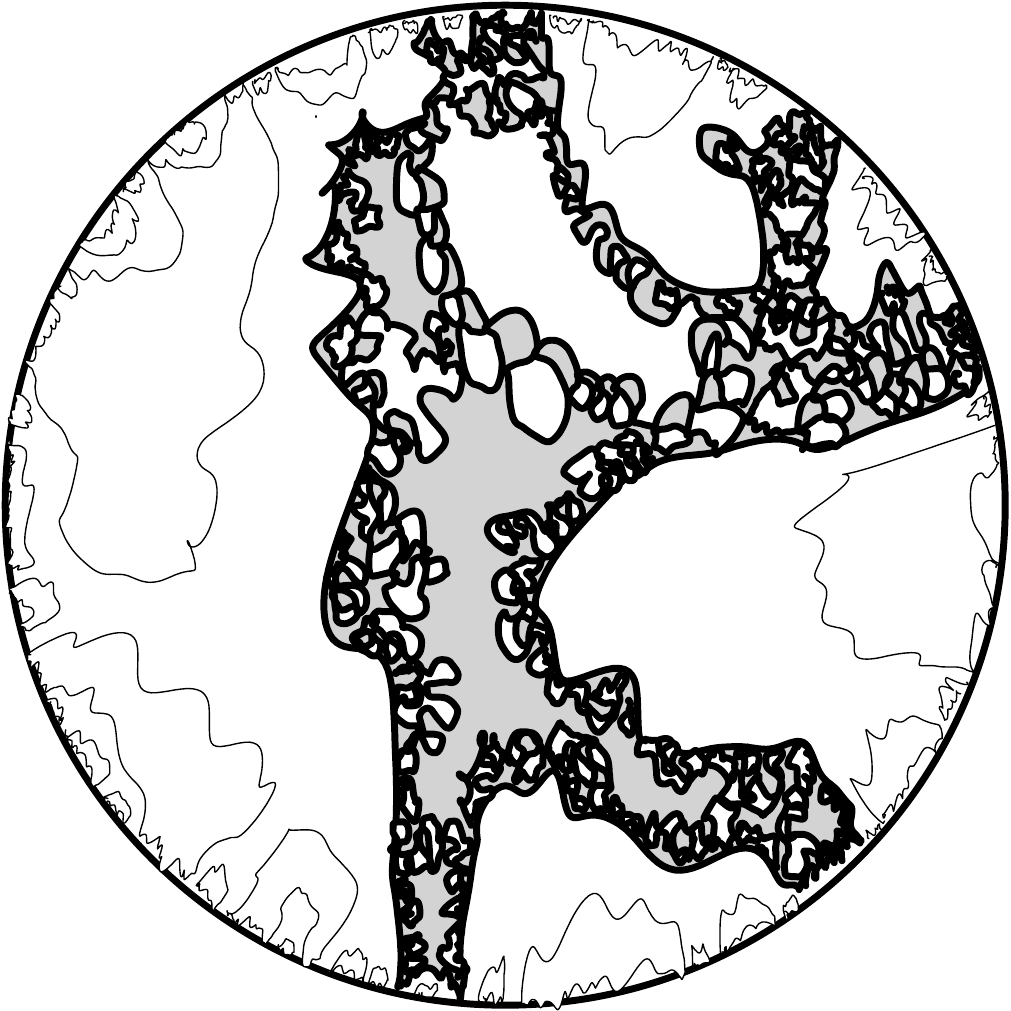}
\caption{\label{fig:bcle_k_rho} A sketch of a $\BCLE_\kappa (\rho)$ and of a nested $\BCLE_{\kappa'} (\rho')$ in one of the BCLE loops.}
\end{center}
\end{figure} 

We now choose a boundary point $x$, from which we will start exploring loops of $\Lambda$ by moving counterclockwise on $\partial D$. 
Assuming that $\Lambda$ is locally finite, we can define a (non-simple) loop $\eta_\Lambda$ from $x$ to $x$ as follows. 
To traverse $\eta_\Lambda$, we follow $\partial D$ counterclockwise starting at $x$ except that each time we encounter a loop of $\Lambda$ for the first time, we traverse the entire loop clockwise before continuing. Thus $\eta_\Lambda$ traces the loops of $\Lambda$ in the order determined by their clockwise-most intersection points with $\partial D$ viewed from $x$. For any given boundary point $y \not= x$, if we parameterize this path according to its half-plane capacity seen from $y$ (thereby excising all the parts of $\eta_\Lambda$ that are hidden from $y$), we obtain a chordal $\bSLE_\kappa (\rho)$ from $x$ to $y$.

Assuming further that the entire collection $\Gamma'$ consisting of all the loops of all the different $\BCLE$s that we defined inside each of the true and false loops of $\Lambda$ is also locally finite, we can then define another loop $\eta'$ from $x$ to $x$ as follows. To trace this path, we follow $\eta_\Lambda$ except that each time we first encounter a point on a loop in $\Gamma'$ we traverse that loop (clockwise or counterclockwise, respectively, depending on whether the loop lies left or right of $\eta_\Lambda$, respectively) before continuing.  We note that $\eta_\Lambda$ never encounters a clockwise and a counterclockwise loop simultaneously, so that the path $\eta_\Lambda$ is well-defined.  This follows because, given $\eta_\Lambda$, the clockwise and counterclockwise loops are independent of each other, there are only a countable number of each, and the probability that a given point on $\eta_\Lambda$ is the starting point of such a loop is equal to zero.  

For each boundary point $y$, we can also define the path $\eta_y'$ from $x$ to $y$ by moving along the $\bSLE_\kappa(\rho)$ from $x$ to $y$ and traversing similarly the loops of $\Gamma'$ that it encounters.  This path can be constructed by excising certain intervals of time from $\eta'$. 

Deriving the following statement will be one of our two main goals for the rest of the paper: 

\begin{theorem}
\label{thm:duality1}
For each $\kappa \in (2,4)$ and $\beta \in [-1,1]$ there exists an admissible $\rho := \rho ( \beta, \kappa') \in [-2, \kappa -4]$  so that if one then defines
\begin{equation}
\label{eqn:rho_p_equations}
\rho_R' = -\frac{\kappa'}{4}(\rho+2) \quad\text{and}\quad \rho_L' = \frac{\kappa'}{2} - 4 - \rho_R', 	
\end{equation}
the following properties hold: 
The collection of loops $\Lambda$ and $\Gamma'$ in this construction are almost surely locally finite, so that the curve $\eta'$ is indeed continuous, and for any boundary point $y$,  the law of the path $\eta_y'$ is a full $\bSLE_{\kappa'}^\beta$ process as defined at the end of Section~\ref{subsec:side-swapping}. 

In particular, this shows that a full $\bSLE_{\kappa'}^\beta$ is almost surely generated by a continuous curve, and that its trunk 
is a $\bSLE_\kappa (\rho)$ process for this $\rho := \rho (\beta, \kappa')$.
\end{theorem}

We stress already that in the present paper we will not derive the general explicit formula for this function $\rho (\beta, \kappa')$, but that this will be one of the main results of our subsequent paper \cite {msw2016betarho}, see the discussion in Section~\ref{SSdiscussion}. In some special cases though, it is however possible to work out already the value of $\rho$:

\begin {itemize}
 \item 
We can note that by symmetry, $$\rho (\beta, \kappa') + \rho ( -\beta, \kappa') = \kappa - 6.$$
In particular, when $\beta = 0$, one necessarily has 
$\rho_L' = \rho_R' = (\kappa'/4) - 2 $ and
\begin{equation}
\label{eqn:rho_symmetric}	
 \rho ( 0, \kappa') = \frac{\kappa-6}{2}.
\end{equation}
In other words:
\begin {corollary}
 The trunk of a $\bSLE_{\kappa'}^0$ is a $\bSLE_{\kappa} ( (\kappa-6) / 2)$. 
\end {corollary}
\item
When $\beta=1$, the full $\bSLE_{\kappa'}^\beta$ process from $x$ to $y$ is nothing else than the process that traces counterclockwise one after the other the $\CLE_{\kappa'}$ loops that touch the boundary of the domain between $x$ and $y$ in their order of appearance on the clockwise arc from $x$ to $y$. So in this special $\beta=1$ case, this theorem is rather trivial: the first BCLE just consists of the loop that traces along the boundary loop clockwise, and inside this loop, one just samples the boundary-touching loops of a $\CLE_{\kappa'}$ with a counterclockwise orientation. The similar symmetric case holds for $\beta=-1$. So,  one 
has 
\[  \rho (1, \kappa') = -2 \quad\hbox{and}\quad \rho (-1, \kappa') = \kappa-4. \]
\item  In fact, our proof will show that for all 
$\kappa' \in (4, 8)$, the mapping $\beta \mapsto \rho(\beta, \kappa')$  is an {\em decreasing} bijection from $[-1,1]$ onto $[-2, \kappa-4]$.
\end {itemize}

\begin{figure}[ht!]
\begin{center}
\includegraphics[width=4in]{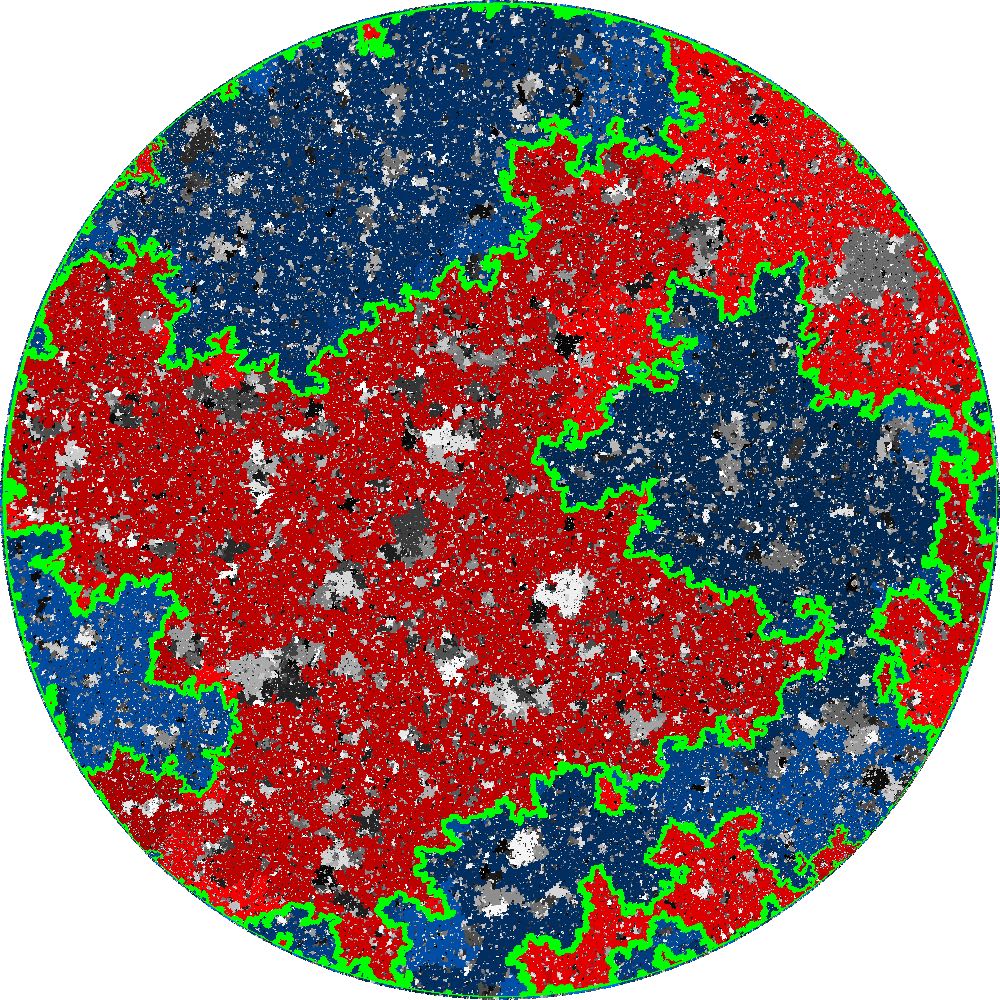}
\end{center}
\caption{\label{fig:perc_sim_later} Continuation of Figure~\ref{fig:cle_6_perc_sim}.  Shown are the red/blue clusters which touch the green interfaces.  In the context of Theorem~\ref{thm:duality1}, the green interfaces correspond to the $\BCLE_\kappa(\rho)$ and the red (resp.\ blue) clusters correspond to the $\BCLE_{\kappa'}(\rho_L')$ (resp.\ $\BCLE_{\kappa'}(\rho_R')$).} 
\end{figure}

\begin{figure}[ht!]
\begin{center}
\includegraphics[width=3.2in]{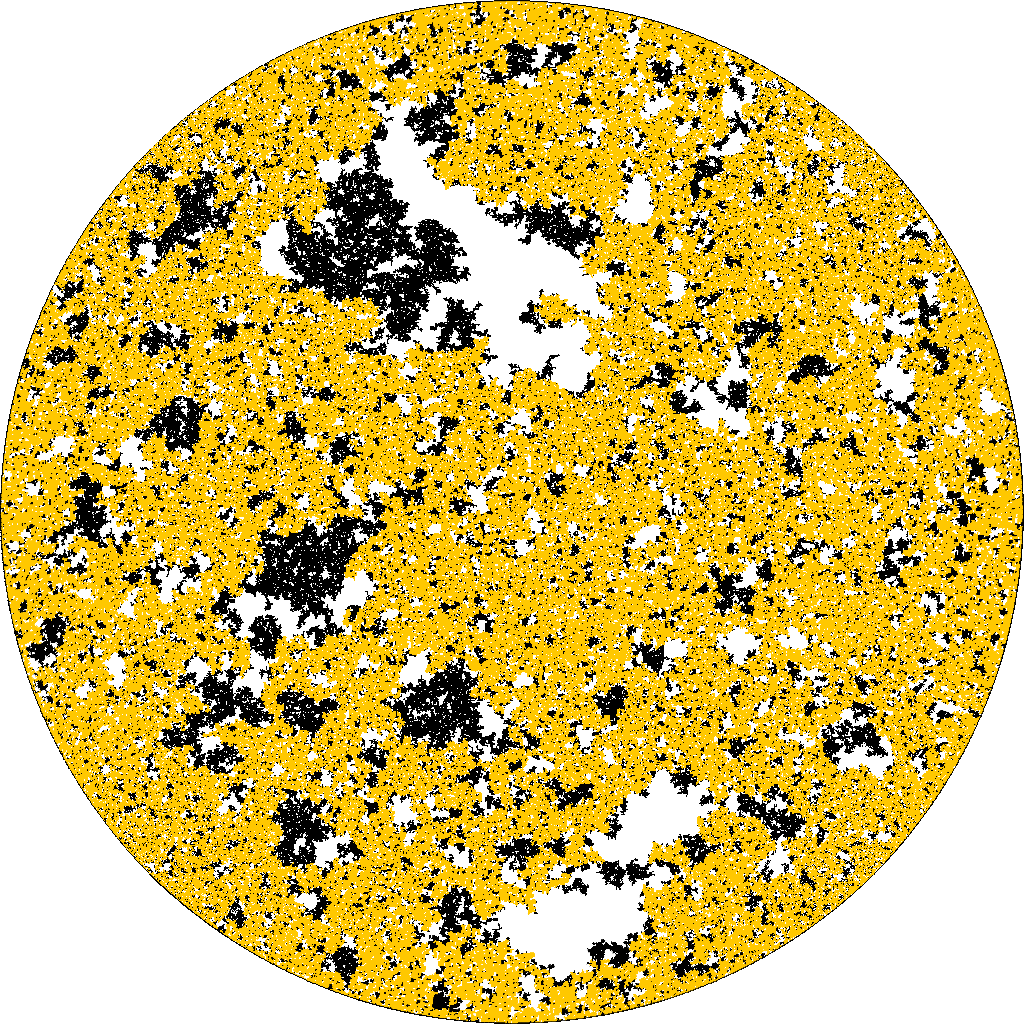}\hspace{0.01\textwidth}\includegraphics[width=3.2in]{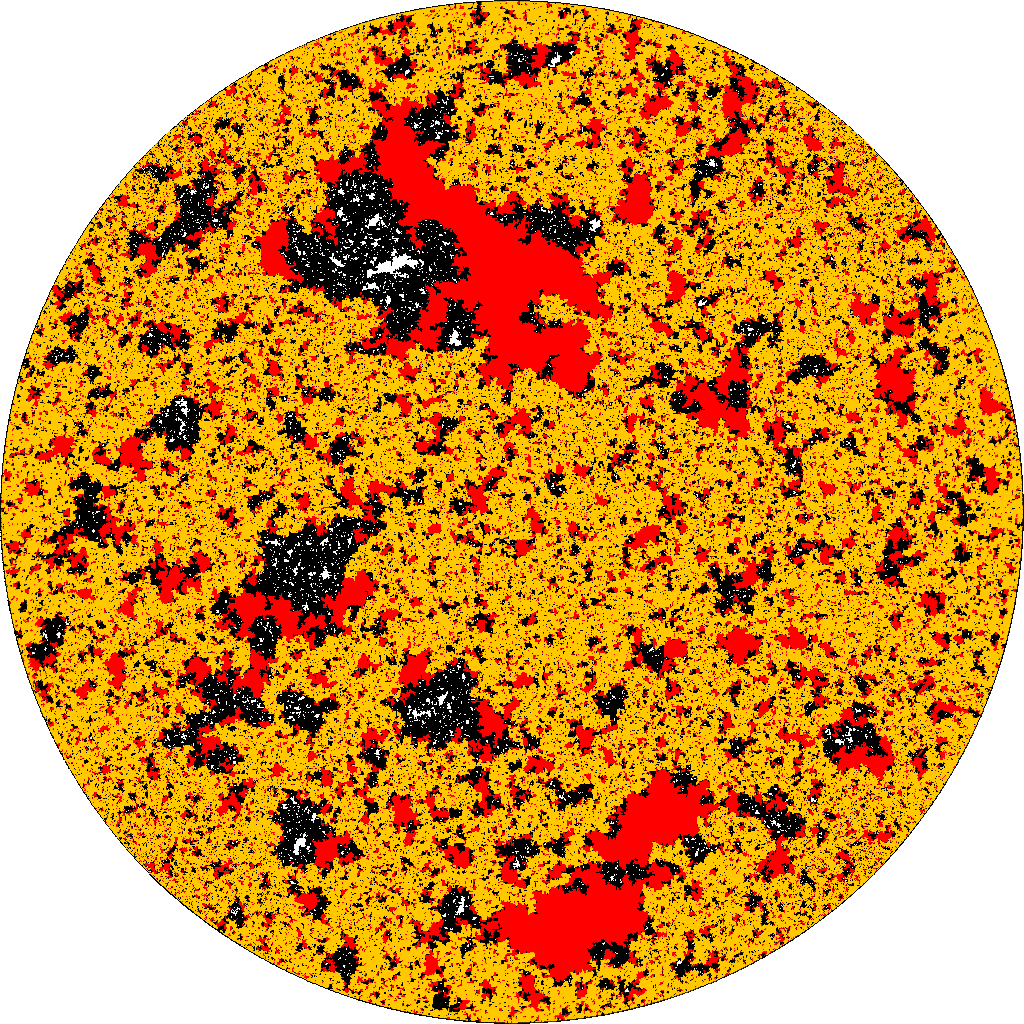}
\end{center}
\vspace{-0.01\textheight}
\caption{\label{fig:ising_sim} $\BCLE_3$ and $\BCLE_{16/3}$ nesting: {\bf Left:} In yellow, the critical FK$_{q=2}$ clusters in an Ising carpet with all $+$ boundary conditions that touch the boundary.  The remainder of the Ising carpet is shown in black.  {\bf Right:} The Ising carpet holes that touch yellow are colored red, together forming a nested $\BCLE_3$.}
\end{figure}

\medbreak

The second main goal will be to derive the corresponding statement when one interchanges~$\kappa$ and~$\kappa'$  (see Figure~\ref{fig:ising_sim}). We now consider $\kappa' \in (4,6)$ (mind that here, the values $\kappa' \in [6,8)$ are excluded), we let $\Lambda'$ be  a $\cwBCLE_{\kappa'} (\rho')$ for an admissible value $\rho'$, and we then iterate by defining an independent $\cwBCLE_\kappa(\rho_L)$ (resp.\ $\ccwBCLE_\kappa(\rho_R)$) in each  false (resp.\ true) loop traced by $\Lambda'$.  We then fix a boundary point $x$ and then define the two loops (from $x$ to $x$) $\eta_{\Lambda'}$ and $\eta$ using the same procedure as before. 

For any other boundary point $y \not= x$, we can also consider the paths $\eta_{\Lambda'}$ and $\eta$ parameterized by their half-plane capacity seen from $y$.  We note that the former is then a $\bSLE_{\kappa'} (\rho')$ 
from $x$ to $y$ and call the latter $\eta_y$.

\begin{theorem}
\label{thm:duality2}
For each $\kappa' \in (4,6)$ and $\beta \in [-1,1]$ there exists   $\rho':= \rho'( \beta, \kappa)$ in the  range $[\kappa'-6, 0]$ so that if one defines 
\begin{equation}
\label{eqn:rho_equations}
	\rho_R = -\frac{\kappa}{4}(\rho'+2) \quad\text{and}\quad \rho_L =  \frac {\kappa}2 - 4 - \rho_R 
\end {equation}
and constructs $\eta$ as just indicated then the following holds: $\eta$ is almost surely a continuous path and the law of the path $\eta_y$ is a $\bSLE_\kappa^\beta$ process. 
In particular, we get that a $\bSLE_\kappa^\beta$ process is almost surely generated by a continuous curve and 
that its trunk is a $\bSLE_{\kappa'} (\rho')$ for this value $\rho' = \rho' ( \beta, \kappa)$.
\end{theorem}

 For a discussion of the general formula giving $\rho'$ as a function of $\beta$ and $\kappa$, see again Section~\ref{SSdiscussion}. 

The same symmetry argument as for Theorem~\ref{thm:duality1} shows that 
$$ \rho' (\beta, \kappa) + \rho' (-\beta, \kappa) =  {\kappa' -6 }. $$ 
 In particular: 
\begin {corollary}
 The trunk of a $\bSLE_{\kappa}^0$ is a $\bSLE_{\kappa'} ( (\kappa'-6) / 2)$. 
\end {corollary}

The special cases where $\beta = 1$ and $\beta=-1$ are more interesting than in the previous $\kappa'$-loops on $\kappa$-trunk case. 
Indeed, the loops traced by this $\SLE_\kappa^1(\rho)$ process will still be all on the right-hand side of the trunk, but this time the trunk is non-trivial.

More precisely, the fact that loops are all on the right-hand side of the trunk implies that the $\cwBCLE_{\kappa}(\rho_L)$ has in fact no true loops. 
In other words,
the $\bSLE_{\kappa} (\rho_L)$ just goes along the boundary counterclockwise, which means that $\rho_L = \kappa-4$  
 and therefore $ \rho_R =-\kappa /2$, and:
\begin {corollary} 
\label {co:76}
 The trunk of the totally asymmetric $\bSLE_\kappa^{1}$  is a $\bSLE_{\kappa'}(0)$.
\end {corollary}
As we will discuss below, this is not surprising in view of the discrete Edwards-Sokal couplings.  A similar statement holds for the case $\beta=-1$: The trunk 
of the totally asymmetric $\bSLE_\kappa^{-1}$ is a $\bSLE_{\kappa'} (\kappa' -6)$. 

More generally, the constraints on the ranges of $\rho_R$ and $\rho_L$ do show why in fact, in Theorem~\ref{thm:duality2}, $\rho'(\beta, \kappa)$ will take its value only in $[\kappa'-6, 0]$ and does not cover the entire 
range of admissible values $[ \kappa' /2 - 4, \kappa'/2 -2]$ where $\BCLE_{\kappa'}(\rho')$ processes exist.  
In fact, our proof will show that the map $\beta \mapsto \rho' (\beta, \kappa)$ is a {\em increasing} bijection from $[-1,1]$ onto $[\kappa'-6, 0]$

The reader might be a little bit puzzled by the fact that the trunk of $\bSLE_\kappa^{1}$ tends to somehow be more to the right than 
the trunk of $\bSLE_\kappa^{-1}$  for $\kappa \in (8/3, 4)$, whereas clearly, for $\kappa' \in (4,8)$, 
the definition shows that the larger $\beta$ is, the more to the left the trunk of $\bSLE_{\kappa'}^{\beta}$ is. 
More generally, it may seem counterintuitive that $\beta \mapsto \rho' (\beta, \kappa)$ turns out to be increasing when $\kappa \in (8/3, 4)$ while $\beta \mapsto \rho (\beta, \kappa')$ from 
Theorem~\ref{thm:duality1} must clearly be decreasing when $\kappa' \in (4,8)$.
Here, one should bear in mind that the principal value correction in the definition of the driving function for bSLE$_\kappa^\beta$ when $\kappa <4$ creates a compensation 
to the $\beta$-dependent ``push'' that was due to the status of the loops, as opposed to the case $\kappa' >4$ where no such compensation is present.

\subsection{Consequences for $\CLE_\kappa$: Edwards-Sokal couplings in the continuum}

We now go through some consequences of 
 Theorems~\ref{thm:duality1} and~\ref{thm:duality2} for conformal loop ensembles themselves. 
 First, we note that they exactly
 provide the missing items that turn the conditional statements of Propositions~\ref{prop:CME_characterization} and~\ref{prop:unique_interface} into plain statements
 and provide a description of the \hyperref[def:cpi]{CPIs} in $\CLE_\kappa^\beta$ carpets for $\kappa \in (8/3, 4)$ as $\bSLE_{\kappa'} (\rho')$ processes,  and of 
 colored $\CLE_{\kappa'}^\beta$ interfaces as $\bSLE_{\kappa} (\rho)$ processes. For the record and future reference, let us state this  formally: 

\begin{theorem}
\label {thm:CPI}
\begin{enumerate}[(i)]
 \item 
 For each $\kappa' \in (4,8)$ and each $\beta \in (-1, 1)$, the colored $\CLE_{\kappa'}^\beta$ percolation interface is the trunk of the full $\bSLE_{\kappa'}^\beta$ that can be used to construct the $\CLE_{\kappa'}^\beta$. It is a continuous curve and its marginal distribution is that of a $\bSLE_\kappa ( \rho)$ for $\rho = \rho (\beta, \kappa')$. 
\item
For any $\kappa \in (8/3, 4)$ and $\beta \in [-1, 1]$, there exists a unique (in distribution) \hyperref[def:cpi]{CPI} in $\CLE_\kappa^\beta$. The joint distribution of this \hyperref[def:cpi]{CPI} with the $\CLE_\kappa^\beta$ loops that it intersects is that of the trunk and the loops drawn by a $\bSLE_\kappa^\beta$. The marginal (i.e.\  ``annealed'' in the terminology used for random walks in random environments) law of the \hyperref[def:cpi]{CPI} is that of a $\bSLE_{\kappa'} (\rho')$ for $\rho' = \rho' (\beta, \kappa)$.
\end{enumerate}
\end{theorem}

Let us now explain how Theorem~\ref{thm:duality2} makes it possible to construct an entire $\CLE_\kappa$ for $\kappa \in (8/3, 4)$
using an iteration of boundary conformal loop ensembles. 

Let us first consider the case where $\beta=1$. 
Suppose that $\kappa \in (8/3, 4)$, and start with the setup of Theorem~\ref{thm:duality2} for $\beta=1$ (which is that of Corollary~\ref{co:76}):
Sample first a $\cwBCLE_{\kappa'} (0)$ (recall that this corresponds to {the} boundary-touching loops of a $\CLE_{\kappa'}$). Then, inside all the true loops of this 
$\cwBCLE_{\kappa'} (0)$, sample independent 
$\ccwBCLE_\kappa ( - \kappa /2)$ processes. 
Theorem~\ref{thm:duality2} states that the obtained picture can be viewed as a $\bSLE_\kappa^1$ branching tree starting from one boundary point and targeting 
all other boundary points. This means that the collection of all true $\SLE_\kappa$ loops that have been traced can be viewed as being part of the same $\CLE_\kappa$, and that in order 
to find the missing $\CLE_\kappa$ loops, one would have to continue exploring via the branching tree in the remaining unexplored regions. 
Here, we note that there are two types of unexplored regions where one has not launched the tree yet: Those that correspond to 
false loops of the $\cwBCLE_{\kappa'} (0)$, and those that correspond to false loops of one of the $\ccwBCLE_\kappa (-\kappa/2)$. However, in both cases, the conditional joint law of the missing $\CLE_\kappa$ loops in those regions is just given by independent $\CLE_\kappa$ collections in each of these regions (this is just due to the branching $\bSLE_\kappa$ construction of $\CLE_\kappa$). 
In particular, one can just iterate the same construction inside each of these regions: Sample a $\cwBCLE_{\kappa'} (0)$ and then an independent 
$\ccwBCLE_\kappa (- \kappa /2)$ inside each of its true loops and so on.  

For each given point $z$ in $D$, the number of such iteration steps required to find the $\CLE_\kappa$ loop that surrounds it is almost surely finite (and follows a geometric random variable because 
the probability of success at each iteration step is independent of $z$ by conformal invariance).
{It follows that:} 

\begin {theorem}
\label {thm:ES}
This iterative $\cwBCLE_{\kappa'} (0)$ / $\ccwBCLE_\kappa (-\kappa/2)$ procedure constructs exactly an entire $\CLE_\kappa$. 
\end {theorem}

We can furthermore note that in this construction, 
the drawn $\cwBCLE_{\kappa'} (0)$ loops correspond to  ``critical percolation clusters'' drawn by a CPI in the $\CLE_\kappa^1$ carpet, as in (ii) of Theorem~\ref{thm:CPI}. 

Let us make the following comment before discussing the generalization to other values of $\beta$: 
As we have already mentioned, prior to the present paper, the only existence proof of $\CLE_\kappa$ for $\kappa \in (8/3, 4)$  was based on the loop-soup construction in \cite{sw2012cle}. More precisely, this was the only existing proof of the fact that $\bSLE_\kappa^\beta$ exploration trees starting from different points would all construct the same law of locally finite collections of simple loops -- one can for instance keep in mind that the coupling of the $\CLE_\kappa$ with the GFF is more complicated for $\kappa\not=4$ than in the case $\kappa=4$ and depends for instance on the choice of the starting point of the chosen tree.  
However, this new construction of $\CLE_\kappa$ provides an alternative derivation of the fact that this collection of loops is defined in a root-invariant way. Indeed, we use here only the fact that the BCLEs are root-invariant (as one constructs the CLEs via iteration of two BCLEs) and this fact follows from the reversibility of the $\SLE_\kappa (\rho)$ processes for $\rho > -2$ derived via the GFF couplings in \cite{ms2012ig2, ms2012ig3}.

We now turn to the generalization of Theorem~\ref{thm:ES} to other values of $\beta$, which will also follow from Theorem~\ref{thm:duality2}. 
It will provide for each value of $\beta \in [-1,1]$ a 
similar yet different BCLE-based construction of $\CLE_\kappa$, or more precisely, a joint construction of $\CLE_\kappa^\beta$ and of ``critical percolation 
clusters'' in the $\CLE_\kappa^\beta$ (where the $\CLE_\kappa^\beta$-loops are considered to be open or closed depending on their label). 
Let us describe this first in the symmetric case where $\beta=0$. 
Suppose that $\kappa \in (8/3, 4)$, and start again with the setup of Theorem~\ref{thm:duality2}, but  for $\beta=0$: Sample first a $\cwBCLE_{\kappa'} ((\kappa' - 6) / 2)$ process $\Lambda'$. 
Then, inside all the false loops of $\Lambda'$ {\em and} all the true loops of $\Lambda'$, sample independent 
$\BCLE_\kappa ( (\kappa / 4) - 2 )$ processes. However, depending on whether one is in a true or false loop of $\Lambda'$, one will sample a $\ccwBCLE_\kappa ( \kappa / 4 - 2 )$ or a $\cwBCLE_\kappa ( \kappa / 4 - 2 )$.

This time, Theorem~\ref{thm:duality2} states that the obtained picture can be viewed as a $\bSLE_\kappa^0$ branching tree starting from one boundary point and targeting 
all other boundary points. Exactly the same procedure as before then shows that one can just iterate the same construction inside each of the false loops of 
each of these $\BCLE_\kappa ((\kappa / 4) - 2)$'s, and eventually, one constructs the whole symmetric $\bSLE_\kappa^0$ branching tree.  Hence:

\begin {theorem}
\label {thm:ES2}
This iterative $\cwBCLE_{\kappa'} ((\kappa' -6)/ 2)$ / ( $\cwBCLE_\kappa (\kappa /4 -2 )$ or $\ccwBCLE_\kappa ( \kappa /4 - 2)$)  procedure constructs exactly an entire $\CLE_\kappa^0$. 
\end {theorem}

For the case of general $\beta \in (-1,1)$, the similar procedure works and constructs a $\CLE_\kappa^\beta$, 
except that one has to sample at each step a $\cwBCLE_\kappa (\rho_L)$ or a $\ccwBCLE_\kappa (\rho_R)$ depending on 
whether one is a true loop or a false loop of the $\cwBCLE_{\kappa'} (\rho')$.

\medbreak

Finally, let us properly state and prove the actual Edwards-Sokal coupling based on coloring the {\em clusters}
of a nested $\CLE_{\kappa'}$. Here, we consider a {\em nested} $\CLE_{\kappa'}$ $\Lambda'$ for $\kappa' \in (4,6)$. 
  As explained in Section~\ref{sec:intro}, in view of the FK framework, in the setting of wired boundary conditions it is natural to use $\Lambda'$ to construct clusters as follows.  Each loop of $\Lambda'$ is assigned an even or odd parity depending on its level of nesting in $\Lambda'$, because the loops correspond in an alternate fashion to outside and inside boundaries 
  of ``clusters''. We take the boundary of the domain to have even parity (and the outermost loops of $\Lambda'$ have odd parity).  If $\CL$ is an even parity loop (or the boundary of the domain), then we associate with it a cluster $C$ by taking $C$ to be the set of points surrounded by $\CL$ minus all of the points surrounded by a loop of odd parity surrounded by $\CL$.  We set the color of the outermost cluster to be white and for a given 
  $p \in (0,1)$, we assign the colors 
  to the other clusters independently, with probability $p$ to be white and with probability $1-p$ to be black.  We then consider clusters of black clusters. 
\begin{theorem}
\label{thm:cle_k_from_kp}
For each value of $\kappa' \in (4,6)$ there exists
$p(\kappa') \in (0,1)$ such that in this construction, the collection of all outermost outer boundaries of clusters of black $\CLE_{\kappa'}$ clusters has the law of a $\CLE_\kappa$.  

This probability $p(\kappa)$ is expressed in terms of the function $\rho(\beta, \kappa')$ of Theorem~\ref{thm:duality2} via  
$$  \rho (1 - 2p , \kappa' ) = - \kappa /2 .$$ 

In the special case where $\kappa' = 16/3$, we have $p(\kappa')=1/2$.
\end{theorem}

Let us mention already (see the discussion in Section~\ref{SSdiscussion}) that it will be a consequence of \cite{msw2016betarho} that for all $\kappa \in (8/3,4)$, 
 $p(\kappa)= 1/( 4 \cos^2 ( \pi  \kappa / 4 ))$, which will solve \cite[Problem~8.10]{she2009cle}.

As we will now see, the proof will use both Theorem~\ref{thm:duality2} (via Theorem~\ref{thm:ES}) and Theorem~\ref{thm:duality1}, and it will make use of the fact that in the latter, when $\rho' = \kappa' -6$ then $\rho_L = -\kappa/2$ while conversely, in the former, when $\rho = -\kappa /2 $, then $\rho_L' = \kappa' -6$. This relation between these coefficients can somehow be interpreted as the continuous counterpart of the fact that Potts and FK models with free boundary conditions correspond to each other, that Potts and FK models with uniform color/wired boundary conditions correspond to each other too, and that the dual of the critical FK model with wired boundary conditions is exactly the critical  FK model
with free boundary conditions.
 
\begin{proof}[Proof of Theorem~\ref{thm:cle_k_from_kp}]
We are going to construct the nested $\CLE_{\kappa'}$ iteratively, and we will call it~$\Lambda'$.
We start with the same setup as in Theorem~\ref{thm:ES}: 

Suppose that we first sample a $\cwBCLE_{\kappa'}(0)$, which will be the boundary-touching loops of our nested $\CLE_{\kappa'}$ $\Lambda'$.
Recall that this can be interpreted as the continuous analog of the boundary-touching loops of 
an FK model with free boundary conditions and that the corresponding false loops would correspond to the boundary-touching loops of the model with wired boundary conditions.

In order to construct the rest of the nested $\CLE_{\kappa'}$ $\Lambda'$, one then has to iterate this procedure in each of the remaining connected components,
but the $\Lambda'$-parity of the drawn loops at the next iterative step will depend on whether one is inside a true or a false loop of this $\cwBCLE_{\kappa'}(0)$: 
In the false loops, the situation is as at the beginning, and the first loops that one will draw will be odd loops of $\Lambda'$. We leave the exploration 
of what happens in these false loops aside for the moment. 
In the true loops however, the first loops that one will draw will be even loops of $\Lambda'$, and we are going to continue the exploration. 

Now, in each of these true loops of the $\cwBCLE_{\kappa'}(0)$,  instead of simply
only drawing another $\cwBCLE_{\kappa'}(0)$ whose loops correspond to even loops of the nested $\CLE_{\kappa'}$,  we want in fact to 
explore a $\CLE_{\kappa'}^\beta$ in order to also be able to discover the colors of the corresponding clusters.
For this, it will be handy to 
use Theorem~\ref{thm:duality1} in order to draw also the BCLE that corresponds to the interfaces between the colored $\CLE_{\kappa'}^\beta$ clusters.

We start therefore to sample in each of the connected components inside the true loops of the $\cwBCLE_{\kappa'}(0)$, an independent $\cwBCLE_\kappa(-\kappa/2)$.
Note that from Theorem~\ref{thm:ES}, these $\cwBCLE_\kappa(-\kappa/2)$ loops form part of a $\CLE_\kappa$ in the original domain (this is due to the fact that $\rho_R= - \kappa/2$ when $\rho'=0$ in Theorem~\ref{thm:duality2}); recall also from Proposition~\ref{prop::bcle_prop} that, modulo orientation, the loops of a $\cwBCLE_\kappa(-\kappa/2)$ are equal in distribution to to those of a $\ccwBCLE_\kappa(-\kappa/2)$).

But by Theorem~\ref{thm:duality1}, if we then sample independently a $\cwBCLE_{\kappa'}(\kappa'-6)$ (resp.\ a $\ccwBCLE_{\kappa'}(2-\kappa'/2)$)  inside each of the 
false (resp.\ true) loops of these $\cwBCLE_\kappa(-\kappa/2)$ loops, then we have drawn (part of) 
the loops in a nested $\CLE_{\kappa'}$ -- we will consider them as being part of $\Lambda'$. Furthermore, if one assigns a color to
these loops according to their orientation, i.e., to whether they lie in 
false or true loops of these $\cwBCLE_\kappa(-\kappa/2)$, then these colored $\CLE_{\kappa'}$ loops correspond to a randomly colored CLE$_{\kappa'}$ where  
loops are independently  chosen to be black or white according to a $(1+\beta)/2$ versus $(1-\beta)/2$ coin, where $\beta$ is chosen so that 
\[ \rho ( \beta, \kappa') = - \kappa / 2. \] The $\cwBCLE_\kappa(-\kappa/2)$ loops are then exactly interfaces between clusters of black loops and clusters of white loops. 
We note that these colored CLE$_{\kappa'}$ loops are all even loops, i.e., outer loops of a cluster of $\Lambda'$, 
so that the color of that cluster can be simply defined to be  the color of that loop if $p=(1- \beta)/2$. Then, we get readily that the $\cwBCLE_\kappa(-\kappa/2)$ loops correspond exactly to the outer boundary of a cluster of black clusters in $\Lambda'$, and that it also corresponds to an inner boundary of the cluster of white clusters in $\Lambda'$ 
that touches the boundary of the domain. 

To the inside of the $\cwBCLE_\kappa(-\kappa/2)$'s, we do not need to explore anymore, since we are inside the outer boundary of a cluster of black clusters.
However, we need to continue to explore in both the components that are surrounded by false loops of the $\cwBCLE_{\kappa'}(\kappa'-6)$'s and in the components 
that are surrounded by the true loops of the $\cwBCLE_{\kappa'}(\kappa'-6)$. In both cases, the conditional distribution of the loops of $\Lambda'$ will be simply that of a nested CLE$_{\kappa'}$ in these components, but the parity rule for the loops are different. Inside the true loops, the parity rule is like at the beginning (the first encountered loops 
will be odd loops for $\Lambda'$), and we leave the exploration of the inside of these true loops aside for the moment.

In the components which are surrounded by false loops of the $\cwBCLE_{\kappa'}(\kappa'-6)$, 
we are in the situation where the first encountered loops will be even loops for $\Lambda'$. So we 
start the same procedure again by launching first a $\cwBCLE_\kappa(-\kappa/2)$ and then a $\cwBCLE_{\kappa'}(\kappa'-6)$ in its false loops, and so on.

Summing up, we see that in fact we have a mechanism where (using the fact that true loops for a BCLE$_{\kappa'}(\kappa'-6)$ are the false loops for a BCLE$_{\kappa'}(0)$ 
and vice-versa), we iteratively perform the following steps (here we simply omit the orientation of the BCLEs as they play no role in these statements): 
\begin{enumerate}
 \item[Step 1:] Launch a BCLE$_{\kappa'} (0)$, and leave the false loops aside. 
\item[Step 2:] In the true loops defined by Step 1, we launch a BCLE$_{\kappa} ( - \kappa/2)$. These true loops are outermost outer boundaries of clusters of black clusters of $\Lambda'$. 
\item[Step 3:] Inside the false loops defined at Step 2, we go back to Step 1. 
\end{enumerate} 
We see that this is exactly the procedure described in Theorem~\ref{thm:ES} that constructs a CLE$_\kappa$, except that here we stopped exploring in the false loops of the BCLE$_{\kappa'} (0)$. However, as we have noted, inside of all these loops, the conditional law of $\Lambda'$ and of the parity of the loops is (modulo conformal invariance) exactly the same as the one we started with, so we can launch the same procedure at Step 1 inside each of them and iterate. 

If we iterate this indefinitely, then we will on the one hand discover the entire family of 
outer boundaries of outermost black clusters of $\Lambda'$, and on the other hand, by Theorem~\ref{thm:ES}, we know that they form exactly a $\CLE_\kappa$. 
This proves the main statement in the theorem.

To conclude, we note that $p$ is equal to $1/2$ when $\beta=0$, which happens for $\rho= (\kappa-6) /2$ by~\eqref{eqn:rho_symmetric}. 
This quantity is equal to $-\kappa /2$ only for $\kappa =3$. 
\end{proof}

Let us stress that the fact that in Theorem~\ref{thm:cle_k_from_kp} {one has that} $p (\kappa') = 1/2$ for $\kappa' =16/3$, or the fact that in Theorem~\ref{thm:duality1} one has $\rho_R' = \rho_L' = \kappa' -6$ only when $\kappa' =16/3$ provide evidence, based on these CLE considerations only, that the only possible conformally invariant scaling limit of the critical FK$_{q=2}$ model and of the Ising model have respectively to be $\CLE_{16/3}$ and $\CLE_3$.

\medbreak

Finally, let us remark that Theorem~\ref{thm:duality1} shows that for any $\beta$, if one starts with a $\CLE_{\kappa'}^\beta$, then the boundary-touching interfaces between the clusters of black clusters and the clusters of white clusters, will have the law of a $\BCLE_\kappa(\rho)$ for a corresponding value of $\rho$ (note that this is just a feature about non-nested $\CLE_{\kappa'}$). 
One can then couple together these laws when letting $\beta$ varying in $[-1,1]$ for each given $\CLE_{\kappa'}$ by associating to each cluster a uniform random variable in $[0,1]$. 
We will show in \cite{msw2016fan} that if one fixes two boundary points and considers the collection of such interfaces between two fixed boundary points then one gets a family of paths which have the same law as the {\em fan of GFF flow lines} considered in \cite{ms2012ig1} with certain GFF boundary conditions.

\subsection{On the relationship between $\beta$ and $\rho$}
\label{SSdiscussion}

As we have already mentioned several times, this paper does not present a derivation of a general formula relating $\beta$ and $\rho$ with $\kappa'$ or $\kappa$ in Theorems~\ref{thm:duality1} and~\ref{thm:duality2}. 
Similarly, except in the case that $\kappa=3$ and $\kappa'=16/3$, we have not identified in this paper the value of $p (\kappa)$ in Theorem~\ref{thm:cle_k_from_kp}.

However, in our upcoming  \cite{msw2016betarho}, we plan to show,  building
on the results and ideas of present paper and combining them with quantum gravity ideas and techniques from  \cite{dms2014mating}, that:
\begin {itemize}
 \item 
When $\kappa' \in (4,8)$, the relation between $\beta$, $\kappa' = 16/ \kappa$ and $\rho$ in Theorem~\ref{thm:duality1} (so that the trunk of a $\bSLE_{\kappa'}^\beta$ is a $\bSLE_{\kappa} (\rho)$)
is given by
\begin{equation}
\label{eqn:beta_rho_k_formula}
\frac {1 - \beta}{2} =  \frac{ \sin ( \pi \rho/2 ) }{ \sin (\pi \rho/2 ) - \sin ( \pi (\kappa-\rho) / 2 )}  .
\end {equation} 
\item 
When $\kappa  \in (8/3,4)$, the relation between $\beta$, $\kappa' = 16/ \kappa$ and $\rho'$ in Theorem~\ref{thm:duality2} (so that 
 $\bSLE_{\kappa'} (\rho')$ is the trunk of a $\bSLE_\kappa^\beta$) is given by
\begin {equation}
 \label{eqn:beta_rho_k_formula2}
\frac {1 - \beta}{2} = \frac{ \sin ( \pi \rho'  /2 ) }{ \sin (\pi \rho'/2 ) - \sin ( \pi (\kappa'-\rho') / 2 )} .
\end{equation}
\end {itemize}
 Note that the formulas look identical but the range of admissible values for which the formulas apply are different.
 In particular, 
 for $\kappa'$ fixed, when $\beta$ varies from $-1$ to $1$ then $\rho$ spans indeed all the admissible interval $[-2, \kappa -4]$ in the first formula, but in the second formula, when 
 $\kappa \in (8/3, 4)$ is fixed and $\beta$ varies from $-1$ to $1$, then the range of obtained values is only $[\kappa'-6, 0]$ (which is not surprising in view of Theorem~\ref{thm:duality2}). 

Furthermore, this formula for $\rho = - \kappa /2$ will show 
that $p(\kappa')$ in Theorem~\ref{thm:cle_k_from_kp} is equal to 
\begin{equation}
\label{eqn:q_k}	
p(\kappa') =  \frac {1}{4 - 4 \sin^2 (\pi \kappa/4)} 
= \frac 1 {4 \cos^2 ( \pi  \kappa / 4 )} =  \frac 1{ 4 \cos^2 ( 4 \pi / \kappa' )}.
\end{equation}
This is of course of interest because in the discrete Edwards-Sokal coupling for $q$-Potts models and FK$_q$-models, the corresponding probability is equal to $1/q$.
It for instance explains why $\kappa'=4$ and $\kappa'=24/5$ would respectively correspond to the $4$ and $3$-state FK-Potts-models respectively, because they give rise to the values $p = 1/4$ and $p =1/3$. 
It gives more generally a justification for the relation 
\begin {equation}
\label {eqn:qkappa}
q =  4 \cos^2 ( 4 \pi / \kappa' )
\end{equation}
between the value of  $q$ for a critical FK$_q$-model and its conjectured $\CLE_{\kappa'}$ scaling limit. See also \cite{mw2017connection} for another approach to this formula via crossing probabilities.
Note that the existence of the scaling limit remains conjectural for most FK models on lattices, but
that it is established (with the identification of the limit where the same relation (\ref{eqn:qkappa}) appears) for FK models on random planar maps with respect to the so-called peanosphere topology in \cite{she2011qginv,dms2014mating}.  See also \cite{gms2015cone_times,gs2015finite_volume,gs2015finite,quantum_spheres,gm2016topology}.

\subsection{Revisiting $\CLE_4^0$ percolation}

Before turning our attention to the proof of these results, let us briefly revisit our study of $\CLE_4^0$ percolation in this BCLE framework. What follows are a few comments on how to reformulate the ideas of the $\CLE_4^0$ percolation proof that we presented in terms of BCLE, that will serve as a warm-up for the proofs of the two theorems that we have just stated.

Let us first give a BCLE on BCLE version (Figure~\ref{fig:bcle_cle_4_perc}), i.e., a loop version, of our construction of one $\SLE_4 (\rho; -2 - \rho)$  path and of the $\CLE_4^0$ loops that it intersects: Let $h$ be a GFF on $D$ with zero boundary conditions and fix $\rho \in (-2,0)$ (recall that $(-2,0)$ is the range of $\rho$ values such that a classical $\SLE_4(\rho)$ bounces off the boundary -- more precisely of one side of the boundary between its starting and end-points). 
We can naturally define deterministically a $\cwBCLE_4(\rho)$ (that we will denote by $\Lambda$) from $h$ as  the collection of loops which are formed by the boundary branching $\bSLE_4 (\rho)$ tree of level lines with boundary conditions given by $\lambda \rho$ on their left and by $\lambda(2+ \rho)$ on their right. In other words, when $\CL$ is a true loop (resp.\ a false loop) of $\Lambda$ then the boundary data for the conditional law of $h$ inside of $\CL$ given the branching tree is equal to $\lambda(2+\rho)$ (resp.\ $\lambda \rho$).

We now define the nested $\cwBCLE_4(\rho_L)$'s and $\ccwBCLE_4(\rho_R)$'s inside the false and true loops of $\Lambda$ respectively, and for reasons that will become immediately clear, 
we choose to define them as level lines of $-h$ rather than $h$  (this type of feature is reminiscent and closely related to the fact that when one couples an $\SLE_\kappa$ from $a$ to $b$ with an $\SLE_{\kappa'}$-type process from $b$ to $a$ via the GFF the former is coupled with $h$ and the latter is coupled with $-h$).  With this coupling, we get that the boundary data of $h$ in the true loops of the different $\ccwBCLE_4(\rho_R)$'s is $\lambda(4 + \rho+\rho_R)$, while it is $\lambda(\rho- \rho_L-2)$ inside the true loops of the different $\cwBCLE_4 (\rho_L)$'s.  If we choose $\rho_L = \rho =  -2-\rho_R \in (-2,0)$, then the heights inside of the true loops of the $\ccwBCLE_4(\rho_R)$'s and the true loops of the $\cwBCLE_4(\rho_L)$'s are respectively given by $2\lambda$ and $-2\lambda$.  That is, they correspond to the same values that one sees inside the loops of the $\CLE_4$ in the $\CLE_4$/GFF coupling.  We also note that the heights inside of the false loops of the $\ccwBCLE_4(\rho_R)$'s and inside of the false loops of the $\cwBCLE_4(\rho_L)$'s are both equal to $0$.  Modulo checking the Minkowski dimension statement, we can therefore directly apply Lemma~\ref{cle4charact}, and conclude that the collection of the true loops of the $\ccwBCLE_4(\rho_R)$ and of the $\cwBCLE_4(\rho_L)$ are exactly the $\CLE_4^0$ loops that intersect the $\BCLE_4 (\rho)$ process.
\begin{figure}[ht!]
\begin{center}
\includegraphics [width=3in]{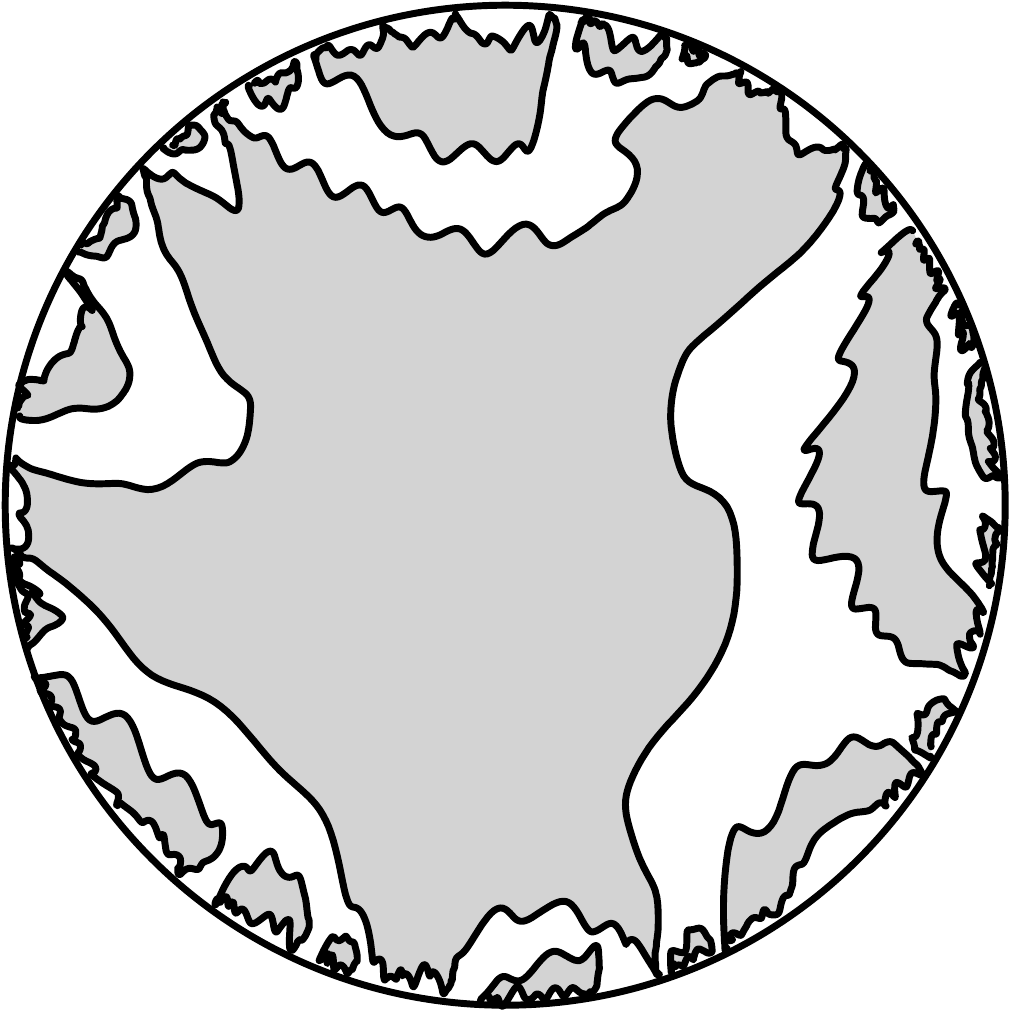}
\includegraphics [width=3in]{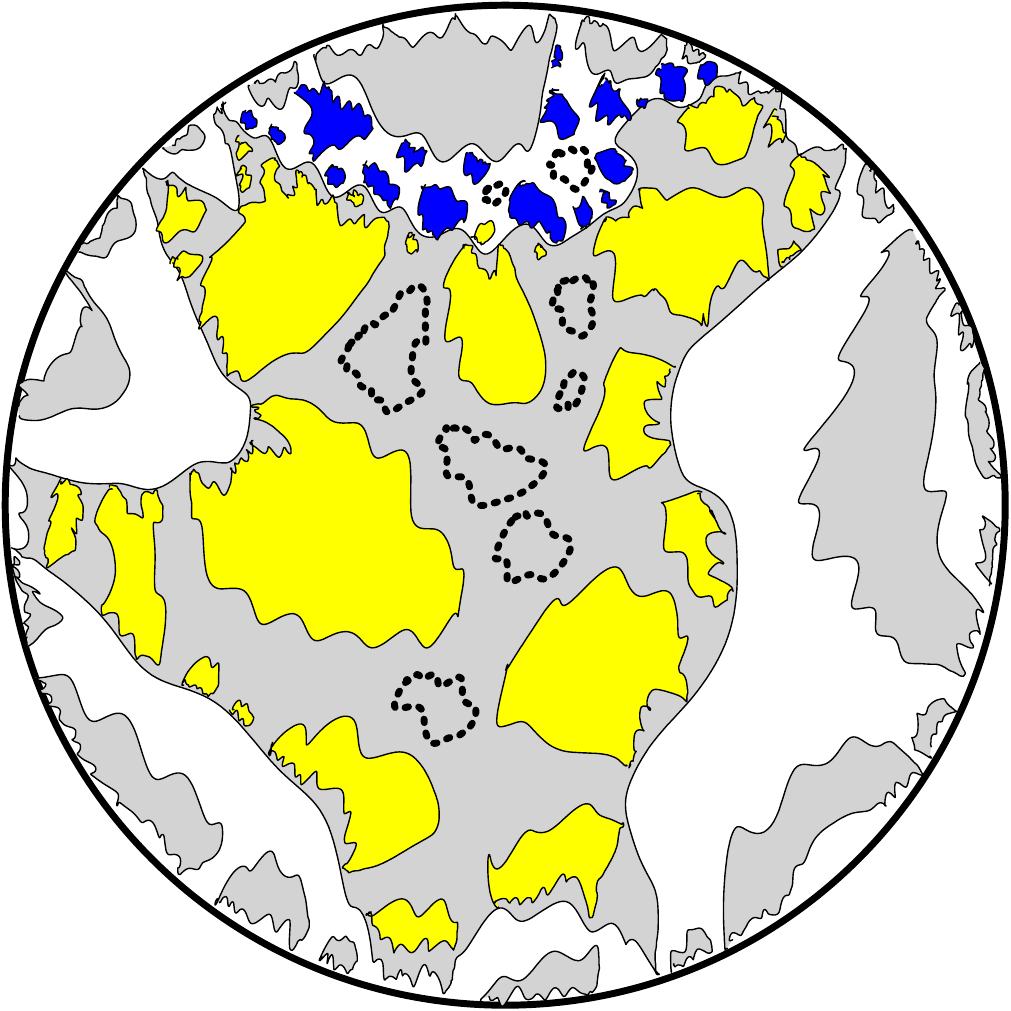}
\caption{\label{fig:bcle_cle_4_perc} A sketch of a $\BCLE_4 (\rho)$ and of a nested $\BCLE_4$ in two of the BCLE loops. The nested (yellow and blue) loops are part of a $\CLE_4^0$, yellow on gray are $+$ loops, blue on white are $-$ loops.
Not all $\CLE_4$ are discovered by this two-level nesting, examples of seven missing loops are drawn in dotted lines.}
\end{center}
\end{figure}

Let $\Gamma$ be the collection of $\cwBCLE_4(\rho_L)$ and $\ccwBCLE_4(\rho_R)$ loops coupled with~$h$ as described just above.  Assuming that~$\Lambda$ is locally finite, one can define a path~$\eta_\Lambda$ which follows along the true (clockwise) loops of~$\Lambda$, in the order in which they are discovered along the counterclockwise arc of~$\partial D$ starting from~$x$.  Assuming further that the entire collection $\Gamma$ is locally finite, we can then define another exploration path $\eta_\Gamma$ which follows along the loops of $\Gamma$ ordered according to when and starting from where they are first visited by $\eta_\Lambda$.  As all of the loops of $\Lambda$ and $\Gamma$ are almost surely determined by $h$ (since they were generated from level lines in the usual sense) it follows that $\eta_\Lambda$ and $\eta_\Gamma$ are almost surely determined by $h$. One can then adapt in this loop-case the arguments of the $\CLE_4^0$ percolation section, to see that $\eta_\Gamma$, when stopped at a stopping time, is a 
local set of the GFF with an explicitly defined harmonic function, and that this property implies that it is a ``full $\bSLE_{\kappa'}^\beta$ loop'' and  that $\eta_\Lambda$ is its ``trunk.''

In the general $\kappa \in (8/3,4)$ and $\kappa' \in (4,6)$ case, the strategy will be quite similar. We will start by defining $\eta_\Lambda$ and $\eta_\Gamma$ in a similar way, then we will study its coupling with the GFF, and in particular see that when one traces $\eta_\Gamma$ up to some stopping time, one constructs a local set with fully described harmonic functions, that in turn will enable us to identify $\eta_\Gamma$ in terms of $\SLE_{\kappa'}^\beta$ paths and $\eta_\Lambda$ as its trunk.  Also, in the general $\kappa,\kappa'$ setting, we will derive the continuity of $\bSLE_\kappa^\beta$ and $\bSLE_{\kappa'}^\beta$ from the continuity of space-filling $\SLE_{\kappa'}$ established in \cite{ms2013ig4}  (note that this fact was instrumental to establish the local finiteness of $\CLE_{\kappa'}$ in \cite{ms2013ig4}).

\section{Imaginary geometry background and some first consequences}
\label{sec:ig}
\label {Sec8}

We begin with a review of imaginary geometry ideas and results.   We then use some of these results to  prove the continuity  of the process which traces the loops of a $\BCLE_\kappa(\rho)$ or $\BCLE_{\kappa'}(\rho')$ process and then of the process that traces the nested $\BCLE$s.  In the present section, $\eta$ and $\eta'$ will stand for curves which are different from earlier.

\subsection{SLE/GFF couplings}

We turn to give a brief overview of the so-called imaginary geometry of the GFF \cite{ms2012ig1,ms2012ig2,ms2012ig3,ms2013ig4}.  This terminology refers to the theory of the flow lines of the formal vector field $e^{i h / \chi}$ where $h$ is an instance of the GFF and $\chi > 0$.  These are paths $\eta$ which formally solve the ODE
$d \eta(t) / dt  =  \exp ({i h(\eta(t))/\chi})$.  As in the case of GFF level lines described in Section~\ref{subsec:cle4_local}, this description is non-rigorous because~$h$ takes values in the space of distributions and does not have values at points, but there is a way make sense of this which is analogous to the construction of GFF level lines and loosely goes as follows.  The general idea is that, if $h$ were a smooth function and $\eta$ its flow line as described by the previous ODE, then $\eta$ determines the values of $h$ along the trajectory of $h$ since the condition that $\eta$ is a flow line of $h$ gives the direction of the arrows of the vector field $\exp ({ih/\chi})$ along $\eta$.  Moreover, one can change the values of $h$ off the range of $\eta$ without affecting that $\eta$ solves the equation.  This suggests the following strategy to make sense of solutions to the flow line ODE in the case that $h$ is a GFF:

First sample the random curve $\eta$ according to a well-chosen distribution (that turns out to be an $\SLE_{\kappa} (\rho)$-type path, depending on the boundary conditions for $h$) and view it as a local set of a certain GFF with corresponding boundary conditions.  That is, we associate deterministically with $\eta$ a harmonic function in the complement of $\eta$, define a GFF in the complement of $\eta$ with these boundary conditions, and then check that, viewed as a distribution on the entire domain, this procedure yields a GFF $h$ with certain boundary conditions.  We then check that in this coupling between $h$ and $\eta$, the curve $\eta$ is in fact a deterministic function of $h$. This strategy has been implemented and made precise in \cite{ss2010continuumcontour,she2010weld,dubedat2009gff,ms2012ig1,ms2013ig4}.  The version of the statements of this type that we will need (and briefly recall now) are given in \cite[Theorem~1.2]{ms2012ig1}.

Note that if $h$ is a smooth function and $\psi \colon \wt{D} \to D$ is a conformal transformation, then by the chain rule $\psi^{-1} \circ \eta$ is a flow line of $h \circ \psi - \chi \arg \psi'$.  With this in mind, we define an \emph{imaginary surface} to be an equivalence class of pairs $(D,h)$ under the equivalence relation:
\begin{equation}
\label{eqn:change_of_coordinates}
(D,h) \to (\psi^{-1}(D), h \circ \psi - \chi \arg \psi') = (\wt{D},\wt{h}).
\end{equation}
We will frequently use this equivalence relation when we describe the GFF and its boundary data in various domains. In particular, a GFF $h$ with boundary data $\Fh_0$ in $D$ will be considered to be equivalent (via the mapping $\psi$) to a GFF $\wt{h}$ with boundary data $\wt{\Fh}_0$ in $\wt{D}$.  Note in particular that~$\Fh_0$ is harmonic if and only if $\wt{\Fh}_0$ is harmonic.  Note also that if $\wt{D} = \h$ and $D$ is a domain with smooth boundary so that $\psi'$ is defined everywhere on the boundary, the function $\arg \psi'$ is equal to the harmonic extension of the winding of the boundary to the interior of the domain.  For general domains with fractal boundary on which $\psi'$ is not defined on the boundary, we have that $\psi'$ is defined in the interior and $\arg \psi'$ still has the interpretation of corresponding to the harmonic extension of the winding of the domain boundary.

\subsubsection{Boundary data}
\label{subsubsec:ig_boundary_data}

Let us now explain in detail how to couple an $\SLE_\kappa (\rho)$ type curve with several force-points with a GFF with a given well-chosen boundary condition $\Fh_0$ on $D$.  This coupling will be invariant under the equivalence rule~\eqref{eqn:change_of_coordinates} for a well-chosen $\chi = \chi (\kappa)$, so it is enough to describe it in one particular domain with given starting and end-points. As is customary in the SLE framework, we describe this now in $\HH$ for an SLE going from $0$ to $\infty$.  In the remainder of this paper, for $\kappa \in (0,4)$ and $\kappa'=16/\kappa$, we let
\begin{equation}
\label{eqn:gff_constants}
 \lambda = \frac{\pi}{\sqrt{\kappa}},\quad \lambda' = \frac{\pi}{\sqrt{\kappa'}} , \quad \text{and}\quad \chi = \frac{2}{\sqrt{\kappa}} - \frac{\sqrt{\kappa}}{2}.
\end{equation}
Note that $2(\lambda - \lambda') = \pi \chi$.  We could also define $\chi' = - \chi$, but we will prefer to keep $\chi$, as this will be easier when we will consider simultaneously $\SLE_\kappa$ and $\SLE_{\kappa'}$ processes coupled with the same instance of the GFF.

Let us now consider an $\SLE_\kappa(\ul{\rho})$ process $\eta$ from $0$ to $\infty$ in $\HH$ for 
\[ \ul{\rho} = (\ul{\rho}^L;\ul{\rho}^R) = (( \rho_{k,L}, \rho_{k-1, L}, \ldots, \rho_{1,L}), (\rho_{1,R}, \ldots, \rho_{\ell,R})) \] 
with force points located at $\ul{x} = (x_{k,L} < \cdots < x_{1,L} < 0 < x_{1,R} < \cdots < x_{\ell,R})$. As explained in \cite[Section~2]{ms2012ig1}, there is no problem to define such a process as long as for all $j \le k$ and $i \le \ell$, $\overline \rho_{j,L}:= \rho_{1,L} + \ldots + \rho_{j,L} > -2$ and $\overline \rho_{i,R} := \rho_{1, R} + \ldots + \rho_{i,R} > -2$ and it is a direct generalization of the $\SLE_\kappa (\rho)$ processes that we have discussed before.  It is shown in \cite{ms2012ig1} that it is generated by a continuous curve $\eta$ in $\ol{\H}$.  In fact, if $\ol \rho_{i,q} \leq -2$ for some $i,q$, there is no difficulty in making sense of the process and it also follows from \cite{ms2012ig1} that it is continuous, but only up to the first time that the driving function collides with one of the force points with $\ol \rho_{i,q} \leq -2$.  This time is called the \emph{continuation threshold} in \cite{ms2012ig1}.

We now define the boundary conditions of the GFF with which we will want to couple with this curve. We let $\Fh_0$ denote the bounded harmonic function in $\HH$ with boundary conditions
\begin{equation}
\label{eqn:gff_flow_bd}
\begin{split}
 -\lambda ( 1+ \overline{\rho}_{j, L}) \quad\text{for}\quad x \in (x_{j+1,L},x_{j,L}] \quad\text{and}\quad 0 \leq j \leq k \\
  \lambda (1+ \overline \rho_{i,R})  \quad\text{for}\quad x \in (x_{i,R},x_{i+1,R}] \quad\text{and}\quad 0 \leq i \leq \ell,
\end{split}
\end{equation}
with $\overline \rho_{0,L} = \overline \rho_{0,R} = 0$, $x_{0,L} = x_{0,R} = 0$, and $x_{k+1,L} = - \infty$, $x_{\ell+1,R} = \infty$.

\begin{figure}[ht!]
\includegraphics[scale=0.85]{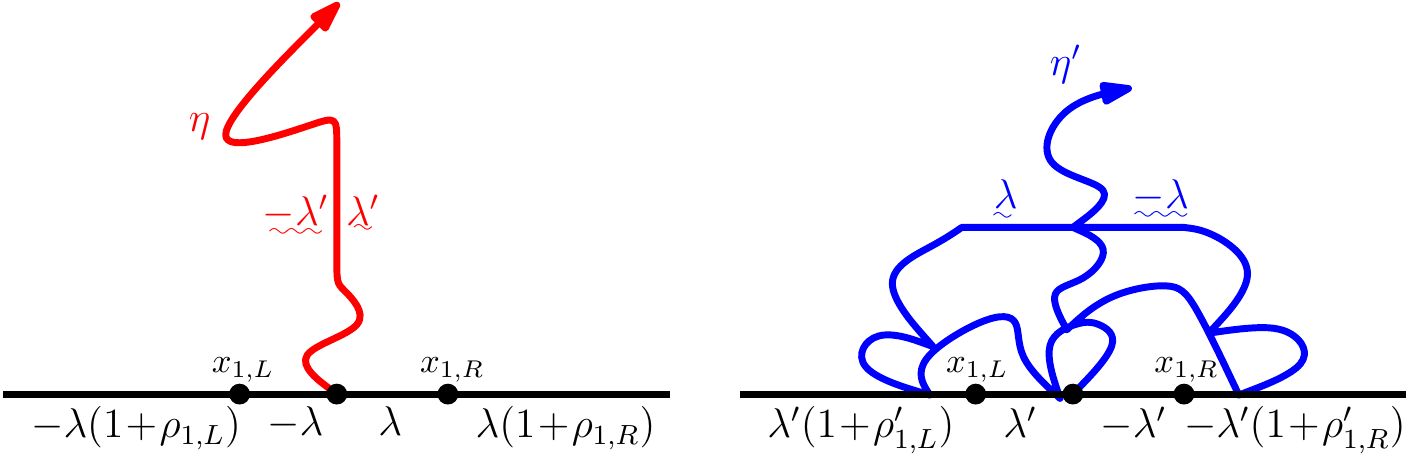}
\caption{\label{fig:gff_boundary_data}
{\bf Left:} Boundary data and the corresponding $\SLE_{\kappa}(\rho_{1,L};\rho_{1,R})$ flow line. {\bf Right:} Boundary data and the corresponding $\SLE_{\kappa'}(\rho_{1,L}';\rho_{1,R}')$ counterflow line.}
\end{figure}

For any $t \ge 0$, we define the harmonic function $\Fh_t$ in the complement of the curve $\eta [0,t]$ with boundary data that can be informally defined as follows: On the left side of the curve $\eta [0,t]$ at $\eta(s)$, it is equal to $-\lambda' + \chi \cdot {\rm winding}$, where the winding is the winding of $\eta$ between $\eta(0)$ and the considered point $\eta(s)$. On the right side of the curve, the boundary data is $\lambda' + \chi \cdot {\rm winding}$, and on $\partial \h$, one uses the same boundary data as $\Fh_0$.  More explicitly, $\Fh_t = \Fh_0 \circ f_t - \chi \arg f_t'$ where $f_t = g_t - W_t$ and~$g_t$ is the unique conformal map from the unbounded component of $\h \setminus \eta([0,t])$ to~$\h$ with $g_t(z) - z \to 0$ as $z \to \infty$.  Note that this definition easily extends to the bounded connected components of the complement of $\eta [0,t]$, if there are any.

It then turns out that for any stopping time $\tau$, the curve $\eta$ up to $\tau$ and the harmonic function $\Fh_\tau$ define a local set of a GFF $h$ in $\HH$ with boundary conditions $\Fh_0$ (this result is stated \cite[Theorem~1.1]{ms2012ig1}, though \cite{ms2012ig1} is not the first place where this result is proved; see \cite{ss2010continuumcontour,dubedat2009gff,she2010weld}).

In our figures, we will often indicate the boundary data along curves and boundary segments using the notation $\uwave{x}$.  This notation is explained in detail in \cite[Figure~1.10]{ms2012ig1} (see also \cite[Figure~1.9]{ms2012ig1}).  This is illustrated in the case that $k=\ell=1$ in Figure~\ref{fig:gff_boundary_data}. 

Again, as shown in \cite{ms2012ig1} the curve $\eta$ can then be deterministically recovered from the GFF. For the reasons mentioned above, it is referred to as the \emph{flow line} starting at $0$ and targeted at $\infty$ of the GFF $h$ with the boundary conditions $\Fh_0$.  This flow line is targeting $\infty$, but it coincides with the flow line targeting another point of the same GFF (defined modulo the imaginary geometry equivalence) until the first time at which $\eta$ disconnects it from $\infty$.  This point will be important when we discuss the couplings of $\BCLE$ with the GFF.

In view of the imaginary geometry, it is natural to define the flow line with angle $\theta$ associated with a GFF (with certain boundary conditions) $h$, to be the flow line of $h + c$, where $c(\theta):= \theta \chi$.  This terminology is motivated by the interpretation of the path as a flow line of the vector field $e^{ih / \chi}$.  In particular, adding $c(\theta)$ to the field has the interpretation of rotating all of the vectors by the angle $\theta$.

The same story works for an $\SLE_{\kappa'}(\ul{\rho}')$ process starting from $0$. However, we will change signs in order to accommodate for the $\chi = - \chi'$ change. The boundary conditions in this case are given by:
\begin{equation}
\label{eqn:gff_cfl_bd}
\begin{split}
 \lambda' (1+\overline \rho_{j,L}') &\quad\text{for}\quad x \in (x_{j+1,L},x_{j,L}] \quad\text{and}\quad 0 \leq j \leq k\\
  -\lambda' (1+\overline \rho_{i,R}') &\quad\text{for}\quad x \in (x_{i,R},x_{i+1,R}] \quad\text{and}\quad 0 \leq i \leq \ell \end{split}
\end{equation}
with otherwise the same conventions as indicated above.  This notation along with the boundary data for the coupling is illustrated in the case that $k=\ell=1$ in Figure~\ref{fig:gff_boundary_data}.  We refer to $\eta'$ as the \emph{counterflow line} of $h$ starting from $0$.  The reason for the differences in signs and terminology is that it enables us to couple flow and counterflow lines with the same field, in such a way that the latter naturally grows in the opposite direction of the former.

These couplings, and their embedded possible change-of-targets make it possible to couple entire BCLE processes with a GFF, once the starting point of the BCLE tree is chosen. Table~\ref{tab:sle_bcle_boundary_conditions} lists the various boundary conditions for $\SLE$ with force points located at $0^-$ and $0^+$.
\begin{table}[ht!]
\rowcolors{1}{lightgray}{white}
\setlength\extrarowheight{3pt}
\begin{tabular}{ c  c  c  }
    & $\R_-$  & $\R_+$\\
 $\SLE_\kappa(\rho_1;\rho_2)$ & $-\lambda(1+\rho_1)$ & $\lambda(1+\rho_2)$ \\
 $\SLE_{\kappa'}(\rho_1';\rho_2')$ & $\lambda'(1+\rho_1')$ & $-\lambda'(1+\rho_2')$\\
  $\cwBCLE_\kappa(\rho)$ & $-\lambda(1+\rho)$ & $-\lambda(1+\rho) - 2\pi \chi$ \\
 $\ccwBCLE_\kappa(\rho)$ & $\lambda(1+\rho) + 2\pi \chi$ & $\lambda(1+\rho) $\\
 $\cwBCLE_{\kappa'}(\rho')$ & $\lambda'(1+\rho')$ & $\lambda'(1+\rho') - 2\pi \chi$ \\
$\ccwBCLE_{\kappa'}(\rho')$ & $-\lambda'(1+\rho') + 2\pi \chi$ & $-\lambda'(1+\rho')$ \\
\end{tabular}
\vspace{0.01\textheight}
\caption{\label{tab:sle_bcle_boundary_conditions}
Boundary data for coupling $\SLE_\kappa(\rho_1;\rho_2)$ and $\SLE_{\kappa'}(\rho_1';\rho_2')$ from $0$ to $\infty$ with force points at $0^-$ and $0^+$ with a GFF on $\h$.  Also shown is the boundary data for coupling clockwise and counterclockwise $\BCLE_\kappa(\rho)$ and $\BCLE_{\kappa'}(\rho')$ with a GFF~$h$ on~$\h$ using a branching flow or counterflow line starting from $0$.}
\end{table}

\subsubsection{Interaction rules} 
\label{subsubsec:ig_interaction_rules}

The description of how the flow and counterflow lines starting from different boundary points and with different angles interact with each other is provided in \cite{ms2012ig1} (paths starting from boundary points) and \cite{ms2013ig4} (paths starting from interior points).  We will now recall the elements of this that will be important for this article.

Suppose that $h$ is a GFF on $\h$ with piecewise constant boundary data as before.  For $x_1 < x_2$ and $\theta_1,\theta_2 \in \R$, let $\eta_{\theta_i}^{x_i}$ be the flow line of $h$ with angle $\theta_i$ starting from $x_i$.  In \cite[Theorem~1.5]{ms2012ig1}, it is described how $\eta_{\theta_1}^{x_1}$ and $\eta_{\theta_2}^{x_2}$ interact (i.e., the conditional law of the latter given the former, the relative position of the latter with respect to the former etc.).  In this article, we will need one particular version of this.  Namely, in the case that the boundary data of $h$ is given by $a$ (resp.\ $b$) on $\R_-$ (resp.\ $\R_+$) and $x_1=x_2 = 0$ (so both paths start from the origin) and $\theta_1 < \theta_2$.  In this case, $\eta_1 = \eta_{\theta_1}^0$ almost surely lies to the right of $\eta_2 = \eta_{\theta_2}^0$.  Moreover, $\eta_i$ is marginally an $\SLE_{\kappa}(\rho_1^i;\rho_2^i)$ process with
\[ \rho_1^i = -\frac{a+\theta_i \chi}{\lambda}-1 \quad\text{and}\quad \rho_2^i = \frac{b+\theta_i \chi}{\lambda}-1.\]
The conditional law of $\eta_1$ given $\eta_2$ is that of an $\SLE_\kappa( (\theta_2-\theta_1) \chi/\lambda-2; (b+\theta_1 \chi)/\lambda-1)$ independently in each of the components of $\h \setminus \eta_2$ which are to the right of $\eta_2$ and the conditional law of $\eta_2$ given $\eta_1$ is that of an $\SLE_\kappa(-(a+\theta_2 \chi)/\lambda-1;(\theta_2 -\theta_1)\chi/\lambda-2)$ independently in each of the components of $\h \setminus \eta_1$ which are to the left of $\eta_1$.  See Figure~\ref{fig:monotonicity_cfl} (see also \cite[Figure~1.20]{ms2012ig1}).

\begin{figure}
\begin{center}
\includegraphics[scale=0.85]{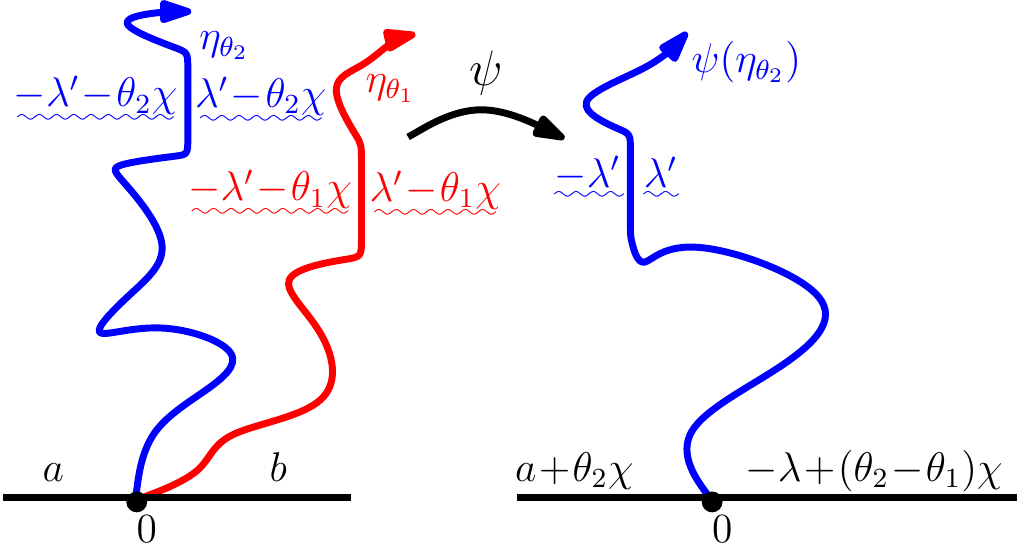} \hspace{0.025\textwidth}\includegraphics[scale=0.85]{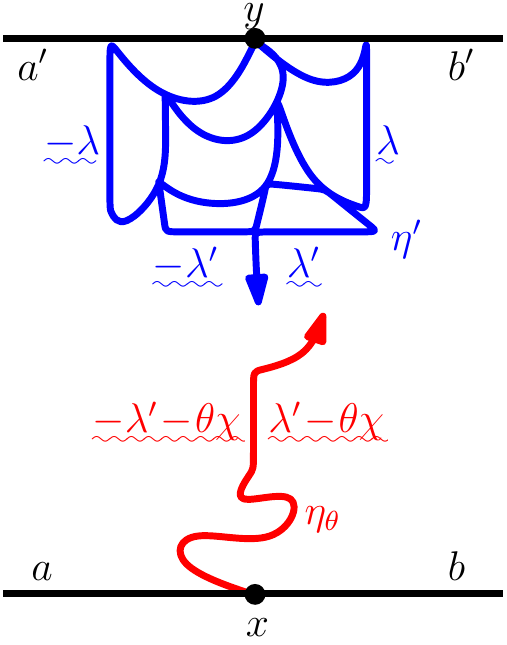}
\end{center}
\caption{\label{fig:monotonicity_cfl} {\bf Left:} Two flow lines $\eta_{\theta_1},\eta_{\theta_2}$ starting from the origin with angles $\theta_1 < \theta_2$ of a common GFF $h$ on $\h$ with the indicated boundary data.  {\bf Middle:} The image of  $\eta_{\theta_2}$ (given $\eta_{\theta_1}$) via the conformal transformation $\psi$ from the component which is to the left of $\eta_{\theta_1}$ to $\h$) is then a flow line of the image GFF with new boundary data. {\bf Right:} We can couple flow and counterflow lines into the same imaginary geometry.  When $\theta = \pi/2$ (resp.\ $\theta=-\pi/2$), the flow line is equal to the left (resp.\ right) side of the counterflow line.  Note that for $\theta = \pi/2$ (resp.\ $\theta = -\pi/2$) the boundary data on the left (resp.\ right) side of the flow line is the same as the boundary data on the corresponding side of the counterflow line.}
\end{figure}

The $\SLE$/GFF coupling will also play an important role in this article because it provides a natural coupling between various $\SLE$ processes.  We will now describe one particular example which will be useful in this article.

Shown on the right side of Figure~\ref{fig:monotonicity_cfl} is a GFF on the strip $\strip = \R \times [0,1]$ with the boundary data indicated.  The flow line $\eta_\theta$ starting from $x$ with angle $\theta$ is marginally an $\SLE_\kappa(\ul{\rho})$ process and the counterflow line $\eta'$ starting from $y$ is marginally an $\SLE_{\kappa'}(\ul{\rho}')$ process for some $\ul \rho$ and $\ul \rho'$.  If $\theta > \tfrac{\pi}{2}$ (resp.\ $\theta < -\tfrac{\pi}{2}$), then $\eta_\theta$ will lie to the left (resp.\ right) of $\eta'$.  If $\theta = \tfrac{\pi}{2}$ (resp.\ $\theta = - \tfrac{\pi}{2}$) then $\eta_\theta$ is equal to the left (resp.\ right) side of $\eta'$.  Finally, if $\theta \in (-\tfrac{\pi}{2},\tfrac{\pi}{2})$ then $\eta_\theta$ is contained in $\eta'$.  Suppose for simplicity that $\theta > \tfrac{\pi}{2}$ so that $\eta_\theta$ is to the left of $\eta'$.  Then the results of \cite{ms2012ig1} imply that we can draw $\eta_\theta$ and $\eta'$ in either order.  In particular, if we first 
draw all of $\eta'$, then $\eta_\theta$ is given by the flow line with angle $\theta$ of the restriction of $h$ to the component of $\strip \setminus \eta'$ which is to the left of $\eta'$  Conversely, if we first draw all of $\eta_\theta$, then $\eta'$ is given by the counterflow line of the restriction of $h$ to the component of $\strip \setminus \eta_\theta$ which is to the right of $\eta_\theta$.  This follows because:
\begin{itemize}
\item In general, the flow and counterflow lines of GFFs are almost surely determined by the GFF \cite[Theorem~1.2]{ms2012ig1} (the version of this for paths which start at interior points is given in \cite[Theorems~1.2, 1.4, 1.6]{ms2013ig4}.)
\item The flow / counterflow lines coupled with a GFF are characterized by the boundary data of the conditional law of the field given the path.  This, in particular, implies that there can be at most one flow line coupled with a GFF with each given angle \cite[Theorems~1.1, 2.4]{ms2012ig1}.
\end{itemize}
In our setting, local set theory tells us that the boundary data for the conditional law of $h$ given all of $\eta'$ and a segment of $\eta_\theta$ takes the form of the flow line of the GFF given by $h$ conditional on $\eta'$ along $\eta_\theta$.  Therefore, by the points above, $\eta_\theta$ is the flow line of this GFF with angle $\theta$.  Conversely, the boundary data for the conditional law of $h$ given all of $\eta_\theta$ and a segment of $\eta'$ takes the form of a counterflow line of the GFF given by $h$ conditional on $\eta_\theta$ along $\eta'$.  Thus, as before, $\eta'$ is the counterflow line of this GFF.  In summary, one can draw $\eta_\theta$ and $\eta'$ in any order without affecting the final path configuration. This type of idea will be of course very useful in order to derive our BCLE duality relations. 

\subsubsection{$\BCLE$ and the GFF}
\label{subsubsec:bcle_gff}

\begin{table}[ht!]
\rowcolors{1}{lightgray}{white}
\setlength\extrarowheight{3pt}
\begin{tabular}{ c  c  c  }
    & $\R_-$  & $\R_+$\\
 CW loop of $\BCLE_\kappa(\rho)$ & $\lambda$ & $\lambda-2\pi \chi$ \\
 CCW loop of $\BCLE_\kappa(\rho)$ & $-\lambda + 2\pi \chi$ & $-\lambda$ \\
 CW loop of $\BCLE_{\kappa'}(\rho')$ & $-\lambda'$ & $-\lambda'-2\pi \chi$ \\
 CCW loop of $\BCLE_{\kappa'}(\rho')$ & $\lambda'+2\pi \chi$ & $\lambda'$\\
\end{tabular}
\vspace{0.01\textheight}
\caption{\label{tab:bcle_inside_loop_bc}
Boundary data for the conditional law of a GFF $h$ inside of the clockwise and counterclockwise loops of a $\BCLE$ after applying a conformal change of coordinates from the region surrounded by the loop to $\h$ which sends the marked point of the loop to $0$ and any other point to $\infty$.  The boundary data for the conditional law does not depend on whether we use $\cwBCLE$ or $\ccwBCLE$.  Note that the heights on $\R_-$ and $\R_+$ always differ by $2\pi\chi$; this is due to the change of coordinates formula~\eqref{eqn:change_of_coordinates}.}
\end{table}

\begin{figure}
\begin{center}
\includegraphics[scale=0.85]{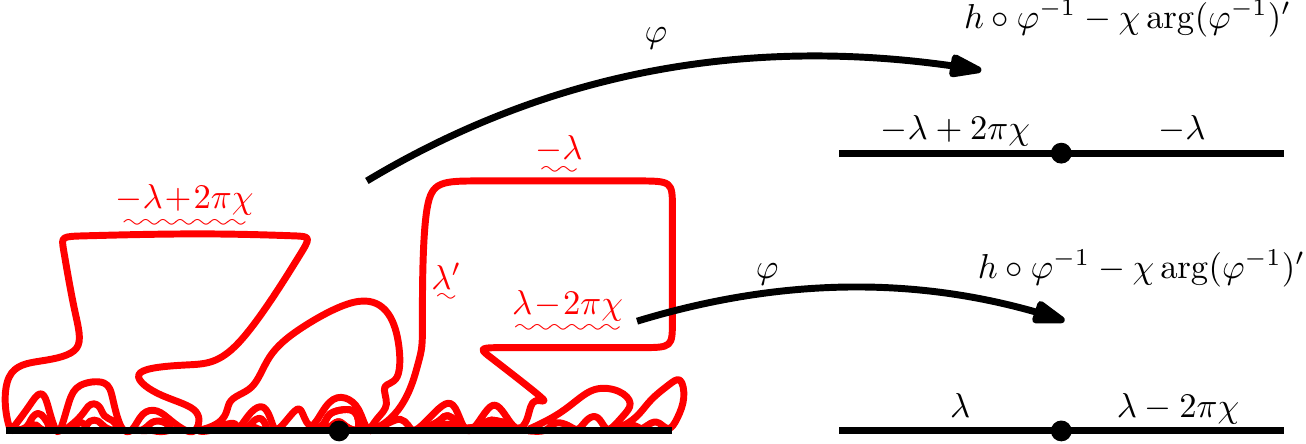}
\end{center}
\caption{\label{fig:bcle_boundary_data}{\bf Left:} A $\cwBCLE_\kappa(\rho)$ generated using a boundary branching flow line of a GFF on $\h$ with boundary conditions given by $-\lambda(1+\rho)$ (resp.\ $-\lambda(1+\rho)-2\pi \chi$) on $\R_-$ (resp.\ $\R_+$).  {\bf Right:} The boundary data for the field $h \circ \varphi^{-1} - \chi \arg (\varphi^{-1})'$ where $\varphi$ is a conformal transformation from a counterclockwise (top) and clockwise (bottom) loop of the $\cwBCLE_\kappa(\rho)$.}
\end{figure}

As mentioned earlier, the target invariance which is built into the $\SLE$/GFF coupling makes it possible to couple the $\BCLE$'s with the GFF in a natural way.  Namely, a $\BCLE_\kappa(\rho)$ is constructed as a boundary branching flow line targeted at every boundary point and a $\BCLE_{\kappa'}(\rho')$ is coupled as a boundary branching counterflow line targeted at every boundary point.  The reason that this construction works is that, by \cite[Theorem~1.2]{ms2012ig1} we have that the paths targeted at different points almost surely agree until the two points are first separated.  That is, the $\bSLE_\kappa (\rho)$ processes naturally fit into the $\SLE$/GFF coupling framework.

The boundary data one uses to couple the different $\BCLE$s (depending on $\kappa$, $\kappa'$, and the orientation) is the same as for the $\SLE$-type processes and is summarized in Table~\ref{tab:sle_bcle_boundary_conditions}.

Since it will be important for our later arguments, we will now explain how one reads off the conditional law of the field inside of each $\BCLE$ loop.  We will explain this in detail in the case of a $\cwBCLE_\kappa(\rho)$; the other possibilities are analogous.  (See Figure~\ref{fig:bcle_boundary_data} for an illustration.)  Suppose that $h$ is a GFF on $\h$ with boundary conditions given by $-\lambda(1+\rho)$ (resp.\ $-\lambda(1+\rho)-2\pi \chi$) on $\R_-$ (resp.\ $\R_+$).  These are the boundary conditions to couple a $\bSLE_\kappa(\rho)$ hence $\cwBCLE_\kappa(\rho)$ starting from $0$ with the GFF.  Let $\Lambda$ be the collection of loops and false loops thus formed and let $\CL \in \Lambda$ be a clockwise loop.  Let $\gamma \colon [0,1] \to \ol{\H}$ be a parameterization of $\CL$ which starts from the first point $x$ on $\CL \cap \partial \h$ which is visited by the path which traverses $\partial \h$ counterclockwise starting from $0$.  Then for each $\epsilon > 0$ there exists $z$ such that the right side of $\gamma([0,1-\epsilon])$ is 
contained in the right side of the flow line of $h$ starting from $0$ and targeted at $z$.  By \cite[Proposition~3.8]{ms2012ig1}, we thus have that the boundary data for the conditional law of $h$ given $\Lambda$ in the region $U$ inside of $\CL$ along the right side of $\gamma([0,1-\epsilon])$ agrees with the boundary data for the conditional law of $h$ given the aforementioned flow line along the same boundary segment.  Since $\epsilon > 0$ was arbitrary, this allows us to determine the boundary data for the conditional law of $h$ given $\Lambda$ in $\CL$ and it is the same as if $\CL$ was equal to the right side of a flow line starting from and terminating at $z$.  Consequently, if $\varphi$ is a conformal transformation $U \to \h$ which takes $x$ to $0$ and any other point to $\infty$, then $h \circ \varphi^{-1} - \chi \arg (\varphi^{-1})'$ has the law of a GFF on $\h$ with boundary conditions given by $\lambda$ (resp.\ $\lambda-2\pi \chi$) on $\R_-$ (resp.\ $\R_+$).

One can similarly read off the boundary conditions for $h$ given $\Lambda$ along the boundary of a region $U$ which is surrounded by a counterclockwise loop in $\Lambda$.  Namely, if $x$ is the point on $\partial \h$ as described above and $\varphi$ is a conformal transformation as described above, then $h \circ \varphi^{-1} - \chi \arg (\varphi^{-1})'$ has the law of a GFF on $\h$ with boundary conditions given by $-\lambda + 2\pi \chi$ (resp.\ $-\lambda$) on $\R_-$ (resp.\ $\R_+$).

We note that the boundary data inside of the clockwise false loops and counterclockwise loops of a $\ccwBCLE_\kappa(\rho)$ takes exactly the same form.

The boundary data for the conditional law after conformally mapping to $\h$ in all cases is summarized in Table~\ref{tab:bcle_inside_loop_bc}.

\subsection{Space-filling $\SLE$ and local finiteness}

We turn to remind the reader of the construction and continuity of space-filling $\SLE_{\kappa'}$.  We will then describe how it is possible to extract the local finiteness of $\BCLE$s from the continuity of space-filling $\SLE_{\kappa'}$.

\subsubsection{Space-filling $\SLE_{\kappa'}$}

We begin by describing the construction of space-filling $\SLE_{\kappa'}$ in the context of $\CLE_{\kappa'}$ and then subsequently describe its construction in the framework of imaginary geometry.  This latter framework is the setting in which many of the properties of space-filling $\SLE_{\kappa'}$, including continuity, are actually proved in \cite{ms2013ig4}, and this is the setting in which we will use it to prove the local finiteness of $\BCLE$.

Suppose that $D \subseteq \C$ is a bounded Jordan domain and that $\Gamma'$ is a nested $\CLE_{\kappa'}$ in $D$.  Fix $x \in \partial D$ and consider the path which is defined as follows.  Let $\eta_0'$ be the path which parameterizes $\partial D$ in clockwise order, starting from and ending at $x$.  Let $\Gamma_1'$ be the collection of loops in $\Gamma'$ which have non-empty intersection with $\partial D$ and let $\eta_1'$ be the path which traces each of the loops of $\Gamma_1'$ starting from and in the order in which they are first hit by the time-reversal of $\eta_0'$, with a clockwise direction.  Assuming the local finiteness of $\CLE_{\kappa'}$ (i.e., for each $\epsilon > 0$ the number of loops of $\Gamma'$ with diameter at least $\epsilon$ is finite), note that $\eta_1'$ does in fact define a continuous path.  Assuming that $\eta_1',\ldots,\eta_k'$ and $\Gamma_1',\ldots,\Gamma_k'$ have been defined for some $k$, we let $\Gamma_{k+1}'$ consist of those loops of $\Gamma'$ which are contained in 
the closure of a component of $D \setminus \eta_k'$ and which intersect the range of $\eta_k'$.  Equivalently, $\Gamma_{k+1}'$ consists of those loops of $\Gamma'$ which are not in $\Gamma_1',\ldots,\Gamma_k'$ and have non-empty intersection with a loop in $\Gamma_k'$.  We then let $\eta_{k+1}'$ be the path which is given by following $\eta_k'$ and, whenever $\eta_k'$ intersects itself and cuts off a component $U$, it follows the loops of $\Gamma_{k+1}'$ contained in $\ol{U}$ as follows.  If $\eta_k'$ has drawn $\partial U$ with a clockwise (resp.\ counterclockwise) orientation, then the aforementioned loops of $\Gamma_{k+1}'$ are each drawn with a clockwise (resp.\ counterclockwise) orientation starting from and ordered according to where/when the path which traverses $\partial U$ in counterclockwise (resp.\ clockwise) order starting from the first (equivalently, last) point on $\partial U$ visited by $\eta_k'$.  Assuming the local finiteness of $\Gamma'$, it is not difficult to see that $\eta_k'$ converges uniformly to a limiting path $\eta'$ as $k \to \infty$.  This limiting path is space-filling $\SLE_{\kappa'}$ and it is the Peano curve associated with the exploration tree defined in \cite{she2009cle} to construct $\CLE_{\kappa'}$.

\begin{figure}[ht!]
\begin{center}
\subfigure{
\includegraphics[scale=0.85, page=1]{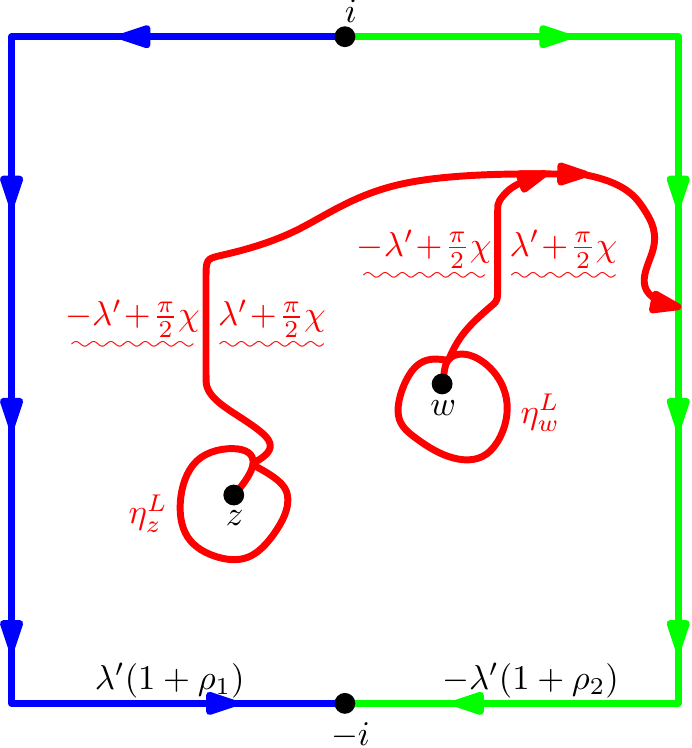}}
\hspace{0.05\textwidth}
\subfigure{
\includegraphics[scale=0.85, page=2]{figures/space_filling_sle}}
\end{center}
\caption{\label{fig:space_filling_sle} {\bf Left:} Suppose that $h$ is a GFF in $D = [-1,1]^2$ with the illustrated boundary data where $\rho_1',\rho_2' \in (-2,\tfrac{\kappa'}{2}-2)$.  For each $z \in D$, let $\eta_z^L$ be the flow line of $h$ starting from $z$ with angle $-\tfrac{\pi}{2}$.  Then we can order the points of $D$ using $h$ by declaring for $z,w \in D$ distinct that $w$ comes before $z$ if $\eta_w^L$ merges with $\eta_z^L$ on its right side.  It is shown in \cite[Theorem~1.16]{ms2013ig4} that there exists a non-crossing, non-self-tracing continuous path $\eta'$, so-called space-filling $\SLE_{\kappa'}(\rho_1';\rho_2')$, which visits the points of $D$ according to this order.  Flow lines bounce off the left (blue) and right (green) sides of $\partial D$ as if they were flow lines using the interaction rules from \cite[Theorem~1.7]{ms2013ig4}.  The direction of these flow lines depends on the values of $\rho_1',\rho_2'$.  In particular, if $\rho_j' \in (-2,\kappa'/2-4]$ (resp.\ $\rho_j' \in (\kappa'/2-4,\kappa'/2-2)$) then the direction of the corresponding boundary segment is from $i$ to $-i$ (resp.\ $-i$ to $i$).  Shown is the case that $\rho_1',\rho_2' \in (-2,\kappa'/2-4]$.  {\bf Right:} Suppose that $h$ is a GFF with the given boundary data.  Then we can define a space-filling $\SLE_{\kappa'}$ loop by ordering the points of $D$ in the same manner.  It follows from \cite[Theorem~1.16]{ms2013ig4} that this is also a continuous curve which fills $\partial D$ counterclockwise.  A similar construction yields a clockwise space-filling $\SLE_{\kappa'}$ loop.  This path is equal in law to the path on the left when it fills a loop that it has cut off from its target point.  These paths exactly correspond to the Peano curve associated with the $\CLE_{\kappa'}$ exploration tree explored in the counterclockwise and clockwise directions.}
\end{figure}

As mentioned just above, space-filling $\SLE_{\kappa'}$ also fits into the imaginary geometry framework of \cite{ms2012ig1,ms2012ig2,ms2012ig3,ms2013ig4} and, in fact, was first proved to be continuous in \cite{ms2013ig4} which in turn led to the first proof of the local finiteness of $\CLE_{\kappa'}$.  This was established by understanding the manner in which the imaginary geometry construction of space-filling $\SLE_{\kappa'}$ interacts with $\CLE_{\kappa'}$.  In other words, the idea for the construction of space-filling $\SLE_{\kappa'}$ is motivated by its connection to $\CLE_{\kappa'}$ described just above, but its existence as a continuous path (hence the local finiteness of $\CLE_{\kappa'}$) was first proved in \cite{ms2013ig4}.  As we will explain below, we will see in a similar manner that the loops formed by admissible $\BCLE_\kappa(\rho)$ and $\BCLE_{\kappa'}(\rho')$ ensembles are also locally finite.

We will now recall the construction of space-filling $\SLE_{\kappa'}$ in the context of imaginary geometry \cite{ms2013ig4} (see Figure~\ref{fig:space_filling_sle}).  We begin with the case of \emph{chordal} space-filling $\SLE_{\kappa'}(\rho_1';\rho_2')$.  Like chordal $\SLE$, this is a continuous process which connects two distinct points on the boundary of a domain.  Specifically, we suppose that $h$ is a GFF on $\h$ with boundary conditions given by $\lambda'(1+\rho_1')$ on $\R_-$ and $-\lambda'(1+\rho_2')$ on $\R_+$.  Fix a deterministic countable dense set $(r_k)$ in $\h$.  Then we can construct an ordering on the $(r_k)$ by saying that $r_j$ comes before $r_k$ if it is true that the flow line of $h$ starting from $r_j$ with angle $-\pi/2$ merges with the flow line of $h$ starting from $r_k$ with angle $-\pi/2$ on its right side (flow lines started at interior points were constructed in \cite{ms2013ig4}).  In this construction, it is important to describe how these flow lines bounce off $\partial \h$.  For chordal space-filling $\SLE_{\kappa'}
(\rho_1';\rho_2')$, the reflection rule is defined by viewing $\R_-$ (resp.\ $\R_+$) as a flow line and then using the interaction rules for flow lines \cite[Theorem~1.7]{ms2013ig4} to determine how the flow lines used to define $\eta'$ behave when they hit $\partial \h$.  The orientations of the two boundary segments $\R_-$ and $\R_+$ depend on the values of $\rho_1',\rho_2'$.  In particular, for $\rho_j' \in (-2,\kappa'/2-4]$, the boundary segment is oriented towards $0$ from $\infty$ and if $\rho_j' \in (\kappa'/2-4,\kappa'/2-2)$, the boundary segment is oriented from $0$ towards $\infty$.  It is shown in \cite{ms2013ig4} that there is a continuous path coupled with and determined by $h$ which respects this ordering.  This is the construction/definition of chordal space-filling $\SLE_{\kappa'}(\rho_1';\rho_2')$ from $0$ to $\infty$ in $\h$.  It is related to ordinary chordal $\SLE_{\kappa'}(\rho_1';\rho_2')$ in $\h$ from $0$ to $\infty$ in that if one parameterizes it according to half-capacity then it is 
an ordinary chordal $\SLE_{\kappa'}(\rho_1';\rho_2')$.  Chordal space-filling $\SLE_{\kappa'}(\rho_1';\rho_2')$ from $x$ to $y$ on a bounded Jordan domain $D$ with $x,y \in \partial D$ distinct is defined by applying a conformal transformation $\h \to D$ which takes $0$ to $x$ and $\infty$ to $y$.

Due to the flow line interaction rules \cite[Theorem~1.7]{ms2013ig4}, if we start a flow line from the target point of a chordal space-filling $\SLE_{\kappa'}(\rho_1';\rho_2')$, the space-filling path will visit that flow line in reverse chronological order.  If we draw a counterflow line with the same starting and ending points as the space-filling $\SLE_{\kappa'}(\rho_1';\rho_2')$, then the space-filling $\SLE_{\kappa'}(\rho_1';\rho_2')$ will visit the points of the counterflow line in the same order.  Whenever the counterflow line cuts off a component from $\infty$, the space-filling $\SLE$ branches in and fills up this component before continuing along the trajectory of the counterflow line.

There is another version of space-filling $\SLE_{\kappa'}$ which is a loop starting and ending at a given boundary point.  This is the version which corresponds to the space-filling loop constructed out of a $\CLE_{\kappa'}$ as described just above.  It is defined in the same way as chordal space-filling $\SLE_{\kappa'}$ except the reflection rule for the flow lines which define the ordering interact with the domain boundary is different.  Namely, in this construction on $\h$, we view $\partial \h$ as a flow line loop which starts and ends at $0$ with a counterclockwise orientation.  It is not explicitly stated in \cite{ms2013ig4} that this ordering extends to a continuous path (after conformally mapping to a bounded domain, say).  However, it is actually an immediate consequence of the continuity of chordal space-filling $\SLE_{\kappa'}(\rho_1';\rho_2')$ that this is the case because whenever the latter makes a clockwise loop and then fills it up, the conditional law of the curve while it is filling the 
loop visits the points inside of the loop according to the aforementioned ordering.  This defines a counterclockwise space-filling $\SLE_{\kappa'}$ loop.  We similarly define a clockwise space-filling $\SLE_{\kappa'}$ loop by viewing $\partial \h$ as a flow line loop which starts and ends at $0$ with a clockwise orientation.

\subsubsection{Local finiteness of $\BCLE$}

We are now ready to state and derive some consequences of the properties of space-filling SLEs for local finiteness of BCLEs. The following lemma describes the ``relative position'' of such a space-filling SLE with respect to a BCLE.

\begin{figure}[ht!]
\begin{center}
\includegraphics[scale=0.85, page=3]{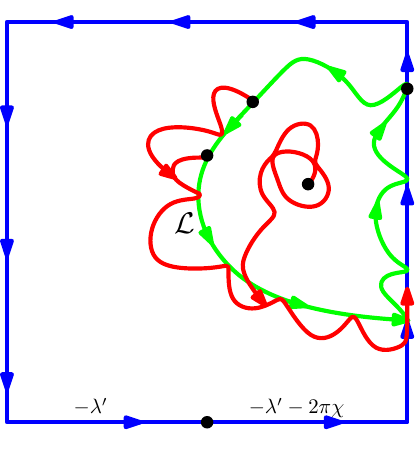}
\end{center}
\caption{\label{fig:bcle_space_filling_sle_interaction} Shown is a single loop $\CL$ of a $\ccwBCLE_\kappa(\rho)$, say $\Lambda$, coupled with a GFF~$h$.  The boundary data for $h$ and the boundary is oriented as in Figure~\ref{fig:space_filling_sle} so as to be compatible with a coupling with a counterclockwise space-filling $\SLE_{\kappa'}$ loop $\eta'$.  Shown are three flow lines of angle $-\pi/2$, one starting from inside of $\CL$ and two others starting along $\CL$ (more precisely, at points in the countable dense set used to define $\eta'$ which are very close to $\CL$).  As $\Lambda$ is formed by flow lines of angle $c/\chi$ for $c = \lambda' + \lambda(1+\rho) + 2\pi \chi$, which is always at least $3\pi/2$ for $\rho \in (-2,\kappa-4)$ (i.e., point to the left of flow lines with angle $-\pi/2$), it follows from \cite[Theorem~1.7]{ms2013ig4} that the former will cross $\CL$ upon intersecting and then merge into the domain boundary and the latter will stay to the right of $\CL$.  Consequently, $\eta'$ 
will visit the points of $\CL$ in chronological order and the start and end points for each excursion it makes from $\CL$ are equal.  Therefore $\CL$ (and all of $\Lambda$) can be generated from $\eta'$ by excising the intervals of time in which $\eta'$ spends in the loops of $\Lambda$.  This implies that there exists a continuous path whose range is equal to the union of the loops in $\Lambda$ and therefore $\Lambda$ is almost surely locally finite.}
\end{figure}

\begin{lemma}
\label{lem:bcle_k_space_filling_sle_interaction}
Suppose that $h$ is a GFF on $\D$ with the boundary conditions so that it can be coupled with a counterclockwise space-filling $\SLE_{\kappa'}$ loop from $-i$ to $-i$ as in Figure~\ref{fig:bcle_space_filling_sle_interaction} and let $\eta'$ be the associated space-filling $\SLE_{\kappa'}$ loop from $-i$ back to $-i$.  Let $c = \lambda' + \lambda(1+\rho) + 2\pi \chi$ and let $\Lambda$ be the $\ccwBCLE_\kappa(\rho)$ coupled with $h$ as being the loop ensemble formed by a boundary branching flow line of $h+c$.  Let $I = \cup_j I_j$ be the disjoint union of open intervals of times $t$ in which $\eta'(t) \notin \Lambda$ and, for each $j$, write $I_j = (a_j,b_j)$.  Then we almost surely have that $\eta'(a_j) = \eta'(b_j)$ for all~$j$.
\end{lemma}

Note that the analogous statement holds if we replace $\ccwBCLE_\kappa(\rho)$ with $\cwBCLE_\kappa(\rho)$ and the counterclockwise space-filling $\SLE_{\kappa'}$ loop with a clockwise space-filling $\SLE_{\kappa'}$ loop.

This lemma implies the following: 
\begin{proposition}
\label{prop:locfinitek}
The collection of loops of a $\ccwBCLE_\kappa(\rho)$ (for $\kappa \in (2, 4]$ and admissible $\rho$) is locally finite (if defined in the unit disk, then for all $\eps$, there are only finitely many loops of diameter greater than $\eps$). 
\end{proposition}
\begin{proof}
Let $I = \cup_j I_j = \cup_j (a_j,b_j)$ denote the open set of times at which $\eta'$ is in the parts of $\D$ which are surrounded by a loop of $\Lambda$.  By Lemma~\ref{lem:bcle_k_space_filling_sle_interaction}, we almost surely have that $\eta'(a_j) = \eta'(b_j)$ for all $j$.  This implies that the path $\wt{\eta}'$ which is taken to be equal to $\eta'$ on $[0,\infty) \setminus I$ and on each interval $(a_j,b_j)$ is taken to be equal to $\eta'(a_j) = \eta'(b_j)$ is almost surely continuous.  Since the complement of the range of $\wt{\eta}'$ is equal to the set of points in $D$ which are surrounded by a loop in $\Lambda$, the desired local finiteness follows from the continuity of $\wt{\eta}'$.
\end{proof}

\begin{proof}[Proof of Lemma~\ref{lem:bcle_k_space_filling_sle_interaction}]
By applying a conformal transformation, we may assume that we are working on $\h$.  Suppose that $h$ is a GFF on $\h$ with boundary conditions given by $-\lambda'$ on $\R_-$ and $-\lambda' - 2\pi \chi$ on $\R_+$ and let $\eta'$ be the counterclockwise space-filling $\SLE_{\kappa'}$ loop associated with~$h$ from~$0$ to~$0$.  Let $c = \lambda' + \lambda(1+\rho) + 2\pi \chi$ and let $\Lambda$ be the $\ccwBCLE_\kappa(\rho)$ which is generated by considering the boundary branching flow line of $h+c$ targeted at every point of $\partial \h$.  As $c/\chi > 3\pi/2$ for all $\rho \in (-2,\kappa-4)$, it follows from \cite[Theorem~1.7]{ms2013ig4} that the flow lines which generate the left boundary of $\eta'$ always point to the right of the flow lines which generate $\Lambda$.  Consequently, as illustrated in Figure~\ref{fig:bcle_space_filling_sle_interaction}, it follows that $\eta'$ visits the points of a loop $\CL$ of $\Lambda$ in chronological order, which proves the proposition.
\end{proof}

We now derive the corresponding results for the $\cwBCLE_{\kappa'} (\rho')$ for $\kappa' \in (4,8)$.
\begin{proposition}
\label{prop:locfinitekprime}
The collection of loops of a $\cwBCLE_{\kappa'}(\rho')$ for $\kappa' \in (4, 8)$ and admissible $\rho'$ is locally finite
\end{proposition}
This is a consequence as before of the following lemma. 
\begin{lemma}
\label{lem:bcle_kp_space_filling_sle_interaction}
Suppose that $h$ is a GFF on $\D$ with the boundary conditions so that it can be coupled with a counterclockwise space-filling $\SLE_{\kappa'}$ loop $\eta'$ from $-i$ to $-i$ as in Figure~\ref{fig:bcle_space_filling_sle_interaction}.  Let $c = \lambda'(2+\rho)$ and let $\Lambda'$ be the $\cwBCLE_{\kappa'}(\rho')$ process associated with $h+c$ using a boundary branching counterflow line starting from $-i$.  Let $I = \cup_j I_j$ be the disjoint union of open intervals of times~$t$ in which $\eta'(t) \notin \Lambda$ and, for each $j$, write $I_j = (a_j,b_j)$.  Then we almost surely have that $\eta'(a_j) = \eta'(b_j)$ for all~$j$.
\end{lemma}
Again, the analogous statement holds if we replace $\cwBCLE_{\kappa'}(\rho')$ with $\ccwBCLE_{\kappa'}(\rho')$ and the counterclockwise space-filling $\SLE_{\kappa'}$ loop with a clockwise space-filling $\SLE_{\kappa'}$ loop.
\begin{proof}
This follows from essentially the same argument used to prove Lemma~\ref{lem:bcle_k_space_filling_sle_interaction}.
\end{proof}

\section{Conclusion of the proofs of Theorem~\ref{thm:duality1} and Theorem~\ref{thm:duality2}}
\label{sec:proofs}
\label {Sec9}

The proofs of Theorem~\ref{thm:duality1} and Theorem~\ref{thm:duality2} have many similarities.  We will first describe in detail the various steps in the derivation of Theorem~\ref{thm:duality1}, and we then explain what minor differences one has to implement in order to derive Theorem~\ref{thm:duality2}.

\subsection{Locality of the boundary-loop-tracing process}
Consider a $\cwBCLE_{\kappa}(\rho)$ process $\Lambda$ in the unit disk $\D$, traced by a boundary branching $\SLE_{\kappa}(\rho)$ beginning at $-i$.  Assume that the $\bSLE_{\kappa}(\rho)$ processes used to trace $\Lambda$ are all coupled with an instance $h$ of the GFF as described in Section~\ref{subsubsec:bcle_gff}: If $\varphi \colon \D \to \h$ is a conformal transformation which takes $-i$ to $0$ and $i$ to $\infty$ then $h = \wt{h} \circ \varphi - \chi \arg \varphi'$ where $\wt{h}$ is a GFF on $\h$ with boundary conditions given by $-\lambda(1+\rho)$ on $\R_-$ and $-\lambda(1+\rho)-2\pi \chi$ on~$\R_+$ (recall Table~\ref{tab:sle_bcle_boundary_conditions}).  Then, we note that the set of points belonging to the loops of $\Lambda$ that intersect a given boundary arc of $\D$ {containing $-i$} is a local set of $h$, because this set can be described by a boundary branching $\SLE_{\kappa}(\rho)$ starting from $-i$ that only targets a dense set of points in this arc (instead of a dense set of points on all of~$\partial \D$). For every loop $\CL \in \Lambda$, let us define $\theta(\CL)$ to be the first point at which one encounters the loop $\CL$ as one traces the boundary $\partial \D$ counterclockwise starting from $-i$ (note that $\theta(\CL)$ determines an ordering of the loops of $\Lambda$).

Let $\eta_\Lambda$ be the single loop that traces through all of the loops of $\Lambda$ in the order described above, with each individual loop of $\Lambda$ being traversed clockwise.  To be more explicit, if we are given any finite collection $\{\CL_1, \CL_2, \ldots, \CL_k \}$, one can define a path that traverses $\partial \D$ in counterclockwise order except that each time it first hits one of the $\CL_i$ it traverses that entire loop clockwise before continuing.  By the local finiteness of $\cwBCLE_\kappa(\rho)$ established in Proposition~\ref{prop:locfinitek}, one can then construct $\eta_\Lambda$ as a uniform limit of (appropriate parameterizations of) the paths defined this way. This ensures that almost surely, $\eta_\Lambda$ indeed traces a continuous path.

We are then interested in the ``dynamic'' locality property of this path:  
\begin{lemma}
Suppose that $\tau$ is any stopping time for the filtration generated by the path $\eta_\Lambda$.  Then $\eta_\Lambda([0,\tau])$ is a local set for $h$.
\end{lemma}
\begin{proof}
We fix an open set $U \subseteq \D$.  On the event that $\eta_\Lambda|_{[0,\tau]}$ does not hit~$U$, we can determine $\eta_\Lambda|_{[0,\tau]}$ from the boundary branching flow line of~$h$ started from $-i$ and targeted at every point on $\partial \D$, stopped upon hitting~$U$.  Therefore the lemma follows from Proposition~\ref{prop:local_set_char}.
\end{proof}
We note that we can determine the boundary conditions for the conditional law of~$h$ given $\eta_\Lambda|_{[0,\tau]}$ using \cite[Proposition~3.8]{ms2012ig1} to make a comparison to the conditional law of $h$ given the flow lines of~$h$ used to generate $\eta_\Lambda|_{[0,\tau]}$.  In particular, if $U$ is a component of $\D \setminus \eta_\Lambda([0,\tau])$ and $x \in \partial \D$ is such that the boundary branching flow line of $h$ starting from $-i$ and targeted at $x$ agrees with $\partial U$ along a boundary segment $L$, then \cite[Proposition~3.8]{ms2012ig1} implies that the boundary conditions for the restriction of $h$ to $U$ given $\eta_\Lambda|_{[0,\tau]}$ along $L$ are the same as the boundary conditions for the conditional law of $h$ given the flow line branch targeted at $x$ along~$L$.

\subsection{Locality of the nested boundary-loop-tracing process}

We now iterate the BCLE construction, and consider $\Gamma$ and $\eta=\eta_\Gamma$ as described in Theorem~\ref{thm:duality1}.  Let us first note that the same argument as for $\eta_\Lambda$ can be applied to deduce that the nested loop-tracer $\eta$ is a continuous path. Let us now describe what GFF boundary conditions this nested loop-tracer corresponds to (when one looks at the whole path at once).  

We note that $\Lambda$ and $\Gamma$ are coupled with a GFF as follows.  We will assume that we are working on~$\h$ because this is the setting in which it is easiest to specify the boundary conditions.  Specifically, we suppose that~$h$ is a GFF on~$\h$ with boundary conditions given by $-\lambda(1+\rho)$ on~$\R_-$ and $-\lambda(1+\rho)-2\pi \chi$ on $\R_+$ (recall Table~\ref{tab:sle_bcle_boundary_conditions}).  These are the same boundary conditions as considered in the previous section for generating $\eta_\Lambda$.  We note that each component $V$ of $\h \setminus \eta_\Lambda$ has a marked point corresponding to where $\eta_\Lambda$ first visits $\partial V$.  Let $\varphi \colon V \to \h$ be a conformal map which sends this special point to $0$ and any other distinct point on $\partial V$ to $\infty$.

We now consider two cases depending on whether $V$ is to the left or to the right of $\eta_\Lambda$ (i.e., whether~$V$ is surrounded by a clockwise or counterclockwise loop of~$\Lambda$). When $V$ is to the right of $\eta_\Lambda$  then we know that $\wt{h} = h \circ \varphi^{-1} - \chi \arg (\varphi^{-1})'$ is a GFF on $\h$ with boundary conditions given by $\lambda$ on $\R_-$ and $\lambda - 2\pi \chi$ on $\R_+$ (recall Table~\ref{tab:bcle_inside_loop_bc} and Figure~\ref{fig:bcle_boundary_data}).  Let
\begin{equation}
\label{eqn:c_r_def}
c_R = -\lambda'(1+\rho_R') + 2\pi \chi - \lambda.
\end{equation}
Then the boundary conditions for $\wt{h} + c_R$ are given by $-\lambda'(1+\rho_R')+ 2\pi \chi$ on~$\R_-$ and $-\lambda'(1+\rho_R')$ on~$\R_+$.  These are the boundary conditions for a $\ccwBCLE_{\kappa'}(\rho_R')$ process (recall Table~\ref{tab:sle_bcle_boundary_conditions}).  So, we take the $\ccwBCLE_{\kappa'}(\rho_R')$ inside of $V$ to be given by the image under $\varphi^{-1}$ of the $\ccwBCLE_{\kappa'}(\rho_R')$ generated by the boundary branching counterflow line of $\wt{h} + c_R$ starting from~$0$ and targeted at every point of~$\partial \h$.

When $V$ is to the left of $\eta_\Lambda$, then we know that $\wt{h} = h \circ \varphi^{-1} - \chi \arg (\varphi^{-1})'$ is a GFF on $\h$ with boundary conditions given by $-\lambda+2\pi \chi$ on $\R_-$ and $-\lambda$ on $\R_+$ (recall Table~\ref{tab:bcle_inside_loop_bc} and Figure~\ref{fig:bcle_boundary_data}).  Let
\begin{equation}
\label{eqn:c_l_def}
c_L = \lambda'(1+\rho_L') - 2\pi \chi + \lambda.
\end{equation}
Then the boundary conditions for $\wt{h} + c_L$ are given by $\lambda'(1+\rho_L')$ on $\R_-$ and $\lambda'(1+\rho_L')-2\pi \chi$ on $\R_+$.  These are the boundary conditions for $\cwBCLE_{\kappa'}(\rho_L')$ (recall Table~\ref{tab:sle_bcle_boundary_conditions}).  So, we take the $\cwBCLE_{\kappa'}(\rho_L')$ inside of $V$ to be given by the image under $\varphi^{-1}$ of the $\cwBCLE_{\kappa'}(\rho_L')$ generated by $\wt{h} + c_L$.

Let $\eta'$ be the path which is given by following along the $\cwBCLE_{\kappa'}(\rho_L')$ (resp.\ $\ccwBCLE_{\kappa'}(\rho_R')$) loops described just above starting from the first point visited by $\eta_\Lambda$ and with a counterclockwise (resp.\ clockwise) orientation.

We now study the iterated loop-tracing path $\eta'$ when stopped at a given stopping time. 

\begin{proposition}
\label{prop:localityofnestedlooptracer}
Suppose that $\tau$ is a stopping time for the filtration generated by the iterated loop-tracing path $\eta'$.  Then $\eta'([0,\tau])$ is a local set for $h$.
\end{proposition}
\begin{proof}
Fix an open set $U \subseteq \h$.  Let $\tau_U$ (resp.\ $\tau_{\Lambda,U}$) be the first time that $\eta'$ (resp.\ $\eta_\Lambda$) hits $U$.  By Proposition~\ref{prop:local_set_char}, it suffices to show that the event that $\tau \leq \tau_U$ is determined by the projection of $h$ onto the subspace of functions which are harmonic in $U$.  In order to show that this is the case, we will consider the following alternative method of generating $\eta'$ (see Figure~\ref{fig:commutation_arg} for an illustration). 
\begin{enumerate}
\item[Step 1:] Generate $\eta_\Lambda|_{[0,\tau_{\Lambda,U}]}$ from $h$.
\item[Step 2:] Generate the $\cwBCLE_{\kappa'}(\rho_L')$'s and $\ccwBCLE_{\kappa'}(\rho_R')$'s associated with the components which are completely surrounded by $\eta_\Lambda|_{[0,\tau_{\Lambda,U}]}$ as branching counterflow lines stopped at the first time they hit $U$ from the GFF given by $h$ conditioned on the result of the previous step.
\item[Step 3:] In each component $W$ not completely surrounded by $\eta_\Lambda|_{[0,\tau_{\Lambda,U}]}$, we draw a branching counterflow line of the field plus $c_L$ (resp.\ $c_R$) conditioned on the previous steps starting from the point on $\partial W$ first drawn by the left (resp.\ right) side of $\eta_\Lambda|_{[0,\tau_{\Lambda,U}]}$ targeted at every point on $\partial W$ which is drawn by the left (resp.\ right) side of $\eta_\Lambda|_{[0,\tau_{\Lambda,U}]}$, with each counterflow line branch stopped upon first hitting $U$.
\item[Step 4:] Generate the \emph{rest} of the stopped counterflow lines from the previous step (i.e., until they have reached their target point).
\item[Step 5:] Generate the \emph{rest} of $\eta_\Lambda$ from the GFF given by $h$ conditioned on the previous steps.
\item[Step 6:] Generate the \emph{rest} of the $\cwBCLE_{\kappa'}(\rho_L')$'s and $\ccwBCLE_{\kappa'}(\rho_R')$'s as branching counterflow lines of~$h$ conditioned on the previous steps.
\end{enumerate}

\begin{figure}
\includegraphics[width=2in]{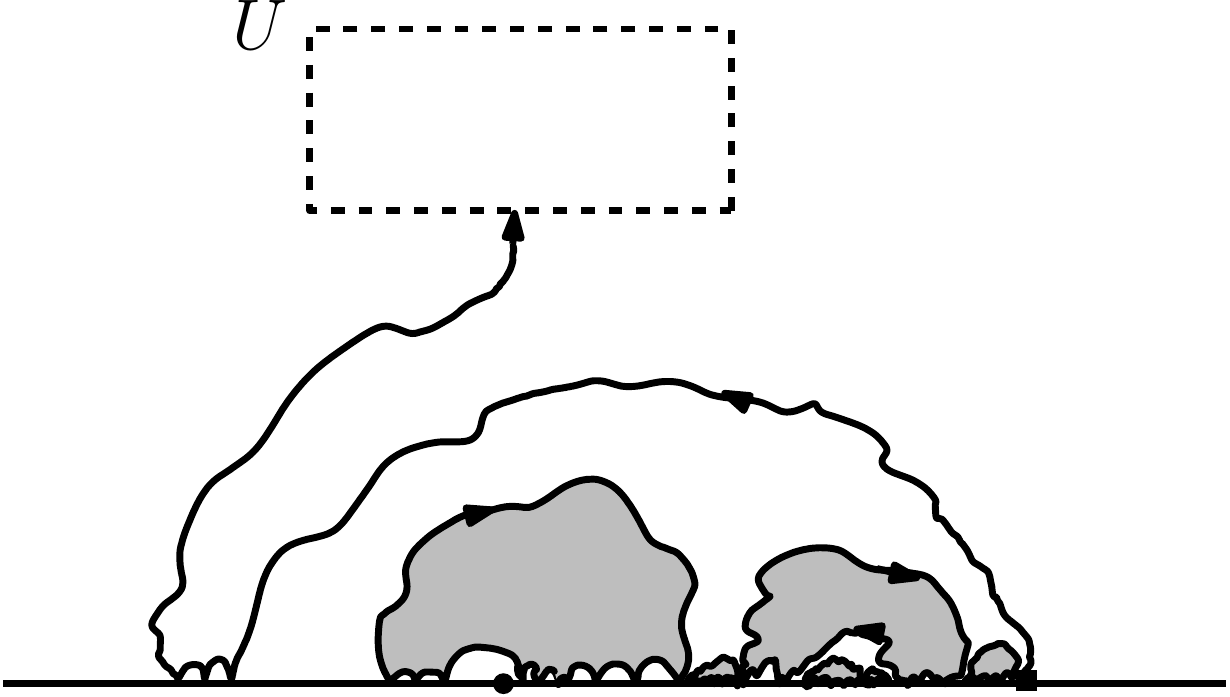}
\includegraphics[width=2in]{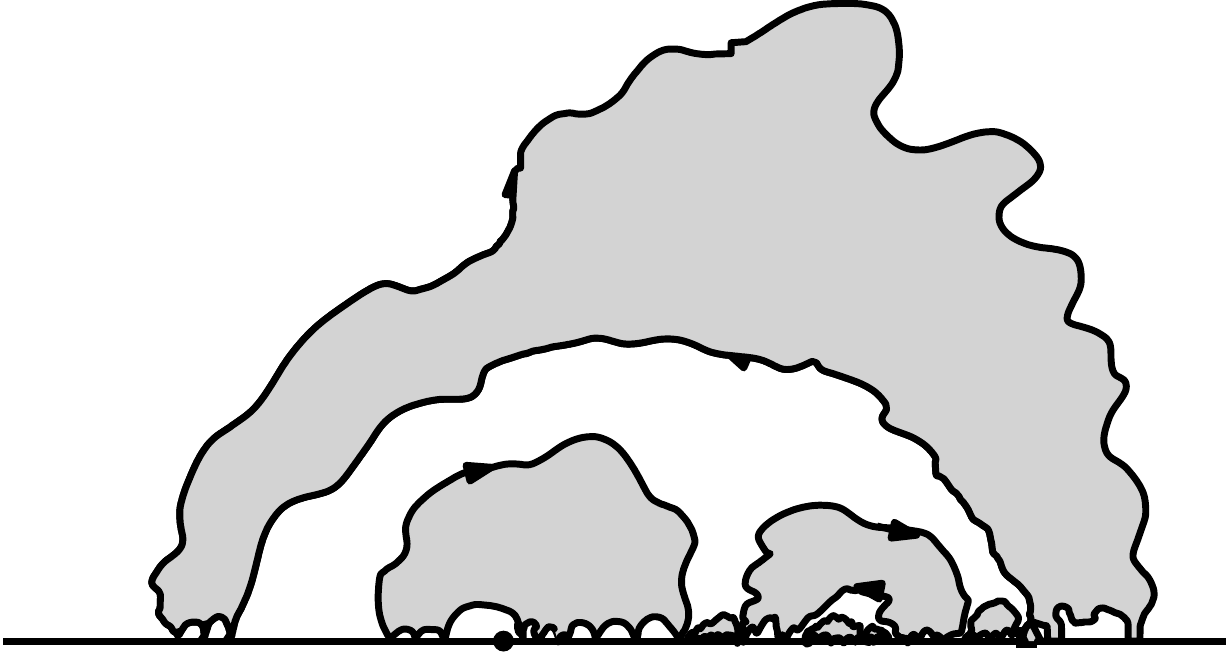}
\includegraphics[width=2in]{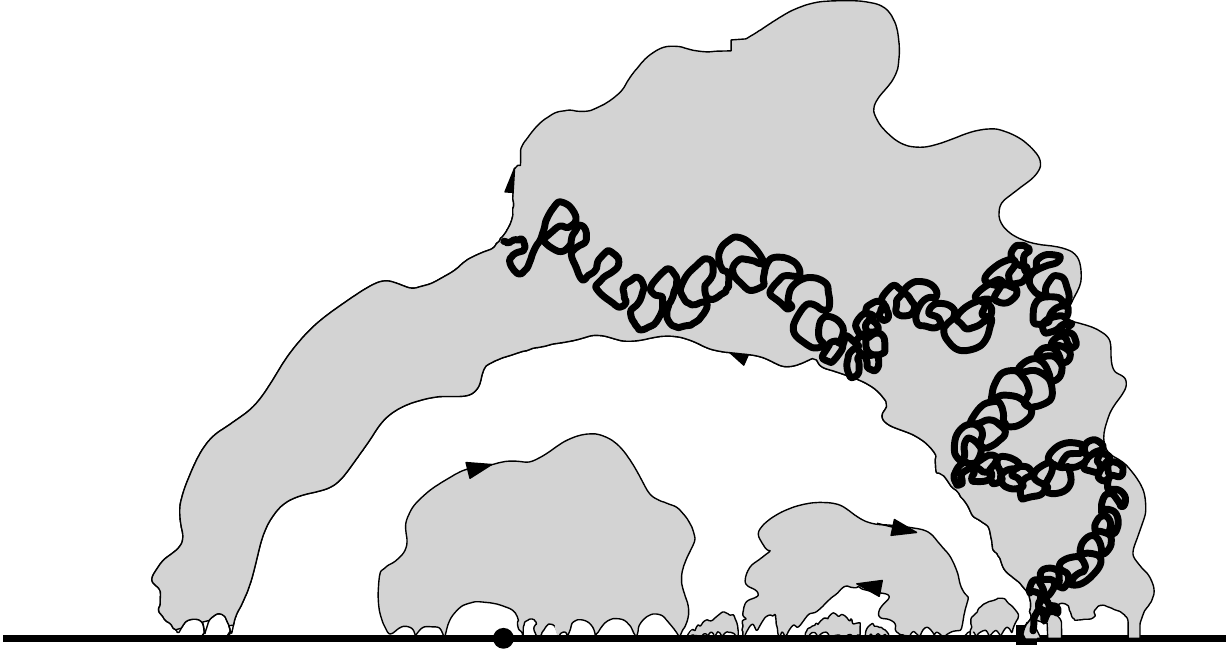}
\caption{\label{fig:commutation_arg} Illustration of the commutation argument used in the proof of Theorem~\ref{thm:duality1}. {\bf Left:} $\BCLE_\kappa(\rho)$ exploration path $\eta_\Lambda$ up until hitting $U$.  Regions surrounded clockwise in grey are ``true loops'' and regions surrounded counterclockwise  are ``false loops.''  {\bf Middle:} Path until finishing loop it is currently drawing.  {\bf Right:} Left/right boundaries of counterflow line to marked point.}
\end{figure}
The difference between this method of generating loops using $h$ from that used in the original definition of $\eta'$ is that we have (partially) generated some of the $\cwBCLE_{\kappa'}(\rho_L')$'s and $\ccwBCLE_{\kappa'}(\rho_R')$'s \emph{before} seeing the entire realization of $\eta_\Lambda$.  One could therefore worry that the loops thus formed are \emph{not} the same as those if we \emph{had} seen the entire realization of~$\eta_\Lambda$.  We will however now explain that this is not the case (i.e., that this method of generating loops does in fact generate the same family as in the definition of~$\eta'$).  This claim indeed implies the proposition because the first three steps in this method of generating~$\eta'$ fully determine~$\eta'$ up until hitting $U$ and do not require the observation of the values of~$h$ on~$U$ itself.  In particular, this implies that the event $\tau \leq \tau_U$ is determined by the projection of~$h$ onto the functions which are harmonic in~$U$.

First, we note that it is obvious that both methods of generating loops produce the same result inside of those components which are completely surrounded by $\eta_\Lambda|_{[0,\tau_{\Lambda,U}]}$ (corresponding to Step~2 above).  Indeed, the conditional law of $h$ given $\eta_\Lambda|_{[0,\tau_{\Lambda,U}]}$ restricted to such a component is the same as the corresponding conditional law given all of~$\eta_\Lambda$.  Moreover, conditionally on $\eta_\Lambda|_{[0,\tau_{\Lambda,U}]}$, the restriction of $h$ to such a component is independent of the restriction of $h$ to the other components of $\h \setminus \eta_\Lambda([0,\tau_{\Lambda,U}])$.

It therefore remains to show that the joint law of the loops generated in Step~3 and Step~4 above given the previous steps is the \emph{same} as the joint law which results by first generating the rest of $\eta_\Lambda$ and then sampling the $\cwBCLE_{\kappa'}(\rho_L')$'s and $\ccwBCLE_{\kappa'}(\rho_R')$'s from $h$ given all of $\eta_\Lambda$.  That is, we must show that the operation of drawing the rest of $\eta_\Lambda$ commutes with Step~3 and Step~4 above.  This type of commutation between flow and counterflow lines was already explained in Section~\ref{subsubsec:ig_interaction_rules} (and indeed follows since flow and counterflow lines are locally determined by the GFF).
\end{proof}

\subsection{Identification of the law of $\eta'$}

Assume that we have chosen $\rho_L',\rho_R'$ as in~\eqref{eqn:rho_p_equations}.  That is, we let
\[ \rho_L' = \frac{\kappa'}{4} \rho + \kappa'-4 \quad\text{and}\quad \rho_R' = -\frac{\kappa'}{4}(\rho+2)\]
and note that $\rho_L'  +\rho_R' = \kappa'/2-4$.  These choices ensure that $-\lambda(1+\rho) + c_L = \lambda'$ and $-\lambda(1+\rho) -2\pi \chi + c_R = -\lambda'$.  Let us first fix a boundary point which is distinct from the starting point of the path and compute the law of the path $\wt{\eta}'$ which is given by parameterizing $\eta'$ by capacity (i.e., targeted at) as seen from this marked point.  By applying a conformal mapping to $\h$, we may assume that $\wt{\eta}'$ starts from $0$ and that this marked point is equal to $\infty$.  Suppose that $\tau$ is a stopping time for the filtration generated by $\wt{\eta}'$.  We now want to describe the conditional law of $h$ given $\wt{\eta}'|_{[0,\tau]}$ in the unbounded component of $\h \setminus \wt{\eta}' ([0,\tau])$ (see Figure~\ref{fig:cond_law_bd_fig} for an illustration).  As $\wt{\eta}'([0,\tau])$ is a local set for $h$, we know that the conditional law of $h$ given $\wt{\eta}'|_{[0,\tau]}$ is that of a GFF on $\h \setminus \wt{\eta}' ([0,\tau])$ with boundary conditions that can be determined\footnote{In this particular setting, one can apply \cite[Proposition~3.8]{ms2012ig1} to the family of local sets $A_n$ defined as follows.  Given $\wt{\eta}'|_{[0,\tau]}$, let $U_n$ be a decreasing sequence of open sets consisting of finite unions of balls with rational radii centered at points with rational coordinates and with $\cap_n U_n = \wt{\eta}'([0,\tau])$ and then take $A_n$ to be given by the union of $\eta_\Lambda$ stopped upon exiting $U_n$ together with the branching counterflow lines on its left and right sides used to generate the $\cwBCLE_{\kappa'}(\rho_L')$ and $\ccwBCLE_{\kappa'}(\rho_R')$, also stopped upon exiting $U_n$.} using \cite[Proposition~3.8]{ms2012ig1} as follows (here, our choice of $\rho_L',\rho_R'$ will ensure that one gets a $\bSLE_{\kappa'}$): Let $\varphi$ be a conformal transformation which takes this unbounded component to $\h$ fixing $\infty$ and with the tip $\wt{\eta}'(\tau)$ taken to $0$.  Let $X$ be the image of the point on $\eta_\Lambda$ most recently visited by $\wt{\eta}'$ before time $\tau$ under $\varphi$.

\begin{figure}[ht!]
\includegraphics[scale=0.85, page=3]{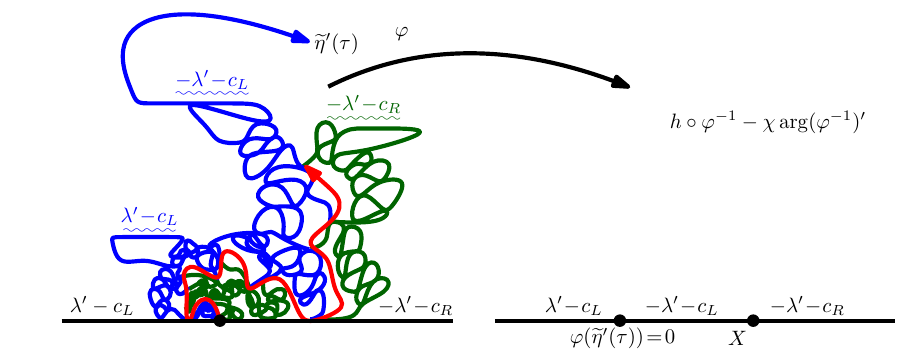}
\caption{\label{fig:cond_law_bd_fig} Illustration of the computation of the conditional law of $h$ given $\eta'$ up to time $\tau$.  Note that the values of $\rho,\rho_L',\rho_R'$ are chosen so that, after applying the conformal change of coordinates $\varphi$ which takes the unbounded component of $\h \setminus \eta'([0,\tau])$ to $\h$ fixing $\infty$ and with $\varphi(\wt{\eta}'(\tau)) = 0$, the boundary data for the field changes only at the image of the tip of the path and the most recently visited point on the trunk.  That is, there are no marked points corresponding to where the left and right sides of $\eta'$ have most recently hit $\R$.}	
\end{figure}

If $X < 0$, then $\wt{\eta}'$ is in the process of drawing a counterclockwise loop on the right side of its trunk and the continuation of this counterflow line is given by the image under $\varphi^{-1}$ of the counterflow line starting from $0$ of the GFF $\wt{h} = h \circ \varphi^{-1} - \chi \arg (\varphi^{-1})' + c_R$ on $\h$.  This GFF has boundary conditions given by:
\[ \lambda' + c_R - c_L =  -\lambda'(1+\rho_L'+\rho_R') + 4\pi \chi -2\lambda \quad\text{on}\quad (-\infty,X),\quad {\lambda'} \quad\text{on}\quad [X,0),\quad\text{and}\quad -\lambda' \quad\text{on}\quad \R_+.\]
As $\rho_L'  +\rho_R' = \kappa'/2-4$ so that the boundary data on $(-\infty,X)$ is equal to $-\lambda'+2\pi \chi$, the martingale characterization \cite[Theorem~2.4]{ms2012ig1} implies that this counterflow line is evolving as an $\SLE_{\kappa'}(\kappa'-6)$ with a single force point located at the most recent point on $\eta_\Lambda$ visited by $\wt{\eta}'$ before time $\tau$.  The very same argument can be applied when $X > 0$, and we can then conclude $\wt{\eta}'$ evolves as an $\SLE_{\kappa'}(\kappa'-6)$ process when it is not hitting the trunk~$\eta_{\Lambda}$ (with marked point located at the last point visited on~$\eta_\Lambda$).

Let us now consider the path $\wt{\eta}'$ when parameterized by capacity seen from infinity. This will in particular erase all its loops that are ``hidden from infinity'' when they are made (this is exactly like looking at the usual $\bSLE_{\kappa'}^\beta$ instead of at the full $\bSLE_{\kappa'}^\beta$).  It is easy to see that this curve has a continuous Loewner driving function because it is easy to see that the corresponding family of hulls grows continuously (see, e.g., \cite{LAW05}).  Moreover, by Proposition~\ref{prop:localityofnestedlooptracer}, it evolves as an $\SLE_{\kappa'} ( \kappa'- 6)$ when it is ``tracing a loop'', its law is invariant under scaling (because the whole construction is invariant under scaling), and satisfies the conformal Markov property (because the path is determined by the field and we have determined the conditional law of the field given the path up to a stopping time just above).  Therefore Lemma~\ref{lem:slekrcharacterization} implies that it evolves as an $\SLE_{\kappa'}^\beta (\kappa'-6)$ process for all times.   We note that the same argument also applies for interior points $z$ to get that the law of $\eta'$ reparameterized by capacity as seen from $z$ is that of an $\SLE_{\kappa'}^\beta(\kappa'-6)$ up until the first time that $z$ has been surrounded by a loop.  Indeed, recall that $\SLE_{\kappa'}^\beta(\kappa'-6)$ is target invariant.  Consequently, to show that $\eta'$ as seen from an interior point $z$ has the law of an $\SLE_{\kappa'}^\beta(\kappa'-6)$ process, it suffices to show that $\eta'$ as seen from a boundary point has this property as this can be iterated by choosing successive boundary points in a measurable manner.  That this is the case follows from the argument given just above.  Note also that from the construction we have that if $z$ and $w$ are distinct points, then the branch of $\eta'$ targeted at $z$ is the same as the branch of $\eta'$ targeted at $w$ up until the two points are disconnected, after which the two branches evolve independently.

We also note that the value of $\beta$ has to be the same for all target points (because the corresponding paths $\wt{\eta}'$ are identical at small times). We are now getting closer to identifying the distribution of the path $\eta_\infty'$, as defined in Theorem~\ref{thm:duality1}, as a full $\bSLE_{\kappa'}^\beta$, but we are not quite there yet, because we need to also describe the behavior of $\eta'$ within the ``pockets.''

Recall that the nested loop tracing path $\eta'$ constructed just above has the property that if we fix any point $z \in D$, then the law of $\eta'$ parameterized by capacity from $z'$ evolves as an $\SLE_{\kappa'}^\beta(\kappa'-6)$ targeted at $z$.  Moreover, it follows from the construction that if we consider the path targeted at two different points, once these two points are first separated, the two continuations of the targeted paths become conditionally independent.  It is easy to see that these properties characterize the law of $\eta'$ and identifies $\eta_\infty'$ as a full $\bSLE_{\kappa'}^\beta$.

\subsection{Proof of Theorem~\ref{thm:duality1}}
We have shown so far that for each admissible triple $\rho$, $\rho_L'$, $\rho_R'$ there exists $\beta \in [-1,1]$ such that the path $\eta_y'$ is a full $\bSLE_{\kappa'}^\beta$ process. 
It is immediate from the construction that this map that to each admissible triple the corresponding $\beta$ is injective (because the law of the 
trunk of a $\bSLE_{\kappa'}^\beta$ has to be unique). It therefore remains to show that the map from admissible triples to $\beta$ is surjective onto $[-1,1]$.  This, in turn, follows from the continuity argument described at the end of the proof of Proposition~\ref{prop:CME_characterization2} and the fact that $\beta=1$ and $\beta=-1$ are reached for the extremal values of $\rho$.
\qed

\subsection{Proof of Theorem~\ref{thm:duality2}}
We now turn our attention to the case of $\BCLE_{\kappa}$ loops attached to $\BCLE_{\kappa'}$. The discussion is essentially the same, and can be to a large extent copied and pasted from the  proof of Theorem~\ref{thm:duality1}. Let $\eta$, $\Lambda$ and $\Gamma$ now be as in Theorem~\ref{thm:duality2}. First of all, the very same arguments as above using the local finiteness of the BCLEs show that the iterated loop-tracer $\eta$ is a continuous path. 

\begin{figure}
\begin{center}
\subfigure{\includegraphics[width=2in]{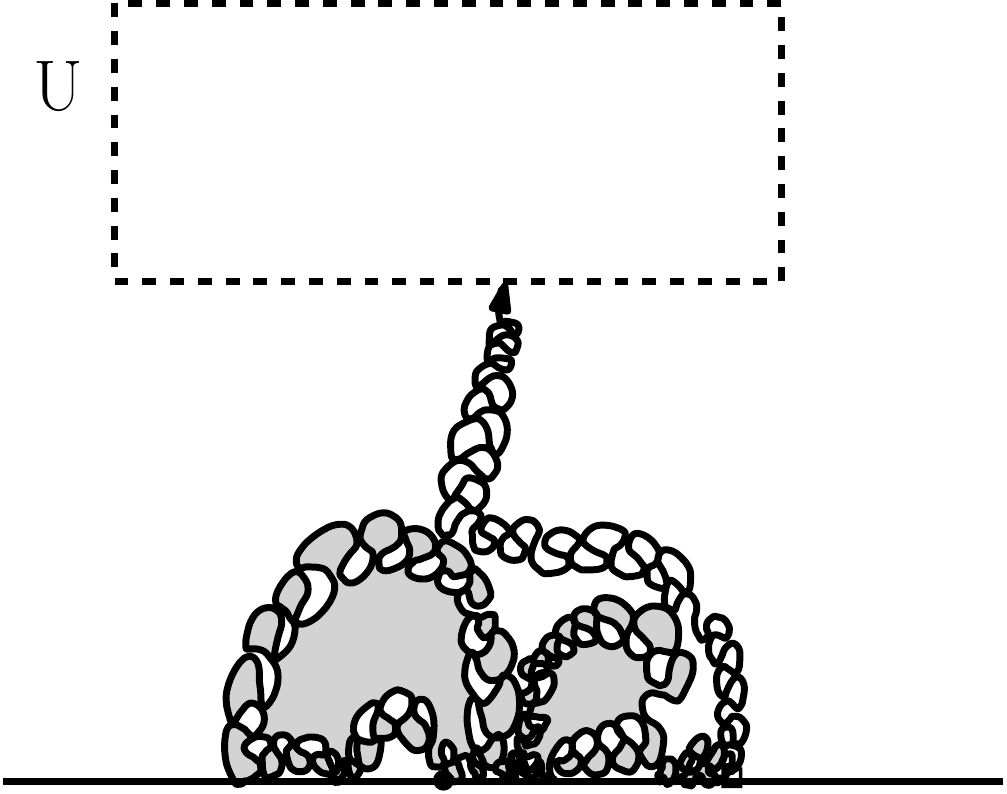}}
\subfigure{\includegraphics[width=2in]{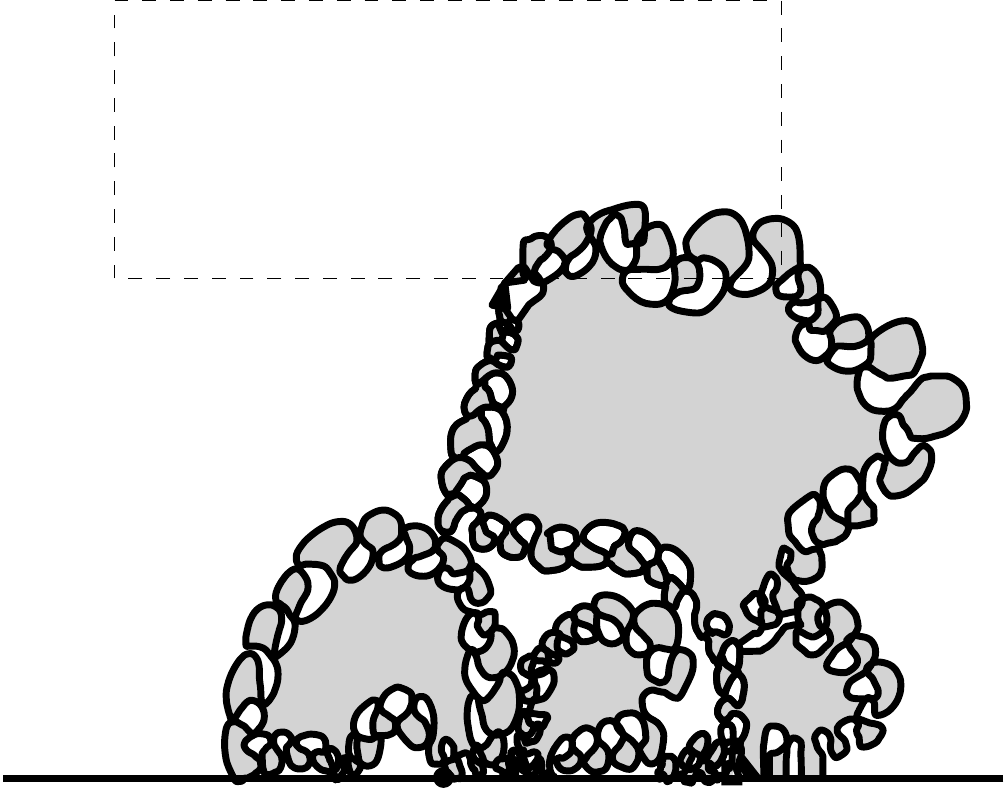}}
\subfigure{\includegraphics[width=2in]{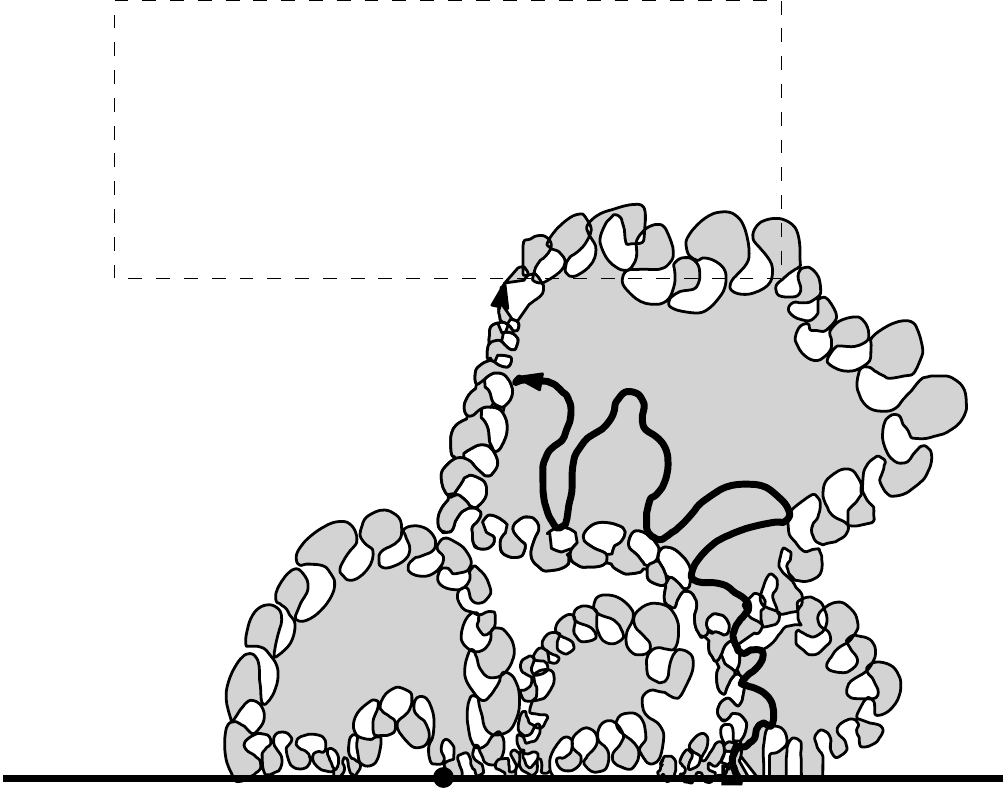}}
\end{center}
\caption{\label{fig:duality2} Illustration of the commutation argument used in the proof of Theorem~\ref{thm:duality2}.  {\bf Left:} $\cwBCLE_{\kappa'}(\rho')$ exploration path $\eta_\Lambda$ drawn up to hitting given open set $U$.  {\bf Middle:}  The continuation of $\eta_\Lambda$ up until the current loop being drawn is finished, i.e., up to its most recent intersection of $\R_+$.  {\bf Right:} Flow line starting from this point }
\end{figure}

The proof of the fact that $\eta([0,\tau])$ is a local set for $h$ when $\tau$ is a stopping time does also follow the same line as above, with some minor difference only in the formulas for the boundary values of the field and the commutation argument (see Figure~\ref{fig:duality2}).

Identifying the law of $\eta_y$ follows again the same lines as before, and it is actually a simpler task than in the previous case, because one has to show that it is a $\bSLE_\kappa^\beta$ process, and the issues of ``pockets''  and full $\bSLE_{\kappa'}^\beta$ versus $\bSLE_{\kappa'}^\beta$ do not arise. \qed

\section{Variants and comments}
\label{sec:conclusion}
\label {Sec10}

\subsection{Generalized $\SLE_\kappa(\rho)$ processes}

The iteration scheme (sampling a flow line and well-chosen counterflow lines) that we  used in the previous section can be generalized to other values of $\rho$ and iterated in different ways.  In this section, we will illustrate this by first  explaining how one can use the same ideas to construct the totally asymmetric generalized  $\SLE_{\kappa}(\rho)$ processes (in the ``non-light cone regime'') and thereby derive some of their properties (continuity of the trace, and the fact that they can be viewed as a deterministic function of the GFF with which they are naturally coupled).

We stress that while it is possible to construct the couplings of these processes with  the GFF directly, by adapting the arguments of \cite{dubedat2009gff,ss2010continuumcontour,she2010weld,ms2012ig1} (even if there is a slight technicality to be dealt with because Bessel processes with dimension $\delta \in (0,1)$ are not semimartingales due to the principal value correction in the drift when the process is hitting $0$), deriving these additional properties is a priori not an easy task. 

An almost identical analysis, with a little extra work, can be applied to construct these $\SLE_\kappa^\beta (\rho)$ processes for the other values of $\beta$ (and prove similarly the continuity of the trace and properties of the coupling with the GFF). One then has, like in the iterated BCLE scheme, to start with an $\SLE_\kappa(\rho_L ; \rho_R)$ trunk and to attach loops on both of its sides.

As before, we will separately treat the cases $\kappa \le 4$ and $\kappa' > 4$ (even if this is actually not really needed in the present case, because the arguments are now exactly the same). The corresponding ranges of values for $\rho$ and $\rho'$ are 
\[ \rho \in (-2-(\kappa/2),(\kappa/2)-4) \quad\text{and}\quad \rho' \in (-2-(\kappa'/2),(\kappa'/2)-2).\] 
Recall that this does not quite cover all of the generalized $\SLE_{\kappa} (\rho)$ cases with $\rho < -2$. The missing light-cone regime where $\kappa \le 4$ and $-2 \ge \rho \ge \max ( -2 - \kappa /2, -4 + \kappa/2 )$ is addressed in \cite{ms2016lightcone}.

\subsubsection{Generalized totally asymmetric $\SLE_{\kappa'}(\rho')$ processes}
\label{subsubsec::ayssm_sle_kp}

Fix $\kappa' >4$ and $\rho' \in (-2-(\kappa'/2),-2)$.  (In this case we assume that $\rho' < -2$ because the continuity and coupling properties with the GFF of the $\SLE_{\kappa'}(\rho')$ processes with $\rho' > -2$ and $\beta=1$ or $\beta=-1$ have already been analyzed in \cite{ms2012ig1}.)   Define then  
\[ \rho  = - \frac{\kappa}{4}(\rho' +2 + ( \kappa' /2)) = - \frac{\kappa}{4}(\rho' +2) - 2.\]
Let $h$ be a GFF on $\h$ with boundary conditions given by $-\lambda$ on $\R_-$ and $\lambda(1+\rho)$ on $\R_+$ and let $\eta$ be the flow line of $h$ from $0$ to $\infty$.  Note that $\eta$ is an $\SLE_{\kappa}(\rho)$ process with force point at  $0^+$.  Since $\rho \in (-2,0)$ for $\rho' \in (-2-(\kappa'/2),-2)$, it follows that $\eta$ may or may not hit $\R_+$.  Suppose that we have fixed $\eta$ and let $U$ be a component of $\h \setminus \eta$ which is to the right of $\eta$ and let $\varphi \colon U \to \h$ be a conformal map which takes the first (resp.\ last) point on $\partial U$ visited by $\eta$ to $0$ (resp.\ $\infty$).  We then add
\begin{equation}
\label{eqn:c_r_kp_on_k_new_def2}
c = \lambda + \lambda'(1+\rho')   
\end{equation}
to the field and draw in the branching counterflow line starting from $0$ and targeted at every point on $\varphi(\partial U \cap \eta) = \varphi(\partial U \setminus \partial \h) = \R_-$.  We note that the counterflow line starting from $0$ and targeted at $\infty$ has the law of an $\SLE_{\kappa'}(\rho'+\kappa'/2)$.  Note that $\rho' + \kappa'/2 \in (-2,\kappa'/2-2)$ so that the counterflow line almost surely hits $\R_-$.  The same argument used to prove Proposition~\ref{prop:locfinitekprime} implies that there exists a continuous path $\gamma$ which follows the loops generated by the branching counterflow lines in the order in which they are visited by $\eta$.

As in the previous subsection, the key to determining the law of $\gamma$ is to describe its evolution when it is not hitting $\eta$.  This amounts to showing that $\gamma([0,\tau])$ is a local set for $h$ for each stopping time $\tau$ and then identifying the conditional law of $h$ given $\gamma|_{[0,\tau]}$.  The proof of the locality of $\gamma([0,\tau])$ is analogous to the proof of Proposition~\ref{prop:localityofnestedlooptracer}.  We will omit the details of this here, and instead focus on identifying the conditional law of $h$ given $\gamma|_{[0,\tau]}$.

Let $\varphi$ be a conformal transformation from the unbounded component of $\h \setminus \gamma([0,\tau])$ to $\h$ which fixes $\infty$ and takes $\gamma(\tau)$ to $0$.  Let $X$ be the image under~$\varphi$ of the point on~$\eta$ most recently hit by~$\gamma$.  Suppose that $X \neq 0$, i.e., that $X < 0$.  Then $\wt{h} = h \circ \varphi^{-1} - \chi \arg (\varphi^{-1})' + c$ has the law of a GFF on $\h$ with boundary conditions given by (and this is why we chose the previous values for $\rho$ and $c$):
\[ \lambda'  (1 + \rho')  \quad\text{on}\quad (-\infty,X),\quad \lambda' \quad\text{on}\quad [X,0), \quad\text{and}\quad -\lambda' \quad\text{on}\quad \R_+.\]
Combining this with the general characterization of \cite[Theorem~2.4]{ms2012ig1} implies that $\gamma$ is evolving as an $\SLE_{\kappa'}(\rho')$ process with a single force point at the most recent point visited by $\eta'$ before time $\tau$.  Hence, using the characterization of the generalized $\SLE_\kappa (\rho)$ processes and the scale-invariance of the law of $\gamma$, we can conclude that:
\begin {itemize}
 \item The path $\gamma$ is an $\SLE_{\kappa'}^{\beta}(\rho')$ process for~$\beta=1$, as the loops it makes are always to the right of the trunk. 
\item It is almost surely a continuous curve.
\item It is a deterministic function of the GFF $h$.
\item The joint law of $(h,\gamma)$ is characterized by the form of the conditional law of $h$ given $\gamma|_{[0,\tau]}$ for each $\gamma$-stopping time $\tau$.
\end {itemize}
We note that the difference between the boundary data to the left and to the right of $0$ for the GFF used to construct $\gamma$ is given by $\lambda'(2+\rho')$, which is the same difference as in the case of $\SLE_{\kappa'}(\rho')$ for $\rho' > -2$.  The construction in the case that $\beta = -1$ instead of $\beta = 1$ is analogous, except the trunk will be on the right rather than the left side.

\subsubsection{Generalized totally asymmetric $\SLE_\kappa(\rho)$ processes}
\label{subsubsec::ayssm_sle_k}

We now fix $\kappa \in (0,4]$ (note that we include $\kappa=4$ here) and $\rho \in (-2-(\kappa/2),(\kappa/2)-4)$. Note that this second condition in fact implies that $\kappa > 2$ as otherwise, this interval of possible values for $\rho$ would be empty.  
This time, we choose
\[ \rho' = -\frac{\kappa'}{4}(\rho +2 + (\kappa/2)) = -\frac{\kappa'}{4}(\rho+2) -2.\]
Let $h$ be a GFF on $\h$ with boundary conditions given by~$\lambda'$ on~$\R_-$ and $-\lambda'(1+\rho')$ on~$\R_+$ and let~$\eta'$ be the counterflow line of~$h$ from~$0$ to~$\infty$.  Note that $\eta'$ is an $\SLE_{\kappa'}(\rho')$ process with a single force point located at $0^+$.  Suppose that we have fixed $\eta'$ and let $U$ be a component of $\h \setminus \eta'$ which is either surrounded by $\eta'$ with a clockwise orientation or is to the right of $\eta'$ and let $\varphi \colon U \to \h$ be a conformal map which takes the first point visited by $\eta'$ on $\partial U$ to $0$ and any other point to $\infty$.  We then add $c = -\lambda' -\lambda(1+\rho)$ to the field and draw in a branching flow line starting from $0$ and targeted at every point on $\varphi(\partial U \cap \eta') = \varphi(\partial U \setminus \partial \h)$.  

The argument of Proposition~\ref{prop:locfinitek} implies that there exists a continuous path $\gamma$ which follows the loops generated by the branching flow lines in the order in which they are visited by $\eta'$. Exactly as in the previous case, the key to determining the law of $\gamma$ is to describe its evolution when it is not hitting $\eta'$. This is then done word for word as in the previous case, and leads to the following:
\begin{itemize}
 \item This path $\gamma$ is an $\SLE_\kappa^{\beta}(\rho)$ process for $\beta=1$, as the loops it makes away from its trunk are always to the right of the trunk.
 \item It is almost surely a continuous curve.
 \item It is almost surely determined by the GFF $h$.
 \item The joint law of $(h,\gamma)$ is characterized by the form of the conditional law of $h$ given $\gamma|_{[0,\tau]}$ for each $\gamma$-stopping time $\tau$. 
 \end{itemize}
 We note that the difference between the boundary data to the right and to the left of $0$ for the GFF used to construct $\gamma$ is given by $\lambda(2+\rho)$, which is the same difference as in the case of $\SLE_{\kappa}(\rho)$ for $\rho > -2$.  The construction in the case that $\beta = -1$ instead of $\beta = 1$ is analogous, except the trunk will be on the right rather than the left side.

\begin{remark}
It is interesting to remark that the critical value $\rho=\kappa/2-4$ is where the $\SLE_\kappa(\rho)$ processes with $\rho < -2$ make the transition from the trunk phase to the light cone phase.  In this case, the branching flow lines in each of the components surrounded by $\eta'$ with a clockwise orientation have the law of an $\SLE_{\kappa}(\kappa-4;-2)$ with force points immediately to the left and right of the seed.  Thus these paths can be interpreted as tracing along the boundary of the corresponding component cut off by $\eta'$, but in the \emph{opposite} order in which the boundary was drawn by $\eta'$, i.e., counterclockwise.  This means that an $\SLE_\kappa(\kappa/2-4)$ process has the same range as an $\SLE_{\kappa'}(\rho')$ process with $\rho' = (\kappa'/2)-4$ but visits the points of its range in a different order.  This point is elaborated on further in \cite{ms2016lightcone}.

Note also that $\rho=-2-(\kappa/2)$ is the critical value at or below which the $\SLE_\kappa(\rho)$ processes are not defined.  Note that, in this case, the value of $\rho'$ for the $\SLE_{\kappa'}(\rho')$ trunk converges to $0$ as $\rho \downarrow -2-(\kappa/2)$.  Thus we can interpret an $\SLE_\kappa(\rho)$ with $\rho = -2-(\kappa/2)$ as corresponding exactly to an $\SLE_{\kappa'}$.

In summary, as $\rho$ varies from $-2-(\kappa/2)$ to $(\kappa/2)-4$, the law of an $\SLE_\kappa(\rho)$ process interpolates between the law of a curve which has the same range as an $\SLE_{\kappa'}$-type process but visits its points in a different order and the law of an $\SLE_{\kappa'}$ itself.
\end{remark}

\subsubsection{Other values of $\beta$}

The results of Sections~\ref{subsubsec::ayssm_sle_kp}, \ref{subsubsec::ayssm_sle_k} can be extended to the cases of the non-totally-asymmetric $\SLE_\kappa^\beta (\rho)$ and $\SLE_{\kappa'}^\beta (\rho')$ processes for the same range of values of $\rho$ and $\rho'$. Let us just state without proof the type of results that one obtains in this way.

Fix $\kappa' >4$, $\rho' \in (-2-(\kappa'/2),\kappa'/2-2)$, and choose $\rho_L',\rho_R' \in (-2,(\kappa'/2)-2)$ such that $\rho_L' + \rho_R' = \rho' + (\kappa'/2)-2$.  We then let
\[ \rho_L = - \frac{\kappa}{4}(\rho_L'+2) \quad\text{and}\quad  \rho_R = - \frac{\kappa}{4}(\rho_R'+2).\]
Let $h$ be a GFF on $\h$ with boundary conditions given by $-\lambda(1+\rho_L)$ on $\R_-$ and $\lambda(1+\rho_R)$ on $\R_+$ and let $\eta$ be the flow line of $h$ from $0$ to $\infty$.  Note that $\eta$ is an $\SLE_{\kappa}(\rho_L; \rho_R)$ process with force points located at $0^-$ and $0^+$.  Suppose that we have fixed $\eta$ and let $U$ be a component of $\h \setminus \eta$ which is to the right of $\eta$ and let $\varphi \colon U \to \h$ be a conformal map which takes the first (resp.\ last) point visited by $\eta$ on $\partial U$ to $0$ (resp.\ $\infty$).  We then add
\begin{equation}
\label{eqn:c_r_kp_on_k_new_def3}
c_R= -\lambda  + \lambda'(1+\rho_R')
\end{equation}
to the field and draw in a branching counterflow line starting from $0$ and targeted at every point on $\varphi(\partial U \cap \eta) = \varphi(\partial U \setminus \partial \h)$.  Similarly, if $U$ is to the left of $\eta$, then we add
\begin{equation}
\label{eqn:c_l_k_on_kp_new_def}
c_L = -\lambda'(1+\rho_L') + \lambda
\end{equation} 
to the field and draw in a branching counterflow line starting from $0$ and targeted at every point on $\varphi(\partial U \cap \eta) = \varphi(\partial U \setminus \partial \h)$.  As before, one can see that there exists a continuous path~$\gamma$ which follows the loops generated by the branching counterflow lines in the order in which they are visited by~$\eta$. Note that the law of this path is scale-invariant.

Again, in order to determine the law of~$\gamma$, we first describe its evolution when it is not hitting~$\eta$.  This amounts to showing that $\gamma([0,\tau])$ is a local set for $h$ for each stopping time $\tau$ and then identifying the conditional law of~$h$ given $\gamma|_{[0,\tau]}$.  The proof of the locality of $\gamma([0,\tau])$ is analogous to the proof of Proposition~\ref{prop:localityofnestedlooptracer}, and one can work out the boundary conditions of the GFF given $\gamma|_{[0, \tau]}$ which leads to a statement of the following type:

\begin{theorem}
\label{thm:generalized_kp_law}
For each $\beta \in [-1,1]$ and $\rho' \in (-2-(\kappa'/2),\kappa'/2-2) \setminus \{-2\}$ there exists $\rho_L',\rho_R' \in (-2,(\kappa'/2)-2)$ with $\rho_L' + \rho_R' = \rho' + (\kappa'/2)-2$ such that with $\rho_L = -{\kappa}(\rho_L'+2)/4$ and $\rho_R = -{\kappa}(\rho_R'+2)/4$, the path $\gamma$ constructed above with these parameters is distributed like a ``full''  $\SLE_{\kappa'}^\beta(\rho')$ process. In particular, if we parameterize it according to half-plane capacity seen from infinity (thereby excising all ``loops hidden from infinity''), we get exactly a path $\gamma^\#$ distributed like an $\SLE_{\kappa'}^\beta(\rho')$.  Similarly, for each $\mu \in \R$ there exists $\rho_L',\rho_R',\rho_L,\rho_R$ satisfying the above relations with $\rho' = -2$ such that the path constructed is distributed as an $\SLE_{\kappa'}^\mu(-2)$ (with $\beta=0$).  In particular, in all cases these processes are uniquely coupled with and determined by the GFF and are almost surely generated by continuous curves.  
\end{theorem}

We note that there are some admissible triples $\rho_L',\rho_R',\rho'$ in which either $\rho_L' \leq \kappa'/2-4$ , $\rho_R' \leq \kappa'/2-4$, or both.  If $\rho_L'$ (resp.\ $\rho_R'$) is in this range, then it means that the excursions that $\gamma$ makes to the left (resp.\ right) of its trunk terminate on the most recently visited point of the trunk before the excursion started.  This simply follows because the corresponding counterflow line completely fills the trunk.  The commutation argument is the same as described in the proof of Theorem~\ref{thm:duality1} and Theorem~\ref{thm:duality2}, except in this case the counterflow lines considered actually \emph{swallow} entirely the flow lines considered.  As explained in \cite{ms2012ig1}, in particular in the proofs of continuity of the $\SLE_{\kappa'}(\rho')$ processes with $\rho' \in (-2,\kappa'/2-4]$, this type of commutation naturally fits into the imaginary geometry framework.

We also note that in the case $\rho' \in (-2-(\kappa'/2),(\kappa'/2)-4]$, the path $\gamma$ does not branch into the components that it cuts off from $\infty$ so that there is no distinction between $\gamma$ and the path which arises by reparameterizing $\gamma$ according to capacity as seen from $\infty$.  This follows because if $\eta'$ is an $\SLE_{\kappa'}(\rho')$ process with $\rho' \leq (\kappa'/2)-4$ in $\h$ from $0$ to $\infty$ with a single boundary force point located at $x > 0$ then $\eta'$ first hits $[x,\infty)$ at $x$.

Using the same argument as in the proof of Theorem~\ref{thm:generalized_kp_law}, we also obtain the following.

\begin{theorem}
\label{thm:generalized_k_law}
For each $\beta \in [-1,1]$ and $\rho \in (-2-(\kappa/2),(\kappa/2)-4)$ there exists $\rho_L,\rho_R \in (-2,\kappa-4)$ with $\rho_L + \rho_R = \rho + (\kappa/2)-2$ such that with $\rho_L' = -\kappa'(\rho_L+2)/4$ and $\rho_R' = -\kappa'(\rho_R+2)/4$, the path $\gamma$ constructed above, switching the roles of flow and counterflow lines, with these parameters is distributed like an $\SLE_{\kappa}^\beta(\rho)$ process.  In particular, these processes are uniquely coupled with and determined by the GFF and are almost surely generated by continuous by curves.
\end{theorem}

As in the case $\beta=1$ mentioned above, the gap in the boundary data for the GFF in the couplings of Theorem~\ref{thm:generalized_kp_law} and Theorem~\ref{thm:generalized_k_law} is the same as in the coupling of $\SLE_{\kappa'}(\rho')$ for $\rho' > -2$ and $\SLE_\kappa(\rho)$ for $\rho > -2$ with the GFF.

\subsection{Brief discussion of $\CLE_\kappa$ and $\CLE_{\kappa'}$ couplings with the GFF}

Let us now very briefly comment on the couplings of conformal loop ensembles with the GFF that we constructed and used in the present paper: 
As $\BCLE$ is coupled with the GFF as a local set as either a boundary branching flow or counterflow line and we have shown that it is possible to obtain $\CLE_\kappa^\beta$ by appropriately iterating $\BCLE$s, we get that $\CLE_\kappa^\beta$ for all $\kappa \in (8/3,8) \setminus \{4\}$ is naturally coupled with the GFF as a local set.  Similarly, $\CLE_4^0$ is also coupled with the GFF as a local set and the coupling can be constructed by iterating $\BCLE_4$'s.  In particular, it is possible to read off the boundary data for the GFF along the 
CLE loop boundaries because we know the boundary data for the GFF along the loop boundaries for $\BCLE$. 
Some important aspects of the nature of this coupling of $\CLE$ with the GFF are different for $\kappa \in (8/3,4)$, $\kappa=4$, and $\kappa' \in (4,8)$: 

\begin {itemize}
\item In the case of $\CLE_4^0$, the boundary data for the field along the loops is only determined by the orientations of the loops.  That is, the boundary data in the inside of a loop is $2\lambda$ (resp.\ $-2\lambda$) if the loop has a counterclockwise (resp.\ clockwise) orientation.  
\item When $\kappa' \in (4,8)$, once we choose the starting point of the $\bSLE_{\kappa'}^\beta$ exploration tree and we know the loops of the $\CLE_{\kappa'}^\beta$, together with their orientations, then the boundary of the field along all the loop boundaries is 
fully determined.  Indeed, the exploration tree is a function of its root and of the entire oriented $\CLE_{\kappa'}^\beta$. 
Then, in this exploration tree, each loop $\gamma$ of the $\CLE_{\kappa'}$ will be traced starting from one of its points $w(\gamma)$ which is the first one that the trunk hits. 
The  GFF boundary data along this loop is then given by a constant plus a winding term which is counted starting from $w$ and will have a $ \pm 2\pi \chi$ jump at $w$ (depending on the orientation of the loop). 
So, the boundary data is  more complicated to work out than in the  $\CLE_4^0$ case, but each choice of a boundary point on $\partial D$ makes it possible to determine all the boundary conditions inside all the loops from the knowledge of $\CLE_{\kappa'}^\beta$. 
\item 
In the setting of $\CLE_\kappa^\beta$ for $\kappa \in (8/3,4)$, one needs to specify much more information to determine the boundary data of the field along the loop boundaries. 
The reason is that (as is shown in \cite{msw2016randomness}), the $\CLE_\kappa^\beta$ and the root of the tree are not sufficient to recover the trunk of the $\bSLE_\kappa^\beta$. 
So, neither the position of the special point $w$ on each $\CLE_\kappa^\beta$ loop where the boundary value jumps by $\pm 2 \pi \chi$ nor the actual additional constant that comes from 
the winding of the trunk before it hits that point can be worked out from the knowledge of the $\bSLE_{\kappa}^\beta$ and the root of the tree only. 
\end {itemize}

\subsection*{Acknowledgements}

J.M.\ thanks FIM at ETH Zurich for its hospitality on several occasions during which part of the work for this project was completed.  J.M.'s work was partially supported by DMS-1204894.  S.S.\ also thanks FIM at ETH Zurich for its hospitality during a visit during which part of this work was completed.  S.S.'s work was also partially supported by a grant from the Simons Foundation and NSF grant DMS-1209044.  W.W. acknowledges the support of  SNF grant 155922, and the support of the Clay Foundation. W.W.\ is part of the NCCR Swissmap. All authors acknowledge the hospitality of the Isaac Newton Institute where part of this work was completed.

\bibliographystyle{abbrv}
\bibliography{sle_kappa_rho}

\end{document}